\documentclass[a4paper]{amsart}
\usepackage{amsthm}
\usepackage{color}
\usepackage{amssymb,amsmath,amsthm,amstext,amsfonts}
\usepackage[dvips]{graphicx}
\usepackage{psfrag}
\usepackage{url}
\usepackage{amsfonts}
\usepackage{amsmath}
\usepackage{mathrsfs}
\usepackage{graphicx}
\usepackage{amsmath,amsthm,amssymb,amscd}
\usepackage[shortlabels]{enumitem}
\setlist[enumerate,1]{label=\textup{(\arabic*)},leftmargin=2.2em}
\usepackage{mathtools}
\usepackage[
bookmarks=true,
hypertexnames=false,
colorlinks=true,
linkcolor=blue,
urlcolor=blue,
bookmarks=true,
citecolor=blue,
linktocpage=true,
hyperindex=true
]{hyperref}

\usepackage{epsfig}
\makeatletter \@addtoreset{equation}{section} \makeatother

\renewcommand\thetable{\thesection.\@arabic\c@table}

\theoremstyle{plain}
\newtheorem{maintheorem}{Theorem}

\newtheorem{maincorollary}{Corollary}

\newtheorem{theorem}{Theorem}[section]
\newtheorem{proposition}[theorem]{Proposition}
\newtheorem{lemma}[theorem]{Lemma}
\newtheorem{corollary}[theorem]{Corollary}

\newtheorem{definition}[theorem]{Definition}

\newtheorem{remark}[theorem]{Remark}
\newtheorem{claim}[theorem]{Claim}
\newtheorem{example}[theorem]{Example}

\theoremstyle{remark}

\makeatletter
\@namedef{subjclassname@2020}{2020 Mathematics Subject Classification}
\makeatother

\allowdisplaybreaks[2]
\begin{document}

\title{Pesin's entropy formula on attractors with continuous invariant splittings}

\author{Xinyu Bai, Wanshan Lin, and Xueting Tian}

\address{Xinyu Bai, School of Mathematical Sciences,  Fudan University\\Shanghai 200433, People's Republic of China}
\email{25110180001@m.fudan.edu.cn}

\address{Wanshan Lin, School of Mathematical Sciences,  Fudan University\\Shanghai 200433, People's Republic of China}
\email{wanshanlin@fudan.edu.cn}

\address{Xueting Tian, School of Mathematical Sciences,  Fudan University\\Shanghai 200433, People's Republic of China}
\email{xuetingtian@fudan.edu.cn}

\begin{abstract}
We consider an attractor $\Lambda$ of a $C^1$ diffeomorphism
of $\mathbf M^d$ with a trapping neighborhood $U$ and an
extendable continuous invariant splitting
$T_{\overline U}\mathbf M^d=E\oplus F$.
Under quantitative asymptotic hyperbolicity assumptions, we
prove the existence of hyperbolic ergodic invariant probability
measures on $\Lambda$ that satisfy Pesin's entropy formula and
have negative Lyapunov exponents on $E$ and positive exponents
on $F$.
No dominated splitting is required on the attractor. Our results apply to mostly expanding partially hyperbolic systems
and to systems with separated Lyapunov exponents under
suitable statistical assumptions. For $C^{1+\alpha}$ diffeomorphisms, the resulting measures are SRB measures, and the corresponding sets of asymptotically
hyperbolic points are covered, modulo Lebesgue null sets, by
the basins of at most countably many such measures. We give
necessary and sufficient criteria for the existence of ergodic SRB measures with these exponent signs in terms
of derivative growth and orbit statistics, and characterize the relationship between
physical measures, strongly observable measures, physical-like measures and such SRB measures. We also construct examples with unique hyperbolic SRB measures of full support on attractors admitting no nontrivial dominated splitting, demonstrating the applicability of our criteria beyond results that require domination.
\end{abstract}
\keywords{metric entropy, Lyapunov exponents, SRB measures, basin coverage,
continuous invariant splitting}
\subjclass[2020]{37A35; 37B40; 37C40; 37D25}
\maketitle
\section{Introduction}
\subsection{Background}
The long-term statistical behavior of a smooth dynamical system is governed by the interaction between the local expansion or contraction of orbits and the global complexity measured by invariant probability measures. Metric entropy, introduced by Kolmogorov and Sinai, quantifies the average rate of information production of a dynamical system relative to an invariant measure. For an \(f\)-invariant probability measure \(\mu\), we denote the metric entropy of \(f\) by \(h_{\mu}(f)\). On the other hand, Lyapunov exponents describe the asymptotic exponential behavior of tangent vectors along typical orbits. The relationship between these two quantities is captured by two fundamental results in differentiable ergodic theory.

Ruelle's inequality provides a general upper bound for metric entropy in terms of the sum of positive Lyapunov exponents, counted with multiplicity. Pesin's entropy formula asserts that equality holds under suitable regularity and absolute-continuity assumptions. Thus, Pesin's entropy formula identifies the metric entropy of an invariant measure with the average expansion rate in the unstable directions. It provides a quantitative link between the measure-theoretic complexity of a system and the geometric behavior of its derivative, and has become a central tool in the study of nonuniform hyperbolicity.

The entropy formula is also closely related to the existence and regularity of invariant measures with geometric significance. The theory of SRB measures originated in the work of Sinai, Ruelle, and Bowen on uniformly hyperbolic systems, particularly Axiom A attractors \cite{B75,R76,S72}. Their statistical significance is closely related to the notion of a physical measure, namely an invariant probability measure whose basin has positive Lebesgue measure. Indeed, for a \(C^{1+\alpha}\) diffeomorphism, every ergodic hyperbolic SRB measure is physical; see, for example, \cite{Y02}. For a $C^{1+\alpha}$ diffeomorphism, an invariant probability measure
with at least one positive Lyapunov exponent almost everywhere is an
SRB measure if and only if it satisfies Pesin's entropy formula.
For the $C^2$ case, see Ledrappier--Young \cite[Theorem A]{LY85-1}.
At $C^{1+\alpha}$ regularity, the ergodic case follows from
Brown \cite[Theorem 1.4]{Brown22}, and the general case follows
by ergodic decomposition as in \cite[Section 6.2]{LY85-1}. This characterization shows that the entropy formula can serve as a criterion for detecting SRB measures and provides a natural route from entropy estimates to statistical properties. Beyond uniform hyperbolicity, the construction of SRB measures becomes considerably more delicate. However, the case of dominated splitting is relatively classical. A number of results have been obtained for attractors with dominated splitting and various forms of asymptotic hyperbolicity; see, for example, \cite{ABV00,BY25,CY16,CCE15}. Other works have considered partially hyperbolic attractors with more intricate center dynamics, including several one-dimensional center directions; see \cite{B25,CMY22,CY05,CYZ20}. 

Significant progress on Pesin's entropy formula in the \(C^1\) setting has been made under domination assumptions. Sun and Tian \cite{ST12} established an entropy lower bound for absolutely continuous invariant measures under an almost-everywhere orbitwise domination hypothesis, obtaining Pesin's entropy formula when the splitting separates the nonpositive and nonnegative Lyapunov exponents. Catsigeras, Cerminara, and Enrich \cite{CCE15} proved a corresponding lower bound for physical-like measures under a global dominated splitting, again yielding the entropy formula under suitable sign conditions. 

A central difficulty in extending Pesin's entropy formula arguments beyond the dominated setting is obtaining suitable control along finite orbit segments. A dominated splitting provides uniform exponential separation between the derivative growth rates on complementary invariant subbundles. This separation yields invariant cone fields after passing to a suitable iterate and facilitates the geometric estimates used in entropy arguments. A continuous invariant splitting over a compact set still guarantees uniform transversality, but continuity alone does not provide this dynamical separation. The challenge is therefore to obtain the geometric and volume estimates needed for an entropy lower bound without assuming domination. We address this issue through strengthened asymptotic hyperbolicity assumptions that provide the required control along suitable orbit segments.

This problem is related to Viana's conjecture that a smooth map
admits a physical measure if all its Lyapunov exponents are nonzero
at Lebesgue-almost every point \cite[p. 572]{Viana98}.

The setting without dominated splitting remains substantially less understood. Climenhaga, Dolgopyat, and Pesin \cite[Theorem A]{CDP16}
used measurable invariant cone families to construct SRB measures
for $C^{1+\alpha}$ diffeomorphisms with a positive-volume set of
effectively hyperbolic points. Burguet \cite{B24-2} proved that a $C^\infty$ surface diffeomorphism
$f$ admits an SRB measure if and only if
$\limsup_{n\to+\infty} \frac{1}{n}\log\|D_xf^n\|>0$
on a set of positive volume.
For a $C^1$ diffeomorphism, Liao and Yu \cite{LY25}
proved Pesin's entropy formula for every ergodic physical-like measure
whose Oseledets splitting $E^{cs}\oplus E^u$, with $\dim E^u=1$,
extends to a continuous $Df$-invariant splitting on a trapping
neighborhood of its support.
If $f$ is $C^{1+\alpha}$, each such measure is an SRB
measure. These results provide important evidence that suitable asymptotic hyperbolicity can compensate, at least in part, for the absence of domination.

\subsection{Motivations and contributions}

The purpose of this paper is to determine how derivative growth and orbit statistics can be used to construct invariant measures satisfying Pesin's entropy formula on attractors admitting an extendable continuous invariant splitting $E\oplus F$. For $C^r$ diffeomorphisms with $r>1$, we seek necessary and sufficient conditions for the existence of hyperbolic ergodic SRB measures with negative Lyapunov exponents on $E$ and positive exponents on $F$. We also investigate which sets of initial conditions are covered, up to Lebesgue null sets, by the basins of these measures, and how this family of SRB measures and its weak-star closure are related to physical measures, strongly observable measures, and physical-like measures.

For $C^1$ diffeomorphisms, we establish quantitative asymptotic hyperbolicity criteria that yield hyperbolic ergodic invariant probability measures satisfying Pesin's entropy formula, with negative Lyapunov exponents on $E$ and positive exponents on $F$. The splitting need not be dominated. The main difficulty is that continuity of the splitting alone gives no uniform separation between the growth rates along $E$ and $F$ over finite orbit segments. To link the orbit segments, we select a set of Pliss times of positive
lower density for a suitable iterate of the diffeomorphism such that, between any two
selected times, the product of the corresponding one-step domination
ratios decays exponentially with their separation. Averaged domination from time zero gives the required narrowing of the $F$-cones, while backward contraction at the selected times controls the geometry of cone-tangent discs. These estimates give the volume upper bound needed to establish the entropy lower bound.

For $C^r$ diffeomorphisms with $r>1$, the Ledrappier--Young entropy characterization identifies these measures as SRB measures. Furthermore, we prove that the basins of at most countably many hyperbolic ergodic SRB measures with these exponent signs cover the corresponding sets of asymptotically hyperbolic points modulo Lebesgue null sets. To prove coverage, we must rule out a positive-volume set of such points outside all these basins. We localize the measure construction to the invariant complement of the union of these basins. On hyperbolic preballs, the intermediate volume factors in the Jacobian product cancel, and backward contraction together with Hölder continuity of the derivative makes the remaining logarithmic errors summable. This gives bounded distortion without requiring Hölder continuity of \(F\). By the density estimate on hyperbolic preballs and bounded distortion, the relative leaf volume of the complement tends to one along a suitable sequence of terminal balls of fixed radius. On these same balls, the stable holonomy estimate on a fixed Pesin block gives a uniform positive lower bound for the proportion occupied by an SRB basin.

For $r>1$, Corollary \ref{cor:srb-statistical-characterization}
shows that a hyperbolic ergodic SRB measure with negative exponents
on $E$ and positive exponents on $F$ exists if and only if the
coordinated quantitative asymptotic hyperbolic condition holds on a set of positive
volume. Under an additional contraction assumption on $E$,
Corollary \ref{cor:srb-expansion-characterization} identifies
the subfamily detected by the mostly expanding property along $F$.
Both corollaries give equivalent criteria in terms of empirical
limits, physical measures, and strongly observable measures. For an invariant measure satisfying the relevant positivity condition,
physicality is equivalent to membership in the family of hyperbolic
ergodic SRB measures with these exponent signs, and strong
observability is equivalent to membership in its weak-star closure. If the corresponding quantitative asymptotic hyperbolic condition holds for
Lebesgue-almost every point of $U$, ordinary and strong
observability coincide for measures supported on $\Lambda$,
and their common class is the weak-star closure of hyperbolic ergodic SRB measures with these exponent signs.

We also construct attractors that admit no nontrivial dominated splitting and each support a unique SRB measure. This measure is hyperbolic and has full support on the attractor, so criteria requiring a dominated splitting on the support of the measure do not apply directly to these examples.

\subsection{Setting and notation}
Throughout this paper, $(X,\mathrm d)$ denotes a compact metric space, and  $T:X\to X$ denotes a homeomorphism.
Let $\mathcal M(X),\:\mathcal M(X,T),\:\mathcal M^{erg}(X,T)$ be the sets of all Borel probability measures, $T$-invariant Borel probability measures, $T$-invariant and ergodic Borel probability measures on $X$, respectively. We write $\mathbb N=\{0,1,2,\ldots\}$ and $\mathbb N^+=\{1,2,\ldots\}$. Let $\mathbb R^{+}$ be the set of all non-negative real numbers. 
Let $C(X)$ be the set of all continuous functions on $X$. We define a metric $\rho$ inducing the weak-star topology on $\mathcal M(X)$ by
$$\rho(.,.):\mathcal M(X)\times\mathcal M(X)\to\mathbb R^+,\:(\mu,\nu)\mapsto\sum_{i\in\mathbb N^+}\frac{1}{2^i}\left|\int g_id\mu-\int g_id\nu\right|,$$ where $\left\{g_i\right\}_{i\in\mathbb N^+}$ is a dense subset of the unit sphere of $C(X)$. With the metric $\rho$, both $\mathcal M(X)\:\:\mathrm{and}\:\:\mathcal M(X,T)$ are compact metric spaces. We say that a sequence of Borel probability measures $\left\{\mu_n\right\}_{n\in\mathbb N^+}$ converges to a measure $\mu$, if $\lim_{n\to+\infty}\rho(\mu_n,\mu)=0.$ For a linear isomorphism $A:V\to W$ between finite-dimensional
normed spaces, we write
$\left\|A\right\|:=\sup_{\|v\|=1}\|Av\|,\:m(A):=\inf_{\|v\|=1}\|Av\|=\|A^{-1}\|^{-1}.$

Let $\mathbf M^d$ be a compact $d$-dimensional Riemannian manifold. For any real number $r\ge1$, $\mathrm{Diff}^r(\mathbf M^d)$ denotes the set of all $C^r$ diffeomorphisms on $\mathbf M^d$, and for any $\nu\in\mathcal M(\mathbf M^d)$,
$\mathrm{Diff}^r_{\nu}(\mathbf M^d)$ denotes the set of all $C^r$ diffeomorphisms on $\mathbf M^d$ with $\nu\in\mathcal M(\mathbf M^d,f)$. We write $\mathrm{Diff}(\mathbf M^d):=\mathrm{Diff}^1(\mathbf M^d)$, $\mathrm{Diff}_{\nu}(\mathbf M^d):=\mathrm{Diff}^1_{\nu}(\mathbf M^d)$, for $\nu\in\mathcal M(\mathbf M^d)$. Let $\mathrm{Leb}$ denote the normalized Riemannian
volume on $\mathbf M^d$. Lebesgue-measurable sets are
understood with respect to the completion of this measure. For a finite measure $\nu$ and a measurable set $A$, we denote by
$\nu_A$ the restriction of $\nu$ to $A$, defined by
$\nu_A(B)=\nu(B\cap A)$ for every measurable set $B$.
When $\nu(A)>0$, we write $\nu|_A:=\frac{\nu_A}{\nu(A)}$ for the
normalized restriction. For a disc or a leaf $S$, let $\mathrm{Leb}_S$ denote the
Riemannian volume induced on $S$. For a linear subspace $V$
of a tangent space, let $\mathrm{Leb}_V$ denote the Euclidean
volume induced by the Riemannian inner product.

For $f\in\mathrm{Diff}^r(\mathbf M^d)$,  let $\Gamma_{f}$ be the set of points that are regular in the sense of Oseledets. For $x\in\Gamma_{f}$, let
$\lambda_1(f,x)>\lambda_2(f,x)>\cdots>\lambda_{\ell(f,x)}(f,x)$
denote its distinct Lyapunov exponents and let
$\mathrm T_x\mathbf M^d=E^{1}_{f}(x)\oplus\cdots\oplus E^{\ell(f,x)}_{f}(x)$
be the corresponding decomposition of its tangent space. If $\nu\in \mathcal{M}^{erg}(\mathbf M^d,f)$, the functions $x\mapsto \ell(f,x),\:x\mapsto\lambda_i(f,x),\:x\mapsto\mathrm{dim}(E^i_{f}(x))$ 
are constant for $\nu$-$\mathrm{a.e.}\:x\in\mathbf M^d$. 

Let $f\in\mathrm{Diff}^r(\mathbf M^d)\:(r>1)$. Assume  $\mu\in\mathcal{M}(\mathbf M^d,f)$ is such that $(f,\mu)$ has a positive Lyapunov exponent a.e. For $x\in\mathbf M^d$, let
$$W^{u}(f,x)=\left\{y\in \mathbf M^d:\limsup_{n\to+\infty}\frac{1}{n}\log{\mathrm d(f^{-n}x,f^{-n}y)}< 0\right\}.$$
By the unstable manifold theorem
(see \cite{P76,RS80}), for $\mu$-almost every $x$,
$W^u(f,x)$ is a $C^r$ immersed submanifold of $\mathbf M^d$ tangent at $x$ to $\bigoplus_{\lambda_i(f,x)>0} E_f^i(x).$ It is called the unstable manifold at $x$.
A measurable partition $\xi$ of $\mathbf M^d$ is said to be subordinate to $W^{u}$, if for $\mu$-a.e.$\:\:x\in\mathbf M^d$,\: the element of $\xi$ that contains $x$ is contained in $W^{u}(f,x)$, and contains an open neighborhood of $x$ in $W^{u}(f,x)$. By Rokhlin's disintegration theorem, 
$\mu=\int\mu_x^{\xi}d\mu(x),$
where $\mu_x^{\xi}$ is the conditional measure  with respect to $\xi(x)\subset W^{u}(f,x)$.
\begin{definition}
\cite{LY85-1}
Let $f\in\mathrm{Diff}^{r}(\mathbf M^d)\:(r>1)$. The measure $\mu\in\mathcal{M}(\mathbf M^d,f)$ is called an SRB measure, if  $(f,\mu)$ has a positive Lyapunov exponent a.e. and, for any measurable partition $\xi$ subordinate to $W^u$, one has $\mu_x^{\xi}\ll \mathrm{Leb}_{W^u(f,x)},$ for $\mu$-a.e.\:$x\in\mathbf M^d$.
\end{definition}
For a homeomorphism $T:\mathbf M^d\to\mathbf M^d$, a point
$x\in\mathbf M^d$, and $n\ge1$, define the empirical measure
$\mathcal E_n^T(x):=\frac1n\sum_{j=0}^{n-1}\delta_{T^j x}.$
The set of accumulation points of the sequence $\{\mathcal E^T_n(x)\}_{n\ge 1}$ in $\mathcal M(\mathbf M^d)$ is denoted by
\[
p\omega_{T}(x):=
\left\{
\nu\in\mathcal M(\mathbf M^d):
\text{there exists a sequence } n_k\to+\infty
\text{ such that }
\mathcal E^T_{n_k}(x)\to\nu
\right\},
\]
and $B_T(\mu)=\{x\in\mathbf M^d:\mathcal E_n^T(x)\to\mu,\:\mathrm{as}\:n\to+\infty\}$.
Since $\mathbf M^d$ is compact, the space $\mathcal M(\mathbf M^d)$ is compact with respect to the metric $\rho(.,.)$. Consequently, $p\omega_{T}(x)$ is a nonempty compact subset of $\mathcal M(\mathbf M^d)$ for every $x\in\mathbf M^d$. When $T=f$, we abbreviate $p\omega_f(x)$ to $p\omega(x)$. Let $\mathrm{SRB}_{f},\:\mathrm{SRB}_{f}^{eh}$ denote, respectively, the set of all SRB measures and the set of all hyperbolic ergodic SRB measures. Moreover, $$\mathrm{PE}_{f}:=\left\{\mu\in\mathcal{M}(\mathbf M^d,f):h_{\mu}(f)=\int\sum_{\lambda_i(f,x)>0}\lambda_i(f,x)\mathrm{dim}\:E^{i}_{f}(x)d\mu(x),\:\lambda_1(f,x)>0,\:\mu\mathrm{-a.e.}\:x\in\mathbf M^d\right\},$$
$$\mathrm{PE}_{f}^{eh}:=\left\{\mu\in\mathcal{M}^{erg}(\mathbf M^d,f)\cap\mathrm{PE}_f:\:\mu\:\mathrm{is\:hyperbolic}\right\},$$
and
$$\mathrm{P}_f:=\left\{\mu\in\mathcal M(\mathbf M^d,f):\mathrm{Leb}\left\{x\in\mathbf M^d:\mathcal E_n^f(x)\to\mu,\:n\to+\infty\right\}>0\right\}.$$
By \cite{Brown22}, \cite{LY85-1}, and \cite{Y02}, for $f\in\mathrm{Diff}^r(\mathbf M^d)$ with $r>1$, we have 
$\mathrm{PE}_{f}^{eh}=\mathrm{SRB}_{f}^{eh}\subset \mathrm{P}_{f}.$

We say that a compact $f$-invariant set $\Lambda$ is an attractor, if there exists an open neighborhood $U$ of $\Lambda$, such that
$f(\overline U)\subset U,$
$\Lambda=\cap_{n\in\mathbb N}f^n(U).$
We say that $f$ admits an extendable continuous invariant splitting on the attractor $\Lambda$, if there exists a continuous  $f$-invariant splitting $T_{\overline U}\mathbf M^d=E\oplus F$ on $\overline U$. 
Throughout, both subbundles are assumed to be nontrivial; that is,
$
1\le \dim E(x),\,\dim F(x)\le d-1,
\:\text{for every }x\in\overline U.$
If $T_\Lambda\mathbf M^d$ admits a dominated splitting,
then, after shrinking the trapping neighborhood $U$,
the splitting extends to a continuous $Df$-invariant
splitting on $\overline U$.
This follows from the backward graph transform for
the slower bundle and a fundamental-domain construction
for its complement, using boundary interpolation and
forward transport; domination ensures continuity at
$\Lambda$.
For this construction in a partially hyperbolic setting, see the proof
of Lemma 6.4 in Crovisier--Yang--Zhang \cite{CYZ20}.

 The empirical limits of
points of $U$ are supported on $\Lambda$. For $x\in\overline U$, let
$$\lambda^{\min}_{F}(x):=\liminf_{n\to+\infty}\frac{1}{n}\log m(Df^n|_{F(x)}),\:\lambda^{\max}_{E}(x):=\limsup_{n\to+\infty}\frac{1}{n}\log\left\|Df^n|_{E(x)}\right\|,$$
\[
 \mathscr S_{E,F}(\Lambda,f)=
 \{\mu\in\mathrm{SRB}^{eh}_f:\mathrm{supp}\mu\subset\Lambda,
                \ \lambda_E^{\max}(z)<0<\lambda_F^{\min}(z),\ \mu\text{-a.e.}\:z\in\Lambda\},
 \:
 \mathcal B_{E,F}=\bigcup_{\mu\in\mathscr S_{E,F}(\Lambda,f)}B_f(\mu).
\]
When $r>1$, $\mathscr S_{E,F}(\Lambda,f)$ is at most countable: its members have disjoint
positive-volume basins in $U$.

\subsection{Main results}
Let \(g \in \mathrm{Diff}(\mathbf M^d)\), and let \(E\) and \(F\) be two continuous
\(Dg\)-invariant subbundles of \(T_{\overline U}\mathbf M^d\) over the closed neighborhood of the attractor $\Lambda$
with
$T_{\overline U}\mathbf M^d=E\oplus F.$ Let $R_{\mathrm{inj}}$ denote the injectivity radius of $\mathbf M^d$. Let $\rho=
\min\left\{
\frac{R_{\mathrm{inj}}}{4},\frac{\mathrm d(\Lambda,\partial U)}{8}
\right\}$. We use the convention $\mathrm d(A,\emptyset)=+\infty$ for nonempty $A$.
Fix an open neighborhood $U_1$ of $\Lambda$ such that
$U_1\subset U\setminus B(\partial U,4\rho)
  \:\text{and}\:g(U_1)\subset U_1.$
For $0<\lambda<1$, define
\[
\Delta_n(\lambda,g,E,F)
:=
\left\{
x\in U_1:
\prod_{j=0}^{k-1}
\frac{\left\|Dg\big|_{E(g^jx)}\right\|}{m\left(Dg\big|_{F(g^jx)}\right)}
\le \lambda^k,\:
\text{for any } 1\le k\le n
\right\}.
\]
If \(x\in \Delta_n(\lambda,g,E,F)\), then the segment
$(x,g^nx)$
is \(\lambda\)-average dominated with respect to the splitting $T_{\overline U}\mathbf M^d=E\oplus F$. Let 
$\Delta(\lambda,g,E,F):=\cap_{n\in\mathbb N^+}\Delta_n(\lambda,g,E,F),$
$$\Delta(a,\lambda,f^p,E,F):=\left\{x\in \Delta(\lambda,f^p,E,F):\liminf_{n\to+\infty}\frac{1}{np}\sum_{j=0}^{n-1}\log m(Df^p|_{F(f^{jp}x)})>a\right\},$$
and
\begin{equation*}
\begin{aligned}
&\mathrm{HP}(a,\lambda,f^p,E,F):=\\
&\left\{x\in \Delta(\lambda,f^p,E,F):\liminf_{n\to+\infty}\frac{1}{np}\sum_{j=0}^{n-1}\min\left\{\log m(Df^p|_{F(f^{jp}x)})-ap,ap+\log m(Df^{-p}|_{E(f^{(j+1)p}x)})\right\}>0\right\},
\end{aligned}
\end{equation*}
for $p\in\mathbb N^+$ and $a\ge 0$.
\begin{maintheorem}\label{thm:A}
Let $f\in\mathrm{Diff}(\mathbf M^d)$, and let $\Lambda$ be an attractor admitting an extendable continuous $f$-invariant splitting  $T_{\overline U}\mathbf M^d=E\oplus F$.\\
(1) Suppose that, for some $a\ge 0$, $0<\lambda<1$, and $p\in\mathbb N^+$, 
$$\mathrm{Leb}(\Delta(a,\lambda,f^p,E,F))>0,$$
$$\left\|Df^p|_{E(x)}\right\|\le e^{ap},\:\forall\:x\in\Lambda.$$
Then there exists $\mu\in\mathcal M^{erg}(\Lambda,f)$ such that
\[
\begin{gathered}
\lambda_F^{\min}(x)>a\ge\lambda_E^{\max}(x)
\quad\text{for $\mu$-a.e. }x\in\Lambda,\\
\text{and}\quad
h_\mu(f)\ge
\int_\Lambda \log\bigl|\det(Df|_{F(x)})\bigr|\,d\mu(x).
\end{gathered}
\]
(2) Suppose that, for some $a\ge 0$, $0<\lambda<1$, and $p\in\mathbb N^+$, 
$$\mathrm{Leb}(\mathrm{HP}(a,\lambda,f^p,E,F))>0.$$
Then there exists $\mu\in\mathcal M^{erg}(\Lambda,f)$ such that
\[
\begin{gathered}
\lambda_F^{\min}(x)>a>\lambda_E^{\max}(x)
\quad\text{for $\mu$-a.e. }x\in\Lambda,\\
\text{and}\quad
h_\mu(f)\ge
\int_\Lambda \log\bigl|\det(Df|_{F(x)})\bigr|\,d\mu(x).
\end{gathered}
\]
\end{maintheorem}
The main contribution of our theorem, compared with previous
results establishing the existence of SRB measures via Pesin's
entropy formula, is that it relaxes the domination requirement
for diffeomorphisms on arbitrary compact Riemannian manifolds.
More precisely, we require averaged domination on a set of
positive Lebesgue measure, together with the quantitative
expansion and contraction conditions specified in
Theorem \ref{thm:A}, without assuming a dominated splitting
on the whole attractor.
When $a=0$, the measures in both parts satisfy Pesin's
entropy formula and, for $C^{1+\alpha}$ diffeomorphisms,
are SRB measures. In this case, the measure obtained in part (2)
is hyperbolic. The proofs of Theorems \ref{thm:B},
\ref{thm:D}, and \ref{thm:E} show that the required
averaged domination can be derived from their respective
hypotheses.
Thus, in these settings, averaged domination need not be
imposed as an independent assumption.
In comparison with \cite{ABV00,BY25,CY16,CCE15}, our results
allow nondominated splittings under the quantitative
asymptotic hypotheses specified here.
Compared with the surface setting of \cite{B24-2}, our
framework accommodates higher-dimensional manifolds and
bundles, while requiring an extendable continuous invariant
splitting and the asymptotic conditions stated above.
In what follows, we present several applications of the main
theorem, including settings in which no dominated splitting
is present.

In terms of the strategy of proof and core techniques, we draw on the ideas of \cite{B24-1,B24-2}: invariant measures are constructed by Cesàro averages of Lebesgue-type reference measures along suitable times of iterates, and a Gibbs-type estimate yields the required entropy lower bound. Our proof replaces the reparametrization argument used in \cite{B24-2} with a selection of hyperbolic times
of positive density and geometric estimates on discs tangent to the cone field. More precisely, we first obtain hyperbolic times of positive density and use them to generate the required time sequences. To prove the Gibbs-like property, we combine the narrowing of the
$F$-cones along the selected average-dominated orbit segments with
backward contraction on discs at hyperbolic times
(cf. \cite{CMZ25,CMY18}). 

For $a\ge0$ and $p\ge1$ define the continuous superadditive block functions
\begin{equation}\label{eq:Phi-def}
 \Phi_{p,a}(x)=\min\{\log m(Df^p|_{F(x)})-ap,
 \ ap-\log\|Df^p|_{E(x)}\|\}.
\end{equation}
 Write
\begin{equation}\label{eq:good-H}
 \mathcal H_a=\left\{x\in U:
 \limsup_{p\to+\infty}\liminf_{n\to+\infty}
 \frac1{np}\sum_{j=0}^{n-1}\Phi_{p,a}(f^{jp}x)>0\right\}.
\end{equation}
\begin{maintheorem}\label{thm:B}
Let $f\in\mathrm{Diff}(\mathbf M^d)$, and let $\Lambda$ be an attractor admitting an extendable continuous $f$-invariant splitting  $T_{\overline U}\mathbf M^d=E\oplus F$. Suppose that, for some
$a\ge0$, $\mathrm{Leb}(\mathcal H_a)>0$.
\begin{enumerate}
\item Then there exists $\mu\in\mathcal M^{erg}(\Lambda,f)$ such that
\[
\begin{gathered}
\lambda_F^{\min}(x)>a>\lambda_E^{\max}(x)
\quad\text{for $\mu$-a.e. }x\in\Lambda,\\
\text{and}\quad
h_\mu(f)\ge
\int_\Lambda \log\bigl|\det(Df|_{F(x)})\bigr|\,d\mu(x).
\end{gathered}
\]
\item If $a=0$, the measure $\mu$ in part (1) belongs to
$\mathrm{PE}^{eh}_f$. If $a=0$ and $f\in\mathrm{Diff}^{r}(\mathbf M^d)$
for some $r>1$, then $\mu\in\mathrm{SRB}^{eh}_f$.
\item If $a=0$ and $f\in\mathrm{Diff}^r(\mathbf M^d)$ for some $r>1$,
the at most countable family
$\mathscr S_{E,F}(\Lambda,f)$ satisfies
\[
 \mathrm{Leb}\bigl(\mathcal H_0\mathbin{\triangle}
(\mathcal B_{E,F}\cap U)\bigr)=0.
\]
\end{enumerate}
\end{maintheorem}

When $a=0$, we say that $f$ satisfies the
mostly coordinated asymptotic hyperbolic condition
on $\Lambda$, with respect to the extendable continuous
invariant splitting
$T_{\overline U}\mathbf M^d=E\oplus F$, if
$\mathrm{Leb}(\mathcal H_0)>0$.
This condition coordinates asymptotic expansion along
$F$ and asymptotic contraction along $E$ through the
block averages defining $\mathcal H_0$. Under this condition, parts (1)--(2) yield a hyperbolic
ergodic measure supported on $\Lambda$ and satisfying
Pesin's entropy formula, which is an SRB measure if
$f\in\mathrm{Diff}^r(\mathbf M^d)$ for some $r>1$.

Under the additional regularity
$f\in\mathrm{Diff}^r(\mathbf M^d)$ for some $r>1$,
part (3) gives a statistical description of the entire
set $\mathcal H_0$: modulo a Lebesgue null set, it
coincides with the union, inside $U$, of the basins
of the hyperbolic ergodic SRB measures in $\mathscr S_{E,F}(\Lambda,f)$.
In particular, for Leb-a.e. $x\in\mathcal H_0$, the empirical measures converge to one of these SRB
measures. In these settings,
the covering family may be infinite, as shown in
Proposition \ref{prop:infinitely-many-basins}.

For general $a\ge0$, Theorem \ref{thm:E} provides
a criterion for the coordinated block-average hypothesis
of Theorem \ref{thm:B} in terms of Lyapunov exponents
and orbit statistics.
Each of its alternative hypotheses implies
$\mathrm{Leb}(\mathcal H_a)>0$.
Its conclusions therefore follow from the corresponding
parts of Theorem \ref{thm:B}.

A dominated splitting with a uniformly expanding subbundle
does not, by itself, guarantee the existence of an invariant
probability measure supported on the attractor and satisfying
Pesin's entropy formula. The following example, based on Burguet \cite[p. 1022]{B24-2},
illustrates this point.

\begin{example}\label{Example 1.2}
Let $h:S\to S$ be the time-one map of a smooth dissipative
figure-eight flow. The only $h$-invariant probability measure supported on its
figure-eight attractor $K$ is $\delta_o$, where $o$ is a hyperbolic
saddle with exponents $\alpha$ and $-\beta$, with $\beta>\alpha>0$.
Let $B:\mathbb T^2\to\mathbb T^2$ be a sufficiently large iterate
of a linear Anosov automorphism, with exponents $\pm\ell$.
Then $f=h\times B$ has the attractor
$\Lambda=K\times\mathbb T^2$ and the smooth dominated splitting
$E=TS\oplus E_B^s,\:F=E_B^u$,
with $F$ uniformly expanding. Every invariant probability on $\Lambda$ is
$\mu=\delta_o\times\nu$, with $\nu$ invariant under $B$, and
$h_\mu(f)=h_\nu(B)\le\ell<\alpha+\ell
=\int\sum_{\lambda_s(f,x)>0}
\lambda_s(f,x)\dim E_f^s(x)\,d\mu(x).$
Thus no invariant probability on $\Lambda$ satisfies
Pesin's entropy formula.
 Moreover, every empirical limit of a point in the trapping neighborhood
is supported on $\{o\}\times\mathbb T^2$, where
$\lambda_F^{\min}=\ell$ and $\lambda_E^{\max}=\alpha$.
By Lemma \ref{stat:identity}, the limsup defining $\mathcal H_0$
therefore equals $-\alpha$ throughout the
trapping neighborhood. Hence $\mathcal H_0=\emptyset$.
\end{example}

We say that a Borel set $\Gamma\subset U$ has total probability
with respect to $f$ on $\Lambda$ if
$\nu(\Gamma)=1,
\:\text{for every }\nu\in\mathcal M(\Lambda,f).$

\begin{maintheorem}\label{thm:D}
Let $f\in\mathrm{Diff}(\mathbf M^d)$, and let $\Lambda$ be an attractor admitting an extendable continuous $f$-invariant splitting  $T_{\overline U}\mathbf M^d=E\oplus F$. Suppose that, for some $a\ge 0$, 
$$\mathrm{Leb}\left\{x\in U:\liminf_{n\to+\infty}\frac{1}{n}\sum_{j=0}^{n-1}\log m(Df|_{F(f^{j}x)})>a\right\}>0.$$ 
Assume also that there exists a Borel set $\Gamma\subset U$
of total probability with respect to $f$ on $\Lambda$ such that
\[
\limsup_{n\to+\infty}
\frac{1}{n}\log\left\|Df^n|_{E(x)}\right\|<a,
\:\text{for every }x\in\Gamma.
\]
\begin{enumerate}
\item Then there exists $\mu\in\mathcal M^{erg}(\Lambda,f)$ such that
\[
\begin{gathered}
\lambda_F^{\min}(x)>a>\lambda_E^{\max}(x)
\quad\text{for $\mu$-a.e. }x\in\Lambda,\\
\text{and}\quad
h_\mu(f)\ge
\int_\Lambda \log\bigl|\det(Df|_{F(x)})\bigr|\,d\mu(x).
\end{gathered}
\]
\item If $a=0$, the measure $\mu$ in part (1) belongs to
$\mathrm{PE}^{eh}_f$. If $a=0$ and $f\in\mathrm{Diff}^{r}(\mathbf M^d)$
for some $r>1$, then $\mu\in\mathrm{SRB}^{eh}_f$.
\item If $a=0$ and $f\in\mathrm{Diff}^r(\mathbf M^d)$ for some $r>1$, put
\[
 H_C=\left\{x\in U:\liminf_{n\to+\infty}\frac1n
      \sum_{j=0}^{n-1}\log m(Df|_{F(f^jx)})>0\right\}.
\]
Then $\mathscr S_{E,F}(\Lambda,f)$ is at most countable and
\[
\mathrm {Leb}\left(
H_C\:\triangle
\bigcup_{\substack{
\mu\in\mathscr S_{E,F}(\Lambda,f)\\
\int\log m(Df|_F)\,d\mu>0
}}
\bigl(B_f(\mu)\cap U\bigr)
\right)=0.
\]
\end{enumerate}
\end{maintheorem}
We denote $\mathcal B_{E,F}^+:=\bigcup_{\substack{
\mu\in\mathscr S_{E,F}(\Lambda,f),\:
\int\log m(Df|_F)\,d\mu>0
}}B_f(\mu)$. Compared with the results in \cite{CMY18}, our result is formulated from the perspective of Pesin’s entropy formula: it applies when $a>0$ and does not require the extendable continuous invariant splitting to be 
Hölder continuous. Their approach uses the definition of SRB measures to establish existence directly.  In this paper, we say that \(F\) is mostly expanding if
\[
\mathrm{Leb}\left\{x\in U:
\liminf_{n\to+\infty}\frac1n
\sum_{j=0}^{n-1}\log m(Df|_{F(f^jx)})>0
\right\}>0.
\]
For partially hyperbolic attractors with $E$ uniformly
contracting and $F$ mostly expanding in the above sense, Theorem \ref{thm:D} yields hyperbolic ergodic
measures satisfying Pesin's entropy formula.
For $C^{1+\alpha}$ diffeomorphisms, these measures are
SRB measures, recovering the corresponding existence
conclusions in \cite{ABV00,BY25} under these common
assumptions. 

Under the assumptions of part \textup{(3)}, $H_C$
coincides, modulo a Lebesgue null set, with the union
in $U$ of the basins of those
$\mu\in\mathscr S_{E,F}(\Lambda,f)$ satisfying
$\int\log m(Df|_F)\,d\mu>0$.
Thus, for Lebesgue-almost every $x\in H_C$, the
empirical measures converge to one of these measures.
This family may be infinite, as explained in
Remark \ref{Remark 7.8}.

Example \ref{Example 1.2} shows that neither of the two hypotheses of Theorem \ref{thm:D} can be omitted. For the splitting
$E=E_B^s$, $F=TS\oplus E_B^u$, the bundle $E$ is uniformly
contracting, but the mostly expanding property on $F$ fails. For the original splitting
$E=TS\oplus E_B^s$, $F=E_B^u$, the bundle $F$ is uniformly
expanding, whereas the maximal Lyapunov exponent along $E$
is positive almost everywhere with respect to every invariant
probability measure on $\Lambda$. Nevertheless, $\Lambda$
supports no invariant probability measure satisfying
Pesin's entropy formula.

For any two nonempty compact subsets $A,B\subset\mathcal M(\mathbf M^d)$, define their Hausdorff distance by
$d_H(A,B):=
\max\left\{
\sup_{\nu_1\in A}\rho(\nu_1,B),
\sup_{\nu_2\in B}\rho(\nu_2,A)
\right\},$
where $\rho(\nu,B):=\inf_{\nu'\in B}\rho(\nu,\nu').$ Write $B_\rho(\mu,\varepsilon)$ for the open ball of radius
$\varepsilon$ about $\mu$ in the space of probability measures.
An $f$-invariant probability measure $\mu$ is called physical-like
if, for every $\varepsilon>0$,
$\mathrm{Leb}\{x\in\mathbf M^d:
  p\omega(x)\cap B_\rho(\mu,\varepsilon)\ne\emptyset\}>0.$
It is called strongly observable if, for every $\varepsilon>0$,
$\mathrm {Leb}\{x\in\mathbf M^d:
  p\omega(x)\subset B_\rho(\mu,\varepsilon)\}>0.$
 Denote by $\mathcal O_f^H$ the set of strongly observable measures with respect to $f$. Denote by $\mathcal O_f$ the set of all physical-like measures with respect to $f$. Our definition of strong observability agrees with
\cite[Definition 2.1]{CZ13}. The set \(\mathcal O_f^H\) is weak\(^*\)-compact and satisfies
$\mathrm{P}_f\subseteq\mathcal O_f^H\subseteq\mathcal O_f.$

For every continuous map on a compact manifold, $\mathcal O_f$
is nonempty by \cite[Theorem 1.3]{CE11}, whereas $\mathcal O_f^H$
may be empty \cite[Proposition 2(C)]{CZ13}. Nevertheless, \(\mathcal O_f^H\) is nonempty whenever the empirical measures converge on a set of positive Lebesgue measure, even if the limiting measure varies with the initial point. The nonemptiness of \(\mathcal O_f^H\) admits a characterization in terms of the asymptotic oscillation of empirical measures. Let \(D(x):=\mathrm{diam}_{\rho}p\omega(x)\). Then \(\mathcal O_f^H\neq\emptyset\) if and only if \(\mathrm{Leb}\{x:D(x)<\varepsilon\}>0\) for every \(\varepsilon>0\), or equivalently, \(\mathrm{ess\,inf}_{\mathrm{Leb}}D=0\). This is also equivalent to requiring that, for every \(\varepsilon>0\), there exists \(N\ge1\) such that
\[
\mathrm{Leb}\left\{x:\sup_{n,k\ge N}\rho\bigl(\mathcal E_n^f(x),\mathcal E_k^f(x)\bigr)<\varepsilon\right\}>0.
\]
Both \(N\) and the positive-measure set may depend on \(\varepsilon\); convergence of the empirical measures on a fixed set of positive measure is not required.

For $x\in U$, the limit inferior of the Birkhoff averages of
$\varphi=\log m(Df|_F)$ along the orbit of $x$ equals
$\min_{\tau\in p\omega(x)}\int_\Lambda\varphi\,d\tau$.
For $r>1$, under the total-probability contraction hypothesis on $E$,
this identity and the basin coverage in part \textup{(3)} of
Theorem \ref{thm:D} yield necessary and sufficient conditions for
the existence of a measure $\mu\in\mathscr S_{E,F}(\Lambda,f)$
with $\int_\Lambda\varphi\,d\mu>0$.
The following corollary also describes the physical and strongly
observable measures for which this integral is positive. 
When \(H_C\) has full volume in \(U\), the corollary identifies all physical-like measures supported on \(\Lambda\).

\begin{maincorollary}
\label{cor:srb-expansion-characterization}
Let $f\in\mathrm{Diff}^r(\mathbf M^d)$ for some $r>1$, and let $\Lambda$ be an attractor admitting an extendable continuous $f$-invariant splitting  $T_{\overline U}\mathbf M^d=E\oplus F$. Assume that there exists a Borel set $\Gamma\subset U$
of total probability with respect to $f$ on $\Lambda$ such that
\[
\limsup_{n\to+\infty}
\frac{1}{n}\log\left\|Df^n|_{E(x)}\right\|<0,
\:\text{for every }x\in\Gamma.
\]
Let $X_F^+:=\left\{\nu\in\mathcal M(\mathbf M^d,f):\int_{\Lambda}\varphi d\nu>0\right\}$. The following conditions are equivalent:
\begin{enumerate}
\item
$\mathscr S_{E,F}(\Lambda,f)\cap X_F^+\ne\emptyset$;
\item
$\mathrm {Leb}(H_C)>0$;
\item
$\mathcal O_f^H\cap X_F^+\ne\emptyset$;
\item
$\mathrm P_f\cap X_F^+\ne\emptyset$;
\item
$\mathrm {Leb}\left\{
x\in U:
p\omega(x)\subset X_F^+
\right\}>0.$
\end{enumerate}
Moreover,
\[
\mathrm P_f\cap X_{F}^+=\mathscr S_{E,F}(\Lambda,f)\cap X_{F}^+,
\:
\mathcal O_f^H\cap X_{F}^+
=
\overline{\mathscr S_{E,F}(\Lambda,f)}^{\,w^*}
\cap X_F^+.\]
If $\mathrm {Leb}(U\setminus H_C)=0$, then
$\mathscr S_{E,F}(\Lambda,f)\subset X_F^+$ and
\[
\mathcal O_f\cap\mathcal M(\Lambda,f)
=
\mathcal O_f^H\cap\mathcal M(\Lambda,f)
=
\overline{\mathscr S_{E,F}(\Lambda,f)}^{\,w^*}.
\]

\end{maincorollary}

We now turn to statistical criteria for the coordinated block-average
condition in Theorem \ref{thm:B}. 
On the common Lyapunov-regular set in
$\Lambda$ put $\gamma_a=\min\{\lambda_F^{\min}-a,a-\lambda_E^{\max}\}$, and set
$\gamma_a(x)=0$ at all remaining points of $\Lambda$. For $\tau\in\mathcal M(\mathbf M^d,f)$, write
\[
 G_a=\{x\in\Lambda:x\text{ is Lyapunov regular and }\gamma_a(x)>0\},
 \:\mathcal G_a(\tau)=\int_{\Lambda}\gamma_a\,d\tau.
\]

\begin{maintheorem}
\label{thm:E}
Let $f\in\mathrm{Diff}(\mathbf M^d)$, and let $\Lambda$ be an attractor admitting an extendable continuous $f$-invariant splitting  $T_{\overline U}\mathbf M^d=E\oplus F$. Fix $a\ge 0$. Suppose that at least
one of the following conditions holds:
\begin{enumerate}[label=\textup{(\roman*)}]
\item There is an invariant probability $\nu\ll\mathrm{Leb}$ on
$\mathbf M^d$ such that $\nu(G_a)>0$.
\item There exists $\nu\in\mathcal O^H_{f}$ with $\mathcal G_a(\nu)>0$.
\item $\mathrm{Leb}(\left\{x\in U:\:\forall\:m\in p\omega(x),\:\mathcal G_a(m)>0\right\})>0.$
\item  $\mathcal G_a(\tau)>0$ for every
$\tau\in\mathcal O_{f}$.
\end{enumerate}
Then the following conclusions hold.
\begin{enumerate}
\item $\mathrm{Leb}(\mathcal H_a)>0$, and there exists
$\mu\in\mathcal M^{erg}(\Lambda,f)$ with
\[
\begin{gathered}
\lambda_F^{\min}(x)>a>\lambda_E^{\max}(x)
\quad\text{for $\mu$-a.e. }x\in\Lambda,\\
\text{and}\quad
h_\mu(f)\ge
\int_\Lambda \log\bigl|\det(Df|_{F(x)})\bigr|\,d\mu(x).
\end{gathered}
\]
\item If $a=0$, the measure in (1) belongs to $\mathrm{PE}^{eh}_f$;
if $a=0$ and $f\in\mathrm{Diff}^r(\mathbf M^d)$ for some $r>1$, it belongs to
$\mathrm{SRB}^{eh}_f$.
\item If $a=0$ and $f\in\mathrm{Diff}^r(\mathbf M^d)$ for some $r>1$, then the at most countable family $\mathscr S_{E,F}(\Lambda,f)$ satisfies
$$\mathrm{Leb}(\mathcal H_0\triangle(\mathcal B_{E,F}\cap U))=0.$$
Furthermore, in case \textup{(i)}, $\nu(G_0\setminus\mathcal B_{E,F})=0$;
in case \textup{(iii)}, $\mathrm{Leb}(\left\{x\in U:\:\forall\:m\in p\omega(x),\:\mathcal G_0(m)>0\right\}\setminus\mathcal B_{E,F})=0$;
in case \textup{(iv)}, $\mathrm{Leb}(U\setminus\mathcal B_{E,F})=0$.
\end{enumerate}
\end{maintheorem}

Each of the four statistical hypotheses implies
$\mathrm{Leb}(\mathcal H_a)>0$, without assuming ergodicity of
the reference measures or convergence of the empirical measures.
At $a=0$, this yields a hyperbolic ergodic measure satisfying
Pesin's entropy formula for $C^1$ diffeomorphisms. For
$C^{1+\alpha}$ diffeomorphisms, the basins of the hyperbolic
ergodic SRB measures in $\mathscr S_{E,F}(\Lambda,f)$ cover
$\mathcal H_0$ modulo a Lebesgue null set.
Proposition \ref{prop:infinitely-many-basins} gives an example
satisfying both \textup{(ii)} and \textup{(iii)} at $a=0$ in which
no finite family of basins covers $\mathcal H_0$ modulo a Lebesgue
null set. These criteria use physical-like and strongly observable measures
in the sense of \cite{CE11,CZ13}. Related entropy estimates and
criteria for physical measures under dominated splittings appear
in \cite{CCE15,MWC24}. Here the variational approach of
Burguet--Yang \cite{BY25} is adapted to an extendable continuous
invariant splitting that need not be dominated. For $C^\infty$
diffeomorphisms, Ben Ovadia and Burguet \cite[Theorem C]{BOB26}
obtain basin coverage under different growth conditions without
a prescribed splitting.

The coordinated block-average hypothesis of Theorem \ref{thm:B} and the statistical conditions of Theorem \ref{thm:E} are related by the identity in Lemma \ref{stat:identity}. Combined with the basin-coverage theorem, this identity yields a necessary and sufficient criterion for the existence of hyperbolic ergodic SRB measures whose exponents are negative on $E$ and positive on $F$. The following corollary also describes the physical and strongly observable measures with positive coordinated mean gap and, under full-volume basin coverage in $U$, identifies all physical-like measures supported on $\Lambda$.

\begin{maincorollary}
\label{cor:srb-statistical-characterization}
Let $f\in\mathrm{Diff}^r(\mathbf M^d)$ for some $r>1$, and let $\Lambda$ be an attractor admitting an extendable continuous $f$-invariant splitting  $T_{\overline U}\mathbf M^d=E\oplus F$.
The following conditions are equivalent:
\begin{enumerate}
\item $\mathscr S_{E,F}(\Lambda,f)\ne\emptyset$;
\item $\mathrm {Leb}(\mathcal H_0)>0$;
\item for some $p\ge1$, $\mathrm {Leb}\{x\in U:\liminf_{n\to+\infty}
 \frac1{np}\sum_{j=0}^{n-1}\Phi_{p,0}(f^{jp}x)>0\}>0$;
\item $\mathcal O_f^H\cap\left\{\nu\in\mathcal M(\mathbf M^d,f):\:\mathcal G_0(\nu)>0\right\}\ne\emptyset$;
\item $\mathrm P_f\cap\left\{\nu\in\mathcal M(\mathbf M^d,f):\:\mathcal G_0(\nu)>0\right\}\ne\emptyset$;
\item $\mathrm {Leb}
 \{x\in U:\:\forall\:m\in p\omega(x),\:\mathcal G_0(m)>0\}>0.$
\end{enumerate}
Moreover,
\[
 \mathrm P_f\cap\left\{\nu\in\mathcal M(\mathbf M^d,f):\:\mathcal G_0(\nu)>0\right\}=\mathscr S_{E,F}(\Lambda,f),
\]
\[
 \mathcal O_f^H\cap\left\{\nu\in\mathcal M(\mathbf M^d,f):\:\mathcal G_0(\nu)>0\right\}
 =
 \overline{\mathscr S_{E,F}(\Lambda,f)}^{\,w^*}\cap\left\{\nu\in\mathcal M(\mathbf M^d,f):\:\mathcal G_0(\nu)>0\right\}.
\]
If $\mathrm {Leb}(U\setminus \mathcal H_0)=0$, then
\[
 \mathcal O_f\cap\mathcal M(\Lambda,f)
 =
 \mathcal O_f^H\cap\mathcal M(\Lambda,f)
 =
 \overline{\mathscr S_{E,F}(\Lambda,f)}^{\,w^*}.
\]
\end{maincorollary}

We construct an uncountable family of
diffeomorphisms that are pairwise not locally
$C^1$-conjugate near their attractors, each possessing an attractor that admits an extendable continuous invariant splitting satisfying the mostly coordinated asymptotic hyperbolic condition, but no dominated splitting. Applying our main theorem, we establish the existence of SRB measures supported on these attractors. These examples show that the hypotheses of Theorem \ref{thm:B} can hold on attractors with no nontrivial
dominated splitting. Their full-support SRB measures cannot be obtained by applying a criterion
requiring a dominated splitting on a compact invariant subset that supports such a measure.

\begin{definition}\label{Definition 1.3}
Let $\mathbf M$ and $\mathbf N$ be compact Riemannian manifolds. Suppose that
$F:V\to \mathbf M
\:\:\text{and}\:\:
G:W\to \mathbf N$
are $C^1$ maps defined on open neighborhoods $V\subset \mathbf M$ and
$W\subset \mathbf N$, respectively. Let $\Lambda\subset V$ and
$\Gamma\subset W$ be invariant sets, meaning that
$F(\Lambda)\subset \Lambda
\:\:\text{and}\:\:
G(\Gamma)\subset \Gamma.$
We say that the dynamical systems
$F|_{\Lambda}:\Lambda\to\Lambda
\:\:\text{and}\:\:
G|_{\Gamma}:\Gamma\to\Gamma$
are locally $C^1$-conjugate near $\Lambda$ and $\Gamma$, if there exist
open neighborhoods $V'\subset V$ and $V''\subset W$ of $\Lambda$ and $\Gamma$,
respectively, and a $C^1$ diffeomorphism
$\Phi:V'\to V''$
such that
$\Phi(\Lambda)=\Gamma$
and there exists an open neighborhood $\mathcal U\subset V'$ of $\Lambda$ for
which
$\Phi\circ F=G\circ\Phi
\:\text{on }\:\mathcal U.$
\end{definition}

\begin{maintheorem}\label{thm:F}
For every $\alpha\in(1,2)$ there is a $C^2$ diffeomorphism $f_\alpha$ of
$\mathbb T^4$ and a trapping neighborhood $U_\alpha$ with
$f_\alpha(\overline U_\alpha)\subset U_\alpha$ such that the attractor
$\Lambda_\alpha=\bigcap_{n\ge0}f_\alpha^n(\overline U_\alpha)$
has the following properties:
\begin{enumerate}
\item $T_{\overline U_\alpha}\mathbb T^4=E_\alpha\oplus F_\alpha$ is a
continuous invariant splitting into two-dimensional bundles.
\item $\Lambda_\alpha$ contains a fixed point $p_\alpha$ with
$Df_\alpha(p_\alpha)=I$. In particular, $\Lambda_\alpha$ admits no nontrivial
dominated splitting of any index.
\item
The pointwise inequality in the hypothesis of
Theorem \ref{thm:B} with $a=0$ holds for
Lebesgue-almost every $x\in U_\alpha$.
\item There is a unique SRB probability measure of $f_\alpha$
supported on $\Lambda_\alpha$.
It is hyperbolic and ergodic, and its support is exactly
$\Lambda_\alpha$.
\item No proper compact invariant subset of $\Lambda_\alpha$ carries
an SRB measure of $f_\alpha$. Also $f_\alpha|_{\Lambda_\alpha}$ is
topologically transitive and has no nonempty proper relative subattractor.
\item For $\alpha\ne\alpha'\in (1,2)$, the systems
$f_\alpha|_{\Lambda_\alpha}$ and
$f_{\alpha'}|_{\Lambda_{\alpha'}}$
are not locally $C^1$-conjugate in the sense of
Definition \ref{Definition 1.3}.
\end{enumerate}
\end{maintheorem}

The absence of any nontrivial dominated splitting on
$\Lambda_\alpha$ prevents the direct application of
SRB-existence criteria requiring such a splitting there.
Moreover, every SRB measure of $f_\alpha$
supported on $\Lambda_\alpha$ has support exactly
$\Lambda_\alpha$.
Hence no proper compact invariant subset of
$\Lambda_\alpha$ supports an SRB measure of $f_\alpha$,
so restricting to such a subset does not provide a way
to apply those criteria.

The construction admits a smooth projection that semiconjugates
$f_\alpha$ on its trapping neighborhood to the base map $g_\alpha$.
The latter is a noninvertible covering whose derivative is the
identity at its indifferent fixed point. Hence it admits no nontrivial
dominated splitting on the whole base. Its absolutely continuous
invariant probability $\nu$ is obtained from intermittent-map estimates. This factor gives an alternative proof of the existence of an SRB
measure on $\Lambda_\alpha$. By Liu's theorem \cite[Theorem 1.1]{Liu98},
$\nu$ satisfies Pesin's entropy formula. Its lift $\mu_\nu$ to the
natural extension has the same metric entropy, and the derivative
comparison in Section 8.5 shows that $\mu_\nu$ and $\nu$ have the
same positive Lyapunov exponents, counted with multiplicity.
Thus $\mu_\nu$ also satisfies Pesin's entropy formula and is an SRB
measure by \cite[Theorem A]{LY85-1}.
 Therefore we do not claim that the
example has no factor argument whatsoever, or that no conceivable
indirect use of another theorem could prove the existence of the SRB measures.

\subsection{Organization of the proofs}
Section 2 collects the geometric tools and entropy estimates
used throughout the paper. Section 3 proves Theorem \ref{thm:A} by
constructing invariant measures, proving the Gibbs-like property and establishing the required
entropy lower bound. Sections 4 and 5 prove
parts \textup{(1)}--\textup{(2)} of Theorems \ref{thm:B} and \ref{thm:D},
respectively. Section 6 establishes a statistical
characterization of the coordinated asymptotic hyperbolic
condition and deduces Theorem \ref{thm:E} from Theorem \ref{thm:B}.
Building on the measure construction and entropy estimates
of Section 3, Section 7 uses full hyperbolic preballs and
the absolute continuity of stable holonomies to prove
part \textup{(3)} of Theorems \ref{thm:B} and \ref{thm:D}, thereby completing
the basin-coverage conclusions of Theorem \ref{thm:E}.
The same section discusses why the covering family need not be
finite and proves Corollaries \ref{cor:srb-expansion-characterization}
and \ref{cor:srb-statistical-characterization}.
Finally, Section 8 constructs the examples and proves
Theorem \ref{thm:F}.

\section{Preliminaries}\label{2}
\subsection{Metric entropy}
Let $\mu\in\mathcal{M}(X)$. Given a finite measurable partition $\mathcal Q=\left\{A_1,\cdots,A_k\right\}$ of $X$, we define the entropy of $\mathcal Q$ by
$$H_{\mu}(\mathcal Q)=-\sum_{1\le i\le k}\mu(A_i)\log{\mu(A_i)},$$
with the convention $0\log 0=0$.
If $\mu\in\mathcal M(X,T)$, metric entropy of $T$ with respect to $\mu$ and $\mathcal Q$ is defined by
$$h_{\mu}(T,\mathcal Q)=\lim_{n\to+\infty}\frac{1}{n}H_{\mu}(\bigvee_{i=0}^{n-1}T^{-i}\mathcal Q)$$
where
$\bigvee_{i=0}^{n-1}T^{-i}\mathcal Q:=\left\{B_1\cap\cdots\cap T^{-(n-1)}B_{n}:B_i\in\mathcal Q,\:1\le i\le n\right\}.$
Metric entropy of $T$ with respect to $\mu$ is given by
$h_{\mu}(T)=\sup_{\mathcal Q}h_{\mu}(T,\mathcal Q),$
where $\mathcal Q$ ranges over all finite measurable partitions of $X$.

\begin{lemma}\cite[Lemma 8]{B24-1}\label{Lemma 2.1}
Let $\mu\in\mathcal{M}(X)$ and $E$ be a nonempty finite subset of $\mathbb N$. For any finite partition $Q$ of $X$, for any $m\in\mathbb{N^+}$, set $Q^E:=\vee_{j\in E}T^{-j}Q,\:Q^m:=\vee_{j=0}^{m-1}T^{-j}Q$. Then one has
$$\frac{1}{\#E}H_{\mu}(Q^E)\le\frac{1}{m}H_{\frac{1}{\#E}\sum_{j\in E}T^j_{*}\mu}(Q^m)+6m\frac{\#\bigl((E+1)\triangle E\bigr)}{\#E}\log\#Q.$$
\end{lemma}

\subsection{Pliss lemma}
\begin{lemma}[Pliss lemma]\label{Pliss}
Let $A>0$ and let $\lambda_1<\lambda_2$. Then there exist constants
$N=N(A,\lambda_1,\lambda_2)\in\mathbb N
\:\text{and}\:
\rho=\rho(A,\lambda_1,\lambda_2)\in(0,1),$
such that for any $n\ge N$, and any finite sequence
$b_1,\ldots,b_n\in\mathbb R$ satisfying
$|b_i|\le A\:\text{for all }1\le i\le n
\:\text{and}\:
\frac{1}{n}\sum_{i=1}^{n}b_i\le \lambda_1,$
there exist 
$1\le j_1<\cdots<j_k\le n$
with $k>\rho n$, such that for any $1\le m\le k$ and any
$1\le l\le n-j_m+1$, one has
\[
\sum_{i=j_m}^{j_m+l-1} b_i\le \lambda_2 l.
\]
\end{lemma}
The indices $j_1,\ldots,j_k$ are referred to as Pliss times. The corresponding backward form follows by applying the
lemma to the reversed finite sequence.

\subsection{The hyperbolic times}
Our approach in this section is inspired by \cite{BY25,CMY18}. Assume that $g\in\mathrm{Diff}(\mathbf M^d)$ with an extendable continuous invariant splitting $T_{\overline U}\mathbf M^d=E\oplus F$ on the attractor $\Lambda$. The rank of $F$ is locally constant and is preserved by $g$.
Since it takes only finitely many values, we may restrict to a
trapping neighborhood on which $F$ has constant rank and which
intersects the positive-Lebesgue-measure set used in the
construction with positive Lebesgue measure.
This restriction does not affect the existence of the measures below. We denote the constant rank of $F$ by $\dim F$.

We say that \(n\) is a $\lambda_1$-hyperbolic time for \(x\) with respect to \(g\) and $F$, if for every \(1\le k\le n\),
\[
\prod_{j=n-k+1}^{n}
\left\|Dg^{-1}\big|_{F(g^jx)}\right\|
\le \lambda_1^k,\:0<\lambda_1<1.
\] 
 For $N\in\mathbb N^+$, let \(A\subset \mathbb{R}^N\) be a linear subspace and let \(v\in \mathbb{R}^N\).
We define the distance from \(v\) to \(A\) by
$\mathrm{dist}(v,A)
=
\min_{w\in A}\|v-w\|.$
Equivalently, \(\mathrm{dist}(v,A)\) is the distance from \(v\) to its
orthogonal projection onto \(A\). For two nonzero linear subspaces \(A,B\subset \mathbb{R}^N\), we define their distance by
$$\mathrm{dist}(A,B)
=
\max\left\{
\max_{\substack{u\in A\\ \|u\|=1}}
\mathrm{dist}(u,B),
\,
\max_{\substack{v\in B\\ \|v\|=1}}
\mathrm{dist}(v,A)
\right\}.$$
For subspaces of a common tangent space $T_x\mathbf M^d$,
we use the same definition with the Riemannian norm.

In the intermediate constructions below, we consider compact
sets contained in an embedded $C^1$ submanifold $S$.
Their tangency refers to the tangent spaces of $S$.
No regularity of the relative boundary is assumed. For such a compact set $\mathcal K$, its intrinsic distance
is the induced length distance
\[
d_{\mathcal K}(u,v)
=
\inf\left\{
\ell(\gamma):
\gamma:[0,1]\to\mathcal K
\:\text{is absolutely continuous},
\ \gamma(0)=u,\ \gamma(1)=v
\right\},
\]
where the infimum is $+\infty$ if no such curve exists.
The standard $C^1$ graph discs used subsequently are obtained
from these compact sets by Lemma \ref{lem:uniform-cone-graph}.

We use the following notation throughout the paper. For a compact set $\mathcal D\subset S$, we write
$\mathrm{Int}\mathcal D= \mathrm{Int}S\cap\mathrm{Int}_S\mathcal D$
and $\partial\mathcal D=\mathcal D\setminus\mathrm{Int}\mathcal D$,
where $\mathrm{Int}S$ denotes the manifold interior of $S$
and $\mathrm{Int}_S$ denotes relative interior in $S$. We call such a compact
set $\mathcal D\subset S$ a compact disc of size $r>0$ centered
at $x\in\mathrm{Int}\mathcal D$ if
$\mathrm d_{\mathcal D}(x,y)=r$ for every $y\in\partial\mathcal D$ and $\mathrm d_{\mathcal D}(x,y)\le r$ for every $y\in\mathrm{Int}\mathcal D$. Let $\mathcal D\subset \mathbf M^d$ be an embedded submanifold. For $x\in\mathbf M^d$, $r>0,\:n\in\mathbb N^+$, let
$$B(x,r):=\left\{y\in\mathbf M^d:\mathrm d(x,y)\le r\right\},\:B_{\mathcal D}(x,r):=\left\{y\in\mathcal D:\mathrm d_{\mathcal D}(x,y)\le r\right\},\:x\in\mathcal D.$$
Here $\mathrm d$ denotes the Riemannian distance on $\mathbf M^d$,
and $\mathrm d_{\mathcal D}$ denotes the intrinsic distance on
$\mathcal D$ induced by the restricted Riemannian metric.
Moreover, $B_n(x,r)=\left\{y\in\mathbf M^d:\mathrm d(g^jx,g^jy)\le r,\:0\le j<n\right\}$.
For two subsets $A,B\subset \mathbf M^d$, denote
$\mathrm{d}(A,B)=\inf_{x\in A,y\in B}\mathrm{d}(x,y),$
$\mathrm{d}(x,A)=\mathrm{d}(\left\{x\right\},A).$
Similarly, we can define $B(A,r)$ and $B_{\mathcal D}(B,r)$, for $r>0$, $A\subset\mathbf M^d$ and $B\subset\mathcal D$. 

Let \(0<\theta<1\). The \(F\)-cone field of width \(\theta\) is the family
$\mathcal C_\theta^F=\{\mathcal C_\theta^F(x)\}_{x\in U},$
where
$\mathcal C_\theta^F(x)
=
\left\{
v=v_E+v_F\in E(x)\oplus F(x):
\|v_E\|\le \theta\|v_F\|
\right\}.$
The \(E\)-cone field \(\mathcal C_\theta^E\) of width \(\theta\) is defined analogously. An embedded submanifold $\mathcal D\subset U$ is tangent to
$\mathcal C_\theta^F$ if $\dim\mathcal D=\dim F$ and
$T_x\mathcal D\subset\mathcal C_\theta^F(x)$ for every
$x\in\mathcal D$.
 More generally, for $x\in U$, we call the orbit segment
$(x,g^n x)$ $\lambda$-average dominated if it satisfies the
product inequalities defining $\Delta_n(\lambda,g,E,F)$ in
Section 1.4, whether or not $x\in U_1$.

\begin{lemma}\cite[Lemma 3.4]{CMY18}\label{Lemma 2.3}
For any $0<\gamma_1<\gamma_2<1$, there exists $r_0:=r_0(\gamma_1,\gamma_2)>0$ such that, for any $x\in U$ with a $\gamma_1$-average dominated segment $(x,g^nx)$ and any $y\in U\cap B_n(x,r_0)$, the segment $(y,g^ny)$ is $\gamma_2$-average dominated.     
\end{lemma}
\begin{lemma}\label{Lemma 2.4}\cite[Lemma 3.5]{CMY18}
For $0<\lambda<1$ and $0<\theta<1$, if $(x,g^nx)$ is $\lambda$-average dominated, one has
$$D_{x}g^{j}\mathcal C^{F}_{\theta}(x)\subset \mathcal C^F_{\lambda^j\theta}(g^{j}x),\:\mathrm{for\:any\:}1\le j\le n.$$
\end{lemma}
\begin{lemma}\cite[Lemma 3.6]{CMY18}\label{Lemma 2.5}
Let $0<\lambda<1$. Then there exists $r_0>0$ such that, for any $\theta\in(0,1)$, if $(x,g^nx)$ is $\lambda$-average dominated and $\mathcal D$ is an embedded submanifold containing $x$ tangent to the $F$-cone field
$\mathcal C_\theta^F$ with
$
g^j\mathcal D\subset B(g^jx,r_0),
\:\text{for all}\:0\le j\le n-1,
$
then for any $1\le j\le n$,\\

$\bullet\:
g^{j}\mathcal D \text{ is tangent to } \mathcal C_{\lambda^{\frac{j}{2}}\theta}^F$,\\

$
\bullet\:\mathrm{dist}\left(
F(g^{j}y),T_{g^{j}(y)}g^{j}\mathcal D
\right)
\le
\lambda^{\frac j2}\theta,\:\text{for any}\:y\in \mathcal D.
$
\end{lemma}
At hyperbolic times, one can find a compact subset
$\mathcal D\subset\widetilde{\mathcal D}$ whose image
$g^n\mathcal D$ is a compact disc of uniform intrinsic size. Inspired by \cite[Lemma 3.7]{CMY18}, we have the following lemma. 

\begin{lemma}\label{Lemma 2.6}
Let $g\in\mathrm{Diff}^1(\mathbf M^d)$ and assume that $T_{\overline U}\mathbf M^d=E\oplus F$
is an extendable continuous invariant splitting on the attractor $\Lambda$. Let
$
0<\lambda_1<\lambda_2<1.$
Then there exist constants
$0<r_1<r_0,\:\theta_0>0,$
such that the following holds. Let $0<r<r_1$, $0<\theta\le \theta_0$, and suppose that
$(x,g^nx)$ is $\lambda_1$-average dominated and that $n$
is a $\lambda_2$-hyperbolic time for $x$ with respect to $g$ and $F$. Let $0\le j<n$. Assume that $\widetilde{\mathcal D}$ is a compact embedded submanifold with $x\in\mathrm{Int}\widetilde{\mathcal D}$ tangent to $\mathcal C_\theta^F$, satisfying 
\[
\mathrm d_{g^j\widetilde{\mathcal D}}
(g^jx,\partial g^j\widetilde{\mathcal D})
>
\lambda_2^{\frac {n-j}2}r
\]
and
\[
g^s\widetilde{\mathcal D}
\subset B(g^sx,r_1),
\:\forall\:0\le s\le j.
\]
Then there exists a compact subset $\mathcal D$ of the embedded $C^1$ submanifold $\widetilde{\mathcal D}$ such that $x\in\mathrm{Int}\mathcal D$, $\mathcal D$ is
tangent to $\mathcal C^F_{\theta}$, and the following properties hold.\\

$
\bullet\:g^k\mathcal D\subset B(g^kx,r),
\:\mathrm{for\:any\:}0\le k\le n$;\\

$\bullet\:g^n\mathcal D$ is a compact disc
centered at $g^nx$ with size $r$;\\

$\bullet\:$for every $y\in\mathcal D$ and $0\le k\le n$,
\[
\mathrm d_{g^{n-k}\mathcal D}
(g^{n-k}x,g^{n-k}y)
\le
\lambda_2^{\frac k2}
\mathrm d_{g^n\mathcal D}(g^nx,g^ny).
\]
\end{lemma}

\begin{proof}
Write
$x_i=g^ix,
\:
q=\sqrt{\lambda_2}.$
Choose $\rho$ such that
$q<\rho<1,
\:
\vartheta:=\frac{\lambda_2}{\rho}<q.$

By uniform continuity of $Dg$ and the invariant bundles,
we may choose $r_1>0$ and $\theta_0>0$ such that
\begin{equation*}
m(D_zg|_H)
\ge
\rho\,m(Dg|_{F(x_i)}),
\end{equation*}
whenever $\mathrm d(z,x_i)\le r_1$
and $H\subset \mathcal C^F_{\theta_0}(z)$ is a
$\dim F$-dimensional subspace. 
For $0\le s\le j$, set
\[
\mathcal D_s=g^s\widetilde{\mathcal D}.
\]
For $j<s\le n$, define
\[
\mathcal D_s
=
\overline{\mathrm{Comp}_{x_s}
\bigl(g(\mathcal D_{s-1})\cap B(x_s,r_1)\bigr)}.
\]
 Since $g^{-k}\mathcal D_s\subset B(x_{s-k},r_1)$ for all $0\le k\le s$, Lemma \ref{Lemma 2.5} implies that $\mathcal D_{s}$ is tangent to $\mathcal C^F_{\theta}$. By the definitions above, for any $j<s\le n$, one has 
$\partial\mathcal D_s\subset g(\partial\mathcal D_{s-1})\cup(g\mathcal D_{s-1}\cap \partial B(x_s,r_1)),$
and $g^{s_1-s_2}\mathcal D_{s_2}\subset\mathcal D_{s_1},$
for all $0\le s_1\le s_2\le n$. For $0\le s\le n$, let
$m_s=m(Dg|_{F(x_s)}).$
If $\gamma\subset\mathcal D_n$ is an absolutely continuous curve, then $g^{s-n}\gamma\subset\mathcal D_s$, for every $0\le s\le n$, and hence
$\ell(\gamma)
\ge
\rho^{n-s}
\prod_{t=s}^{n-1}m_t\,
\ell(g^{s-n}\gamma).$
Since $n$ is a $\lambda_2$-hyperbolic time,
\[
\prod_{t=s}^{n-1}m_t
=
\prod_{t=s}^{n-1}
\frac{1}{
\left\|Dg^{-1}|_{F(x_{t+1})}\right\|}
\ge
\lambda_2^{-(n-s)}.
\]
Hence
$\ell(g^{s-n}\gamma)
\le
\vartheta^{n-s}
\ell(\gamma).$

Put $\delta_j=
\mathrm d_{\mathcal D_j}(x_j,\partial\mathcal D_j).$
By assumption, $\delta_j>q^{n-j}r.$
Choose $R$ such that
$r<R<r_1,
\:
\vartheta^{n-j}R<\delta_j.$
We claim that
\[
\mathrm d_{\mathcal D_n}(x_n,\partial\mathcal D_n)>r.
\]
Otherwise, there would exist an absolutely continuous curve $\gamma\subset \mathcal D_n$ from $x_n$ to
$\partial\mathcal D_n$ with $\ell(\gamma)<R$. If there exists $z\in\gamma$ such that
$g^{s-n}z\in\partial B(x_s,r_1)$ for some $j<s\le n$, then 
$r_1
\le
\mathrm d(x_s,g^{s-n}z)
\le
\ell(g^{s-n}\gamma)
\le
\vartheta^{n-s}\ell(\gamma)
<R<r_1,$
a contradiction. Let $z=\gamma(1)\in\partial \mathcal D_n$. The boundary inclusion above then gives $g^{j-n}z\in\partial\mathcal D_j$. Hence, 
$\delta_j
\le
\mathrm d_{\mathcal D_j}(x_j,g^{j-n}z)
\le
\vartheta^{n-j}\ell(\gamma)
<
\vartheta^{n-j}R
<
\delta_j,$
which is again impossible. We have proved that
$b:=\mathrm d_{\mathcal D_n}(x_n,\partial\mathcal D_n)>r.$
Define
$$Q=
\left\{
z\in\mathcal D_n:
\mathrm d_{\mathcal D_n}(x_n,z)\le r
\right\}.$$
We first verify that $Q$ is compact.
Given a sequence $\left\{z_j\right\}_{j\in\mathbb N^+}$ such that $z_j\in Q$ for all $j$, choose absolutely continuous
curves $\gamma_j\subset\mathcal D_n$ joining $x_n$ to $z_j$
with $\ell(\gamma_j)\le r+\frac 1j$, and parameterize them with
constant speed on $[0,1]$.
Compactness of $\mathcal D_n$ and the Arzel\`a--Ascoli theorem
give a uniformly convergent subsequence.
The limit curve lies in $\mathcal D_n$ and has length at most
$r$, by lower semicontinuity of length.
Its endpoint therefore belongs to $Q$.
Moreover, $Q\subset\mathrm{Int}\mathcal D_n$, since $b>r$. For each $z\in Q$, the same argument applied to a minimizing
sequence gives a length-minimizing absolutely continuous
curve from $x_n$ to $z$ in $\mathcal D_n$.
Every initial segment of this curve is minimizing, so the
whole curve lies in $Q$. Consequently,
\[
\mathrm d_Q(x_n,z)=\mathrm d_{\mathcal D_n}(x_n,z)\le r,
\:z\in Q.
\]
The function $z\mapsto \mathrm d_{\mathcal D_n}(x_n,z)$ is locally
continuous on the interior wherever it is finite, as follows
by joining nearby points in a $C^1$ coordinate neighborhood.
Thus $x_n\in\mathrm{Int}Q$, and
\[
\mathrm d_Q(x_n,z)=r,
\:z\in\partial Q.
\]
Points at distance $r$ are approached along their minimizing
curves by points at distance less than $r$; hence
$Q=\overline{\mathrm{Int}Q}$.
Define $\mathcal D=g^{-n}Q$.

Fix $y\in\mathcal D$ and let $z=g^ny\in Q$. Choose an absolutely continuous curve $\gamma\subset Q$ joining $x_n$ to $z$
with $\ell(\gamma)\le \mathrm d_{Q}(x_n,z)+\varepsilon$. We obtain
$\ell(g^{-k}\gamma)
\le
\vartheta^k\ell(\gamma)
\le
\lambda_2^{\frac k2}\ell(\gamma),$
for all $0\le k\le n$.
Therefore
$
\mathrm d_{g^{n-k}\mathcal D}
(g^{n-k}x,g^{n-k}y)
\le
\ell(g^{-k}\gamma)\le \lambda_2^{\frac k2}(\mathrm d_{g^n\mathcal D}(g^nx,g^ny)+\varepsilon),$
for all $0\le k\le n$. As $\varepsilon\to 0$, one has
$$\mathrm d_{g^{n-k}\mathcal D}
(g^{n-k}x,g^{n-k}y)\le
\lambda_2^{\frac k2}
\mathrm d_{g^n\mathcal D}(g^nx,g^ny),$$
for all $0\le k\le n$.
This proves the backward contraction estimate.

Finally, for $1\le k\le n$,
$\mathrm d(g^{n-k}x,g^{n-k}y)
\le
\lambda_2^{\frac k2}r<r,$
while
$\mathrm d(g^nx,g^ny)\le r.$
Hence
\[
g^k\mathcal D\subset B(g^kx,r),\:\forall\:0\le k\le n.
\]
\end{proof}
For any $0\le\ell<n$,
we say that \(n\) is a $\lambda_1$-hyperbolic time for \(x\) with respect to \(g\), $F$ and $\ell$, if for every \(1\le k\le n-\ell\),
\[
\prod_{j=n-k+1}^{n}
\left\|Dg^{-1}\big|_{F(g^jx)}\right\|
\le \lambda_1^k,\:0<\lambda_1<1.
\] 
It is a basic fact that there exists $r_0'>0$ such that if $n$ is a $\lambda_1$-hyperbolic time for \(x\) with respect to \(g\), $F$ and $\ell$, then, for any $y\in \left\{z\in\overline U:d(g^jx,g^jz)\le r_0',\:\ell+1\le j\le n\right\}$, $n$ is a $\lambda_1^{\frac{1}{2}}$-hyperbolic time for \(y\) with respect to \(g\), $F$ and $\ell$. An argument similar to that used in the proof of Lemma \ref{Lemma 2.6} gives the following lemma.
\begin{lemma}\label{Lemmas 2.6}
Let $g\in\mathrm{Diff}^1(\mathbf M^d)$ and assume that $T_{\overline U}\mathbf M^d=E\oplus F$
is an extendable continuous invariant splitting on the attractor $\Lambda$. Let
$
0<\lambda_1<\lambda_2<1.$
Then there exist constants
$0<r_1'<r_0,\:\theta_0'>0,$
such that the following holds. Let $0<r<r_1'$, $0<\theta\le \theta_0'$, and suppose that
$(x,g^nx)$ is $\lambda_1$-average dominated and that $n$
is a $\lambda_2$-hyperbolic time for $x$ with respect to $g$, $F$ and $\ell$, for $0\le \ell<n$. Let $\ell\le j<n$. Assume that $\widetilde{\mathcal D}$ is a compact embedded submanifold with $x\in\mathrm{Int}\widetilde{\mathcal D}$ tangent to $\mathcal C_\theta^F$, satisfying 
\[
\mathrm d_{g^j\widetilde{\mathcal D}}
(g^jx,\partial g^j\widetilde{\mathcal D})
>
\lambda_2^{\frac {n-j}2}r
\]
and
\[
g^s\widetilde{\mathcal D}
\subset B(g^sx,r_1'),
\:\forall\:0\le s\le j.
\]
Then there exists a compact subset $\mathcal D$ of the embedded $C^1$ submanifold $\widetilde{\mathcal D}$ such that $x\in\mathrm{Int}\mathcal D$, $\mathcal D$ is
tangent to $\mathcal C^F_{\theta}$, and the following properties hold.\\

$
\bullet\:g^k\mathcal D\subset B(g^kx,r),
\:\mathrm{for\:any\:}\ell\le k\le n$;\\

$\bullet\:g^n\mathcal D$ is a compact disc
centered at $g^nx$ with size $r$;\\

$\bullet\:$for every $y\in\mathcal D$ and $0\le k\le n-\ell$,
\[
\mathrm d_{g^{n-k}\mathcal D}
(g^{n-k}x,g^{n-k}y)
\le
\lambda_2^{\frac k2}
\mathrm d_{g^n\mathcal D}(g^nx,g^ny).
\]
\end{lemma}

Let $g\in\mathrm{Diff}(\mathbf M^d)$ and assume that $T_{\overline U}\mathbf M^d=E\oplus F$
is an extendable continuous invariant splitting on the attractor $\Lambda$.
For $x\in\overline U$ and $0<\lambda<1$, we say that the orbit
segment $(x,gx)$ is $\lambda$-dominated if
\[
\frac{\|Dg|_{E(x)}\|}
     {m(Dg|_{F(x)})}
\le \lambda.
\]
For distinct points $y,z\in U$ with
$\mathrm d(y,z)<R_{\mathrm{inj}}$, suppose that their unique
minimizing geodesic segment is contained in $U$. For $0<\theta<1$, we write
\[
y\sim_{F,\theta}z
\]
if this geodesic segment is tangent to $\mathcal C_\theta^F$
at every point. We adopt the convention that $z\sim_{F,\theta}z$, for all $z\in U$.

\begin{lemma}\label{Lemmas 2.7}
For any $0<\theta<1$ and any $0<\lambda<1$, there exists
$0<r_2:=r_2(g,\theta,\lambda)<\rho=
\min\left\{
\frac{R_{\mathrm{inj}}}{4},\frac{\mathrm d(\Lambda,\partial U)}{8}
\right\}$
such that the following holds. For any $0<r<r_2$, for any $x,y,z\in \bigl(U\setminus B(\partial U,4\rho)\bigr)
\cap g^{-1}\bigl(U\setminus B(\partial U,4\rho)\bigr)$, if the orbit segment $(x,gx)$ is
$\lambda$-dominated, and $y,z\in B(x,r),
\:
gy,gz\in B(gx,r),$
then
\[
y\sim_{F,\theta}z
\:\Longrightarrow\:
gy\sim_{F,\theta}gz.
\]
\end{lemma}
\begin{proof}
If two points belong to $B(x,\rho)$, their unique minimizing geodesic is contained in
$B(x,3\rho)$.
Under the hypotheses of the lemma, the minimizing geodesics from
$y$ to $z$ and from $gy$ to $gz$ are therefore uniquely defined
and contained in $U$.

Suppose, by contradiction, that no radius
$r_2\in(0,\rho)$ satisfies the conclusion.
There exist sequences $\left\{r_j\right\}_{j\in\mathbb N^+},\:\left\{x_j\right\}_{j\in\mathbb N^+},\:\left\{y_j\right\}_{j\in\mathbb N^+},\:\left\{z_j\right\}_{j\in\mathbb N^+}$ with
$r_j\downarrow0,\:
x_j,
y_j,z_j\in \bigl(U\setminus B(\partial U,4\rho)\bigr)
\cap g^{-1}\bigl(U\setminus B(\partial U,4\rho)\bigr)$
and $r_j<\rho$, such that
$$\frac{\|Dg|_{E(x_j)}\|}
     {m(Dg|_{F(x_j)})}
\le\lambda,\:y_j,z_j\in B(x_j,r_j),
\:
gy_j,gz_j\in B(gx_j,r_j),
$$
and
\[
y_j\sim_{F,\theta}z_j,
\:
gy_j\not\sim_{F,\theta}gz_j,
\]
for all $j\in\mathbb N^+$.
In particular, $y_j\ne z_j$, for all $j\in\mathbb N^+$.  By compactness of $\overline{U\setminus B(\partial U,4\rho)}\subset U\setminus B(\partial U,3\rho)$, after passing to a subsequence,
we may assume that $x_j\to x\in U\setminus B(\partial U,3\rho)$.
Consequently,
$y_j,z_j\to x,
\:
gy_j,gz_j\to gx.$
Since \(g(U)\subset U\), we have \(gx\in U\). As
\(gx_j\to gx\) and \(r_j\to0\), the minimizing geodesic
from \(gy_j\) to \(gz_j\) is contained in \(U\) for all
sufficiently large \(j\).
Continuity of $Dg$ and the splitting implies
\begin{equation*}
\frac{\|Dg|_{E(x)}\|}
     {m(Dg|_{F(x)})}
\le\lambda.
\end{equation*}

Set $\ell_j=\mathrm d(y_j,z_j),
\:
v_j=\exp_{y_j}^{-1}(z_j),
\:
u_j=\frac{v_j}{\ell_j}.$
Then $\ell_j>0$, $\|u_j\|=1$, and
$u_j\in\mathcal C_\theta^F(y_j),$
because the minimizing geodesic from $y_j$ to $z_j$
is tangent to $\mathcal C_\theta^F$, in particular at $y_j$. Using a smooth local trivialization of $T\mathbf M^d$ near $x$,
and passing to a further subsequence, we obtain
$u_j\to u\in T_x\mathbf M^d,
\:
\|u\|=1.$
By continuity of the splitting, $u\in\mathcal C_\theta^F(x).$
Next, let
$w_j=\exp_{gy_j}^{-1}(gz_j).$
Consider the map
$\Phi(p,v)
=
\exp_{gp}^{-1}\bigl(g(\exp_pv)\bigr),$
defined near the zero section.
Since $g$ is $C^1$ and the Riemannian exponential and
logarithm maps are smooth on the relevant neighborhoods,
$\Phi$ is $C^1$ and
$\Phi(p,0)=0,
\:
D_v\Phi(p,0)=D_pg.$
It follows that
$w_j=D_{y_j}gv_j+o(\ell_j).$
Therefore, in a smooth local trivialization near $gx$,
\begin{equation*}
\frac{w_j}{\ell_j}\to D_xg(u).
\end{equation*}

Let $\eta_j(t)=\exp_{gy_j}(tw_j),
\: 0\le t\le1,$
be the minimizing geodesic from $gy_j$ to $gz_j$.
Let $P_{j,t}$ denote parallel transport along $\eta_j$
from $gy_j$ to $\eta_j(t)$.
Since $\eta_j$ is affinely parametrized, $\dot\eta_j(t)=P_{j,t}w_j.$
The entire geodesic $\eta_j$ converges to $gx$, and its
length tends to zero.
Thus, in a smooth local trivialization near $gx$,
$P_{j,t}$ converges uniformly in $t$ to the identity. This gives
\begin{equation}\label{eq 2.3}
\frac{\dot\eta_j(t_j)}{\ell_j}
\to D_xg(u),
\end{equation}
for every $t_j\in[0,1]$.

Define $
\Psi_\theta(p,v)
=
\|\pi_p^Ev\|-\theta\|\pi_p^Fv\|,$
where $\pi_p^{E}$ and $\pi_p^{F}$ denote the natural projections from $T_{p}\mathbf M^d$ to $E(p)$ and $F(p)$.
This is a continuous function on $T_U\mathbf M^d$, and
\[
v\in\mathcal C_\theta^F(p)
\:\Longleftrightarrow\:
\Psi_\theta(p,v)\le0.
\]
Since $gy_j\not\sim_{F,\theta}gz_j$, there exists
$t_j\in[0,1]$ such that
$\Psi_\theta\bigl(\eta_j(t_j),\dot\eta_j(t_j)\bigr)>0.$
By positive homogeneity,
$\Psi_\theta\left(
\eta_j(t_j),
\frac{\dot\eta_j(t_j)}{\ell_j}
\right)>0.$
Passing to the limit and using
\eqref{eq 2.3}, we obtain
\begin{equation}\label{eq 2.4}
\Psi_\theta(gx,D_xg(u))\ge0.
\end{equation}

On the other hand, write $u=u_E+u_F$ according to
$T_x\mathbf M^d=E(x)\oplus F(x)$.
By $u\in\mathcal C^F_{\theta}(x)$, one has
$\|u_E\|\le \theta\|u_F\|.$
Moreover, $u_F\ne0$, since $\|u\|=1$.
The invariance of the splitting and the $\lambda$-dominated property of the segment $(x,gx)$ give
\[
\begin{aligned}
\|D_xg(u_E)\|
&\le
\|Dg|_{E(x)}\|\,\|u_E\|\le
\lambda \theta\,
m(Dg|_{F(x)})\,\|u_F\|\le
\lambda \theta\,\|D_xg(u_F)\|.
\end{aligned}
\]
Consequently,
\[
\begin{aligned}
\Psi_\theta(gx,D_xg(u))
&=
\|D_xg(u_E)\|-\theta\|D_xg(u_F)\|\le
-(1-\lambda)\theta\|D_xg(u_F)\|<0.
\end{aligned}
\]
This contradicts \eqref{eq 2.4} and
completes the proof.
\end{proof}

For a general disc tangent to \(\mathcal C^F_{\theta}\), it may happen that
\(x\not\sim_{F,\theta} y\) for some points \(x,y\) in the disc. Therefore, we shall
work with special discs which can be represented, via the exponential map,
as graphs of Lipschitz maps over the subbundle \(F\). For $v\in F(x)$ and $s>0$, write
$B_{F(x)}(v,s):=\{u\in F(x):\|u-v\|\le s\}.$
Its interior is taken relative to $F(x)$.

\begin{lemma}\label{lem:uniform-cone-graph}
For every $0<\theta<1$, there exist constants $r_3>0$ and
$0<\theta_1<\theta$, independent of $x\in \overline{U\setminus B(\partial U,4\rho)}$, with the following
property. Let $0<r<r_3$, and set $r'=\frac{r}{16}.$ Suppose that $\mathcal D\subset U$ is a compact disc centered at $x\in \overline{U\setminus B(\partial U,4\rho)}$ with size $r$ tangent to $\mathcal C_{\theta_1}^F$. Then there exists a $C^1$ map
\[
l:\mathrm{Int}(B_{F(x)}(0,2r'))\longrightarrow E(x),
\:l(0)=0,
\]
whose restriction to $B_{F(x)}(0,r')  $ is $\theta$-Lipschitz,
such that
\[
\mathcal D'
=
\exp_x(\Gamma_l)
\subset\mathrm{Int}\mathcal D,
\:
\Gamma_l
=
\{u+l(u):u\in B_{F(x)}(0,r')  \}.
\]
Moreover, $\mathcal D'$ is a compact $C^1$ embedded
disc with $C^1$ boundary,
$x\in\mathrm{Int}\mathcal D'$, and
\[
y\sim_{F,\theta}z,
\:\text{for all }y,z\in\mathcal D'.
\]
\end{lemma}
\begin{proof}
Fix $0<\theta<1$ and set $b=\min\left\{\frac \theta4,\frac14\right\},
\:
\theta_1=\frac b2.$
Thus
$0<\theta_1<b<2b<\theta.$ By continuity of the splitting and compactness of $K:=\overline{U\setminus B(\partial U,4\rho)}$,
$P=
\max\left\{
1,\sup_{x\in K}\|\pi_x^E\|,
\sup_{x\in K}\|\pi_x^F\|
\right\}
<+\infty.$
We first establish uniform properties of normal
coordinates.
Choose $0<\delta_0<\rho$ small enough, such that for every $x\in K$,
$4\delta_0<R_{\mathrm{inj}},
\:
B(x,4\delta_0)\subset U,$
and
$\|D_v\exp_x\|\le2,
\:
\|D_p(\exp_x^{-1})\|\le2,$
whenever $\|v\|<4\delta_0$ and $p=\exp_x(v)$. Since $\theta_1<b$ and $2b<\theta$, continuity and compactness
allow us to choose $0<\delta\le\delta_0$ such that
\begin{equation}\label{eq:pullback-narrow-cone}
D_p(\exp_x^{-1})
\bigl(\mathcal C_{\theta_1}^F(p)\bigr)
\subset \mathcal C_b^F(x)
\end{equation}
and
\begin{equation}\label{eq:pushforward-fixed-cone}
D_v\exp_x\bigl(\mathcal C_{2b}^F(x)\bigr)
\subset\mathcal C_\theta^F(p)
\end{equation}
for all $x\in K$, $\|v\|<4\delta$, and $p=\exp_x(v)$.

We next show that, after reducing $\delta$ further,
\begin{equation}\label{eq:secant-geodesic-cone}
v_1,v_2\in B_{T_x\mathbf M^d}(0,\delta),
\:
v_2-v_1\in\mathcal C_b^F(x)
\:\Longrightarrow\:
\exp_x(v_1)\sim_{F,\theta}\exp_x(v_2).
\end{equation}
The case $v_1=v_2$ is immediate.
For $v_1\ne v_2$, let $\eta:[0,1]\to \mathbf M^d$ be the affinely
parametrized minimizing geodesic from $\exp_x(v_1)$
to $\exp_x(v_2)$.
Set
$h(t)=\exp_x^{-1}(\eta(t)),
\:
\Delta=v_2-v_1.$
The geodesic is contained in $B(x,3\delta)$.
The coordinate line segment joining $v_1$ and $v_2$
has Riemannian length at most $2\|\Delta\|$.
Therefore, $\ell(\eta)\le2\|\Delta\|,
\:
\|\dot h(t)\|\le4\|\Delta\|.$

The Christoffel symbols in these normal coordinates
are uniformly bounded for $x\in K$ and
$\|h\|\le3\delta_0$.
Consequently, there is a constant $C\ge0$, independent
of $x$, such that
$\|\ddot h(t)\|
\le C\|\dot h(t)\|^2
\le16C\|\Delta\|^2.$
Since $\Delta=\int_0^1\dot h(s)\,ds,$ we obtain
$\|\dot h(t)-\Delta\|
\le16C\|\Delta\|^2
\le32C\delta\|\Delta\|.$
Choose
$0<\varepsilon<
\frac{b}{2P(1+b)(1+2b)}$
and reduce $\delta$ so that $32C\delta\le\varepsilon$.
Then $\dot h(t)=\Delta+e_t,
\:
\|e_t\|\le\varepsilon\|\Delta\|.$
Writing $\Delta=\Delta_E+\Delta_F$, the assumed cone
condition gives
$
\|\Delta_E\|\le b\|\Delta_F\|,
\:
\|\Delta\|\le(1+b)\|\Delta_F\|.$
In particular, $\Delta_F\ne0$, and
\[
\frac{\|\pi_x^E\dot h(t)\|}
     {\|\pi_x^F\dot h(t)\|}
\le
\frac{b+P\varepsilon(1+b)}
     {1-P\varepsilon(1+b)}
\le2b.
\]
Thus $\dot h(t)\in\mathcal C_{2b}^F(x)$.
By \eqref{eq:pushforward-fixed-cone},
$\dot\eta(t)
=
D_{h(t)}\exp_x(\dot h(t))
\in\mathcal C_\theta^F(\eta(t)).$
This proves \eqref{eq:secant-geodesic-cone}.

Set $r_3=\delta$.
Fix $0<r<r_3$, $x\in K$, and a disc $\mathcal D$
as in the statement.
We use only the condition
\begin{equation}\label{eq:intrinsic-boundary-distance}
d_{\mathcal D}(x,\partial\mathcal D)\ge r.
\end{equation}Define
\[
\Pi_x:
\mathrm{Int}\mathcal D\cap \mathrm{Int}(B(x,\delta))
\to F(x),
\:
\Pi_x(p)=\pi_x^F(\exp_x^{-1}(p)),
\]
where $\mathrm{Int}({B}(x,\delta))$ denotes the open ball centered at $x$ with radius $\delta$. For $v\in T_p\mathcal D$, write
$D_p(\exp_x^{-1})v=e_{E}+e_{F},
\:e_E\in E(x),\:e_{F}\in F(x).$ Tangency of $\mathcal D$ to $\mathcal C_{\theta_1}^F$
and \eqref{eq:pullback-narrow-cone} imply $\|e_{E}\|\le b\|e_{F}\|.$
Since $D_p\Pi_x(v)=e_{F}$, the derivative $D_p\Pi_x$
is injective.
Both its domain and codomain have dimension $\mathrm{dim}F$,
so it is an isomorphism.
Moreover,
$\|v\|
\le2\|e_{E}+e_{F}\|
\le2(1+b)\|e_{F}\|
\le4\|D_p\Pi_x(v)\|.$
Hence
\begin{equation}\label{eq:inverse-projection-bound}
\|(D_p\Pi_x)^{-1}\|\le4.
\end{equation}
Let $u\in B_{F(x)}(0,\frac r8)$.
The inverse function theorem provides a unique local
lift $\gamma_u$ of the radial path $t\mapsto tu$,
starting at $x$:
$\Pi_x(\gamma_u(t))=tu,
\:
\gamma_u(0)=x.$
While this lift exists in
$\mathrm{Int}\mathcal D\cap \mathrm{Int}({B}(x,\delta))$,
\eqref{eq:inverse-projection-bound} gives $\|\dot\gamma_u(t)\|\le4\|u\|.$ In particular,
\begin{equation}\label{eq:radial-lift-length}
\ell(\gamma_u|_{[0,t]})
\le4t\|u\|\le\frac r 2
\: (0\le t\le1).
\end{equation}
The lift extends to all of $[0,1]$.
Indeed, at any finite endpoint of its interval of
existence, the speed bound and compactness of
$\mathcal D$ give a limit point $p_*\in\mathcal D$.
By \eqref{eq:radial-lift-length},
$\mathrm d(x,p_*)\le \frac r2<\delta.$
If $p_*\in\partial\mathcal D$, the lifted curve,
extended continuously to $p_*$, would give
$\mathrm d_{\mathcal D}(x,p_*)\le \frac r2<r,$
contrary to \eqref{eq:intrinsic-boundary-distance}.
Thus $p_*\in\mathrm{Int}\mathcal D\cap\mathrm{Int}({B}(x,\delta))$,
where the inverse function theorem extends the lift.

Define $\psi(u)=\gamma_u(1),
\:
u\in B_{F(x)}(0,\frac r8).$
Then
$\Pi_x(\psi(u))=u,
\:
\psi(0)=x.$
The map $\psi$ is continuous.
To see this, fix a radial lift and cover its compact
image by finitely many inverse neighborhoods for
$\Pi_x$.
For nearby values of $u$, the radial paths can be
lifted through the same inverse neighborhoods.
Uniqueness of lifts then gives continuous dependence
of their endpoints on $u$.
Since $\Pi_x$ is a local $C^1$ diffeomorphism and
$\Pi_x\circ\psi=\mathrm{id}$, continuity implies
that $\psi$ locally agrees with a $C^1$ inverse branch.
Thus $\psi$ is $C^1$.

Set
$l(u)=\pi_x^E(\exp_x^{-1}(\psi(u))).$
Then $\psi(u)=\exp_x(u+l(u)),
\: l(0)=0.$
By \eqref{eq:pullback-narrow-cone}, every tangent vector
to the graph of $l$ belongs to $\mathcal C_b^F(x)$.
Therefore
$\|Dl(u)\|\le b.$
Since $B_{F(x)}(0,\frac r8)$ is convex,
\begin{equation}\label{eq:graph-lipschitz-bound}
\|l(u)-l(v)\|\le b\|u-v\|
\end{equation}
for all $u,v$ in this ball. Now take
\[
r'=\frac r{16},
\:
\Gamma_l=\{u+l(u):u\in B_{F(x)}(0,r')  \},
\:
\mathcal D'=\exp_x(\Gamma_l).
\]
The map $l$ is $C^1$ on the larger open ball
$\mathrm{Int}({B}_{F(x)}(0,2r'))$.
Consequently, $\mathcal D'$ is a compact $C^1$
embedded disc with $C^1$ boundary and dimension $\mathrm{dim }F$ centered at $x\in\mathrm{Int}\mathcal D'$ satisfying
$\mathcal D'\subset\mathrm{Int}\mathcal D.$
Moreover, \eqref{eq:graph-lipschitz-bound} shows that
$l|_{B_{F(x)}(0,r')  }$ is $b$-Lipschitz, hence $\theta$-Lipschitz.

Finally, let
$y=\exp_x(u+l(u)),
\:
z=\exp_x(v+l(v))$
be distinct points of $\mathcal D'$.
By \eqref{eq:graph-lipschitz-bound},
$$(v-u)+(l(v)-l(u))\in\mathcal C_b^F(x).$$
Also,
$\|u+l(u)\|\le(1+b)r'<\delta,
\:
\|v+l(v)\|\le(1+b)r'<\delta.$
Applying \eqref{eq:secant-geodesic-cone} yields
$y\sim_{F,\theta}z.$

All choices of constants were made uniformly over
the compact set $K$ and before choosing $x$, $r$,
or $\mathcal D$.
Thus $r_3$ and $\theta_1$ are independent of $x$ and
$\mathcal D$, and $r'=\frac r{16}$ has the required
uniformity.
\end{proof}
With the parameters \(\theta,\theta_1,r,r'\) fixed, any disc
of the form \(\exp_x(\Gamma_l)\) as above will be called an \(F\)-disc centered at
\(x\) with respect to $(\theta,\theta_1,r,r')$. For the terminal domains obtained in Lemma \ref{Lemma 2.6} and Lemma \ref{Lemmas 2.6},
assume also that $r<r_3$ and set $r'=\frac r{16}$.
Choose $N_{\mathrm{geo}}\in\mathbb N^+$ such that
$\lambda_1^{\frac{N_{\mathrm{geo}}}2}\theta\le\theta_1.$
For every $n\ge N_{\mathrm{geo}}$ such that
$g^n x\in\overline{U\setminus B(\partial U,4\rho)}$,
cone narrowing along the original orbit segment implies that
$g^n\mathcal D$ is tangent to $\mathcal C_{\theta_1}^F$. Since $g^n\mathcal D$ is centered at
$g^n x$ with size $r$, Lemma \ref{lem:uniform-cone-graph} yields an $F$-disc contained
in $\mathrm{Int}(g^n\mathcal D)$ with respect to
$(\theta,\theta_1,r,r')$.
The integer $N_{\mathrm{geo}}$ depends only on
$\theta,\theta_1$, and $\lambda_1$. 

The following is a version of the geometric dichotomy in
\cite[Lemma 3.4]{BY25}. We give the proof to show that the radius
can be chosen independently of the embedded submanifold $\mathcal D$.

\begin{lemma}\label{Lemma 2.7}
 Fix $0<\theta<1$ and $0<r<r_3$. Let $\theta_1$ and
$r'=\frac r{16}$ be as in Lemma \ref{lem:uniform-cone-graph}.  With the background geometric data fixed, there exists $0<r_4:=r_4(\theta,r)<r_3$, independent of the
choice of $\mathcal D$ below, such that the following holds. Let \(\mathcal D\) be a compact $\mathrm{dim}F$-dimensional embedded manifold. Suppose that
\(\mathcal D'\subset \mathrm{Int}\mathcal D\) and \(\mathcal D''\subset \mathrm{Int}\mathcal D\) are \(F\)-discs
centered at \(y'\) and \(y''\) with respect to $(\theta,\theta_1,r,r')$ from the
preceding construction in Lemma \ref{lem:uniform-cone-graph}, and that \(y'\) and \(y''\)
belong to a ball \(B(x,r'')\) with \(r''<r_4\). Then either
\[
\mathcal D'\cap B(x,r'')=\mathcal D''\cap B(x,r''),
\]
or
$\mathcal D'\cap \mathcal D''=\emptyset$
and there exist points
\[
z'\in \mathcal D'\cap B(x,2r''),
\:
z''\in \mathcal D''\cap B(x,2r''),
\]
such that
\[
z'\not\sim_{F,\theta} z''.
\]
\end{lemma}
\begin{proof}
Put $s=r''$ and recall that $r'=\frac r{16}$. Let $K$ be the fixed
compact set containing the admissible disc centers. We use the standard discs from the
preceding construction: in normal coordinates at its center $y$,
each is the graph of a $C^1$ map $\ell$ on the closed $F(y)$-ball
of radius $r'$, with $\ell(0)=0$ and
$\|D\ell\|_{\infty}\le b=\min\{\frac \theta 4,\frac 1 4\}$.
All coordinate and cone estimates below are taken on a fixed compact
neighborhood of $K$ in $U$. The ball center $x$ lies in this
neighborhood as soon as $r''$ is small enough.

For a disc centered at $y$, use an
isometry $J:T_y\mathbf M^d\to T_x\mathbf M^d$ taking $F(y)$ to $F(x)$
and uniformly close to parallel transport as $\mathrm d(x,y)\to0$.
Identifying $F(y)$ with $F(x)$ by $J$, write the disc in the normal
chart at $x$ as
$(h(u),k(u))\in F(x)\oplus {F(x)}^\perp,\:u\in B_{F(x)}(0,r').$
There is $r_{\mathrm{chart}}>0$, independent of $\mathcal D$ and of
the particular discs, such that, when $0<s<r_{\mathrm{chart}}$,
\[
  \|Dh-I\|_{\infty}\le \frac{3}{10},\: \|Dk\|_{\infty}\le \frac{3}{10},\:
  \|(h(0),k(0))\|\le s .
\]
Indeed, $D(\exp_x^{-1}\circ\exp_y)$ differs from $J$ by at most
$\omega(s)$ on the tangent ball of radius $(1+b)r'$, with
$\omega(s)\to0$ uniformly by compactness, smoothness of the normal
charts, and continuity of $F$. Orthogonally projecting the
derivative of the graph gives the bound $b+(1+b)\omega(s)$
for each of the first two quantities.

We next show that $\Omega=h(B_{F(x)}(0,r'))$ is star-shaped about
$0$ if $s<\frac{r'}{100}$. Write $h(u)=c+u+e(u)$ with $c=h(0)$, $e(0)=0$,
and $\mathrm{Lip}(e)\le \frac{3}{10}$. Thus $h$ is injective. Whenever $B_{F(x)}(v,L)\subset B_{F(x)}(0,r')$, we have 
\[
 B_{F(x)}(h(v),\frac{7}{10}L)\subset h(B_{F(x)}(v,L)).
\]
Indeed, for each $w$ in the ball on the left, the map
$\xi\mapsto w-c-e(\xi)$ is a contraction of the closed ball
$B_{F(x)}(v,L)$ into itself, and its fixed point satisfies
$h(\xi)=w$.
In particular $\Omega$ contains $B_{F(x)}(0,s)$.
For $\|u\|\le r'$ and $0\le t<1$,
\[
 \|h(tu)-t h(u)\|
 \le (1-t)(s+\frac{6}{10}r')<\frac{7}{10}(1-t)r' .
\]
The open ball of radius $(1-t)r'$ about $tu$ lies in the open
radius-$r'$ parameter ball. The ball inclusion therefore puts
$t h(u)$ in the interior of $\Omega$. This also treats points
$h(u)$ on the boundary of $\Omega$.

Let $\Omega_1,\Omega_2$ be the two projection domains. Their
intersection is star-shaped about $0$ and hence connected. On
this intersection the coincidence set of the two graph functions
is relatively closed. It is also relatively open: at a common
point $p$, horizontal projection in the normal chart at $x$ is
locally invertible on $\mathrm{Int}\mathcal D$, since its
derivative is invertible on the tangent space of either disc and
that tangent space equals $T_p\mathcal D$. This applies also at a
boundary point of a disc, because $p\in\mathrm{Int}\mathcal D$.
Consequently, if the discs meet anywhere, their graph functions
agree on $\Omega_1\cap\Omega_2$. Every point in $B(x,s)$ projects
into $ B_{F(x)}(0,s)\subset\Omega_1\cap\Omega_2$, proving the
first alternative.

Suppose the discs are disjoint. Take the unique point of each disc
with horizontal coordinate zero. The graph function $k\circ h^{-1}$
has Lipschitz constant at most $\frac37$, so the distance of
each point from $x$ is at most
\[
 \|k(0)\|+\tfrac37\|h(0)\|
 \le\sqrt{1+(\frac37)^2}\,s<2s .
\]
Their nonzero coordinate chord belongs to $F(x)^\perp$. For any unit
$v\in F(x)^\perp$, writing $v=v_E+v_F$ in the given invariant
splitting gives $\|v_E\|^2=1+\|v_F\|^2$. Hence
$\|v_E\|-\theta\|v_F\|\ge\sqrt{1-\theta^2}>0.$
After parallel transport to $x$, the initial velocity of the
minimizing geodesic between two points of $B(x,2s)$ differs from
their coordinate chord by $o(1)$ times the chord length, uniformly
as $s\to0$ (with the geodesic parametrized on $[0,1]$).
Indeed, for $v,w$ the two normal coordinates at $x$, the map
$w\mapsto P_{\exp_x v\to x}\exp_{\exp_x v}^{-1}(\exp_x w)$ vanishes
at $w=v$ and its derivative in $w$ tends uniformly to the identity
as $\|v\|+\|w\|\to 0$; here $P$ is parallel transport along the
radial geodesic. The mean value estimate gives the claimed relative
error. The strict cone margin and continuity of the splitting therefore
give $r_{\mathrm{cone}}>0$ such that, whenever $0<s<r_{\mathrm{cone}}$,
the initial velocity of the geodesic joining the two chosen points
lies outside the $F$-cone. Thus the two points are not
$\sim_{F,\theta}$-related.

Choose $r_4=\tfrac12\min\{r_3,\frac{r'}{100},r_{\mathrm{chart}},r_{\mathrm{cone}}\}$,
including the compact-neighborhood and injectivity requirements in
$r_{\mathrm{chart}}$ and $r_{\mathrm{cone}}$. These radii are fixed by
$\theta,r$ and the standing geometric data. The local inverse on
$\mathrm{Int}\mathcal D$ was used only to prove relative
openness, and its radius does not enter this choice.
\end{proof}

The following lemma gives the volume bounds for \(F\)-discs.
\begin{lemma}\label{Lemma 2.10}
Let $m=\dim F,
\:
\omega_m=\mathrm{Leb}^m\bigl(B_{\mathbb R^m}(0,1)\bigr).$ There exists a constant $r_5>0$ such that the following holds. For any \(0<\theta_1<\theta<1\), any \(0<r'<r<r_5\) and any $x\in K$, let
\[
\mathcal D_x=\exp_x(\Gamma_l),
\:
\Gamma_l
=
\{u+l(u):u\in B_{F(x)}(0,r')  \},
\]
be an \(m\)-dimensional \(F\)-disc with respect to
\((\theta,\theta_1,r,r')\), where $l:\mathrm{Int}(B_{F(x)}(0,2r'))\longrightarrow E(x)$
is of class \(C^1\) and \(l|_{B_{F(x)}(0,r')  }\) is \(\theta\)-Lipschitz. Then
\[
2^{-m}(1-\theta)^m\omega_m(r')^m
\le
\mathrm{Leb}_{\mathcal D_x}(\mathcal D_x)
\le
2^m(1+\theta)^m\omega_m(r')^m.
\]
\end{lemma}

\begin{proof}
By compactness, choose $r_5>0$ such that, for every $x\in K$,
the exponential map $\exp_x$ is injective on
$\{w\in T_x\mathbf M^d:\|w\|<2r_5\}$ and
$\frac12\|v\|
\le \|D_w\exp_x(v)\|
\le 2\|v\|$
whenever $\|w\|<2r_5$, $v\in T_x\mathbf M^d$.

Put $G(u)=u+l(u)$ and $\Phi=\exp_x\circ G$.
Since $l(0)=0$ and $l$ is $\theta$-Lipschitz on
$B_{F(x)}(0,r')  $,
\[
\|G(u)\|\le(1+\theta)r'<2r_5,
\:\text{for all }u\in B_{F(x)}(0,r')  .
\]
Moreover, for $u\in\mathrm{Int}B_{F(x)}(0,r')  $
and $v\in F(x)$,
$(1-\theta)\|v\|
\le \|D_uG(v)\|
\le (1+\theta)\|v\|.$
Consequently, all $m$ singular values of $D_u\Phi$
lie in $[\frac{1-\theta}2,\,2(1+\theta)]$, and hence
\[
2^{-m}(1-\theta)^m
\le J_u^m\Phi
\le 2^m(1+\theta)^m.
\]
The map $\Phi$ is a $C^1$ embedding: $G$ is an immersion
and $\|G(u)-G(v)\|\ge(1-\theta)\|u-v\|,
\:u,v\in B_{F(x)}(0,r')  ,$
while $\exp_x$ is a diffeomorphism on the chosen neighborhood.
The area formula therefore gives
\[
\mathrm{Leb}_{\mathcal D_x}(\mathcal D_x)
=
\int_{B_{F(x)}(0,r')  }J_u^m\Phi\,
d\mathrm{Leb}_{F(x)}(u).
\]
Since
$\mathrm{Leb}_{F(x)}(B_{F(x)}(0,r')  )=\omega_m(r')^m$,
the asserted bounds follow.
\end{proof}

\section{Proof of Theorem \ref{thm:A}}\label{3}

\subsection{Pliss times with positive density} 
\begin{lemma}\label{Lemma 3.1}
The following statements hold.\\
 $\bullet\:$
 Under the assumptions of Theorem \ref{thm:A} (1), there exist $\Delta_1\subset \Delta(a,\lambda,f^p,E,F)$ with $\mathrm{Leb}(\Delta_1)>0$, $a_1'>a$, $\beta_1>0$ and $0<\lambda_1'<1$, such that for any $x\in\Delta_1$, there exists $P_1(x)\subset \mathbb N$ with $\underline d(P_1(x))>\beta_1$, such that for any $n\in P_1(x)$, for any $0\le j\le n-1$, one has
\begin{equation}\label{3.1}
\begin{aligned}
\frac{1}{(n-j)p}\sum_{s=0}^{n-j-1}\log m(Df^p|_{F(f^{(s+j)p}x)})&
\ge \frac{a+a_1'}{2}.
\end{aligned}
\end{equation}
Moreover, for any $x\in\Delta_1$, for any $j_1,\:j_2\in P_1(x)$ with $j_1>j_2$, one has
$$\frac{1}{(j_1-j_2)p}\sum_{j=j_2}^{j_1-1}\log\frac{m(Df^p|_{F(f^{jp}x)})}{\left\|Df^p|_{E(f^{jp}x)}\right\|}\ge-\log\lambda_1'.$$
$\bullet\:$Under the assumptions of Theorem \ref{thm:A} (2), there exist $\Delta_2\subset \mathrm{HP}(a,\lambda,f^p,E,F)$ with $\mathrm{Leb}(\Delta_2)>0$, $a_1''>a>a_2''$, $\beta_2>0$ and $0<\lambda_2'<1$, such that for any $x\in\Delta_2$, there exists $P_2(x)\subset \mathbb N$ with $\underline d(P_2(x))>\beta_2$, such that for any $n\in P_2(x)$, for any $0\le j\le n-1$, one has
\begin{equation}\label{3.2}
\begin{aligned}
\frac{1}{(n-j)p}\sum_{s=0}^{n-j-1}\log m(Df^p|_{F(f^{(s+j)p}x)})&
\ge\frac{a_1''+a}{2}\\
&>\frac{a+a_2''}{2}\ge\frac{1}{(n-j)p}\sum_{s=0}^{n-j-1}\log \left\|Df^p|_{E(f^{(s+j)p}x)}\right\|.
\end{aligned}
\end{equation}
Moreover, for any $x\in\Delta_2$, for any $j_1,\:j_2\in P_2(x)$ with $j_1>j_2$, one has
$$\frac{1}{(j_1-j_2)p}\sum_{j=j_2}^{j_1-1}\log\frac{m(Df^p|_{F(f^{jp}x)})}{\left\|Df^p|_{E(f^{jp}x)}\right\|}\ge-\log\lambda_2'.$$
\end{lemma}
\begin{proof}
\text{The case of Theorem \ref{thm:A} (1).}
It is clear that there exists $a_1'>a$ such that $\mathrm{Leb}(\Delta_1)>0$, where
\begin{equation}
\begin{aligned}
\Delta_1=\left\{x\in \Delta(\lambda,f^p,E,F):\liminf_{n\to+\infty}\frac{1}{np}\sum_{j=0}^{n-1}\log m(Df^p|_{F(f^{jp}x)})>a_1'\right\}.
\end{aligned}
\end{equation}
Let 
$\lambda_1=-a_1',\:\lambda_2=-\frac{a+a_1'}{2}.$
By Lemma \ref{Pliss}, for any $x\in\Delta_1$,
there exists $P_1(x)\subset \mathbb N$, such that for any $n\in P_1(x)$, for any $0\le j\le n-1$, one has
\begin{equation}\label{3.4}
\begin{aligned}
\frac{1}{(n-j)p}\sum_{s=0}^{n-j-1}\log m(Df^p|_{F(f^{(s+j)p}x)})&
\ge\frac{a+a_1'}{2}.
\end{aligned}
\end{equation}
Moreover, there exists $\beta_1>0$, such that for any $x\in \Delta_1$, one has
$\underline{d}(P_1(x))>\beta_1.$
Put $b=\frac{3a+a_1'}4$. Since $b>a$ and $\|Df^p|_{E(x)}\|\le e^{ap},
\:x\in\Lambda,$
continuity gives a neighborhood $V$ of $\Lambda$ such that
$\|Df^p|_{E(z)}\|\le e^{bp},\:
z\in V.$
Every point of the original trapping neighborhood eventually
remains in $V$. Hence
\[
\Delta_1
=
\bigcup_{T\ge0}
\left\{
x\in\Delta_1:
f^{kp}x\in V,\text{ for every }k\ge T
\right\}.
\]
Choose $T$ for which the corresponding set has positive
Lebesgue measure, and replace $\Delta_1$ by this set.
For each remaining $x$, replace $P_1(x)$ by
$P_1(x)\cap[T,+\infty)$, retaining the same notation. Deleting finitely many times does not change the lower
density of $P_1(x)$, and every retained time still satisfies
the estimate \eqref{3.4}.
Whenever $j_1>j_2$ belong to the retained set $P_1(x)$,
the derivative bound above applies throughout the interval
$j_2\le j<j_1$.
We set
$\lambda_1'=e^{\frac{a-a'_1}{4}}.$
It implies that, for any $x\in\Delta_1$, for any $j_1,\:j_2\in P_1(x)$ with $j_1>j_2$, one has
\begin{equation*}
\begin{aligned}
\sum_{j=j_2}^{j_1-1}\log\frac{m(Df^p|_{F(f^{jp}x)})}{\left\|Df^p|_{E(f^{jp}x)}\right\|}&\ge\sum_{j=j_2}^{j_1-1}\log m(Df^p|_{F(f^{jp}x)})-(j_1-j_2)bp\\
&\ge (j_1-j_2)p(\frac{a+a_1'}{2}-\frac{3}{4}a-\frac{1}{4}a_1')\ge (j_2-j_1)p\log\lambda_1'.
\end{aligned}    
\end{equation*}
\text{The case of Theorem \ref{thm:A} (2).}
It is clear that there exist $a_1''>a>a_2''$ such that $\mathrm{Leb}(\Delta_2)>0$, where
\begin{equation}
\begin{aligned}
&\Delta_2=\left\{x\in \Delta(\lambda,f^p,E,F):\right.\\
&\left.\liminf_{n\to+\infty}\frac{1}{np}\sum_{s=0}^{n-1}\min\left\{\log m(Df^p|_{F(f^{sp}x)})-a_1''p,a_2''p+\log m(Df^{-p}|_{E(f^{(s+1)p}x)})\right\}>0\right\}.
\end{aligned}
\end{equation}
Let 
$\lambda_1=0,\:\lambda_2=\min\left\{\frac{a_1''-a}{2},\frac{a-a_2''}{2}\right\}.$
By Lemma \ref{Pliss}, for any $x\in\Delta_2$,
there exists $P_2(x)\subset \mathbb N$, such that for any $n\in P_2(x)$, for any $0\le j\le n-1$, one has
\begin{equation}\label{eq 3.6}
\begin{aligned}
\frac{1}{(n-j)p}\sum_{s=0}^{n-j-1}\log m(Df^p|_{F(f^{(s+j)p}x)})&
\ge\frac{a_1''+a}{2}\\
&>\frac{a+a_2''}{2}\ge\frac{1}{(n-j)p}\sum_{s=0}^{n-j-1}\log \left\|Df^p|_{E(f^{(s+j)p}x)}\right\|.
\end{aligned}
\end{equation}
Moreover, there exists $\beta_2>0$, such that for any $x\in\Delta_2$, one has
$\underline{d}(P_2(x))>\beta_2.$
We set 
$\lambda_2'=e^{\frac{a_{2}''-a_1''}{2}}.$
It implies that, for any $x\in\Delta_2$, for any $j_1,\:j_2\in P_2(x)$ with $j_1>j_2$, one has
\begin{equation*}
\begin{aligned}
\sum_{j=j_2}^{j_1-1}\log\frac{m(Df^p|_{F(f^{jp}x)})}{\left\|Df^p|_{E(f^{jp}x)}\right\|}&=\sum_{j=j_2}^{j_1-1}\log m(Df^p|_{F(f^{jp}x)})-\sum_{j=j_2}^{j_1-1}\log\left\|Df^p|_{E(f^{jp}x)}\right\|\\
&\ge (j_1-j_2)p(\frac{a+a_1''}{2}-\frac{a+a_2''}{2})\ge (j_2-j_1)p\log\lambda_2'.
\end{aligned}    
\end{equation*}
\end{proof}
By Lemma \ref{Lemma 3.1}, for each $i\in\{1,2\}$ and $x\in\Delta_i$,
every $n\in P_i(x)$ is a $\gamma_i'$-hyperbolic time
for $x$ with respect to $f^p$ and $F$, where
$\gamma_1'
=e^{-\frac{p(a_1'+a)}{2}},
\:
\gamma_2'
=e^{-\frac{p(a_1''+a)}{2}}.$
Since a $\gamma$-hyperbolic time is also a
$\widetilde\gamma$-hyperbolic time whenever
$0<\gamma\le\widetilde\gamma<1$, we choose and fix
$\gamma_i\in\bigl(\max\{\lambda,\gamma_i'\},1\bigr),
\:i\in\{1,2\}.$
Then
$\lambda<\min\{\gamma_1,\gamma_2\}
\le\max\{\gamma_1,\gamma_2\}<1,$
and every $n\in P_i(x)$ remains a $\gamma_i$-hyperbolic
time for $x$ with respect to $f^p$ and $F$.
These coefficients are fixed throughout the remainder
of the proof of Theorem \ref{thm:A}.
For $i=1,2$, apply Lemma \ref{Lemma 2.6} with $(\lambda,\gamma_i)$ and
Lemma \ref{Lemmas 2.6} with $(\lambda,\sqrt{\gamma_i})$. Taking the smaller
radii and cone widths over $i=1,2$ gives common constants
$r_1,\theta_0,r'_1,\theta'_0$.

\subsection{Construction of the invariant measures}
Let $\mathcal P_{\mathbb N}$ and $\mathcal P_n$ be the power sets of $\mathbb N$ and $\left\{0,\cdots,n-1\right\}$. For $E\in\mathcal P_{\mathbb N}$ or $E\in\mathcal P_{n}$, the boundary $\partial E$ of $E$ is the symmetric difference of $E$ and $E+1$ with $E+1:=\left\{k+1:k\in E\right\}$. For $P\subset\mathbb N$, we call $[a,b)\cap\mathbb N$
a $P$-irreducible interval if $a,b\in P$, $a<b$, and
$P\cap(a,b)=\emptyset$. For $M\in\mathbb N^+$, we denote by $E_{M}$ the union of the integer intervals $[a,b)$ with $a,b\in E$ and $0<b-a\le M$. For $n\in\mathbb N^+$ and $F\in\mathcal P_{\mathbb N}$, let
$$d_n(F)=\frac{\#\bigl(F\cap[0,n)\bigr)}{n},\:\overline{d}(F)=\limsup_{n\to+\infty}d_n(F),\:\underline{d}(F)=\liminf_{n\to+\infty}d_n(F).$$

For $x\in\Delta_i$, take $P_i(x)$ to consist of all positive
integers $n$ for which the inequalities in \eqref{3.1} if $i=1$, or
in \eqref{3.2} if $i=2$, hold for every $0\le j<n$, with the additional
requirement $n\ge T$ when $i=1$.
For each fixed $n$, the set $\{x\in\Delta_i:n\in P_i(x)\}$ is
measurable, since the defining conditions involve only finitely
many inequalities between continuous functions.
Let $\mathcal D_i\subset U_1$ be an $F$-disc with respect to
\((\theta,\theta_1,r,r')\) tangent to $\mathcal C^F_{\theta}$ with $\theta<\min\left\{\frac{\theta_0}{16},\frac{\theta_0'}{16}\right\}$ and $r<\min\left\{\frac{r_0}{16},\frac{r_0'}{16},\frac{r_1}{16},\frac{r_1'}{16},\frac{\rho}{16},\frac{r_3(\theta)}{16}\right\}$. By Fubini's theorem, without loss of generality, we assume that $\mathrm{Leb}_{\mathcal D_i}(\Delta_{i}\cap\mathcal D_i)>0,\: i=1,2$. Let 
$A_{i}=\left\{x\in \Delta_i\cap\mathcal D_i:\underline{d}(P_i(x))>\beta_i\right\}.$ 
Since $\mathrm{Leb}_{\mathcal D_i}(\partial\mathcal D_i)=0$,
we may replace $A_i$ by a subset of positive
$\mathrm{Leb}_{\mathcal D_i}$-measure, retaining the same notation,
and choose $\eta>0$ such that 
\begin{equation}\label{equation 3.9}
A_i\cap B(\partial\mathcal D_i,\eta)=\emptyset,
\end{equation}
$$\mathrm{Leb}_{\mathcal D_i}(A_i)>0.$$
Set $\gamma_*:=\max\{\gamma_1,\gamma_2\}<1$.
Choose an integer $T_{*}\ge 1$, independently of the error
parameter $\varepsilon$ used below, such that
$\gamma_*^{\frac {T_*}2}r<\eta,
\:
\lambda^{\frac{T_*}2}\theta\le\theta_1.$
For each $i\in\{1,2\}$ and $x\in A_i$, replace $P_i(x)$ by
$P_i(x)\cap[T_*,+\infty),$
retaining the same notation.
This finite truncation preserves the measurability and
lower-density properties of these time sets, as well as
all the estimates in Lemma \ref{Lemma 3.1}.
All the time sets and reference probabilities below are
constructed using these truncated sets.
For any $M\in\mathbb N^+$, we define 
$$E^i_{M}(x):=(P_i(x))_{M},\:E_{n,M}^i(x):=E^i_M(x)\cap [0,n),\:n\in\mathbb N^+,$$
for $i=1,2,\:x\in A_i$. For each $x\in A_i$, by the measurable diagonal argument, there exists a sequence $\mathbf p_i(x)\subset P_i(x)$ of positive integers, such that\\
$\bullet\:x\mapsto \mathbf p_i(x)$ is measurable;\\
$\bullet\:$for each $n\in\mathbb N^+$, $\left\{x\in A_i:n\in\mathbf p_i(x)\right\}$ is measurable;\\
$\bullet\:$for any $M\in\mathbb N^+$,\:$\Delta^i_{M}(x)=\lim_{n\in\mathbf p_i(x)\to+\infty}d_{n}(E^i_{M}(x))$;\\
$\bullet\:\Delta^i_{\infty}(x)=\lim_{M\to+\infty}\Delta^i_{M}(x)$.\\ 
It is a basic fact that, for any $x\in A_i$, one has
$$\liminf_{n\in\mathbf p_i(x)\to+\infty}d_{n}(E^i_{n,M}(x)\cap P_i(x))\ge \beta_i-\frac{1}{M}.$$
By Egorov's theorem, without loss of generality, we assume that
\begin{equation*}
\liminf_{n\in\mathbf p_i(x)\to+\infty}d_n(E^i_{n,M}(x)\cap P_i(x))\ge \beta_i-\frac{1}{M},\:\mathrm{uniformly\:with\:respect\:to}\:x\in A_i,
\end{equation*}
\begin{equation}\label{eq 3.7}
\Delta_{M}^i(x)=\lim_{n\in \mathbf p_i(x)}d_{n}(E^i_M(x)),\:\mathrm{uniformly\:with\:respect\:to}\:x\in A_i,
\end{equation}
\begin{equation}\label{eq 3.8}
 \Delta_{\infty}^i(x)=\lim_{M\to+\infty}\Delta^i_{M}(x),\:\mathrm{uniformly\:with\:respect\:to}\:x\in A_i.   
\end{equation}
Let $A^i_n=\left\{x\in A_i:n\in\mathbf p_i(x)\right\}$, and by the Borel-Cantelli lemma, $\mathrm{Leb}_{\mathcal D_i}(A_n^i)>n^{-2}$ for infinitely many $n\in\mathbb N^+$. We let $\mathbf n_i$ be this infinite subset of integers. For $n\in\mathbf n_i$, as $\partial E^i_{M}(x)\cap [0,n)$ determines $ E^i_M(x)\cap [0,n)$, there are at most $\sum_{k=1}^{[\frac{2n}{M}]+1}\binom{n}{k}$ choices for $E^i_M(x)\cap [0,n),\:x\in A^i_n$, and thus it may be fixed by dividing the measure of $A^i_n$ by $\sum_{k=1}^{[\frac{2n}{M}]+1}\binom{n}{k}=e^{n\delta_{n,M}}$ for some $\delta_{n,M}\to 0$, as $n\in\mathbf n_i\to+\infty$ and $M\to+\infty$. We let $E_{n,M}^i$ be the common value of $E^i_{M}(x)\cap [0,n)$ for $x$ in this new set $A_{n,M}^i$ for any $n\in\mathbf n_i$. It implies that, for any $n\in\mathbf n_i$, and any $M\in\mathbb N^+$, there exist a Borel subset $A_{n,M}^i$, a common value $E^i_{n,M}$ and a
constant $\delta_{n,M}$, such that
$$\forall\:x\in A_{n,M}^i,\:E^i_{n,M}=E^i_{n,M}(x),$$
\begin{equation}\label{eq 3.9}
\mathrm{Leb}_{\mathcal D_i}(A_{n,M}^i)\ge \frac{1}{n^2}e^{-n\delta_{n,M}}.
\end{equation}
Moreover,
$\delta_{n,M}\to 0,\:\mathrm{as}\:n\in\mathbf n_i\to+\infty\:\mathrm{and}\:M\to+\infty.$

For $M$ and $n\in\mathbf n_i$ large enough (such that $E^i_{n,M}\ne\emptyset$), we let
$$\nu_{n,M}^i:=\frac{1}{\# E^i_{n,M}}\sum_{j\in E^i_{n,M}}f^{jp}_{*}((\mathrm{Leb}_{\mathcal D_i})|_{A^i_{n,M}}),$$
where $(\mathrm{Leb}_{\mathcal D_i})|_{A^i_{n,M}}(.)=\frac{\mathrm{Leb}_{\mathcal D_i}(.\cap A^i_{n,M})}{\mathrm{Leb}_{\mathcal D_i}(A^i_{n,M})}$.
After passing to a common subsequence of $\mathbf n_i$, we may
assume that $\nu_{n,M}^i\to\nu_M^i$ as $n\to+\infty$
for every sufficiently large fixed $M$. Passing to a subsequence
in $M$, we may further assume that $\nu_M^i\to\nu_i$ as $M\to+\infty$.
The following lemma shows that $\nu_i$ is $f^p$-invariant.

\begin{lemma}\label{Lemma 3.2}
For $i\in\left\{1,2\right\}$, one has
$\nu_i\in\mathcal M(\mathbf M^d,f^p).$
\end{lemma}
\begin{proof}
 It is enough to show that, for every continuous function
\(\varphi\colon \mathbf M^d\to\mathbb{R}\),
\[
\left|\int\varphi\,d\nu_{n,M}^{i}
-
\int\varphi\circ f^p\,d\nu_{n,M}^{i}\right|
\to 0,
\:\text{as }n\in\mathbf n_i\to+\infty\:\mathrm{and}\:M\to+\infty.
\]
Indeed, for $n\in\mathbf n_i$,
\begin{align*}
&\left|\int\varphi\,d\nu_{n,M}^{i}
-
\int\varphi\circ f^p\,d\nu_{n,M}^{i}\right|\\
&\le
\frac{1}{\# E^i_{n,M}}
\int
\left|
\sum_{\substack{k\in E^i_{n,M}\\ k-1\notin E^i_{n,M}}}
\varphi\circ f^{kp}
-
\sum_{\substack{k\notin E^i_{n,M}\\ k-1\in E^i_{n,M}}}
\varphi\circ f^{kp}
\right|
\,d((\mathrm{Leb}_{\mathcal D_i})|_{A_{n,M}^i})\le \|\varphi\|_{\infty}\,
\frac{\#\partial E^i_{n,M}}{\#E_{n,M}^i},
\end{align*}
where
$\|\varphi\|_{\infty}
:=
\sup_{x\in \mathbf M^d}|\varphi(x)|.$
Therefore,
\begin{align*}
\limsup_{M\to+\infty}\limsup_{n\in\mathbf n_i\to+\infty}
\left|\int\varphi\,d\nu_{n,M}^{i}
-
\int\varphi\circ f^p\,d\nu_{n,M}^{i}\right|
&\le
\|\varphi\|_{\infty}
\limsup_{M\to+\infty}\limsup_{n\in\mathbf n_i\to+\infty}
\frac{\#\partial E^i_{n,M}}{\#E_{n,M}^i}\\
&\le
\|\varphi\|_{\infty}\limsup_{M\to+\infty}\frac{\frac{2}{M}}{\beta_i-\frac{1}{M}}
=0.
\end{align*}    
\end{proof}
It is a basic fact that $\nu_i\in\mathcal M(\Lambda,f^p)$, for any $i\in\left\{1,2\right\}$. For any $n\in\mathbf n_i,M\in\mathbb N^+$, write
$E_{n,M}^i=\cup_{j=0}^{m_n-1}[a_j,b_j),$
and let
$\overline E_{n,M}^i=\cup_{j=0}^{m_n-1}[a_j,b_j].$
Moreover, for sufficiently large \(M\) and \(n\in\mathbf n_i\) such that \(E_{n,M}^i\ne\emptyset\), we let
$$\overline\nu^i_{n,M}:=\frac{1}{\#\overline E_{n,M}^i}\sum_{j\in \overline E_{n,M}^i}f^{jp}_{*}((\mathrm{Leb}_{\mathcal D_i})|_{A_{n,M}^i}).$$
Passing, if necessary, to a further common subsequence of
$\mathbf n_i$, we may assume that, for every sufficiently large
fixed $M$, the measures $\overline\nu_{n,M}^i$ converge weakly
to a probability measure $\overline\nu_M^i$ as $n\to+\infty$
along this subsequence.
Then for any $\psi\in C(\mathbf M^d)$ with $\left\|\psi\right\|_{\infty}\le 1$, 
\begin{equation*}
\begin{aligned}
\left|\int\psi d\nu^i_{n,M}-\int\psi d\overline \nu^i_{n,M}\right|&\le\int \left|\frac{1}{\#E_{n,M}^i}\sum_{j\in E_{n,M}^i}\psi\circ f^{jp}-\frac{1}{\#\overline E_{n,M}^i}\sum_{j\in\overline E_{n,M}^i}\psi\circ f^{jp}\right|d((\mathrm{Leb}_{\mathcal D_i})|_{A^i_{n,M}})\\
&\le \frac{2(\#\overline E_{n,M}^i-\#E_{n,M}^i)}{\#\overline E_{n,M}^i}.
\end{aligned}    
\end{equation*}
Therefore, as $n\in\mathbf n_i\to+\infty$, one has
$\left|\int\psi d\nu^i_{M}-\int\psi d\overline \nu^i_{M}\right|\le\frac{\frac{2}{M}}{\beta_i}.$
It implies that
$$\nu_i=\lim_{M\to+\infty}\nu^i_M=\lim_{M\to+\infty}\overline\nu_M^i.$$

\subsection{Measures with separated Lyapunov exponents}
We aim to prove the following lemma.
\begin{lemma}\label{Lemma 3.3}
(1) Under the assumptions of Theorem \ref{thm:A} (1),
for $\nu_1$-a.e.$\:x\in\mathbf M^d$, one has
$$\lambda^{\min}_{F}(x):=\liminf_{n\to+\infty}\frac{1}{n}\log m(Df^n|_{F(x)})>a\ge\lambda^{\max}_{E}(x):=\limsup_{n\to+\infty}\frac{1}{n}\log\left\|Df^n|_{E(x)}\right\|.$$
(2) Under the assumptions of Theorem \ref{thm:A} (2),
for $\nu_2$-a.e.$\:x\in\mathbf M^d$, one has
$$\lambda^{\min}_{F}(x)>a>\lambda^{\max}_{E}(x).$$
\end{lemma}
\begin{proof}
We prove (2). The estimate in the $F$-direction in (1)
follows by the same argument, while the estimate in the
$E$-direction follows directly from the assumed bound
on $\Lambda$. Let
$$\phi^F_{p}(x)=\frac{1}{p}\log m(Df^p|_{F(x)})-\frac{a_1''+a}{2},\:x\in\overline U.$$ 
For every finite set $I\subset\mathbb N$, write
\[
  \phi^F_{p,I}(x):=\sum_{j\in I}\phi_p^F(f^{jp}x),
  \:
  \phi^F_{p,n}(x):=\phi^F_{p,\{0,\ldots,n-1\}}(x),
  \:n\ge1.
\]
For $I\ne\emptyset$, let
$\delta_x^I:=\frac{1}{\#I}\sum_{j\in I}\delta_{f^{jp}x}.$
We first prove that
$$\lim_{n\to+\infty}\frac{1}{n}\phi^F_{p,n}(x)\ge0,\:\mathrm{for}\:\nu_2\mathrm{-a.e.}\:x\in\mathbf M^d.$$
By Birkhoff's ergodic theorem, Lemma \ref{Lemma 3.2}, and the definition of $\nu_2$, the limit
$\lim_{n\to+\infty}\frac{1}{n}\phi^F_{p,n}(x)$
exists, for $\nu_2$-a.e.\:$x\in\mathbf M^d$. By \eqref{eq 3.6}, for any $x\in A_2$, for any $t_1,\:t_2\in P_2(x)$ with $t_1>t_2$, one has
\begin{equation}\label{eq 3.10}
\phi^F_{p,t_1-t_2}(f^{t_2p}x)\ge 0.
\end{equation}
For any $\alpha<0$, 
we aim to prove that
$\lim_{N\to+\infty}\nu_2(\left\{x:\frac{\phi^F_{p,N}(x)}{N}\le \alpha\right\})=0.$
It suffices to consider those $\alpha<0$ for which
$
\nu_2\left(
\left\{x:\frac{\phi^F_{p,N}(x)}{N}=\alpha\right\}
\right)=0,\:\text{for all }N\ge1.$

For $M'<N,M\in\mathbb N^+$ with $M\ge M'$, $x\in A^2_{n,M}$, and $n\in\mathbf n_2$, let
$E_{n,M}^{2,\alpha}(x):=\left\{k\in E^2_{n,M},\frac{\phi^F_{p,N}(f^{kp}x)}{N}\le\alpha\right\}.$
For $l\in\mathbb N$ and $k\in \left\{0,\cdots,N-1\right\}$ with $k+lN\in E^2_{n,M}$, the interval of integers $J_{k,l}=[k+lN,k+(l+1)N)$ may be written as
$J_{k,l}=I^1_{k,l}\bigcup I^2_{k,l}\bigcup I^3_{k,l}\bigcup I^4_{k,l},$
where\\

$\bullet\:I^1_{k,l}$ is the union of disjoint $P_2(x)$-irreducible intervals of integers of length at most $M'$ contained in $J_{k,l}$. By \eqref{eq 3.10}, we have $\phi_{p,I^1_{k,l}}^F(x)\ge 0$,  and therefore, 
$$\sum_{k,l,\:k+lN\in E_{n,M}^{2,\alpha}(x)}\phi^F_{p,I^1_{k,l}}(x)\ge 0;$$

$\bullet\:I^2_{k,l}$ is the union of at most two nonempty
intersections $J_{k,l}\cap I$, where $I$ is a
$P_2(x)$-irreducible interval of integers of length
at most $M'$ that is not entirely contained in $J_{k,l}$. As $E^2_{n,M}$ is a subset of $[0,n)$, the cardinality of $\left\{l:k+lN\in E^2_{n,M}\right\}$ is less than or equal to $\lceil\frac{n}{N}\rceil$,  and therefore 
$$\sum_{k,l,\:k+lN\in E_{n,M}^{2,\alpha}(x)}\phi^F_{p,I^2_{k,l}}(x)\ge -2NM'\lceil\frac{n}{N}\rceil\left\|\phi^F_{p}\right\|_{\infty};$$

$\bullet\:I^3_{k,l}\subset \mathbb N\setminus E^2_{n,M}$. If $I^3_{k,l}\ne\emptyset$, then $J_{k,l}$ contains an element of $\partial E^2_{n,M}$. Therefore, for a fixed $k$ we get from $\#(\cup_{l,\:I^3_{k,l}\ne\emptyset}I^3_{k,l})\le N\#\partial E^2_{n,M}$:
$$\sum_{k,l,\:k+lN\in E_{n,M}^{2,\alpha}(x)}\phi^F_{p,I^3_{k,l}}(x)\ge -N^2\#\partial E^2_{n,M}\left\|\phi^F_p\right\|_{\infty};$$

$\bullet\:I^4_{k,l}\subset E^2_{n,M}\setminus E^2_{n,M'}(x)$, and therefore, 
$$\sum_{k,l,\:k+lN\in E_{n,M}^{2,\alpha}(x)}\phi^F_{p,I^4_{k,l}}(x)\ge -N\#(E_{n,M}^{2}\setminus E_{n,M'}^2(x))\left\|\phi^F_{p}\right\|_{\infty}.$$
Combining these estimates, we have
\begin{equation*}
\begin{aligned}
\int_{\left\{\frac{\phi^F_{p,N}}{N}\le\alpha\right\}} \frac{\phi^F_{p,N}}{N}d\delta_{x}^{E^2_{n,M}}&=\frac{1}{N\#E^2_{n,M}}\sum_{k,l,\:k+lN\in E_{n,M}^{2,\alpha}(x)}\phi^F_{p,J_{k,l}}(x)\\
&\ge-\frac{1}{\# E^2_{n,M}}(N\#\partial E_{n,M}^2+2M'\lceil\frac{n}{N}\rceil+\#(E_{n,M}^{2}\setminus E_{n,M'}^2(x)))\left\|\phi^F_{p}\right\|_{\infty}.
\end{aligned}    
\end{equation*}
Integrating with respect to $(\mathrm{Leb}_{\mathcal D_2})|_{A^2_{n,M}}$, we obtain
\begin{equation}
\begin{aligned}
 &\int_{\left\{\frac{\phi^F_{p,N}}{N}\le\alpha\right\}} \frac{\phi^F_{p,N}}{N}d\nu_{n,M}^2\ge -\frac{1}{\# E^2_{n,M}}(N\#\partial E_{n,M}^2+2M'\lceil\frac{n}{N}\rceil+\sup_{x\in A^2_{n,M}}\#(E_{n,M}^{2}\setminus E_{n,M'}^2(x)))\left\|\phi^F_{p}\right\|_{\infty}\\
 &\ge -\frac{1}{d_n(E_{n,M}^2)}(Nd_n(\partial E_{n,M}^2)+\frac{2M'}{N}+O(\frac{N+M'}{n})+\sup_{x\in A^2_{n,M}}d_n(E_{n,M}^{2}\setminus E_{n,M'}^2(x)))\left\|\phi^F_{p}\right\|_{\infty}.
\end{aligned}    
\end{equation}
 Therefore, as $n\in\mathbf n_2\to+\infty$ and $M\to+\infty$, one has
\begin{equation*}
\begin{aligned}
\int_{\left\{\frac{\phi^F_{p,N}}{N}\le\alpha\right\}} \frac{\phi^F_{p,N}}{N}d\nu_{2}&=\int_{\left\{\frac{\phi^F_{p,N}}{N}<\alpha\right\}} \frac{\phi^F_{p,N}}{N}d\nu_{2}\ge\limsup_{M\to+\infty}\int_{\left\{\frac{\phi^F_{p,N}}{N}<\alpha\right\}} \frac{\phi^F_{p,N}}{N}d\nu_{M}^2\\
&\ge\limsup_{M\to+\infty}\limsup_{n\in\mathbf n_2\to+\infty}\int_{\left\{\frac{\phi^F_{p,N}}{N}\le\alpha\right\}} \frac{\phi^F_{p,N}}{N}d\nu_{n,M}^2\\
&\ge\limsup_{M\to+\infty}-\frac{2}{\beta_2-\frac{1}{M}}(\frac{N}{M}+\frac{M'}{N}+\sup_{x\in A_2}\left|\Delta_{M}^2(x)-\Delta_{M'}^2(x)\right|)\left\|\phi^F_p\right\|_{\infty}.
\end{aligned}
\end{equation*}
The preceding estimate gives
$\int_{\left\{\frac{\phi^F_{p,N}}{N}\le\alpha\right\}} \frac{\phi^F_{p,N}}{N}d\nu_{2}\ge-\frac{2}{\beta_2}(\frac{M'}{N}+\sup_{x\in A_2}\left|\Delta_{\infty}^2(x)-\Delta_{M'}^2(x)\right|)\left\|\phi^F_p\right\|_{\infty}.$
And therefore
\begin{equation}\label{3.13}
 \begin{aligned}
 &\limsup_{N\to+\infty}\nu_2(\left\{x:\frac{\phi^F_{p,N}(x)}{N}\le \alpha\right\})\le\limsup_{N\to+\infty}\frac{1}{\alpha}\int_{\left\{\frac{\phi^F_{p,N}}{N}\le\alpha\right\}} \frac{\phi^F_{p,N}}{N}d\nu_{2}\\
&\le\limsup_{N\to+\infty}\frac{1}{\alpha}\limsup_{M\to+\infty}\limsup_{n\in\mathbf n_2\to+\infty}\int_{\left\{\frac{\phi^F_{p,N}}{N}\le\alpha\right\}} \frac{\phi^F_{p,N}}{N}d\nu_{n,M}^2\le-\frac{2}{\alpha\beta_2}\sup_{x\in A_2}\left|\Delta_{\infty}^2(x)-\Delta_{M'}^2(x)\right|\left\|\phi^F_p\right\|_{\infty},    
 \end{aligned}   
\end{equation}
for all $M'\in\mathbb N^+$. Letting $M'\to+\infty$ in \eqref{3.13}, we obtain
$\lim_{N\to+\infty}\nu_2(\left\{x:\frac{\phi^F_{p,N}(x)}{N}\le \alpha\right\})=0.$
Therefore, under the assumptions of Theorem \ref{thm:A} (2), one has
$$\lim_{n\to+\infty}\frac{1}{np}\sum_{j=0}^{n-1}\log m(Df^p|_{F(f^{jp}x)})>a,\:\mathrm{for}\:\nu_2\mathrm{-a.e.}\:x\in\mathbf M^d.$$
Hence
$$\liminf_{m\to+\infty}\frac{1}{m}\log m(Df^{m}|_{F(x)})\ge\lim_{n\to+\infty}\frac{1}{np}\sum_{j=0}^{n-1}\log m(Df^p|_{F(f^{jp}x)})>a,\:\mathrm{for}\:\nu_2\mathrm{-a.e.}\:x\in\mathbf M^d.$$
 Now consider 
$$\phi^E_{p}(x)=\frac{a+a_{2}''}{2}-\frac{1}{p}\log \left\|Df^p|_{E(x)}\right\|,\:\phi^E_{p,n}(x)=\sum_{0\le j<n}\phi^E_{p}\circ f^{jp}(x),\:n\in\mathbb N,\:x\in\overline U.$$
By \eqref{eq 3.6}, for every $x\in A_2$ and every
$t_1,t_2\in P_2(x)$ with $t_1>t_2$, we have
$\phi^E_{p,t_1-t_2}(f^{t_2p}x)\ge0.$
Under the assumptions of Theorem \ref{thm:A} (2), the same argument therefore gives
 $$\lim_{n\to+\infty}\frac{1}{np}\sum_{j=0}^{n-1}\log \left\|Df^p|_{E(f^{jp}x)}\right\|<a,\:\mathrm{for}\:\nu_2\mathrm{-a.e.}\:x\in\mathbf M^d.$$
 It follows that
$$\limsup_{m\to+\infty}\frac{1}{m}\log \left\|Df^{m}|_{E(x)}\right\|\le\lim_{n\to+\infty}\frac{1}{np}\sum_{j=0}^{n-1}\log \left\|Df^p|_{E(f^{jp}x)}\right\|<a,\:\mathrm{for}\:\nu_2\mathrm{-a.e.}\:x\in\mathbf M^d.$$
\end{proof}

\subsection{Upper bound of the volume of the Bowen-like balls}
We aim to prove the following Gibbs-like property with respect to the volume measure. 
\begin{proposition}\label{Proposition 3.4}
For any $\varepsilon>0$ and any $i\in\left\{1,2\right\}$, there exists $K:=K(\varepsilon)\in\mathbb N^+$ such that, for any $M\ge N\ge K$, there exist $\delta:=\delta(\varepsilon,N)>0$ and $k:=k(\varepsilon,N)\in\mathbb N^+$ such that, for any finite Borel partition $\mathcal P$ of $\mathbf M^d$ with $\mathrm{diam}\mathcal P<\delta$, any $n\in\mathbf n_i$ with $n\ge k$, and any $x\in A_{n,M}^i$, one has
\[
\mathrm{Leb}_{\mathcal D_i}\!\left(
  \left(\bigvee_{j\in\overline E^i_{n,M}}
              f^{-jp}\mathcal P\right)(x)\cap A_{n,M}^i
\right)
\le e^{4n\varepsilon}
    \exp\!\left(-\psi^F_{p,E^i_{n,N}(x)}(x)\right),
\]
where, for $y\in\overline U$ and a finite set
$I\subset\mathbb N$, we write
$\psi^F_{p,1}(y):=
\log\left|\det\left(Df^p|_{F(y)}\right)\right|,
  \:
  \psi^F_{p,I}(y):=
  \sum_{j\in I}\psi^F_{p,1}(f^{jp}y).$
\end{proposition}
\begin{proof}
\textbf{Step 1: the bound of the volume of Bowen balls.}
\begin{lemma}\label{Lemma 3.5}
For any $\varepsilon>0$, there exist constants
\(0<\delta_\varepsilon<r_0\) and \(\theta_\varepsilon>0\) such that the following holds. Assume that $x\in \Delta(\lambda,f^p,E,F)$, and let \(\Delta\) be a \(\dim F\)-dimensional submanifold tangent to the
\(F\)-cone field \(\mathcal C_{\theta_\varepsilon}^F\) with \(x\in\Delta\). For any \(m>n\ge 0\), and any measurable subset
$$\Gamma\subset\left\{
y\in\Delta:
\mathrm{d}\bigl(f^{ip}y,f^{ip}x\bigr)<\delta_\varepsilon,
\:n\le i\le m-1
\right\}\cap\left\{
y\in\Delta:
\mathrm{d}\bigl(f^{ip}y,f^{ip}x\bigr)<r_0,
\:0\le i\le n-1
\right\},$$
one has
\[
e^{(n-m)\varepsilon}\,
\mathrm{Leb}_{f^{mp}\Delta}\bigl(f^{mp}\Gamma\bigr)
\le
\mathrm{Leb}_{f^{np}\Delta}(f^{np}\Gamma)
\prod_{j=n}^{m-1}
\left|\mathrm{det}\left(
Df^p\big|_{F(f^{jp}x)}
\right)\right|
\le
e^{(m-n)\varepsilon}\,
\mathrm{Leb}_{f^{mp}\Delta}\bigl(f^{mp}\Gamma\bigr).
\]
\end{lemma}

\begin{proof}
Fix \(\varepsilon>0\). By the continuity of \(Df^p\) and of the bundle \(F\),
there exists \(\theta_\varepsilon>0\) such that, for every
\(\dim F(z)\)-dimensional subspace
\(V\subset\mathcal C_{\theta_\varepsilon}^F(z)\) with $z\in U$, one has
\begin{equation}\label{eq 3.13}
e^{-\frac \varepsilon2}
\le
\frac{
\left|\mathrm{det}\left(D_{z}f^p\big|_{V}\right)
\right|}{
\left|\mathrm{det}\left(Df^p\big|_{F(z)}\right)
\right|}
\le
e^{\frac \varepsilon2}.
\end{equation}
Moreover, by uniform continuity of the map
$z\mapsto\log
\left| \mathrm{det}\left(Df^p\big|_{F(z)}\right)\right|,$
there exists \(0<\delta_\varepsilon<r_0\) such that, whenever
\(z,z'\in U\) satisfy \(\mathrm d(z,z')<\delta_\varepsilon\), then
\begin{equation}\label{eq 3.14}
e^{-\frac \varepsilon2}
\le
\frac{\left|
\mathrm{det}\left(Df^p\big|_{F(z)}\right)\right|
}{\left|
\mathrm{det}\!\left(Df^p\big|_{F(z')}\right)\right|
}
\le
e^{\frac \varepsilon2}.
\end{equation}

Let \(z\in f^{mp}\Gamma\), and put \(w=f^{-mp}z\in \Gamma\). For any $n\le j\le m-1$, we have
$\mathrm d\bigl(f^{jp}w,f^{jp}x\bigr)<\delta_\varepsilon,$
and Lemmas \ref{Lemma 2.3} and \ref{Lemma 2.4} give
$T_{f^{jp}w}f^{jp}\Delta\subset\mathcal C_{\theta_\varepsilon}^F(f^{jp}w).$
It follows from
\eqref{eq 3.13} and
\eqref{eq 3.14} that, for every \(n\le j\le m-1\),
\[
e^{-\varepsilon}
\le
\frac{\left|
\mathrm{det}\left(
Df^p\big|_{T_{f^{jp}w}f^{jp}\Delta}
\right)
\right|}{\left|
\mathrm{det}\left(
Df^p\big|_{F(f^{jp}x)}
\right)
\right|}
\le
e^{\varepsilon}.
\]
Multiplying these inequalities for \(j=n,\dots,m-1\), we obtain
\[
e^{(n-m)\varepsilon}
\le
\frac{\left|
\mathrm{det}\left(
Df^{(m-n)p}\big|_{T_{f^{np}w}f^{np}\Delta}
\right)\right|
}{
\displaystyle\prod_{j=n}^{m-1}
\left|\mathrm{det}\left(
Df^p\big|_{F(f^{jp}x)}\right)\right|
}
\le
e^{(m-n)\varepsilon}.
\]

Finally, the change-of-variables formula gives
\[
\mathrm{Leb}_{f^{np}\Delta}(f^{np}\Gamma)
=
\int_{f^{mp}\Gamma}\left|
\mathrm{det}\left(
Df^{(n-m)p}\big|_{T_z f^{mp}\Delta}
\right)\right|
d\mathrm{Leb}_{f^{mp}\Delta}(z).
\]
Multiplying by
$\prod_{j=n}^{m-1}
\left|\mathrm{det}\left(
Df^p\big|_{F(f^{jp}x)}
\right)\right|$
and applying the preceding pointwise estimate yields
\[
e^{(n-m)\varepsilon}\,
\mathrm{Leb}_{f^{mp}\Delta}\bigl(f^{mp}\Gamma\bigr)
\le
\mathrm{Leb}_{f^{np}\Delta}(f^{np}\Gamma)
\prod_{j=n}^{m-1}
\left|\mathrm{det}\left(
Df^p\big|_{F(f^{jp}x)}
\right)\right|
\le
e^{(m-n)\varepsilon}\,
\mathrm{Leb}_{f^{mp}\Delta}\bigl(f^{mp}\Gamma\bigr).
\]
This proves the lemma.
\end{proof}
The estimate also applies, up to a constant depending on
$\varepsilon$, when the initial disc is tangent to the
cone field $\mathcal C_\theta^F$. More precisely, retain the orbit-tracking
assumptions of Lemma \ref{Lemma 3.5} and set $T_\varepsilon
:=
\min\left\{
t\in\mathbb N:
\lambda^{\frac t2}\theta\le\theta_\varepsilon
\right\}.$
Let $L_p
:=
\max\left\{
1,\,
\sup_{z\in\mathbf M^d}\|D_zf^p\|,\,
\sup_{z\in\mathbf M^d}\|D_zf^{-p}\|
\right\},
\:
J:=2d\log L_p.$
Here $\mathrm{Jac}(D_zf^p|_H)$ denotes the positive
Jacobian of the restriction to a linear subspace
$H\subset T_z\mathbf M^d$. For every $k$-dimensional subspace $H$,
$L_p^{-k}
\le
\mathrm{Jac}(D_zf^p|_H)
\le
L_p^k.$
Consequently, for all relevant points $z,z'$ and every
$\dim F$-dimensional subspace $H\subset T_z\mathbf M^d$,
$\left|
\log\mathrm{Jac}(D_zf^p|_H)
-
\log\mathrm{Jac}(D_{z'}f^p|_{F(z')})
\right|
\le J.$
Applying the estimates of Lemma \ref{Lemma 2.3} and Lemma \ref{Lemma 2.4}
along the original orbit segments, we obtain, for every
$w\in\Gamma$,
\[
T_{f^{jp}w}(f^{jp}\Delta)
\subset
\mathcal C_{\lambda^{\frac j2}\theta}^F(f^{jp}w),
\:0\le j\le m-1.
\]
Thus the logarithmic Jacobian discrepancy is bounded by
$\varepsilon$ whenever $n\le j\le m-1$ and
$j\ge T_\varepsilon$. The remaining iterates in $[n,m)$
are at most $T_\varepsilon$ in number, and each contributes
a discrepancy bounded by $J$. Therefore, with $C:=e^{JT_\varepsilon},$
the change-of-variables argument above gives
\[
\begin{aligned}
&C^{-1}e^{-(m-n)\varepsilon}
\mathrm{Leb}_{f^{mp}\Delta}(f^{mp}\Gamma)\le
\mathrm{Leb}_{f^{np}\Delta}(f^{np}\Gamma)
\prod_{j=n}^{m-1}
\mathrm{Jac}
\bigl(D_{f^{jp}x}f^p|_{F(f^{jp}x)}\bigr)\le
Ce^{(m-n)\varepsilon}
\mathrm{Leb}_{f^{mp}\Delta}(f^{mp}\Gamma).
\end{aligned}
\]
The factor $C$ is independent of $n$ and $m$
and is absorbed into the constants below.

\paragraph{\textbf{Step 2: a claim to be proved}}
For every finite subset \(I\subset\mathbb N\), write \(\mathcal P^I:=\bigvee_{j\in I}f^{-jp}\mathcal P\).
\begin{claim}\label{Claim 3.6}
For any $i\in \left\{1,2\right\}$ and any \(\varepsilon>0\), there exists $K:=K(\varepsilon)\in\mathbb N^+$ such that, for any $M\ge N\ge K$,
there exist constants
\(\delta=\delta(\varepsilon,N)>0\) and \(C=C(\varepsilon)>1\) with the
following property. Let $m\in\mathbf n_i$ with $E^i_{m,M}\ne\emptyset$, and write
$E^i_{m,M}=\bigcup_{j=0}^{n-1}[a_j,b_j)$
as the union of its maximal nonempty integer intervals, listed
in increasing order. In particular, $b_{n-1}\le m$.
Let \(\mathcal P\) be a finite Borel partition of \(\mathbf M^d\) such that $\mathrm{diam}(\mathcal P)<\delta.$
Assume that \(\mathcal D_i\) is an $F$-disc with respect to
\((\theta,\theta_{1},r,r')\) tangent to the cone field
\(\mathcal C_{\theta}^{F}\), where $\theta<\min\left\{\frac{\theta_0}{16},\frac{\theta_0'}{16}\right\}$, $r<\min\left\{\frac{r_0}{16},\frac{r_0'}{16},\frac{r_1}{16},\frac{r_1'}{16},\frac{\rho}{16},\frac{r_3(\theta)}{16}\right\}$ and $\theta_1,\:r'$ are as in Lemma \ref{lem:uniform-cone-graph}. Define
\[
(\mathcal D_i)_{E^i_{m,M}}
:=
\left\{
z\in\mathcal D_i\cap A_i:
a_j,b_j\in P_i(z)
\text{ for any }0\le j\le n-1
\right\}.
\]
Then, for any $x\in A^i_{m,M}$,
\begin{equation}
\begin{aligned}
\mathrm{Leb}_{\mathcal D_i}
\bigl(\mathcal P^{\overline{E}^i_{m,M}}(x)\cap(\mathcal D_i)_{E^i_{m,M}}\cap\Delta(\lambda,f^p,E,F)\bigr)
\le
C^{\#\partial E^i_{m,M}+\#\partial E^i_{m,N}(x)+2}\\
e^{3b_{n-1}\varepsilon}
\prod_{j\in E^i_{m,N}(x)}\left|
\mathrm{det}
\left(
Df^p\big|_{F(f^{jp}x)}
\right)\right|^{-1}.
\end{aligned}
\end{equation}
\end{claim}
\paragraph{\textbf{Proof of Proposition \ref{Proposition 3.4} through Claim \ref{Claim 3.6}}}
The reference discs $\mathcal D_i$, the sets $A_i$, and the reference
probabilities of Section 3.2 remain fixed. The finite localizations below
estimate their volumes; they do not define new reference measures.

For every $z\in A^i_{n,M}\subset A_n^i$, the definition of $A_n^i$
gives $n\in\mathbf p_i(z)\subset P_i(z)$. Since
$E^i_{n,M}=E^i_M(z)\cap[0,n)$, each endpoint of a maximal
component of $E^i_{n,M}$ belongs to $P_i(z)$, including any
right endpoint created by truncation at $n$. Therefore, Claim \ref{Claim 3.6} yields
\begin{equation*}
\begin{aligned}
\mathrm{Leb}_{\mathcal D_i}
\bigl(\mathcal P^{\overline{E}^i_{n,M}}(x)\cap A^i_{n,M}\bigr)&\le\mathrm{Leb}_{\mathcal D_i}
\bigl(\mathcal P^{\overline{E}^i_{n,M}}(x)\cap(\mathcal D_i)_{E^i_{n,M}}\cap\Delta(\lambda,f^p,E,F)\bigr)\\
&\le
C^{(\#\partial E^i_{n,M}+\#\partial E^i_{n,N}(x)+2)} e^{3n\varepsilon}
\exp\left(
-\sum_{j\in E^i_{n,N}(x)}
\log\left|\mathrm{det}
\left(
Df^p\big|_{F(f^{jp}x)}
\right)\right|
\right).
\end{aligned}
\end{equation*}
Consecutive maximal components of $E^i_{n,N}(x)$
are separated by a gap of length greater than $N$.
Hence $\#\partial E^i_{n,N}(x)\le\frac{2n}{N}+2.$
Similarly, $\#\partial E^i_{n,M}\le\frac{2n}{M}+2.$
After all fixed losses have been absorbed into
$C=C(\varepsilon)>1$, choose
\[
K'=K'(\varepsilon):=
\max\left\{
1,\left\lceil\frac{10\log C}{\varepsilon}\right\rceil
\right\}.
\]
For $M\ge N\ge\max\{K,K'\}$ and $n\ge N$, we obtain
\[
\begin{aligned}
\left(
\#\partial E^i_{n,M}
+\#\partial E^i_{n,N}(x)+2
\right)\log C&\le
\left(\frac{2n}{M}+\frac{2n}{N}+6\right)\log C\\
&\le
\left(\frac{4n}{N}+6\right)\log C\le
\frac{10n\log C}{N}
\le n\varepsilon.
\end{aligned}
\]
Therefore,
$C^{\#\partial E^i_{n,M}
+\#\partial E^i_{n,N}(x)+2}
\le e^{n\varepsilon}.$
Increasing $k(\varepsilon,N)$ if necessary, we may
assume that $k(\varepsilon,N)\ge N$.
Consequently,
\[
\mathrm{Leb}_{\mathcal D_i}
\bigl(\mathcal P^{\overline{E}^i_{n,M}}(x)\cap A^i_{n,M}\bigr)
\le
e^{4n\varepsilon}
\exp\left(
-\sum_{j\in E^i_{n,N}(x)}
\log\left|\mathrm{det}
\left(
Df^p\big|_{F(f^{jp}x)}
\right)\right|
\right).
\]
This completes the proof of Proposition \ref{Proposition 3.4}.

\paragraph{\textbf{Proof of Claim \ref{Claim 3.6}}} The volume argument follows the gap--block decomposition of
Burguet--Yang \cite{BY25}, with an additional distinction between the
common coarse set $E^i_{m,M}$ and the point-dependent fine set
$E^i_{m,N}(x)$. The coarse endpoints are hyperbolic times for every point
being estimated. Inside a coarse block, orbit tracking transfers the
relative hyperbolic-time conditions of $x$ to these points, and the
fine blocks have successive hyperbolic time gaps at most $N$.
Uniform domination is not assumed here. Consequently, each auxiliary
domain remains a subset of its ancestors in the original disc, and
Lemmas \ref{Lemma 2.6} and \ref{Lemmas 2.6} are applied with this full
prehistory. Averaged domination is never restarted at an arbitrary
intermediate time. The atom $\mathcal P^{\overline E^i_{m,M}}(x)$ retains
the coarse right endpoints as well as the half-open blocks. We prove Claim \ref{Claim 3.6} by induction over the successive
endpoints processed in the construction.

\paragraph{\textbf{Step 3: assumptions and choice of scales}}
Assume that $E^i_{m,N}(x)=\cup_{j=0}^{n-1}\cup_{s=0}^{\ell_j-1}[a_{j,s},b_{j,s})$, where $\cup_{s=0}^{\ell_j-1}[a_{j,s},b_{j,s})\subset [a_j,b_j)$, $b_{j,s}<a_{j,s+1}$, for all $0\le j\le n-1$, for all $0\le s\le\ell_j-2$, and each $[a_{j,s},b_{j,s})$  consists of integer intervals whose endpoints are consecutive times of $P_i(x)$, with gaps no larger than \(N\). 

For the local constructions in Steps 4--6, we may assume without
loss of generality that each coarse block of $E^i_{m,M}$ contains
at least one fine block of $E^i_{m,N}(x)$ and that the interval
lengths satisfy (3)--(7) below. The reduction for a general block
configuration is given in Step 7.\\

(1) The constant $C$ may depend on $\varepsilon$ and on the fixed dynamical and geometric data of Section 3.2. It may vary from one estimate to the next and is uniform with respect to $m,n,N,M$.\\

(2) $a_0\ge T_*$. We use the intrinsic--ambient distance comparison only
on the standard $F$-discs obtained from Lemma \ref{lem:uniform-cone-graph}.
The normal-coordinate bounds in that lemma give
\[
\mathrm d_D(u,v)\le C_{\mathrm{geo}}\mathrm d(u,v),
\:
C_{\mathrm{geo}}=\frac{4(1+\theta)}{1-\theta}.
\]
For pullback domains, we use the curve-length estimates
in the proofs of Lemmas \ref{Lemma 2.6} and \ref{Lemmas 2.6} along the admissible
orbit segments. The intrinsic--ambient comparison is
made only on the terminal standard disc; no additional
comparison factor is introduced at intermediate times.\\

(3) The length of $[b_j,a_{j+1})$ is large enough such that
$$C_{\mathrm{geo}}\gamma_i^{\frac{a_{j+1}-b_j}{4}}r<\min \left\{(1-\theta)\frac{r'}{16},\frac{\delta_{\varepsilon}}{16},\eta\right\},\:\lambda^{\frac{a_{j+1}-b_j}{2}}\theta<\frac{\theta_{1}}{16},$$
for all $0\le j\le n-2$.\\

(4) The length of $[b_{j,s-1},a_{j,s})$ is large enough such that
$$C_{\mathrm{geo}}\gamma_i^{\frac{a_{j,s}-b_{j,s-1}}{4}}r<\min \left\{(1-\theta)\frac{r'}{16},\frac{\delta_{\varepsilon}}{16}\right\},\:\lambda^{\frac{a_{j,s}-b_{j,s-1}}{2}}\theta<\frac{\theta_{1}}{16},$$
for all $0\le j\le n-1$ and $1\le s\le \ell_j-1$.\\

(5) If $a_{j}<a_{j,0}$, then the length of $[a_j,a_{j,0})$ is large enough such that
$$C_{\mathrm{geo}}\gamma_i^{\frac{a_{j,0}-a_{j}}{4}}r<\min \left\{(1-\theta)\frac{r'}{16},\frac{\delta_{\varepsilon}}{16}\right\},\:\lambda^{\frac{a_{j,0}-a_{j}}{2}}\theta<\frac{\theta_{1}}{16},$$
for all $0\le j\le n-1$.\\

(6) If $b_{j,\ell_j-1}<b_j$, then the length of $[b_{j,\ell_j-1},b_{j})$ is large enough such that
$$C_{\mathrm{geo}}\gamma_i^{\frac{b_{j}-b_{j,\ell_j-1}}{4}}r<\min \left\{(1-\theta)\frac{r'}{16},\frac{\delta_{\varepsilon}}{16}\right\},\:\lambda^{\frac{b_{j}-b_{j,\ell_j-1}}{2}}\theta<\frac{\theta_{1}}{16},$$
for all $0\le j\le n-1$.\\

(7) The length of $[a_{j,s},b_{j,s})$ is large enough such that
$$C_{\mathrm{geo}}\gamma_i^{\frac{b_{j,s}-a_{j,s}}{4}}r<\min \left\{(1-\theta)\frac{r'}{16},\frac{\delta_{\varepsilon}}{16}\right\},\:\lambda^{\frac{b_{j,s}-a_{j,s}}{2}}\theta<\frac{\theta_{1}}{16},$$
for all $0\le j\le n-1$ and $0\le s\le \ell_j-1$.\\

(8) Let $Q\ge 1$ be a fixed uniform upper bound for the
volumes of the standard $F$-discs with parameters
$(\theta,\theta_1,r,r')$ and for the volumes of the
reference discs $\mathcal D_i$, where $r<\min\left\{\frac{r_0}{16},\frac{r_0'}{16},\frac{r_1}{16},\frac{r_{1}'}{16},\frac{\rho}{16},\frac{r_{3}(\theta)}{16}\right\}$. Choose an integer $G=G(\varepsilon)\ge1$ sufficiently large that
$C_{\mathrm{geo}}\gamma_i^{\frac G4}r<\delta_\varepsilon.$

(9) Let $\delta:=\delta(\varepsilon,N)<\min_{1\le j\le N,\:1\le i\le 2}\left\{(1-\theta)\frac{r'}{16},\frac{\delta_{\varepsilon}}{16},\frac{r_0}{16},\frac{r_0'}{16},\frac{r_1}{16},\frac{r_{1}'}{16},\frac{r_2(f^{jp},\theta,\lambda_i')}{16(3C_{\mathrm{geo}}+1)},\frac{\rho}{16},\frac{r_{3}(\theta)}{16},\frac{r_4(\theta,r)}{16}\right\}$, such that for any $x,y\in\overline U$ with $\mathrm d(x,y)<\delta$, one has
$$e^{-\varepsilon}\le\frac{\left|\mathrm{det
}(Df^p|_{F(x)})\right|}{\left|\mathrm{det}(Df^p|_{F(y)})\right|}\le e^{\varepsilon}.$$

\paragraph{\textbf{Step 4: the initial case of $[0,a_0)$}}
Let $\mathcal D_i$ be an $F$-disc with parameters  \((\theta,\theta_{1},r,r')\) tangent  to $\mathcal C^F_{\theta}$. Denote $\Delta_{E^i_{m,M}}:=\mathcal P^{\overline{E}^i_{m,M}}(x)\cap (\mathcal D_i)_{E^i_{m,M}}\cap\Delta(\lambda,f^p,E,F)$.
Then, there exists $y_{\varepsilon}\in \Delta_{E^i_{m,M}}$ such that
\[
\begin{aligned}
\mathrm{Leb}_{\mathcal D_i}
 (\Delta_{E^i_{m,M}})
\le C\mathrm{Leb}_{\mathcal D_i}
 (\Delta_{E^i_{m,M}}\cap B(y_\varepsilon,\delta'_\varepsilon)),
\end{aligned}
\]
    where $\delta_{\varepsilon}'=\min\left\{\frac{(1-\theta)r'}{16},\frac{\delta_{\varepsilon}}{16}\right\}$.
\begin{lemma}\label{Lemma 3.8}
There exist finitely many points $y_1^0,\ldots,y_{m_0}^0\in\Delta_{E^i_{m,M}}\cap B(y_\varepsilon,\delta'_\varepsilon)$ and finitely many embedded submanifolds $\Delta_{y^0_1}^i,\cdots,\Delta_{y^0_{m_0}}^i\subset \mathcal D_i$ tangent to the cone field $\mathcal C^{F}_{\theta}$ such that, for every $1\le s\le m_0$,
$y_s^0\in\mathrm{Int}\Delta_{y_s^0}^i$ and
$\Delta_{y_s^0}^i$ satisfies the orbit-containment and
backward-contraction conclusions of
Lemma \ref{Lemma 2.6} at the $\gamma_i$-hyperbolic time
$a_0$ of $y_s^0$ with respect to $f^p$ and $F$.
Moreover, the following properties hold:\\

$\bullet\:\text{for every }1\le s\le m_0,\:f^{a_0p}\Delta_{y^0_s}^i$ is an $F$-disc centered at $f^{a_0p}y_s^0$ with parameters \((\theta,\theta_{1},r,r')\);\\

$\bullet\:\Delta_{y^0_1}^i,\cdots,\Delta_{y^0_{m_0}}^i$ are mutually disjoint;\\

$\bullet\:\Delta_{E^i_{m,M}}\cap B(y_{\varepsilon},\delta_{\varepsilon}')\subset\cup_{j=1}^{m_0}\Delta_{y^0_j}^i$;\\

$\bullet\:\sum_{s=1}^{m_0}
\left(
\prod_{j=0}^{a_0-1}\left|
\mathrm{det}\left(
Df^p\big|_{F(f^{jp}y_s^0)}
\right)\right|
\right)^{-1}
\le Ce^{a_0\varepsilon}$.

\end{lemma}
\begin{proof}
By Lemma \ref{Lemma 2.6}, Lemma \ref{lem:uniform-cone-graph} and Lemma \ref{Lemma 2.7}, there exist finitely many points $y_1^0,\ldots,y_{m_0}^0$ in
$\Delta_{E^i_{m,M}}\cap B(y_\varepsilon,\delta'_\varepsilon)$ and finitely many embedded submanifolds $\Delta_{y^0_1}^i,\cdots,\Delta_{y^0_{m_0}}^i\subset \mathcal D_i$ tangent to the cone field $\mathcal C^{F}_{\theta}$ such that, for every $1\le s\le m_0$,
$y_s^0\in\mathrm{Int}\Delta_{y_s^0}^i$ and
$\Delta_{y_s^0}^i$ satisfies the orbit-containment and
backward-contraction conclusions of
Lemma \ref{Lemma 2.6} at the $\gamma_i$-hyperbolic time
$a_0$ of $y_s^0$ with respect to $f^p$ and $F$.
Moreover, the following properties hold:\\

$\bullet\:\text{for every }1\le s\le m_0,\:f^{a_0p}\Delta_{y^0_s}^i$ is an $F$-disc centered at $f^{a_0p}y_s^0$ with parameters \((\theta,\theta_{1},r,r')\);\\

$\bullet\:\Delta_{y^0_1}^i,\cdots,\Delta_{y^0_{m_0}}^i$ are mutually disjoint;\\

$\bullet\:\Delta_{E^i_{m,M}}\cap B(y_{\varepsilon},\delta_{\varepsilon}')\subset\bigcup_{j=1}^{m_0}\Delta_{y^0_j}^i.$\\

For every $1\le s\le m_0$, $y\in\Delta_{y_s^0}^i$,
and $0\le k\le a_0-G$, the backward-contraction
estimate and the intrinsic--ambient distance comparison
on the terminal standard disc give
\[
\begin{aligned}
\mathrm d(f^{kp}y,f^{kp}y_s^0)
&\le
\gamma_i^{\frac{a_0-k}2}
\mathrm d_{f^{a_0p}\Delta_{y_s^0}^i}
\bigl(f^{a_0p}y,f^{a_0p}y_s^0\bigr)\\
&\le
C_{\mathrm{geo}}\gamma_i^{\frac{a_0-k}2}r\le
C_{\mathrm{geo}}\gamma_i^{\frac G4}r
<\delta_\varepsilon.
\end{aligned}
\]
Consequently,
$\Delta_{y_s^0}^i
\subset
\left\{
y\in\mathcal D_i:
\mathrm d(f^{kp}y,f^{kp}y_s^0)<\delta_\varepsilon
\text{ for every }0\le k\le a_0-G
\right\}.$
Applying Lemma \ref{Lemma 3.5} and absorbing the finite initial
and terminal time losses into $C$, we obtain
\begin{equation}
\begin{aligned}
\mathrm{Leb}_{\mathcal D_i}
\left(
\Delta_{y_s^0}^i
\right)
&\ge C^{-1}\mathrm{Leb}_{f^{a_0p}\mathcal D_i}
\left(
f^{a_0p}\Delta_{y_s^0}^i
\right)
 e^{-a_0\varepsilon}
\left(
\prod_{j=0}^{a_0-1}\left|
\mathrm{det}\left(
Df^p\big|_{F(f^{jp}y_{s}^0)}
\right)\right|
\right)^{-1} \\
&\ge
 C^{-1}e^{-a_0\varepsilon}
\left(
\prod_{j=0}^{a_0-1}\left|
\mathrm{det}\left(
Df^p\big|_{F(f^{jp}y_{s}^0)}
\right)\right|
\right)^{-1}.
\end{aligned}
\end{equation}
Since the sets $\Delta^i_{y^0_{s}},\:1\le s\le m_0$, are mutually disjoint, by Lemma \ref{Lemma 2.10}, we obtain
\begin{equation}
\begin{aligned}
Q&\ge
\mathrm{Leb}_{\mathcal D_i}
\bigl(\mathcal D_i\bigr)\ge
\sum_{s=1}^{m_0}
\mathrm{Leb}_{\mathcal D_i}
\left(\Delta_{y_s^0}^i
\right)
\\
&\ge
 C^{-1}e^{-a_0\varepsilon}
\sum_{s=1}^{m_0}\left(
\prod_{j=0}^{a_0-1}\left|
\mathrm{det}\left(
Df^p\big|_{F(f^{jp}y_{s}^0)}
\right)\right|
\right)^{-1}.
\end{aligned}
\end{equation}
 Therefore,
\[
\sum_{s=1}^{m_0}\left(
\prod_{j=0}^{a_0-1}\left|
\mathrm{det}\left(
Df^p\big|_{F(f^{jp}y_{s}^0)}
\right)\right|
\right)^{-1}
\le Ce^{a_0\varepsilon}.
\]
\end{proof}
By the definition of the points \(y_{s}^0\:(1\le s\le m_0)\) and Lemma \ref{Lemma 3.5}, we obtain
\begin{align*}
\mathrm{Leb}_{\mathcal D_i}
\bigl(\Delta_{E^i_{m,M}}\cap B(y_{\varepsilon},\delta_{\varepsilon}')\bigr)&\le \sum_{s=1}^{m_0}\mathrm{Leb}_{\mathcal D_i}
\bigl(\Delta^{a_0,s}_{E^i_{m,M}}\bigr)\\
&\le C\sum_{s=1}^{m_0}\mathrm{Leb}_{f^{a_0p}\mathcal D_i}
\bigl(f^{a_0p}\Delta^{a_{0},s}_{E^i_{m,M}}\bigr)e^{a_0\varepsilon}
\left(
\prod_{j=0}^{a_0-1}\left|
\mathrm{det}\left(
Df^p\big|_{F(f^{jp}y_{s}^0)}
\right)\right|
\right)^{-1},
\end{align*}
where $\Delta^{a_{0},s}_{E^i_{m,M}}:=\mathcal P^{\overline{E}^i_{m,M}}(x)\cap (\Delta_{y^0_s}^i)_{E^i_{m,M}}\cap\Delta(\lambda,f^p,E,F)$, for all $1\le s\le m_0$.
Consequently, there exists $1\le s_0\le m_0$ such that
\begin{equation}
\begin{aligned}
&\mathrm{Leb}_{\mathcal D_i}
\bigl(\Delta_{E^i_{m,M}}\bigr)\le Ce^{2a_0\varepsilon}\mathrm{Leb}_{f^{a_0p}\mathcal D_i}
\bigl(f^{a_0p}\Delta^{a_0,s_0}_{E^i_{m,M}}\bigr).   
\end{aligned}    
\end{equation}

\paragraph{\textbf{Step 5: the case of $a_0<a_{0,0}$}}
The following lemma is analogous to Lemma \ref{Lemma 3.8}.
\begin{lemma}\label{Lemma 3.9}
There exist finitely many points $y_{1}^{0,0},\cdots,y_{m_{0,0}}^{0,0}\in\Delta^{a_0,s_0}_{E^i_{m,M}}$ and finitely many embedded submanifolds $\Delta_{y^{0,0}_1}^i,\cdots,\Delta_{y^{0,0}_{m_{0,0}}}^i\subset \Delta^i_{y^0_{s_0}}$ tangent to the cone field $\mathcal C^{F}_{\theta}$ such that, for every $1\le s\le m_{0,0}$,
$y_s^{0,0}\in\mathrm{Int}\Delta_{y_s^{0,0}}^i$ and
$\Delta_{y_s^{0,0}}^i$ satisfies the orbit-containment and
backward-contraction conclusions of
Lemma \ref{Lemmas 2.6} at the $\gamma_i^{\frac{1}{2}}$-hyperbolic time
$a_{0,0}$ of $y_s^{0,0}$ with respect to $f^p$, $F$ and $a_0$.
Moreover, the following properties hold:\\

$\bullet\:\text{for every }1\le s\le m_{0,0},\:f^{a_{0,0}p}\Delta_{y^{0,0}_s}^i$ is an $F$-disc centered at $f^{a_{0,0}p}y^{0,0}_s$ with parameters  $(\theta,\theta_1,r,r')$;\\

$\bullet\:\Delta_{y^{0,0}_1}^i,\cdots,\Delta_{y^{0,0}_{m_{0,0}}}^i$ are mutually disjoint;\\

$\bullet\:\Delta^{a_0,s_0}_{E^i_{m,M}}\subset\cup_{j=1}^{m_{0,0}}\Delta_{y^{0,0}_j}^i$;\\

$\bullet\:\sum_{s=1}^{m_{0,0}}
\left(
\prod_{j=a_0}^{a_{0,0}-1}\left|
\mathrm{det}\left(
Df^p\big|_{F(f^{jp}y_s^{0,0})}
\right)\right|
\right)^{-1}
\le Ce^{(a_{0,0}-a_0)\varepsilon}$.

\end{lemma}
\begin{proof}
 By the choice of $\delta$, for any $z\in \Delta^{a_0,s_0}_{E^i_{m,M}}$, $a_{0,0}$ is a $\gamma_i^{\frac{1}{2}}$-hyperbolic time with respect to $f^p$, $F$ and $a_0$. And by the assumptions of Step 3, one has
 $$\mathrm{d}_{f^{a_0p}\Delta^i_{y^0_{s_0}}}(f^{a_0p}z,\partial f^{a_0p}\Delta^i_{y^0_{s_0}})>\gamma_i^{\frac{a_{0,0}-a_0}{4}}r,$$
 for all $z\in \Delta^{a_0,s_0}_{E^i_{m,M}}$. By Lemma \ref{Lemmas 2.6}, Lemma \ref{lem:uniform-cone-graph} and Lemma \ref{Lemma 2.7}, there exist finitely many points $y_1^{0,0},\cdots,y_{m_{0,0}}^{0,0}\in \Delta^{a_0,s_0}_{E^i_{m,M}}$ and finitely many embedded submanifolds $\Delta_{y^{0,0}_1}^i,\cdots,\Delta_{y^{0,0}_{m_{0,0}}}^i\subset \Delta^i_{y^0_{s_0}}$ tangent to the cone field $\mathcal C^{F}_{\theta}$ such that, for every $1\le s\le m_{0,0}$,
$y_s^{0,0}\in\mathrm{Int}\Delta_{y_s^{0,0}}^i$ and
$\Delta_{y_s^{0,0}}^i$ satisfies the orbit-containment and
backward-contraction conclusions of
Lemma \ref{Lemmas 2.6} at the $\gamma_i^{\frac{1}{2}}$-hyperbolic time
$a_{0,0}$ of $y_s^{0,0}$ with respect to $f^p$, $F$ and $a_0$.
Moreover, the following properties hold:\\

$\bullet\:\text{for every }1\le s\le m_{0,0},\:f^{a_{0,0}p}\Delta_{y^{0,0}_s}^i$ is an $F$-disc centered at $f^{a_{0,0}p}y^{0,0}_s$ with parameters $(\theta,\theta_1,r,r')$;\\

$\bullet\:\Delta_{y^{0,0}_1}^i,\cdots,\Delta_{y^{0,0}_{m_{0,0}}}^i$ are mutually disjoint;\\

$\bullet\:\Delta^{a_0,s_0}_{E^i_{m,M}}\subset\cup_{j=1}^{m_{0,0}}\Delta_{y^{0,0}_j}^i.$\\

For every $1\le s\le m_{0,0}$, $y\in\Delta_{y_s^{0,0}}^i$,
and $0\le k\le a_{0,0}-G$, the backward-contraction
estimate and the intrinsic--ambient distance comparison
on the terminal standard disc give
\[
\begin{aligned}
\mathrm d(f^{kp}y,f^{kp}y_s^{0,0})
&\le
\gamma_i^{\frac{a_{0,0}-k}4}
\mathrm d_{f^{a_{0,0}p}\Delta_{y_s^{0,0}}^i}
\bigl(f^{a_{0,0}p}y,f^{a_{0,0}p}y_s^{0,0}\bigr)\\
&\le
C_{\mathrm{geo}}\gamma_i^{\frac{a_{0,0}-k}4}r\le
C_{\mathrm{geo}}\gamma_i^{\frac G4}r
<\delta_\varepsilon.
\end{aligned}
\]
Indeed, for $k<a_0$, let $\sigma$ be a curve in the terminal disc
joining the two endpoints in the first inequality. Since
$\Delta^i_{y_s^{0,0}}\subset\Delta^i_{y_{s_0}^0}$, its pullback
to time $a_0$ lies in $f^{a_0p}\Delta^i_{y_{s_0}^0}$.
Applying the relative backward-contraction estimate from $a_{0,0}$
to $a_0$ and then the curve-length estimate from Step 4 from $a_0$
to $k$ gives a total length factor at most
$\gamma_i^{\frac{a_{0,0}-a_0}4}\gamma_i^{\frac{a_0-k}2}
\le \gamma_i^{\frac{a_{0,0}-k}4}$.
Taking the infimum over such curves gives the first inequality.
Consequently,
$\Delta_{y_s^{0,0}}^i
\subset
\left\{
y\in\mathcal D_i:
\mathrm d(f^{kp}y,f^{kp}y_s^{0,0})<\delta_\varepsilon
\text{ for every }0\le k\le a_{0,0}-G
\right\}.$
Therefore,
Lemma \ref{Lemma 3.5} gives
\begin{equation}
\begin{aligned}
\mathrm{Leb}_{f^{a_{0}p}\mathcal D_i}
\left(
f^{a_0p}\Delta_{y_{s}^{0,0}}^i
\right)
&\ge C^{-1}\mathrm{Leb}_{f^{a_{0,0}p}\mathcal D_i}
\left(
f^{a_{0,0}p}\Delta_{y_s^{0,0}}^i
\right)
 e^{(a_0-a_{0,0})\varepsilon}
\left(
\prod_{j=a_0}^{a_{0,0}-1}\left|
\mathrm{det}\left(
Df^p\big|_{F(f^{jp}y_{s}^{0,0})}
\right)\right|
\right)^{-1} \\
&\ge
 C^{-1}e^{(a_0-a_{0,0})\varepsilon}
\left(
\prod_{j=a_0}^{a_{0,0}-1}\left|
\mathrm{det}\left(
Df^p\big|_{F(f^{jp}y_{s}^{0,0})}
\right)\right|
\right)^{-1}.
\end{aligned}
\end{equation}
Since the sets $\Delta^i_{y^{0,0}_{s}},\:1\le s\le m_{0,0}$, are mutually disjoint, by Lemma \ref{Lemma 2.10}, we obtain
\begin{equation}
\begin{aligned}
Q&\ge
\mathrm{Leb}_{f^{a_{0}p}\mathcal D_i}
\bigl(f^{a_{0}p}\Delta^i_{y^0_{s_0}}\bigr) \ge
\sum_{s=1}^{m_{0,0}}
\mathrm{Leb}_{f^{a_{0}p}\mathcal D_i}
\left(
f^{a_{0}p}\Delta_{y_s^{0,0}}^i
\right)\\
&\ge
 C^{-1}e^{(a_0-a_{0,0})\varepsilon}
\sum_{s=1}^{m_{0,0}}\left(
\prod_{j=a_0}^{a_{0,0}-1}\left|
\mathrm{det}\left(
Df^p\big|_{F(f^{jp}y_{s}^{0,0})}
\right)\right|
\right)^{-1}.
\end{aligned}
\end{equation}
Therefore,
\[
\sum_{s=1}^{m_{0,0}}\left(
\prod_{j=a_0}^{a_{0,0}-1}\left|
\mathrm{det}\left(
Df^p\big|_{F(f^{jp}y_{s}^{0,0})}
\right)\right|
\right)^{-1}
\le Ce^{(a_{0,0}-a_{0})\varepsilon}.
\]
\end{proof}
By the definition of the points \(y_{s}^{0,0}\:(1\le s\le m_{0,0})\) and Lemma \ref{Lemma 3.5}, we obtain
\begin{align*}
\mathrm{Leb}_{f^{a_0p}\mathcal D_i}
\bigl(f^{a_0p}\Delta^{a_0,s_0}_{E^i_{m,M}}\bigr)&\le \sum_{s=1}^{m_{0,0}}\mathrm{Leb}_{f^{a_{0}p}\mathcal D_i}
\bigl(f^{a_{0}p}\Delta^{a_{0,0},s}_{E^i_{m,M}}\bigr)\\
&\le C\sum_{s=1}^{m_{0,0}}\mathrm{Leb}_{f^{a_{0,0}p}\mathcal D_i}
\bigl(f^{a_{0,0}p}\Delta^{a_{0,0},s}_{E^i_{m,M}}\bigr)e^{(a_{0,0}-a_0)\varepsilon}
\left(
\prod_{j=a_0}^{a_{0,0}-1}\left|
\mathrm{det}\left(
Df^p\big|_{F(f^{jp}y_{s}^{0,0})}
\right)\right|
\right)^{-1},
\end{align*}
where $\Delta^{a_{0,0},s}_{E^i_{m,M}}:=\mathcal P^{\overline{E}^i_{m,M}}(x)\cap(\Delta_{y^{0,0}_s}^i)_{E^i_{m,M}}\cap\Delta(\lambda,f^p,E,F)$, for all $1\le s\le m_{0,0}$.
Consequently, there exists $1\le s_{0,0}\le m_{0,0}$ such that
\begin{align*}
&\mathrm{Leb}_{f^{a_0p}\mathcal D_i}
\bigl(f^{a_0p}\Delta^{a_{0},s_0}_{E^i_{m,M}}\bigr)\le Ce^{2(a_{0,0}-a_0)\varepsilon}\mathrm{Leb}_{f^{a_{0,0}p}\mathcal D_i}
\bigl(f^{a_{0,0}p}\Delta^{a_{0,0},s_{0,0}}_{E^i_{m,M}}\bigr).
\end{align*}
Therefore,
\begin{equation}
\begin{aligned}
&\mathrm{Leb}_{\mathcal D_i}
\bigl(\Delta_{E^i_{m,M}}\bigr)\le Ce^{2a_{0,0}\varepsilon}\mathrm{Leb}_{f^{a_{0,0}p}\mathcal D_i}
\bigl(f^{a_{0,0}p}\Delta^{a_{0,0},s_{0,0}}_{E^i_{m,M}}\bigr).
\end{aligned}    
\end{equation}

\begin{lemma}\label{Lemma 3.10}
For any $z\in\Delta^{a_{0,0},s_{0,0}}_{E^i_{m,M}}$, there exists an embedded submanifold $\Delta_{z,b_{0,0}}^i\subset\Delta^i_{y^{0,0}_{s_{0,0}}}$ tangent to the cone field $\mathcal C^F_{\theta}$
such that $z\in\mathrm{Int}\Delta_{z,b_{0,0}}^i$ and
$\Delta_{z,b_{0,0}}^i$ satisfies the orbit-containment and
backward-contraction conclusions of
Lemma \ref{Lemmas 2.6} at the $\gamma_i^{\frac{1}{2}}$-hyperbolic time
$b_{0,0}$ of $z$ with respect to $f^p$, $F$ and $a_{0,0}$.
Moreover, the following properties hold:\\

$\bullet\:f^{b_{0,0}p}\Delta_{z,b_{0,0}}^i$ is an $F$-disc centered at $f^{b_{0,0}p}z$ with parameters  $(\theta,\theta_{1},r,r')$;\\

$\bullet\:
\Delta^{a_{0,0},s_{0,0}}_{E^i_{m,M}}\subset \Delta_{z,b_{0,0}}^i$.
\end{lemma}
\begin{proof}
  By the choice of $\delta$, for any $z\in \Delta^{a_{0,0},s_{0,0}}_{E^i_{m,M}}$, $b_{0,0}$ is a $\gamma_i^{\frac{1}{2}}$-hyperbolic time with respect to $f^p$, $F$ and $a_{0,0}$. And by the assumptions of Step 3, one has
$$\mathrm{d}_{f^{a_{0,0}p}\Delta^i_{y^{0,0}_{s_{0,0}}}}(f^{a_{0,0}p}z,\partial f^{a_{0,0}p}\Delta^i_{y^{0,0}_{s_{0,0}}})>\gamma_i^{\frac{b_{0,0}-a_{0,0}}{4}}r,$$
 for all $z\in \Delta^{a_{0,0},s_{0,0}}_{E^i_{m,M}}$.  By Lemma \ref{Lemmas 2.6} and Lemma \ref{lem:uniform-cone-graph}, for any \(z,z'\in\Delta^{a_{0,0},s_{0,0}}_{E^i_{m,M}}\),
there exist embedded submanifolds $\Delta_{z,b_{0,0}}^i\subset\Delta^i_{y^{0,0}_{s_{0,0}}}$ and $\Delta_{z',b_{0,0}}^i\subset\Delta^i_{y^{0,0}_{s_{0,0}}}$ tangent to the cone field $\mathcal C^F_{\theta}$ such that $z\in\mathrm{Int}\Delta_{z,b_{0,0}}^i,\:z'\in\mathrm{Int}\Delta_{z',b_{0,0}}^i$ and
$\Delta_{z,b_{0,0}}^i,\:\Delta_{z',b_{0,0}}^i$ satisfy the orbit-containment and
backward-contraction conclusions of
Lemma \ref{Lemmas 2.6} at the $\gamma_i^{\frac{1}{2}}$-hyperbolic time
$b_{0,0}$ of $z,\:z'$ with respect to $f^p$, $F$ and $a_{0,0}$. Moreover,  $f^{b_{0,0}p}\Delta_{z,b_{0,0}}^i,\:f^{b_{0,0}p}\Delta_{z',b_{0,0}}^i$ are $F$-discs centered at $f^{b_{0,0}p}z,\:f^{b_{0,0}p}z'$ with parameters  $(\theta,\theta_{1},r,r')$, respectively.

By Lemma \ref{Lemma 2.7}, either
\[
f^{b_{0,0}p}\Delta^i_{z,b_{0,0}}\cap
B(f^{b_{0,0}p}x,\delta)
=
f^{b_{0,0}p}\Delta^i_{z',b_{0,0}}\cap
B(f^{b_{0,0}p}x,\delta),
\]
in which case $f^{b_{0,0}p}z'\in f^{b_{0,0}p}\Delta^i_{z,b_{0,0}},$
or
\[
f^{b_{0,0}p}\Delta^i_{z,b_{0,0}}
\cap
f^{b_{0,0}p}\Delta^i_{z',b_{0,0}}
=
\emptyset.
\]

Suppose that the latter alternative holds. Then there exist points
$f^{b_{0,0}p}y_1\in
f^{b_{0,0}p}\Delta^i_{z,b_{0,0}}
\cap
B(f^{b_{0,0}p}x,2\delta)$
and $f^{b_{0,0}p}y_2\in
f^{b_{0,0}p}\Delta^i_{z',b_{0,0}}
\cap
B(f^{b_{0,0}p}x,2\delta)$
such that
\[
f^{b_{0,0}p}y_1
\not\sim_{F,\theta}
f^{b_{0,0}p}y_2.
\]
Set $z_1=z$ and $z_2=z'$.
For $s\in\{1,2\}$ and
$a_{0,0}\le k\le b_{0,0}$, the backward-contraction
estimate and the intrinsic--ambient distance comparison
on the terminal standard disc give
\[
\begin{aligned}
\mathrm d(f^{kp}y_s,f^{kp}x)
&\le
\mathrm d(f^{kp}y_s,f^{kp}z_s)
+\mathrm d(f^{kp}z_s,f^{kp}x)\\
&\le
C_{\mathrm{geo}}
\gamma_i^{\frac{b_{0,0}-k}4}
\mathrm d(f^{b_{0,0}p}y_s,f^{b_{0,0}p}z_s)
+\delta\\
&\le
(3C_{\mathrm{geo}}+1)\delta
<\min_{1\le j\le N,\:1\le i\le 2}r_2(f^{jp},\theta,\lambda_i').
\end{aligned}
\]
Here we used $\mathrm d(f^{b_{0,0}p}y_s,f^{b_{0,0}p}x)\le2\delta$
and
$\mathrm d(f^{kp}z_s,f^{kp}x)<\delta,
\: a_{0,0}\le k\le b_{0,0}.$
Consequently,
$f^{kp}y_s\in
B\bigl(f^{kp}x,r_2(f^{jp},\theta,\lambda_i')\bigr)$
for every $a_{0,0}\le k\le b_{0,0}$,
$s\in\{1,2\}$, and $1\le j\le N$.
Since
\(f^{a_{0,0}p}y_1,f^{a_{0,0}p}y_2\) belong to the same \(F\)-disc \(f^{a_{0,0}p}\Delta^i_{y^{0,0}_{s_{0,0}}}\), we have
$
f^{a_{0,0}p}y_1\sim_{F,\theta} f^{a_{0,0}p}y_2.$
As $[a_{0,0},b_{0,0})$ consists of integer intervals whose endpoints are consecutive elements of $P_i(x)$, with gaps no larger than \(N\), Lemma \ref{Lemmas 2.7} and Lemma \ref{Lemma 3.1} imply
$
f^{b_{0,0}p}y_1\sim_{F,\theta} f^{b_{0,0}p}y_2.
$
Consequently, the geodesic segment joining
\(f^{b_{0,0}p}y_1\) and \(f^{b_{0,0}p}y_2\) must be tangent to
\(\mathcal C^F_{\theta}\), contradicting the fact that 
\[f^{b_{0,0}p}y_1
\not\sim_{F,\theta}
f^{b_{0,0}p}y_2.
\]

Therefore, the second alternative is impossible. Hence
$f^{b_{0,0}p}z'\in f^{b_{0,0}p}\Delta^i_{z,b_{0,0}}$
for every \(z'\in\Delta^{a_{0,0},s_{0,0}}_{E^i_{m,M}}\), and consequently
\[
\Delta^{a_{0,0},s_{0,0}}_{E^i_{m,M}}\subset \Delta_{z,b_{0,0}}^i.
\]
This proves the lemma.
\end{proof}
Lemma \ref{Lemma 3.10} implies that
\begin{equation*}
\begin{aligned}
\mathrm{Leb}_{f^{a_{0,0}p}\mathcal D_i}(f^{a_{0,0}p}\Delta^{a_{0,0},s_{0,0}}_{E^i_{m,M}})
&\le
C\mathrm{Leb}_{f^{b_{0,0}p}\mathcal D_i}
\bigl(f^{b_{0,0}p}\Delta^{a_{0,0},s_{0,0}}_{E^i_{m,M}}\bigr)e^{(b_{0,0}-a_{0,0})\varepsilon}
\left(
\prod_{j=a_{0,0}}^{b_{0,0}-1}\left|
\mathrm{det}
\left(
Df^p\big|_{F(f^{jp}y^{0,0}_{s_{0,0}})}
\right)\right|
\right)^{-1}\\
&\le C\mathrm{Leb}_{f^{b_{0,0}p}\mathcal D_i}
\bigl(f^{b_{0,0}p}\Delta^{a_{0,0},s_{0,0}}_{E^i_{m,M}}\bigr)e^{2(b_{0,0}-a_{0,0})\varepsilon}
\left(
\prod_{j=a_{0,0}}^{b_{0,0}-1}\left|
\mathrm{det}
\left(
Df^p\big|_{F(f^{jp}x)}
\right)\right|
\right)^{-1}.
\end{aligned}
\end{equation*}
We choose $z_{0,0}\in\Delta^{a_{0,0},s_{0,0}}_{E^i_{m,M}}$, and let
$\Delta^{z_{0,0},b_{0,0}}_{E^i_{m,M}}:=\mathcal P^{\overline{E}^i_{m,M}}(x)\cap(\Delta_{z_{0,0},b_{0,0}}^i)_{E^i_{m,M}}\cap \Delta(\lambda,f^p,E,F)$. Then one has
\begin{equation}\label{eq 3.16}
\begin{aligned}
\mathrm{Leb}_{f^{a_{0,0}p}\mathcal D_i}(f^{a_{0,0}p}\Delta^{a_{0,0},s_{0,0}}_{E^i_{m,M}})\le
C\mathrm{Leb}_{f^{b_{0,0}p}\mathcal D_i}
\bigl(f^{b_{0,0}p}\Delta^{z_{0,0},b_{0,0}}_{E^i_{m,M}}\bigr)e^{2(b_{0,0}-a_{0,0})\varepsilon}
\left(
\prod_{j=a_{0,0}}^{b_{0,0}-1}\left|
\mathrm{det}
\left(
Df^p\big|_{F(f^{jp}x)}
\right)\right|
\right)^{-1}.
\end{aligned}
\end{equation}
Therefore,
\begin{equation}
\begin{aligned}
\mathrm{Leb}_{\mathcal D_i}
\bigl(\Delta_{E^i_{m,M}}\bigr)\le C^3e^{2b_{0,0}\varepsilon}\left(
\prod_{j=a_{0,0}}^{b_{0,0}-1}\left|
\mathrm{det}
\left(
Df^p\big|_{F(f^{jp}x)}
\right)\right|
\right)^{-1}\mathrm{Leb}_{f^{b_{0,0}p}\mathcal D_i}
\bigl(f^{b_{0,0}p}\Delta^{z_{0,0},b_{0,0}}_{E^i_{m,M}}\bigr).
\end{aligned}    
\end{equation}

\paragraph{\textbf{Step 6: the case of $a_{0}=a_{0,0}$}}

\begin{lemma}\label{Lemma 3.7}
For any $z\in\Delta^{a_{0},s_0}_{E^{i}_{m,M}}$, there exists an embedded submanifold $\Delta_{z,b_{0,0}}^i\subset\Delta^i_{y^{0}_{s_{0}}}$ tangent to the cone field $\mathcal C^F_{\theta}$
such that $z\in\mathrm{Int}\Delta_{z,b_{0,0}}^i$ and
$\Delta_{z,b_{0,0}}^i$ satisfies the orbit-containment and
backward-contraction conclusions of
Lemma \ref{Lemmas 2.6} at the $\gamma_i^{\frac{1}{2}}$-hyperbolic time
$b_{0,0}$ of $z$ with respect to $f^p$, $F$ and $a_{0,0}$.
Moreover, the following properties hold:\\

$\bullet\:f^{b_{0,0}p}\Delta_{z,b_{0,0}}^i$ is an $F$-disc centered at $f^{b_{0,0}p}z$ with parameters $(\theta,\theta_{1},r,r')$;\\

$\bullet\:\Delta^{a_{0},s_0}_{E^{i}_{m,M}}\subset \Delta_{z,b_{0,0}}^i$.
\end{lemma}
\begin{proof}
Use the proof of Lemma \ref{Lemma 3.10} with
$a_{0,0}=a_0$ and with the parent domain
$\Delta^i_{y^{0,0}_{s_{0,0}}}$ replaced by
$\Delta^i_{y^0_{s_0}}$. We verify the hypotheses needed for this
substitution. The latter domain has the full prehistory furnished by
Step 4, and its image at $a_{0,0}=a_0$ is a standard $F$-disc.
For every \(z\in\Delta^{a_0,s_0}_{E^i_{m,M}}\), orbit tracking
on the endpoint-completed coarse atom makes $b_{0,0}$ a
$\gamma_i^{\frac 12}$-hyperbolic time relative to $a_{0,0}$.
Step 3 gives
\[\mathrm d_{f^{a_{0,0}p}\Delta^i_{y^0_{s_0}}}
 \bigl(f^{a_{0,0}p}z,\partial f^{a_{0,0}p}\Delta^i_{y^0_{s_0}}\bigr)
 >\gamma_i^{\frac{b_{0,0}-a_{0,0}}{4}}r.
\]
Thus Lemmas \ref{Lemmas 2.6} and \ref{lem:uniform-cone-graph}
provide the required pullback domains and standard terminal discs.

The remaining containment argument is unchanged. If two terminal
discs were disjoint, Lemma \ref{Lemma 2.7} would give the points
$y_1,y_2$ used in the proof of Lemma \ref{Lemma 3.10}.
Backward contraction and orbit tracking give, for $s=1,2$ and
$a_{0,0}\le k\le b_{0,0}$,
\[
\mathrm  d(f^{kp}y_s,f^{kp}x)
 \le(3C_{\mathrm{geo}}+1)\delta
 <\min_{1\le j\le N,\,1\le i\le2}r_2(f^{jp},\theta,\lambda_i').
\]
Their images at $a_{0,0}$ belong to the same standard $F$-disc,
so they are $\sim_{F,\theta}$-related. The consecutive Pliss-time
gaps in $[a_{0,0},b_{0,0})$ are at most $N$.
Lemmas \ref{Lemmas 2.7} and \ref{Lemma 3.1} therefore preserve this
relation up to $b_{0,0}$, contradicting the terminal separation.
The other alternative in Lemma \ref{Lemma 2.7} therefore gives
\(\Delta^{a_0,s_0}_{E^i_{m,M}}\subset\Delta^i_{z,b_{0,0}}\),
as required.
\end{proof} 
Lemma \ref{Lemma 3.7} implies that
\begin{equation*}
\begin{aligned}
\mathrm{Leb}_{f^{a_{0,0}p}\mathcal D_i}(f^{a_{0,0}p}\Delta^{a_{0},s_0}_{E^{i}_{m,M}})&\le
C\mathrm{Leb}_{f^{b_{0,0}p}\mathcal D_i}
\bigl(f^{b_{0,0}p}\Delta^{a_{0},s_0}_{E^{i}_{m,M}}\bigr)e^{(b_{0,0}-a_{0,0})\varepsilon}
\left(
\prod_{j=a_{0,0}}^{b_{0,0}-1}\left|
\mathrm{det}
\left(
Df^p\big|_{F(f^{jp}y^0_{s_0})}
\right)\right|
\right)^{-1}\\
&\le C\mathrm{Leb}_{f^{b_{0,0}p}\mathcal D_i}
\bigl(f^{b_{0,0}p}\Delta^{a_{0},s_0}_{E^{i}_{m,M}}\bigr)e^{2(b_{0,0}-a_{0,0})\varepsilon}
\left(
\prod_{j=a_{0,0}}^{b_{0,0}-1}\left|
\mathrm{det}
\left(
Df^p\big|_{F(f^{jp}x)}
\right)\right|
\right)^{-1}.
\end{aligned}
\end{equation*}
We choose $z_{0,0}\in\Delta^{a_{0},s_0}_{E^{i}_{m,M}}$, and let $\Delta^{z_{0,0},b_{0,0}}_{E^i_{m,M}}:=\mathcal P^{\overline{E}^i_{m,M}}(x)\cap(\Delta^i_{z_{0,0},b_{0,0}})_{E^i_{m,M}}\cap\Delta(\lambda,f^p,E,F)$. Then one has
\begin{equation}\label{eq 3.16b}
\begin{aligned}
&\mathrm{Leb}_{f^{a_{0,0}p}\mathcal D_i}(f^{a_{0,0}p}\Delta^{a_{0},s_0}_{E^{i}_{m,M}})\le
C\mathrm{Leb}_{f^{b_{0,0}p}\mathcal D_i}
\bigl(f^{b_{0,0}p}\Delta^{z_{0,0},b_{0,0}}_{E^{i}_{m,M}}\bigr)e^{2(b_{0,0}-a_{0,0})\varepsilon}
\left(
\prod_{j=a_{0,0}}^{b_{0,0}-1}\left|
\mathrm{det}
\left(
Df^p\big|_{F(f^{jp}x)}
\right)\right|
\right)^{-1}.
\end{aligned}
\end{equation}
Therefore,
\begin{equation}
\begin{aligned}
&\mathrm{Leb}_{\mathcal D_i}
\bigl(\Delta_{E^i_{m,M}}\bigr)\le C^2e^{2b_{0,0}\varepsilon}
\left(
\prod_{j=a_{0,0}}^{b_{0,0}-1}\left|
\mathrm{det}
\left(
Df^p\big|_{F(f^{jp}x)}
\right)\right|
\right)^{-1}\mathrm{Leb}_{f^{b_{0,0}p}\mathcal D_i}
\bigl(f^{b_{0,0}p}\Delta^{z_{0,0},b_{0,0}}_{E^{i}_{m,M}}\bigr).
\end{aligned}    
\end{equation}

\paragraph{\textbf{Step 7: the inductive step}}
Choose an integer $L=L_\varepsilon\ge1$, depending only on
$\varepsilon$ and the fixed data, such that the inequalities
in (3)--(7) hold whenever the corresponding interval has
length at least $L$. Enlarge $K(\varepsilon)$ so that
$K(\varepsilon)\ge L$, and take $M\ge N\ge K(\varepsilon)$. In the geometric construction, retain the maximal fine blocks
of $E^i_{m,N}(x)$ of length at least $L$, and let $O$ be the
union of the remaining fine blocks. We keep the original block
indices, partition atom, and endpoint conditions throughout. 

Between two successive retained fine blocks in the same coarse
block there is an original fine gap of length greater than $N$.
Likewise, if the first retained fine block starts strictly after
the left endpoint of its coarse block, the intervening interval
contains a gap of length greater than $N$ between consecutive
times of $P_i(x)$. Indeed, the coarse endpoint belongs to $P_i(x)$,
and the original fine blocks are maximal unions of consecutive
Pliss-time intervals of length at most $N$. The same observation
applies at the right endpoint. A transition from one coarse block
to the left endpoint of a later coarse block crosses an original
coarse gap of length greater than $M$. Thus all nonempty gaps
used in these transitions satisfy the required length bounds,
since $M\ge N\ge L$.

Define
\[
I_q=\bigl(E^i_{m,N}(x)\setminus O\bigr)\cap[0,q),
\:
B_q=
\#\bigl(\partial E^i_{m,M}\cap[0,q]\bigr)
+
\#\bigl(\partial E^i_{m,N}(x)\cap[0,q]\bigr),\:q\in\mathbb N^+.
\]
We first assume that at least one fine block is retained.

At each processed endpoint $q$, the construction provides a
pullback domain $D(q)\subset\mathcal D_i$, contained in all its
previously selected ancestors, such that $f^{qp}D(q)$ is a
standard $F$-disc with the fixed parameters
$(\theta,\theta_1,r,r')$. Set
$V(q)=
\mathrm{Leb}_{f^{qp}\mathcal D_i}
\bigl(f^{qp}(\Delta_{E^i_{m,M}}\cap D(q))\bigr).$
The induction hypothesis is
\begin{equation}\label{eq:volume-induction}
\mathrm{Leb}_{\mathcal D_i}(\Delta_{E^i_{m,M}})
\le
C^{B_q+2}e^{2q\varepsilon}
\exp\!\left(
-\sum_{s\in I_q}\psi^F_{p,1}(f^{sp}x)
\right)V(q).
\end{equation}
Write $[a_{j_*,s_*},b_{j_*,s_*})$ for the first retained fine
block, with its original indices. We first apply the initial
construction of Step 4 on $[0,a_{j_*})$. Indeed, $a_{j_*}$ is
a common $\gamma_i$-hyperbolic time and
$a_{j_*}\ge a_0\ge T_*$ by the assumption (2). If $a_{j_*}<a_{j_*,s_*}$,
we use Step 5 to treat the intervening gap and then the fine
block. If $a_{j_*}=a_{j_*,s_*}$, we use Step 6 directly.
In either case, \eqref{eq:volume-induction} holds at
$q=b_{j_*,s_*}$. For each subsequent
retained fine block, we first estimate the gap between it and the
preceding retained fine block, and then the block itself. Let
$[q,q')$ denote the full interval under consideration in either
estimate.

The constructions in Steps 4--6 provide
$D(q')\subset D(q)$, with standard terminal image, and give
\begin{equation}\label{eq:volume-transition}
V(q)
\le
Ce^{2(q'-q)\varepsilon}
\exp\!\left(
-\sum_{s\in I_{q'}\setminus I_q}
\psi^F_{p,1}(f^{sp}x)
\right)V(q').
\end{equation}
Suppose first that $[q,q')$ is a gap. If the two retained fine
blocks lie in the same coarse block, set $q''=q$. Otherwise,
let $q''$ be the left endpoint of the latter coarse block.
Then $q\le q''\le q'$. If $q<q''$, then $q''$ is a common $\gamma_i$-hyperbolic time
for all eligible centers. We use the packing argument of Step 4 for the transition from
$q$ to $q''$, keeping the full prehistory in the original disc,
and select $D(q'')\subset D(q)$. When $q''=q$, we simply set
$D(q'')=D(q)$. If $q''<q'$, then
$q'\in P_i(x)$, and orbit tracking at times $q''+1,\ldots,q'$
makes $q'$ a $\sqrt{\gamma_i}$-hyperbolic time relative to
$q''$ for each eligible centre in $D(q'')$. We then apply
Step 5 on $[q'',q')$ within $D(q'')$ to select $D(q')$. Throughout the gap, $I_{q'}=I_q$. At each nonempty stage,
disjointness and the Jacobian-weighted packing estimate give
the volume bound for the selected branch. Since there are at
most two stages, the square of the uniform single-stage
constant suffices in \eqref{eq:volume-transition}.

For a retained fine block $[q,q')\subset E^i_{m,N}(x)$, the constructions in
Steps 5 and 6 give a domain $D(q')$ containing
$\Delta_{E^i_{m,M}}\cap D(q)$. The gaps between successive Pliss-times are at most
$N$, so the same geodesic-cone argument applies.
Lemma \ref{Lemma 3.5} then supplies the inverse
$F$-Jacobian over $[q,q')$, yielding
\eqref{eq:volume-transition}.

These constructions retain the full prehistory in the
original disc. Earlier orbit-containment and cone conditions
are preserved by restriction, while Step 3 supplies the
boundary margins required for the next pullback.
The intrinsic--ambient comparison is used only on the
terminal standard disc, so no additional comparison factor
accumulates at intermediate iterates.

Each new endpoint $q'$ belongs to
$\partial E^i_{m,M}\cup\partial E^i_{m,N}(x)$.
Consequently,
$C^{B_q+2}C\le C^{B_{q'}+2}.$
Since
$I_{q'}=I_q\sqcup(I_{q'}\setminus I_q)$,
substitution of \eqref{eq:volume-transition} into
\eqref{eq:volume-induction} proves
\eqref{eq:volume-induction} at $q'$.
This completes the inductive step.

Stop at the last retained fine endpoint $q_*$. Then
$I_{q_*}=E^i_{m,N}(x)\setminus O,
\:
q_*\le b_{n-1},
\:
V(q_*)\le Q.$
Writing
$B=\#\partial E^i_{m,M}+\#\partial E^i_{m,N}(x),$
we obtain
\[
\mathrm{Leb}_{\mathcal D_i}(\Delta_{E^i_{m,M}})
\le
Q C^{B+2}e^{2b_{n-1}\varepsilon}
\exp\!\left(
-\sum_{s\in E^i_{m,N}(x)\setminus O}
\psi^F_{p,1}(f^{sp}x)
\right).
\]
If no fine block is retained, the same bound holds with
$E^i_{m,N}(x)\setminus O=\varnothing$, since
$\mathrm{Leb}_{\mathcal D_i}(\Delta_{E^i_{m,M}})\le Q$.

In both cases, put $J_F=\|\psi^F_{p,1}\|_\infty$.
Each omitted block has length less than $L$, and each maximal
fine block contributes two boundary points. Hence
$\#O\le \frac L2\,\#\partial E^i_{m,N}(x).$
Restoring the omitted Jacobian factors costs at most
\[
 \exp\!\left(\sum_{s\in O}\psi^F_{p,1}(f^{sp}x)\right)
 \le
 \exp\!\left(\frac{J_FL}{2}\,\#\partial E^i_{m,N}(x)\right).
\]
Enlarging $C$ to absorb this factor into $C^{B+2}$ and to
absorb $Q$, we obtain
\[
\mathrm{Leb}_{\mathcal D_i}(\Delta_{E^i_{m,M}})
\le
C^{B+2}e^{2b_{n-1}\varepsilon}
\left(
\prod_{s\in E^i_{m,N}(x)}
\left|
\det\!\left(
Df^p\big|_{F(f^{sp}x)}
\right)
\right|
\right)^{-1},
\]
where $C$ is independent of $m,n,N,$ and $M$. This completes the proof.
\end{proof}

\subsection{Lower bound of the metric entropy}\label{sec:entropy-lower}
For any $\varepsilon>0$, for any $N\in\mathbb N^+$ with $N\ge K(\varepsilon)$, let $\mathcal P_N$ be a finite partition of $\mathbf M^d$ with $\nu_i(\partial \mathcal P_N)=0$ and $\mathrm{diam}\mathcal P_N<\delta:=\delta(\varepsilon,N)$. Then for any $\iota>0$, we have
\begin{equation}
\begin{aligned}
h_{\nu_i}(f^p)&\ge h_{\nu_i}(f^p,\mathcal P_N)\ge\frac{1}{m_N}H_{\nu_i}(\vee_{j=0}^{m_N-1}f^{-jp}\mathcal P_{N})-\iota\\
&\text{($m_N$ large enough)}\\
&\ge \lim_{M\to+\infty}\liminf_{n\in\mathbf n_i\to+\infty}\frac{1}{m_N}H_{\overline \nu^i_{n,M}}(\vee_{j=0}^{m_N-1}f^{-jp}\mathcal P_N)-\iota\\
&\text{($\nu_i(\partial \mathcal P_N)=0$) }\\
&\ge \liminf_{M\to+\infty}\liminf_{n\in\mathbf n_i\to+\infty}\frac{1}{\#\overline E^i_{n,M}}H_{(\mathrm{Leb}_{\mathcal D_i})|_{A^i_{n,M}}}(\vee_{j\in \overline E^i_{n,M}}f^{-jp}\mathcal P_{N})\\
&-\limsup_{M\to+\infty}\limsup_{n\in\mathbf n_i\to+\infty}\frac{12m_N\#\partial\overline E^i_{n,M}}{\# \overline E^i_{n,M}}\log\#\mathcal P_{N}-\iota\\
&\text{(Lemma \ref{Lemma 2.1})}\\
&\ge \liminf_{M\to+\infty}\liminf_{n\in\mathbf n_i\to+\infty}\frac{1}{\#\overline E^i_{n,M}}\int(-\log\mathrm{Leb}_{\mathcal D_i}\!\left(
  \mathcal P_N^{\overline E^i_{n,M}}(x)\cap A_{n,M}^i
\right)+\log\mathrm{Leb}_{\mathcal D_i}(A_{n,M}^i))\\
&d((\mathrm{Leb}_{\mathcal D_i})|_{A_{n,M}^i})(x)
-\limsup_{M\to+\infty}\frac{24m_N}{M(\beta_i-\frac{1}{M})}\log\#\mathcal P_{N}-\iota\\
&\text{(the definition of partition entropy)}\\
&\ge \liminf_{M\to+\infty}\liminf_{n\in\mathbf n_i\to+\infty}\frac{1}{\#\overline E^i_{n,M}}\int-\log\mathrm{Leb}_{\mathcal D_i}\!\left(
  \mathcal P_N^{\overline E^i_{n,M}}(x)\cap A_{n,M}^i
\right)d((\mathrm{Leb}_{\mathcal D_i})|_{A_{n,M}^i})(x)\\
&-\limsup_{M\to+\infty}\frac{\lim_{n\in\mathbf n_i\to+\infty}\delta_{n,M}}{\beta_{i}-\frac{1}{M}}-\iota\\
&\eqref{eq 3.9}\\
&\ge \liminf_{M\to+\infty}\liminf_{n\in\mathbf n_i\to+\infty}\frac{1}{\#\overline E^i_{n,M}}\int\psi_{p,E^i_{n,N}(x)}^{F}(x)d((\mathrm{Leb}_{\mathcal D_i})|_{A^i_{n,M}})(x)-\iota-\frac{4\varepsilon}{\beta_i}\\
&\text{(Proposition \ref{Proposition 3.4})}\\
&\ge\liminf_{M\to+\infty}\liminf_{n\in\mathbf n_i\to+\infty}\frac{1}{\#\overline E^i_{n,M}}\int\psi_{p,E^i_{n,M}}^{F}(x)d((\mathrm{Leb}_{\mathcal D_i})|_{A^i_{n,M}})(x)\\
&-\frac{\sup_{x\in A_i}\left|\Delta^i_{\infty}(x)-\Delta^i_{N}(x)\right|}{\beta_i}\left\|\psi^F_{p,1}\right\|_{\infty}-\iota-\frac{4\varepsilon}{\beta_i}\\
&\text{\eqref{eq 3.7}}\\
&\ge \liminf_{M\to+\infty}\liminf_{n\in\mathbf n_i\to+\infty}\int\psi_{p,1}^F(x)d\overline{\nu}_{n,M}^i(x)-\frac{\sup_{x\in A_i}\left|\Delta^i_{\infty}(x)-\Delta^i_{N}(x)\right|}{\beta_i}\left\|\psi^F_{p,1}\right\|_{\infty}-\iota-\frac{4\varepsilon}{\beta_i}\\
&\text{(the definition of $\overline\nu_{n,M}^i$, the vanishing
contribution of the added endpoints)}\\
&\ge\int\psi_{p,1}^Fd\nu_{i}-\frac{\sup_{x\in A_i}\left|\Delta^i_{\infty}(x)-\Delta^i_{N}(x)\right|}{\beta_i}\left\|\psi^F_{p,1}\right\|_{\infty}-\iota-\frac{4\varepsilon}{\beta_i},\\
&\text{(the definition of $\nu_i$)}\\
\end{aligned}    
\end{equation}
Keeping $\varepsilon$ and $\iota$ fixed, we first let
$N\to+\infty$. By the uniform convergence of
$\Delta_N^i$ to $\Delta_\infty^i$ on $A_i$, we obtain
\[
h_{\nu_i}(f^p)
\ge
\int\psi_{p,1}^F\,d\nu_i
-\iota-\frac{4\varepsilon}{\beta_i}.
\]
For every $\varepsilon>0$, the preceding estimates apply
to the same sets $A^i_{n,M}$ and the same measures
$\nu^i_{n,M}$ constructed in Section 3.2. Only the constants and the test partitions
depend on $\varepsilon$; $\nu_i$ does not.
Therefore, the final passage $\varepsilon\downarrow0$
is taken with $\nu_i$ fixed.
As $\iota\downarrow0$ and $\varepsilon\downarrow0$, one has
\[
h_{\nu_i}(f^p)\ge\int\psi_{p,1}^F\,d\nu_i.
\]

\subsection{Passing back to an \texorpdfstring{$f$}{f}-invariant measure}Let $\mu_i=\frac{1}{p}\sum_{0\le j<p}f_{*}^j\nu_{i}.$ 
Then $\mu_i\in\mathcal M(\mathbf M^d,f)$. By $h_{\mu_i}(f)=\frac{1}{p}h_{\nu_i}(f^p)$ and 
$\int \log\left|\mathrm{det}(Df|_{F(x)})\right|d\mu_i(x)=\frac{1}{p}\int \log\left|\mathrm{det}(Df^p|_{F(x)})\right|d\nu_i(x),$ one has
$$h_{\mu_i}(f)\ge \int\log\left|\mathrm{det}(Df|_{F(x)})\right|d\mu_i(x).$$
By Lemma \ref{Lemma 3.3}, $\mu_i$ satisfies the asserted Lyapunov
inequalities. Let
$\mu_i=
\int_{\mathcal M^{\mathrm{erg}}(\Lambda,f)}
m\,d\tau_i(m)$
be its ergodic decomposition.
For $\tau_i$-almost every $m$, the same Lyapunov
inequalities hold $m$-almost everywhere. The entropy decomposition formula therefore yields
\[
\begin{aligned}
&\int_{\mathcal M^{\mathrm{erg}}(\Lambda,f)}
\left(
h_m(f)-\int_\Lambda\log\left|\det(Df|_{F(x)})\right|\,dm(x)
\right)\,d\tau_i(m)\\
&=
h_{\mu_i}(f)-\int_\Lambda\log\left|\det(Df|_{F(x)})\right|\,d\mu_i(x)
\ge0.
\end{aligned}
\]
Consequently, there exists an ergodic component $\mu$
satisfying both the asserted Lyapunov inequalities and
$h_\mu(f)\ge\int\log\left|\det(Df|_{F(x)})\right|\,d\mu(x).$
Choosing this component completes the proof.\\

\section{Proof of Theorem \ref{thm:B}}\label{sec:proof-B}
Let $g\in\mathrm{Diff}(\mathbf M^d)$, and let $\phi$ be a continuous function on $\overline U$, where $U$ is an open set with $g(U)\subset U$. Given \(\lambda\in(0,1)\) and \(N\in\mathbb{N}^+\), define
\[
\Lambda_{\lambda,\phi}
=
\left\{
x\in U:
\limsup_{n\to+\infty}
\frac{1}{n}
\sum_{j=0}^{n-1}
\phi(g^jx)
\le 2\log \lambda
\right\},
\]
and
\[
\Lambda_{\lambda,N,\phi}
=
\left\{
x\in \Lambda_{\lambda,\phi}:
\frac{1}{n}
\sum_{j=0}^{n-1}
\phi(g^jx)
\le \log \lambda <0,\:\text{for any } n\ge N
\right\}.
\]
We can obtain the following lemma. It states
that, under the assumption of
$$\mathrm{Leb}\left\{x\in U:\limsup_{n\to+\infty}\frac{1}{n}\sum_{j=0}^{n-1}\phi(g^jx)<0\right\}>0,$$
there exists $0<\lambda<1$ for which $\Lambda_{\lambda,1,\phi}$ has positive volume. This set need not be invariant. This set is more convenient for
our later arguments.

\begin{lemma}\label{lm:4.1}
Assume that
\[
\mathrm{Leb}\left\{
x\in U:
\limsup_{n\to+\infty}
\frac{1}{n}\sum_{j=0}^{n-1}\phi(g^jx)<0
\right\}>0.
\]
Then there exists a constant \(\lambda\in(0,1)\) such that
\[
\mathrm{Leb}(\Lambda_{\lambda,1,\phi})>0.
\]
\end{lemma}
\begin{proof}
By a countable decomposition, choose $c>0$ such that the set $$A_c\subset \left\{
x\in U:
\limsup_{n\to+\infty}
\frac{1}{n}\sum_{j=0}^{n-1}\phi(g^jx)<0
\right\}$$
has positive volume, where the upper average is at most $-2c$. For $x\in A_c$, put $S_0(x)=0$ and
$S_n(x)=\sum_{j=0}^{n-1}(\phi(g^jx)+c)$.
Then $S_n(x)\to-\infty$, so this sequence attains its maximum at some finite index. Let $k(x)$ be its first maximizing index. The sets $\{x\in A_c:k(x)=k\}$ are measurable and form a countable partition of $A_c$; one of them, denoted by $B$, has positive volume. For $x\in B$ and $n\ge1$,
$S_{k+n}(x)-S_k(x)\le0$, which proves the prefix inequality. Boundedness of $\phi$ shows that a finite orbit shift preserves its upper average. Taking $\lambda=e^{-c}$ gives
$g^k(B)\subseteq\Lambda_{\lambda,1,\phi}$.
Since $g^k(B)$ has positive volume, the conclusion follows.
\end{proof}

\begin{proof}[Proof of Theorem \ref{thm:B}(1)]
Choose a trapping neighborhood $U_1$ of $\Lambda$ as in
Section 1.4. Since $\operatorname{Leb}(\mathcal H_a)>0$,
there exists $p\ge1$ such that the set
\[
  X_p=\left\{x\in U:
  \liminf_{n\to+\infty}\frac{1}{np}
  \sum_{j=0}^{n-1}\Phi_{p,a}(f^{jp}x)>0\right\}
\]
has positive volume. Choose $k\ge0$ such that
$f^{kp}(\overline U)\subset U_1$. Removing the first $k$
complete blocks preserves the lower average, so
$f^{kp}(X_p)\subset X_p\cap U_1$. Hence $X_p\cap U_1$
has positive volume.
Let $g=f^p$,
$\phi:x\mapsto-\frac{1}{p}\min\left\{\log m(Df^p|_{F(x)})-ap,ap+\log m(Df^{-p}|_{E(f^{p}x)})\right\}.$
By Lemma \ref{lm:4.1}, there exists $\lambda\in (0,1)$ such that
\begin{equation}\label{eq3.1}
\begin{aligned}
\mathrm{Leb}\left\{x\in U_1:\frac{1}{np}\sum_{j=0}^{n-1}(\log\left\|Df^p|_{E(f^{jp}x)}\right\|-\log m(Df^p|_{F(f^{jp}x)}))\le 2\log\lambda,\right.\\
\left.-\frac{1}{np}\sum_{j=0}^{n-1}\min\left\{\log m(Df^p|_{F(f^{jp}x)})-ap,ap+\log m(Df^{-p}|_{E(f^{(j+1)p}x)})\right\}\le\log\lambda,\:\forall\:n\ge 1\right\}>0.   
\end{aligned}    
\end{equation}
It is obvious that
\begin{equation*}
\begin{aligned}
\left\{x\in U_1:\frac{1}{np}\sum_{j=0}^{n-1}(\log\left\|Df^p|_{E(f^{jp}x)}\right\|-\log m(Df^p|_{F(f^{jp}x)}))\le 2\log\lambda,\:-\frac{1}{np}\sum_{j=0}^{n-1}\right.\\
\left.\min\left\{\log m(Df^p|_{F(f^{jp}x)})-ap,ap+\log m(Df^{-p}|_{E(f^{(j+1)p}x)})\right\}\le\log\lambda,\:\forall\:n\ge 1\right\}\subset \mathrm{HP}(a,\lambda^{2p},f^p,E,F).    
\end{aligned}
\end{equation*}
It follows that
$\mathrm{Leb}(\mathrm{HP}(a,\lambda^{2p},f^p,E,F))>0.$
Theorem \ref{thm:A} (2) now proves Theorem \ref{thm:B} (1). 
\end{proof}
\begin{lemma}\label{Lemma 4.2}
Assume that $T_{\Omega}\mathbf M^d=E\oplus F$ is an $f$-invariant continuous splitting on a compact $f$-invariant set $\Omega\subset\mathbf M^d$, and $\mathrm{dim}F(x)=q$, for all $x\in\Omega$. Let $\widehat\chi_1^F(x)\ge\cdots\ge
\widehat\chi_q^F(x)$
be the Lyapunov exponents of the restricted cocycle
$Df|_F:\Omega\ni x\mapsto D_xf|_{F(x)},$
listed with multiplicity. If $\mu\in\mathcal M(\Omega,f)$ is such that
$\lambda^{\min}_{F}(x)>0>\lambda^{\max}_{E}(x),\:\mathrm{for\:}\mu\mathrm{-a.e.}\:x\in\Omega,$
then
\[
\sum_{\lambda_s(f,x)>0}\mathrm{dim}E^s_{f}(x)\lambda_s(f,x)
=
\sum_{j=1}^{q}\widehat\chi_j^F(x),
\]
$\mathrm{for\:}\mu\mathrm{-a.e.}\:x\in\Omega$. Moreover, 
\[
\lim_{n\to+\infty}
\frac{1}{n}
\log\left|\mathrm{det}(Df^n|_{F(x)})\right|
=
\sum_{j=1}^{q}\widehat\chi_j^F(x)
=
\sum_{\lambda_s(f,x)>0}\mathrm{dim}(E^s_f(x))\lambda_s(f,x),
\]
$\mathrm{for\:}\mu\mathrm{-a.e.}\:x\in\Omega$.
\end{lemma}
\begin{proof}
Make $E$ and $F$ orthogonal without changing their individual metrics. The resulting continuous metric is uniformly equivalent to the original one. In this metric $Df^n$ is block diagonal, and its singular values are the union, with multiplicities, of the singular values of its two restrictions. The multiplicative ergodic theorem therefore identifies its spectrum with the union of the two restricted spectra. The sign hypothesis identifies precisely the $F$-spectrum as positive. Applying the same theorem to the top exterior power $\bigwedge^q(Df|_F)$ gives the determinant limit. 
\end{proof}

\begin{proof}[Proof of Theorem \ref{thm:B}(2)]
Assume that $a=0$, and let
$\mu\in\mathcal M^{\mathrm{erg}}(\Lambda,f)$
be the measure obtained in part (1). Then
$\lambda_F^{\min}(x)>0>\lambda_E^{\max}(x)
,\:\text{for }\mu\text{-almost every }x.$
Since $F$ is continuous and invariant, its rank is constant
on $\mathrm{supp}\mu$ by ergodicity.
Lemma \ref{Lemma 4.2}, Birkhoff's ergodic theorem,
the entropy lower bound in part (1), and Ruelle's
inequality give
\[
\begin{aligned}
h_\mu(f)
&\ge
\int_\Lambda
\log\left|\det(Df|_{F(x)})\right|\,d\mu(x)\\
&=
\int_\Lambda
\sum_{\lambda_s(f,x)>0}
\dim(E_f^s(x))\lambda_s(f,x)\,d\mu(x)\ge h_\mu(f).
\end{aligned}
\]
Thus $\mu$ satisfies Pesin's entropy formula.
Since $\mu$ is hyperbolic and ergodic, we have
$\mu\in\mathrm{PE}_f^{eh}$.
If $f\in\mathrm{Diff}^r(\mathbf M^d)$ for some $r>1$,
then $\mu\in\mathrm{SRB}_f^{eh}$ by
\cite{Brown22,LY85-1}.
This completes the proof of
Theorem \ref{thm:B} (2).
\end{proof}
Part (3) of Theorem \ref{thm:B} is proved by the full-preball argument
in Section \ref{sec:coverage}; no basin conclusion is inferred solely
from the existence assertion above.

\section{Proof of Theorem \ref{thm:D}}\label{sec:proof-C}
\begin{lemma}\label{Proposition 4.1}
Let $\Lambda$ be a compact $f$-invariant set and $E\subset T_{\Lambda}\mathbf M^d$ be a continuous $Df$-invariant bundle. Suppose that there exists a set $\Lambda'\subset\Lambda$ with total probability such that, for any $x\in\Lambda'$, one has
$$\limsup_{n\to+\infty}\frac{1}{n}\log\left\|Df^n|_{E(x)}\right\|\le a.$$
Then for any $\varepsilon>0$, there exists $N\in\mathbb N^+$ such that, for any $n\ge N$ and any $x\in\Lambda$,
$$\left\|Df^n|_{E(x)}\right\|\le e^{(a+\varepsilon)n}.$$
\end{lemma}
\begin{proof}
Put
$\varphi_n(x)=\log\|Df^n|_{E(x)}\|,
\:x\in\Lambda,\:n\ge1.$
Continuity and invariance of $E$ imply that
$\{\varphi_n\}_{n\ge1}$ is a sequence of continuous
subadditive functions on $\Lambda$.
For every $\mu\in\mathcal M^{erg}(\Lambda,f)$,
the total-probability assumption and Kingman's
subadditive ergodic theorem give
\[
\inf_{n\ge1}\frac1n\int_\Lambda\varphi_n\,d\mu
=
\lim_{n\to+\infty}\frac{\varphi_n(x)}n
\le a,\:\mu\text{-a.e. }x\in\Lambda.
\]
We apply \cite[Theorem 1]{Sch98} to the discrete-time
semiflow
$\phi:\Lambda\times\mathbb N\longrightarrow\Lambda,
\:\phi(x,n)=f^n(x),$
and the continuous subadditive function
$F(x,n)=\log\|D_xf^n|_{E(x)}\|,
\:F(x,0)=0.$
By \cite[Theorem 1]{Sch98}, together with Fekete's lemma
applied to $\sup_{x\in\Lambda}\varphi_n(x)$, we obtain
\[
\lim_{n\to+\infty}\frac1n
\sup_{x\in\Lambda}\varphi_n(x)
=
\sup_{\mu\in\mathcal M^{erg}(\Lambda,f)}
\inf_{n\ge1}\frac1n\int_\Lambda\varphi_n\,d\mu
\le a.
\]
Consequently, for any $\varepsilon>0$, there exists
$N\ge1$ such that
$\varphi_n(x)\le(a+\varepsilon)n$, for all $n\ge N$ and all $x\in\Lambda$.
\end{proof}
As in the proof of part (1) of Theorem \ref{thm:B}, we may choose $a_1>a$ and $\Lambda\subset U_1\subset U\setminus B(\partial U,4\rho)$ with $f(U_1)\subset U_1$, such that
$$\mathrm{Leb}\left\{x\in U_1:\liminf_{n\to+\infty}\frac{1}{n}\sum_{j=0}^{n-1}\log m(Df|_{F(f^{j}x)})>a_1\right\}>0.$$
Choose $a_*$ with $a<a_*<a_1$. By Lemma \ref{Proposition 4.1}, there exists
$p\ge1$ such that
$\log\left\|Df^p|_{E(x)}\right\|\le a_*p,
  \:\text{for all }x\in\Lambda.$
By continuity, after shrinking $U_1$ we have
$\log\|Df^p|_{E(x)}\|\le a_1p$ for all $x\in U_1$.
The positive-volume condition above persists under this
shrinking, since a fixed forward iterate enters the smaller
trapping neighborhood and preserves the lower Birkhoff average.
Let $g=f$, 
$\phi:x\mapsto a_1-\log m(Df|_{F(x)}).$
By Lemma \ref{lm:4.1}, there exists $\lambda\in (0,1)$ such that
\begin{equation}
\begin{aligned}
\mathrm{Leb}\left\{x\in U_1:\frac{1}{n}\sum_{j=0}^{n-1}(a_1-\log m(Df|_{F(f^{j}x)}))\le \log\lambda,\:\forall\:n\ge 1\right\}>0.   
\end{aligned}    
\end{equation}
Hence
$\mathrm{Leb}(\Delta(a_1,\lambda^p,f^p,E,F))>0.$
Applying Theorem \ref{thm:A}(1) with $a_1$,
we obtain a measure $\mu\in\mathcal M^{erg}(\Lambda,f)$ such that
$\lambda_F^{\min}(x)>a_1\ge\lambda_E^{\max}(x)
\:\text{for $\mu$-a.e. }x\in\Lambda,$
and $h_\mu(f)\ge
\int \log\left|\det(Df|_{F(x)})\right|\,d\mu(x).$
Since $\Gamma$ has total probability with respect to $f$
over $\Lambda$, we have $\mu(\Gamma)=1$. The assumption on
$\Gamma$ therefore gives
$\lambda_E^{\max}(x)<a,
\:\text{for $\mu$-a.e. }x\in\Lambda.$
Together with $a_1>a$, this yields
\[
\lambda_F^{\min}(x)>a>\lambda_E^{\max}(x),
\:\text{for $\mu$-a.e. }x\in\Lambda.
\]
This proves part (1). When $a=0$, the conclusions of part (2)
follow as in the proof of Theorem \ref{thm:B}(2). Part (3) of Theorem \ref{thm:D} will be proved in Section \ref{sec:coverage}.

\section{Statistical criteria and reference measures}
\label{stat:section}\label{sec:statistics}
With the notation introduced before Theorem \ref{thm:E}, recall that
$\gamma_a$ is set equal to zero on the irregular set in $\Lambda$.
For every common Lyapunov-regular point $x\in\Lambda$ and every
$\tau\in\mathcal M(\Lambda,f)$, we have
\begin{equation}\label{stat:gap}
 \gamma_a(x)=\min\{\lambda_F^{\min}(x)-a,a-\lambda_E^{\max}(x)\},\;
 \mathcal G_a(\tau)=\int\gamma_a\,d\tau.
\end{equation}
The existence conclusion used below is
\begin{equation}\label{stat:conclusion}
 \mu\in\mathcal M^{erg}(\Lambda,f),\:\lambda_F^{\min}(x)>a>\lambda_E^{\max}(x),\:\mu\text{-a.e. }x\in\Lambda,\:h_\mu(f)\ge\int\log|\det(Df|_F)|\,d\mu.
\end{equation}
\subsection{A statistical characterization of coordinated asymptotic hyperbolicity}

For $n\ge1$, recall that
\begin{equation}\label{stat:potential}
 \Phi_{n,a}(x)=\min\{\log m(Df^n|_{F(x)})-an,
 an-\log\|Df^n|_{E(x)}\|\},
 \:\Phi_{0,a}=0,\:\forall\:x\in\overline U.
\end{equation}
These functions are continuous, and,
for some $L>0$,
\begin{equation}\label{stat:super}
 \left\|\Phi_{n,a}\right\|_{\infty}\le Ln,\:
 \Phi_{n+k,a}(x)\ge\Phi_{n,a}(x)+\Phi_{k,a}(f^nx),\;\forall\:x\in\overline U,\:\forall\:k,n\in\mathbb N.
\end{equation}
For $x\in U$ and $\tau\in\mathcal M(\Lambda,f)$, define
\begin{equation}\label{stat:I}
 I_{p,a}(\tau)=\frac1p\int\Phi_{p,a}\,d\tau,\:b_{p,a}(x)=\liminf_{n\to+\infty}\frac1{np}
\sum_{j=0}^{n-1}\Phi_{p,a}(f^{jp}x),\:\mathcal H_a=\{x\in U:\limsup_{p\to+\infty}b_{p,a}(x)>0\}.
 \end{equation}
Theorem \ref{thm:B} applies whenever
$\mathrm{Leb}(\mathcal H_a)>0$, and gives
\eqref{stat:conclusion}.

\begin{lemma}\label{stat:identity}
For any $m\in\mathcal M(\Lambda,f)$,
\begin{equation}\label{stat:G-representation}
 \mathcal G_a(m)=\lim_{p\to+\infty}I_{p,a}(m)=\sup_{p\ge1}I_{p,a}(m).
\end{equation}
Furthermore, for any $x\in U$,
\begin{equation}\label{stat:identity-equation}
 \lim_{p\to+\infty}b_{p,a}(x)
 =\sup_{p\ge1}b_{p,a}(x)
 =\min_{\tau\in p\omega(x)}\mathcal G_a(\tau).
\end{equation}
Consequently,
\begin{equation}\label{stat:I-equivalent}
 \mathcal H_a=
 \{x\in U:\mathcal G_a(\tau)>0,\:\forall\:\tau\in p\omega(x)\}
 .
\end{equation}
\end{lemma}

\begin{proof}
The multiplicative ergodic theorem gives, for any $m\in\mathcal M(\Lambda,f)$, $\frac{\Phi_{p,a}(x)}p\to\gamma_a(x),\:m$-a.e. $x\in\Lambda$.  The bound in \eqref{stat:super}, dominated convergence,
and Fekete's lemma give \eqref{stat:G-representation}. Since each $I_{p,a}$ is weak$^*$-continuous on $\mathcal M(\Lambda,f)$,
\eqref{stat:G-representation} shows that $\mathcal G_a$ is lower
semicontinuous there. Hence $\mathcal G_a$ attains its minimum on
the nonempty compact set $p\omega(x)$ for every $x\in U$.

 Uniform bounds on the one-step restricted
cocycles and their inverses give a constant $C$ such that
\begin{equation}\label{stat:shift}
 |\Phi_{n,a}(f^sy)-\Phi_{n,a}(y)|\le2Cs,\:\forall\:n\ge1,\:s\ge0,\:y\in\overline U.
\end{equation}
Fix $x\in U$, $p\in\mathbb N^+$ and $\tau\in p\omega(x)$.  Choose $n_j\to+\infty$ with
$\mathcal E^{f}_{pn_j}(x)\to\tau$, as $j\to+\infty$; rounding lengths to multiples of $p$ does
not change the limit. After passing to a subsequence of $\left\{n_j\right\}_{j\in\mathbb N^+}$ without relabeling,
we may assume that $\mathcal E^{f^p}_{n_j}(x)\to\eta$ as
$j\to+\infty$.  Then
$\eta$ is $f^p$-invariant,
$\tau=\frac{1}{p}\sum_{0\le s\le p-1}f_*^s\eta$, and
$\frac{1}{p}\int\Phi_{p,a}\,d\eta\ge b_{p,a}(x)$.  Superadditivity and
\eqref{stat:shift} give
$\int\Phi_{mp,a}\,d\tau
       \ge m\int\Phi_{p,a}\,d\eta-C(p-1).$
Keeping $p$ fixed, divide the preceding inequality by $mp$
and let $m\to+\infty$. By \eqref{stat:G-representation}, we obtain
$\mathcal G_a(\tau)
\ge \frac1p\int\Phi_{p,a}\,d\eta
\ge b_{p,a}(x).$
Since $p\ge1$ and $\tau\in p\omega(x)$ were arbitrary, it follows
that
$\sup_{p\ge1} b_{p,a}(x)
   \le \min_{\tau\in p\omega(x)}\mathcal G_a(\tau).
$

For the reverse inequality, the standard all-starts block estimate is
\begin{equation}\label{stat:all-starts}
 \Phi_{q,a}(y)\ge\frac1p\sum_{j=0}^{q-1}\Phi_{p,a}(f^jy)-2L(p-1),\:q\ge p.
\end{equation}
For completeness, write $q=s+k_sp+r_s$, $0\le r_s<p$, for
each $0\le s<p$.  Superadditivity bounds $\Phi_{q,a}$ below by
$\sum_{j<k_s}\Phi_{p,a}(f^{s+jp}y)-L(s+r_s)$.  The starting indices
enumerate $0,\ldots,q-p$ exactly once, and the remainders $r_s$
are a permutation of $0,\ldots,p-1$.  Averaging and adding the
missing $p-1$ terms, bounded by $Lp$ each, proves
\eqref{stat:all-starts}.  Apply it at $f^{sq}x$, sum over $s$,
and take lower limits.  Continuity of $\Phi_{p,a}$ yields
\begin{equation}\label{stat:lower}
 b_{q,a}(x)\ge \min_{\tau\in p\omega(x)}I_{p,a}(\tau)-\frac{2L(p-1)}q,\:q\ge p.
\end{equation}

Given $\beta<\min_{\tau\in p\omega(x)}\mathcal G_a(\tau)$, \eqref{stat:G-representation} supplies 
$p\omega(x)\subset\bigcup_{0\le i\le N_0}\left\{\tau_i\in\mathcal M(\Lambda,f):I_{p_i,a}(\tau_i)>\beta\right\},\:N_0\in\mathbb N^+.$
For $P=\prod_{0\le i\le N_0}p_i$, invariance and
superadditivity imply $I_{P,a}(\tau)\ge I_{p_i,a}(\tau)$, for all $\tau\in\mathcal M(\Lambda,f)$, $0\le i\le N_0$.
Compactness therefore gives $\min_{\tau\in p\omega(x)}I_{P,a}(\tau)>\beta$.
Equation \eqref{stat:lower} implies
$\liminf_{q\to+\infty}b_{q,a}(x)\ge\beta$.  Letting $\beta\uparrow\min_{\tau\in p\omega(x)}\mathcal G_a(\tau)$
and using the upper bound established above for every fixed $p$, we obtain
\[
  \sup_{p\ge1} b_{p,a}(x)
  \le \min_{\tau\in p\omega(x)}\mathcal G_a(\tau)
  \le \liminf_{q\to+\infty} b_{q,a}(x)
  \le \limsup_{q\to+\infty} b_{q,a}(x)
  \le \sup_{p\ge1} b_{p,a}(x).
\]
All inequalities are therefore equalities, proving \eqref{stat:identity-equation}.
Equation \eqref{stat:I-equivalent} follows from the definition of $\mathcal H_a$.
\end{proof}

\subsection{Proof of Theorem \ref{thm:E}}
We first show that each of the four hypotheses implies
$\mathrm{Leb}(\mathcal H_a)>0$.

\smallskip
\noindent\textit{\textbf{Hypothesis \textup{(i)}.}}
Let $\nu$ be the invariant probability in this hypothesis.
Since $G_a\subset\Lambda$ and $\nu(G_a)>0$, we have
$\nu(\Lambda)>0$. Thus $\nu|_{\Lambda}$
is an invariant probability supported on $\Lambda$. Pointwise ergodic decomposition gives, for
$\nu|_\Lambda$-almost every $x\in\Lambda$, an ergodic
measure $\tau_x\in\mathcal M^{erg}(\Lambda,f)$
such that $p\omega(x)=\{\tau_x\}.$
Moreover, $\gamma_a$ is bounded, measurable and
$f$-invariant on $\Lambda$. Hence ergodic decomposition
also gives
$\mathcal G_a(\tau_x)
=
\int\gamma_a\,d\tau_x
=
\gamma_a(x)$
for $\nu|_\Lambda$-a.e. $x\in\Lambda$.
In particular, for $\nu$-a.e. $x\in G_a$, $\mathcal G_a(\tau_x)=\gamma_a(x)>0.$
Lemma \ref{stat:identity} therefore implies that, for $\nu$-a.e. $x\in G_a$, $x\in\mathcal H_a$. We have proved
\begin{equation}\label{stat:invariant-good}
\nu(G_a\setminus\mathcal H_a)=0.
\end{equation}
Consequently,
$\nu(G_a\cap\mathcal H_a)=\nu(G_a)>0.$
Since $\nu\ll\mathrm{Leb}$, this yields
$\mathrm{Leb}(\mathcal H_a)>0$.

\smallskip
\noindent\textit{\textbf{Hypothesis \textup{(ii)}.}} We first prove that $\nu(\Lambda)=1$.
Since $\mathcal G_a(\nu)>0$, we have
$\nu(\Lambda)>0$, and hence $\nu(U)>0$. There exists $\delta_0>0$ such that
$\tau(U)>\frac{\nu(U)}2$, for any $\tau\in\mathcal M(\mathbf M^d)$ with $\rho(\tau,\nu)<\delta_0$. Choose a sequence $\left\{\varepsilon_k\right\}_{k\in\mathbb N^+}$ with $\varepsilon_k\downarrow0$ as $k\to+\infty$ and 
$\varepsilon_k<\delta_0$ for all $k\in\mathbb N^+$.
By the definition of $\mathcal O_f^H$, for each
$k\ge1$ there exists $x_k\in\mathbf M^d$ such that
$d_H(p\omega(x_k),\{\nu\})<\varepsilon_k.$
Choose $\tau_k\in p\omega(x_k)$.
Then $
\rho(\tau_k,\nu)<\varepsilon_k,
\:
\tau_k(U)>\frac{\nu(U)}2>0,$ for all $k\in\mathbb N^+$.
Consequently, the forward orbit of $x_k$ enters $U$, for all $k\in\mathbb N^+$. It is a basic fact that every empirical limit of an orbit entering $U$
is supported on $\Lambda$.
In particular, $\tau_k(\Lambda)=1$ for all $k\in\mathbb N^+$.
Since $\tau_k\to\nu$ as $k\to+\infty$,
we conclude that $\nu(\Lambda)=1$.

We may therefore apply
\eqref{stat:G-representation} to $\nu$.
Choose $p\ge1$ and $c>0$ such that
$I_{p,a}(\nu)>2c.$
Let $\widehat\Phi_{p,a}\in C(\mathbf M^d)$ be a
continuous extension of $\Phi_{p,a}$ from
$\overline U$ to $\mathbf M^d$, and define $\widehat I_{p,a}(\tau)
=
\frac1p\int\widehat\Phi_{p,a}\,d\tau$
for every Borel probability $\tau$ on $\mathbf M^d$.
This functional is continuous with respect to $\tau$, and $\widehat I_{p,a}(\nu)=I_{p,a}(\nu)>2c.$
Since $\nu(U)=1$, we can choose $\delta>0$ such that
\begin{equation}\label{stat:global-observable-neighborhood}
\widehat I_{p,a}(\tau)>c,
\:
\tau(U)>\frac12,
\:
\text{whenever }\rho(\tau,\nu)<\delta.
\end{equation}

By strong observability, the Borel set
$A_\delta
=
\left\{
x\in\mathbf M^d:
d_H(p\omega(x),\{\nu\})<\delta
\right\}$
has positive Lebesgue measure.
For every $x\in A_\delta$, all measures in
$p\omega(x)$ satisfy
\eqref{stat:global-observable-neighborhood}.
The preceding entrance argument therefore shows
that every such $x$ eventually enters $U$.
Hence
$A_\delta
=
\bigcup_{j\ge0}
\bigl(A_\delta\cap f^{-j}(U)\bigr).$
Choose $j_0\ge0$ such that
$\mathrm{Leb}
\bigl(A_\delta\cap f^{-j_0}(U)\bigr)>0,$
and put
$A
=
f^{j_0}\bigl(A_\delta\cap f^{-j_0}(U)\bigr).$
Since $f^{j_0}$ is a diffeomorphism,
$A\subset U$ is measurable and
$\mathrm{Leb}(A)>0$.

A finite time shift does not change the empirical
limit set. Thus, for any $x\in A$ and any
$\tau\in p\omega(x)$, we have
$\tau(\Lambda)=1,
\:
I_{p,a}(\tau)
=
\widehat I_{p,a}(\tau)>c.$
Using \eqref{stat:G-representation}, we obtain
$\mathcal G_a(\tau)\ge I_{p,a}(\tau)>c,
\:\forall\:\tau\in p\omega(x).$
Lemma \ref{stat:identity} now gives
$A\subset\mathcal H_a$.
Consequently,
$\mathrm{Leb}(\mathcal H_a)
\ge\mathrm{Leb}(A)>0.$

\smallskip
\noindent\textit{\textbf{Hypothesis \textup{(iii)}.}}
Let $A\subset U$ be the measurable set in this
hypothesis. For every $x\in A$, all measures
$\tau\in p\omega(x)$ satisfy
$\mathcal G_a(\tau)>0$.
By Lemma \ref{stat:identity},
$A\subset\mathcal H_a.$ Hence
$\mathrm{Leb}(\mathcal H_a)
\ge\mathrm{Leb}(A)>0.$

\smallskip
\noindent\textit{\textbf{Hypothesis \textup{(iv)}.}}
By the results of \cite{CE11},
\begin{equation}\label{stat:global-observable-full}
p\omega(x)\subset\mathcal O_f,
\:\text{Leb-a.e. }x\in\mathbf M^d.
\end{equation}
For Leb-a.e.\:$x\in U$,
by \eqref{stat:global-observable-full} and the
hypothesis, we have
$\mathcal G_a(\tau)>0,
\:\forall\:\tau\in p\omega(x).$
By lower semicontinuity and compactness,
$\min_{\tau\in p\omega(x)}\mathcal G_a(\tau)>0.$
Lemma \ref{stat:identity} gives
$\lim_{p\to+\infty}b_{p,a}(x)
=
\min_{\tau\in p\omega(x)}\mathcal G_a(\tau)>0.$
Consequently, for Leb-a.e. $x\in U$, $x\in\mathcal H_a$.
It follows that
$\mathrm{Leb}(U\setminus\mathcal H_a)=0.$
In particular, because $U$ is nonempty and open,
$\mathrm{Leb}(\mathcal H_a)
=
\mathrm{Leb}(U)>0.$

\smallskip
\noindent\textit{\textbf{Conclusion \textup{(1)}.}}
Under each of the four hypotheses, we have proved
$\mathrm{Leb}(\mathcal H_a)>0$.
Theorem \ref{thm:B}(1) therefore provides
$\mu\in\mathcal M^{\mathrm{erg}}(\Lambda,f)$ such that
$\lambda_F^{\min}(x)>a>\lambda_E^{\max}(x),
\:\mu\text{-a.e. }x\in\Lambda,$
and
$h_\mu(f)
\ge
\int\log|\det(Df|_F)|\,d\mu.$
This proves \eqref{stat:conclusion} and hence conclusion (1).

\smallskip
\noindent\textit{\textbf{Conclusion \textup{(2)}.}} This follows as in the proof of Theorem \ref{thm:B} (2).

\smallskip
\noindent\textit{\textbf{Conclusion \textup{(3)}.}}
Assume that $a=0$ and
$f\in\mathrm{Diff}^r(\mathbf M^d)$ for some
$r>1$.
By Theorem \ref{thm:B}(3), the family
$\mathscr S_{E,F}(\Lambda,f)$ is at most countable,
and its basin union
$\mathcal B_{E,F}
=
\bigcup_{\mu\in\mathscr S_{E,F}(\Lambda,f)}
B_f(\mu)$
satisfies
\begin{equation}\label{stat:coverage-null}
\mathrm{Leb}
(\mathcal H_0\triangle(\mathcal B_{E,F}\cap U))=0.
\end{equation}
We verify the additional coverage assertions
under each hypothesis.

Under hypothesis \textup{(i)},
\eqref{stat:invariant-good} with $a=0$ gives
$\nu(G_0\setminus\mathcal H_0)=0$.
Since $\nu\ll\mathrm{Leb}$,
\eqref{stat:coverage-null} implies
\[
\begin{aligned}
\nu(G_0\setminus\mathcal B_{E,F})
&\le
\nu(G_0\setminus\mathcal H_0)
+
\nu(\mathcal H_0\setminus\mathcal B_{E,F})=0.
\end{aligned}
\]

Under hypothesis \textup{(iii)}, the original
set $A$ satisfies $A\subset\mathcal H_0$.
Therefore \eqref{stat:coverage-null} gives
\[
\mathrm{Leb}
(A\setminus\mathcal B_{E,F})=0.
\]

Under hypothesis \textup{(iv)}, one has $\mathrm{Leb}(U\setminus\mathcal H_0)=0$.
Therefore
\eqref{stat:coverage-null} yields
\[
\begin{aligned}
\mathrm{Leb}(U\setminus\mathcal B_{E,F})
&\le
\mathrm{Leb}(U\setminus\mathcal H_0)
+
\mathrm{Leb}(\mathcal H_0\setminus\mathcal B_{E,F})=0.
\end{aligned}
\]
This proves conclusion \textup{(3)} and completes
the proof.

\section{Countable coverage by hyperbolic SRB basins}
\label{sec:coverage}

In this section $f$ is $C^{1+\alpha}$, $0<\alpha\le1$. Recall that
\[
 \mathscr S_{E,F}(\Lambda,f)=\{\mu:\mu\text{ is an ergodic hyperbolic SRB measure on }\Lambda,
\:\lambda_E^{\max}<0<\lambda_F^{\min}\ \mu\text{-a.e.}\},
\]
$\mathcal B_{E,F}=\bigcup_{\mu\in\mathscr S_{E,F}(\Lambda,f)}B_f(\mu).$
Moreover, $\mathcal H_0$ is defined in
\eqref{stat:I}, and 
\[
 H_C=\left\{x\in U:\liminf_{n\to+\infty}\frac1n
           \sum_{j=0}^{n-1}\log m(Df|_{F(f^jx)})>0\right\}.
\]

\begin{theorem}\label{thm:countable-coverage}
The family $\mathscr S_{E,F}(\Lambda,f)$ is at most countable and
\begin{equation}\label{eq:coverage-B}
 \mathrm{Leb}\bigl(\mathcal H_0\mathbin{\triangle}(\mathcal B_{E,F}\cap U)\bigr)=0.
\end{equation}
If, in addition, the strict total-probability contraction hypothesis of
Theorem \ref{thm:D} holds at $a=0$, then
\begin{equation}\label{eq:coverage-C}
 \mathrm{Leb}(H_C\setminus\mathcal B_{E,F})=0.
\end{equation}
\end{theorem}

The existence conclusion of Theorem \ref{thm:A} alone would not prove this
statement. We retain its limiting measure before ergodic extraction and
add the geometric argument below. The original stagewise discs remain
responsible for the entropy estimates. Separate full-time preballs provide
the differentiation and basin-transport argument.

\subsection{Full preballs and distortion}
Let $g=f^p$ and let $\mathcal D$ be a sufficiently small compact $F$-disc.
Suppose $A\subset\mathrm{Int}\mathcal D$ is measurable, stays a fixed positive
intrinsic distance from $\partial\mathcal  D$, and, for fixed $0<\lambda<\gamma<1$,
every $x\in A$ satisfies
\[
 \prod_{j=0}^{k-1}
 \frac{\|Dg|_{E(g^jx)}\|}{m(Dg|_{F(g^jx)})}\le\lambda^k,\:\forall\:k\ge1.
\]
Suppose also that $P(x)\subset\mathbb N$,
\[
 \prod_{j=n-k}^{n-1}\|Dg^{-1}|_{F(g^{j+1}x)}\|\le\gamma^k,\:\forall\:n\in P(x),\:\forall\:1\le k\le n.
\]
Every \(n\in P(x)\) is a \(\gamma\)-hyperbolic time with starting
time zero. In the application below, \(P(x)\) is the set of times
supplied by Lemma \ref{Lemma 3.1}, after the restrictions made
in Section 3.2. Put $q=\dim F$.
\begin{lemma}\label{lem:coverage-preballs}
There exist $r>0$, $N_0\in\mathbb N$, $0<\vartheta<1$
and $C_1\ge1$, $C_2>0$ such that the following holds.
For any $x\in A$, any $n\in P(x)$ with $n\ge N_0$ and any
$0<s\le4r$, the set
$V_n(x,s)
=
g^{-n}B_{g^n\mathcal D}(g^nx,s)$
satisfies
$V_n(x,s)\subset\mathrm{Int}\mathcal D$
and
\begin{equation}\label{eq:coverage-geometry}
\begin{aligned}
\mathrm d_{g^j\mathcal D}(g^jx,g^jy)
 \le\vartheta^{n-j}s,
 \:y\in V_n(x,s),\:0\le j\le n,\:
C_1^{-1}s^q
\le\mathrm{Leb}_{g^n\mathcal D}
 \bigl(B_{g^n\mathcal D}(g^nx,s)\bigr)
 \le C_1s^q.
\end{aligned}
\end{equation}
The terminal ball $B_{g^n\mathcal D}(g^nx,4r)$
lies in a single graph branch.
For each $0<s\le4r$, the ball
$B_{g^n\mathcal D}(g^nx,s)$ contains a graph core
of base radius $C_2s$, and the graph slope on this
ball, in exponential coordinates based on
$E(g^nx)\oplus F(g^nx)$, is at most
$C_1\lambda^{\frac n2}+\omega(s),$
where $\omega$ is independent of $x$, $n$ and
$\omega(s)\to0$ as $s\to0$.
Moreover, for any $y\in V_n(x,4r)$ and any $0\le j\le n$,
the tangent spaces $T_{g^jy}(g^j\mathcal D)$ lie in
$F$-cones of uniformly bounded width.
\end{lemma}

\begin{proof}
Retain the initial cone width $\theta$ and the
smallness conditions fixed in Section 3.2.
Choose a fixed terminal radius $R>0$ admissible for
Lemmas \ref{Lemma 2.6} and \ref{lem:uniform-cone-graph}.
Set $\vartheta=\sqrt\gamma$ and choose $N_0$ so that
\[
\gamma^{\frac{N_0}2}R<\inf_{x\in A}
\mathrm d_{\mathcal D}(x,\partial\mathcal D),
\:
\lambda^{\frac{N_0}2}\theta\le\theta_1,
\]
where $\theta_1$ is the cone-width threshold in
Lemma \ref{lem:uniform-cone-graph}.

Fix $x\in A$ and $n\in P(x)$ with $n\ge N_0$.
Apply Lemma \ref{Lemma 2.6} with initial index zero.
It gives a compact pullback domain $H_n\subset\mathcal D$
whose terminal image $g^nH_n$
is a compact disc centered at $g^nx$ with
intrinsic size $R$.
A path in $g^n\mathcal D$ starting at $g^nx$ and
having length less than $R$ cannot leave $g^nH_n$.
By approximating intrinsic distances with such paths,
we obtain
$B_{g^nH_n}(g^nx,s)
=
B_{g^n\mathcal D}(g^nx,s),
\:0<s<R.$
Consequently, $V_n(x,s)\subset H_n$, and
Lemma \ref{Lemma 2.6} gives
\[
\begin{aligned}
\mathrm d_{g^j\mathcal D}(g^jx,g^jy)
&\le
\mathrm d_{g^jH_n}(g^jx,g^jy)\le
\gamma^{\frac{n-j}2}
\mathrm d_{g^nH_n}(g^nx,g^ny)\\
&\le
\vartheta^{n-j}s,
\:
y\in V_n(x,s),\:0\le j\le n.
\end{aligned}
\]
The estimate for $j=0$, together with
$\gamma^{\frac{N_0}2}R<\inf_{x\in A}
\mathrm d_{\mathcal D}(x,\partial\mathcal D)$, keeps these preballs away
from $\partial\mathcal D$.

By Lemma \ref{Lemma 2.5}, the terminal domain $g^nH_n$
is tangent to $\mathcal C^F_{\theta\lambda^{\frac n2}}$.
Lemma \ref{lem:uniform-cone-graph} therefore provides
a graph parametrization
$\Psi_n(u)=\exp_{g^nx}(u+l_n(u))$
defined on the open base ball of radius $\frac R8$, with
$G_n:=\Psi_n\bigl(B_{F(g^n x)}(0,\frac R{16})\bigr)
\subset\mathrm{Int}(g^nH_n).$
The chart and graph bounds give a uniform constant
$L\ge1$ such that
$L^{-1}\|v\|
\le\|D_u\Psi_n(v)\|
\le L\|v\|,
\:u\in B_{F(g^n x)}(0,\frac R{16}),\:v\in F(g^nx).$
In particular,
$\mathrm d_{G_n}(g^nx,\partial G_n)
\ge\frac{R}{16L}.$
Choose $r>0$ so that $4r<\frac{R}{16L}$.
Then $B_{g^n\mathcal D}(g^nx,4r)\subset G_n.$

For $0<s\le4r$, radial graph paths and the inverse
parametrization bound yield
\[
\Psi_n(B_{F(g^n x)}(0,\frac sL))
\subset
B_{g^n\mathcal D}(g^nx,s)
\subset
\Psi_n(B_{F(g^n x)}(0,Ls)).
\]
Thus the ball contains a graph core of base radius
$C_2s$, with $C_2=L^{-1}$, and then
\[
C_1^{-1}s^q
\le
\mathrm{Leb}_{g^n\mathcal D}
\bigl(B_{g^n\mathcal D}(g^nx,s)\bigr)
\le C_1s^q
\]
for a uniform constant $C_1\ge1$.

Finally, Lemma \ref{Lemma 2.5} gives the width of the terminal cone
 $\theta\lambda^{\frac n2}$ relative to the splitting
at each point.
On $B_{g^n\mathcal D}(g^nx,s)$, comparison with the splitting at $g^nx$, using uniform continuity
of the bundles and the exponential-coordinate
derivatives, gives the slope bound
$C_1\lambda^{\frac n2}+\omega(s),
\:\omega(s)\to0,
\:\text{as }s\to 0,$
after enlarging $C_1$ if necessary.
The same cone estimates along the full orbit segments
starting at time zero control all intermediate
tangent planes over $g^jV_n(x,4r)$, for all $0\le j\le n$.
\end{proof}

\begin{lemma}\label{lem:coverage-distortion}
After decreasing $r$ if necessary, there exists $K_0\ge1$
such that, for any $x\in A$, any $n\in P(x)$ with $n\ge N_0$,
\begin{equation}\label{eq:coverage-distortion}
K_0^{-1}
\le
\frac{J_{\mathcal D}g^n(y)}
     {J_{\mathcal D}g^n(z)}
\le K_0,\:\forall\:y,z\in V_n(x,4r).
\end{equation}
Here $J_{\mathcal D}g^n(y)$ denotes the
$q$-dimensional Riemannian Jacobian of
$g^n|_{\mathcal D}$ at $y$.
\end{lemma}

\begin{proof}
Put $x_j=g^jx,\:0\le j\le n.$ For $0\le j<n$, consider the coordinate maps
$G_j=\exp_{x_{j+1}}^{-1}\circ g\circ\exp_{x_j}.$
By compactness and the $C^{1+\alpha}$ regularity of $g$,
these maps are defined on uniform neighborhoods of the
origins and have uniform $C^{1+\alpha}$ bounds.
For $y\in V_n(x,4r)$, write
$v_j(y)=\exp_{x_j}^{-1}(g^jy),\:0\le j\le n.$
After decreasing $r$, all these coordinates are well defined,
and Lemma \ref{lem:coverage-preballs} gives
\begin{equation}\label{eq:distortion-orbit-distance}
\|v_j(y)\|
\le 4r\vartheta^{n-j},
\:0\le j\le n.
\end{equation}
All constants below are independent of $x$, $n$, $j$ and $y$.

Consider the coordinate tangent planes
$\widehat H_j(y)
=
D_{g^jy}(\exp_{x_j}^{-1})
\bigl(T_{g^jy}(g^j\mathcal D)\bigr)
\subset T_{x_j}\mathbf M^d.$
The uniform cone bounds in
Lemma \ref{lem:coverage-preballs}, together with the
continuity of the splitting and the
exponential-coordinate derivatives, imply, after
decreasing $r$, that these planes are uniformly
transverse to $E(x_j)$.
Since $\dim\widehat H_j(y)=\dim F(x_j)=q$, there exist
linear maps $L_j(y):F(x_j)\to E(x_j)$
such that
$\widehat H_j(y)
=
\{L_j(y)u+u:u\in F(x_j)\},
\:
\|L_j(y)\|\le L_*$
for a uniform constant $L_*>0$.
This is a linear graph representation of each tangent
plane; no graph representation of the entire intermediate
disc is required.

Let $\pi_j^F$ denote the projection onto $F(x_j)$
along $E(x_j)$, and define
$I_j(y)u=L_j(y)u+u,
\:u\in F(x_j).$
The norms of $\pi_j^F$ and $I_j(y)$ are uniformly bounded.
For $0\le j<n$, put
$D_j^F=D_{x_j}g|_{F(x_j)},
\:
T_j(y)
=
\pi_{j+1}^F D_{v_j(y)}G_jI_j(y).$
Thus both $D_j^F$ and $T_j(y)$ map $F(x_j)$ to
$F(x_{j+1})$. Since $D_0G_j=D_{x_j}g$ and the splitting is invariant,
$\pi_{j+1}^F (D_0G_j)I_j(y)=D_j^F.$
Consequently,
$T_j(y)-D_j^F
=
\pi_{j+1}^F
\bigl(D_{v_j(y)}G_j-D_0G_j\bigr)I_j(y),$
and the uniform $C^{1+\alpha}$ bounds give
$\|T_j(y)-D_j^F\|
\le C\|v_j(y)\|^\alpha.$
Here the comparison is made in fixed fibers, so no
H\"older continuity of the bundles is needed. The conorms $m(D_j^F)$ have a uniform positive lower bound.
By decreasing $r$ once more, we may therefore assume that
all $T_j(y)$ are invertible and
\begin{equation}\label{eq:distortion-coordinate-jacobian}
\begin{aligned}
\bigl|
\log|\det T_j(y)|
-
\log|\det D_j^F|
\bigr|
&\le C\|v_j(y)\|^\alpha\le C(4r)^\alpha
       \vartheta^{\alpha(n-j)}.
\end{aligned}
\end{equation}

The chain rule gives
$D_{v_j(y)}G_j\bigl(\widehat H_j(y)\bigr)
=
\widehat H_{j+1}(y).$
The definition of $T_j(y)$ and the graph representation
of $\widehat H_{j+1}(y)$ therefore yield
$(D_{v_j(y)}G_j)I_j(y)
=
I_{j+1}(y)T_j(y).$
Define
$\mathcal S_j(y)
=
D_{v_j(y)}\exp_{x_j}\circ I_j(y)
:
F(x_j)\to T_{g^jy}(g^j\mathcal D),
\:
h_j(y)=J^q\mathcal S_j(y).$
There exist uniform constants $0<h_-\le h_+<+\infty$
such that
$h_-\le h_j(y)\le h_+.$
Indeed, the upper bound follows from the uniform bounds
for $I_j(y)$ and the chart derivatives.
For the lower bound, use
$\pi_j^F I_j(y)=\mathrm{id}_{F(x_j)}$,
the uniform projection bounds, and the uniform bounds
for the inverse chart derivatives.

By the chain rule,
$D_{g^jy}g\circ\mathcal S_j(y)
=
\mathcal S_{j+1}(y)\circ T_j(y).$
Taking $q$-dimensional Jacobians yields
$J_{g^j\mathcal D}g(g^jy)\,h_j(y)
=
h_{j+1}(y)|\det T_j(y)|.$
Multiplying these identities for $0\le j<n$, the
intermediate volume factors cancel, giving
\begin{equation}\label{eq:distortion-telescoping}
J_{\mathcal D}g^n(y)
=
\frac{h_n(y)}{h_0(y)}
\prod_{j=0}^{n-1}|\det T_j(y)|.
\end{equation}
Set
$S_r=
C(4r)^\alpha
\frac{\vartheta^\alpha}{1-\vartheta^\alpha}.$
Summing \eqref{eq:distortion-coordinate-jacobian} and
using \eqref{eq:distortion-telescoping}, we obtain
\[
\frac{h_-}{h_+}e^{-S_r}
\le
\frac{J_{\mathcal D}g^n(y)}
     {\displaystyle\prod_{j=0}^{n-1}|\det D_j^F|}
\le
\frac{h_+}{h_-}e^{S_r}.
\]
The denominator is independent of $y$.
Comparing this estimate for $y$ and $z$ proves
\eqref{eq:coverage-distortion} with
\[
K_0=
\left(\frac{h_+}{h_-}\right)^2 e^{2S_r}.
\]
\end{proof}

\begin{lemma}\label{lem:coverage-density}
Let $\mathcal D$ be a compact disc and $A\subset\mathrm{Int}\mathcal D$
the measurable set of admissible centers in
Lemmas \ref{lem:coverage-preballs}
and \ref{lem:coverage-distortion}.
Assume that $P(x)$ is unbounded for all $x\in A$
and that $\{x\in A:n\in P(x)\}$ is measurable for all
$n\in\mathbb N$. Then, for $r$ small enough, for any measurable set $B\subset \mathcal D$, 
one has
\begin{equation}\label{eq:coverage-density}
\frac{\mathrm{Leb}_{\mathcal D}(V_n(x,r)\setminus B)}
     {\mathrm{Leb}_{\mathcal D}(V_n(x,r))}
\to0,
\:\text{as }P(x)\ni n\to+\infty,
\end{equation}
for $\mathrm{Leb}_{\mathcal D}$-a.e. $x\in B\cap A$. If $\mathrm{Leb}_{\mathcal D}(B\cap A)>0$, there exists a measurable set
$K\subset B\cap A$ with $\mathrm{Leb}_{\mathcal D}(K)>0$ such that, for any $0<s\le r$,
\[
\lim_{N\to+\infty}
\sup_{x\in K}
\sup_{\substack{n\in P(x)\\ n\ge N}}
\frac{\mathrm{Leb}_{\mathcal D}(V_n(x,s)\setminus B)}
     {\mathrm{Leb}_{\mathcal D}(V_n(x,s))}
=0.
\]
\end{lemma}

\begin{proof}
All time labels below are at least $N_0$. The area formula, the terminal volume estimates, and
bounded distortion give a uniform constant $C\ge1$
such that
\begin{equation}\label{eq:coverage-volume-comparison}
\frac{\mathrm{Leb}_{\mathcal D}(V_n(x,t))}
     {\mathrm{Leb}_{\mathcal D}(V_n(x,s))}
\le C\left(\frac{t}{s}\right)^q,
\:\forall\:0<s\le t\le4r,
\end{equation}
for all $x\in A$ and $n\in P(x)$.
Indeed, writing
$Q_n(x,t)=B_{g^n\mathcal D}(g^nx,t)$, we have
\[
\mathrm{Leb}_{\mathcal D}(V_n(x,t))
=
\int_{Q_n(x,t)}
\frac{1}{J_{\mathcal D}g^n(g^{-n}u)}
\,d\mathrm{Leb}_{g^n\mathcal D}(u).
\]
The ratio of any two inverse Jacobians in this formula,
for points in $Q_n(x,4r)$, is bounded by the distortion
constant, while the terminal volumes are uniformly
comparable to the corresponding $q$th powers of the radii.
After enlarging $C$, we obtain
\begin{equation}\label{eq:coverage-volume-doubling}
\mathrm{Leb}_{\mathcal D}(V_n(x,4r))
\le C\,\mathrm{Leb}_{\mathcal D}(V_n(x,r)).
\end{equation}

We also have the following property: for any $x,\:y\in A$, one has
\begin{equation}\label{eq:coverage-engulfing}
n\ge m,\:
V_n(x,r)\cap V_m(y,r)\neq\emptyset
\:\Longrightarrow\:
V_n(x,r)\subset V_m(y,4r),
\end{equation}
whenever $n\in P(x)$ and $m\in P(y)$.
To see this, choose
$w\in V_n(x,r)\cap V_m(y,r)$.
For any $z\in V_n(x,r)$, the backward-contraction
estimate gives
\[
\begin{aligned}
\mathrm d_{g^m\mathcal D}(g^my,g^mz)
&\le
\mathrm d_{g^m\mathcal D}(g^my,g^mw)
+\mathrm d_{g^m\mathcal D}(g^mw,g^mx)
+\mathrm d_{g^m\mathcal D}(g^mx,g^mz)\\
&\le r+2r\vartheta^{n-m}\le3r<4r.
\end{aligned}
\]
This proves \eqref{eq:coverage-engulfing}.

For $n\ge N_0$, define on $A$
\[
a_n(x)=
\begin{cases}
\displaystyle
\frac{\mathrm{Leb}_{\mathcal D}(V_n(x,r)\setminus B)}
     {\mathrm{Leb}_{\mathcal D}(V_n(x,r))},
& n\in P(x),\\[2mm]
0,
& n\notin P(x).
\end{cases}
\]
These functions are measurable.
Indeed, for each fixed $n$,
the set $\{x\in A:n\in P(x)\}$ is measurable, and the kernel
$(x,y)\mapsto
\mathbf 1_{\{
\mathrm d_{g^n\mathcal D}(g^nx,g^ny)\le r
\}}$
is Borel measurable on $\mathcal D\times\mathcal D$.
Integration with respect to leaf volume shows that the
numerator and denominator are measurable functions of $x$,
and the denominator is strictly positive. Fix $\varepsilon>0$, and put
$Z_\varepsilon
=
\left\{
x\in B\cap A:
a_n(x)>\varepsilon
\text{ for infinitely many }n
\right\}.$
This set is measurable.
Let $O\subset \mathcal D$ be relatively open and contain $B$.
Consider the family
\[
\begin{aligned}
\mathscr V_\varepsilon(O)=
\bigl\{V_n(x,r):\;&
x\in Z_\varepsilon,\:
n\in P(x),\:n\ge N_0,\:V_n(x,r)\subset O,\:
a_n(x)>\varepsilon
\bigr\}.
\end{aligned}
\]
Every $x\in Z_\varepsilon$ is the center of members
of this family with arbitrarily large time labels.
In fact,
$x\in O,\:\mathrm d_{\mathcal D}(x,y)\le r\vartheta^n,
\:\text{for every }y\in V_n(x,r),$
so its preballs are contained in $O$ for all sufficiently
large admissible times. The required threshold may
depend on $x$.

Select a disjoint subfamily of
$\mathscr V_\varepsilon(O)$ in increasing order of
the integer time labels.
At each time, choose a maximal disjoint family among
the candidates disjoint from all previously selected
preballs.
This selection at each fixed time is finite:
the lower bound of the terminal volume and the area formula
give
\[
\mathrm{Leb}_{\mathcal D}(V_n(x,r))
\ge
\frac{C'r^q}{\sup_{u\in\mathcal D}J_{\mathcal D}g^n(u)}
>0,
\]
for a uniform constant $C'>0$.
Since $\mathrm{Leb}_{\mathcal D}(\mathcal D)<+\infty$, a disjoint family with that
fixed time label has bounded cardinality. 

Denote the resulting disjoint family by $\left\{V_{i,j}=V_{n_i}(x_{i,j},r):\:i\le n_{0},\:0\le j\le \ell_i\right\}$, where \(i\) is a positive integer and \(j\) is a nonnegative integer, $n_0\in\mathbb N\cup\left\{\infty\right\}$, and $\ell_i\in\mathbb N$ for each $i$.
Every unselected candidate meets a selected preball
whose time label is no larger than its own.
By \eqref{eq:coverage-engulfing}, it is therefore
contained in the corresponding enlarged preball.
Since every point of $Z_\varepsilon$ is the center
of a candidate, we obtain
$Z_\varepsilon
\subset
\bigcup_{i\le n_0,\:0\le j\le\ell_i}V_{n_i}(x_{i,j},4r).$
Using \eqref{eq:coverage-volume-doubling}, one has
\[
\begin{aligned}
\mathrm{Leb}_{\mathcal D}(Z_\varepsilon)
&\le
\sum_{i\le n_0,0\le j\le\ell_i}\mathrm{Leb}_{\mathcal D}(V_{n_i}(x_{i,j},4r))\le
C\sum_{i\le n_0,0\le j\le\ell_i}\mathrm{Leb}_{\mathcal D}(V_{i,j})\\
&\le
\frac{C}{\varepsilon}
\sum_{i\le n_0,0\le j\le\ell_i}\mathrm{Leb}_{\mathcal D}(V_{i,j}\setminus B)\le
\frac{C}{\varepsilon}
\mathrm{Leb}_{\mathcal D}(O\setminus B).
\end{aligned}
\]
By outer regularity, the last quantity can be made
arbitrarily small. Hence $\mathrm{Leb}_{\mathcal D}(Z_\varepsilon)=0$.
This proves
\eqref{eq:coverage-density}.

Suppose now that $\mathrm{Leb}_{\mathcal D}(B\cap A)>0$.
For $N\ge N_0$, define
$u_N(x)=\sup_{n\ge N}a_n(x).$
These measurable functions decrease to zero at
$\mathrm{Leb}_{\mathcal D}$-almost every point of $B\cap A$.
Since $\mathrm{Leb}_{\mathcal D}(\mathcal D)<+\infty$, Egorov's theorem provides a
measurable set $K\subset B\cap A$ with $\mathrm{Leb}_{\mathcal D}(K)>0$ on which $u_N\to0$ uniformly. Finally, for every fixed $0<s\le r$,
the inclusion $V_n(x,s)\subset V_n(x,r)$ and
\eqref{eq:coverage-volume-comparison} give
\[
\frac{\mathrm{Leb}_{\mathcal D}(V_n(x,s)\setminus B)}
     {\mathrm{Leb}_{\mathcal D}(V_n(x,s))}
\le
C\left(\frac rs\right)^q
\frac{\mathrm{Leb}_{\mathcal D}(V_n(x,r)\setminus B)}
     {\mathrm{Leb}_{\mathcal D}(V_n(x,r))}.
\]
The factor on the right is independent of $x$ and $n$.
Thus, uniform convergence on the same set $K$ holds
at every fixed radius $0<s\le r$, as claimed.
\end{proof}

\subsection{Localization of the entropy construction}

\begin{proposition}\label{prop:coverage-localization}
Let
$R\subset\mathbf M^d$ be Borel and $f$-invariant.
Suppose that at least
one of the following conditions holds:
\begin{enumerate}[label=\textup{(\roman*)}]
\item $\mathrm{Leb}(R\cap\mathcal H_0)>0$.
\item the
total-probability assumption of Theorem \ref{thm:D}
holds at $a=0$ and
$\mathrm{Leb}(R\cap H_C)>0$.
\end{enumerate}
Then, there exist an iterate $g=f^p$, a disc $\mathcal D$, a measurable
set $A\subset R\cap\mathrm{Int}\mathcal D$ with $\mathrm{Leb}_{\mathcal D}(A)>0$, original full hyperbolic-time sets $P(x)$ for all $x\in A$,
a $g$-invariant probability measure $\nu$ supported on
$\Lambda$, and a finite nonzero measure $\eta\le\nu$
such that:
\begin{enumerate}
\item
The preceding preball lemmas apply on $A$, and
\[
\lim_{N\to+\infty}
\sup_{x\in A}
\sup_{\substack{n\in P(x)\\ n\ge N}}
\frac{\mathrm{Leb}_{\mathcal D}(V_n(x,r)\setminus R)}
     {\mathrm{Leb}_{\mathcal D}(V_n(x,r))}
=0.
\]

\item
Almost every ergodic component of $\nu$ is a hyperbolic
SRB measure for $g$, with negative $E$ exponents and
positive $F$ exponents.

\item
For any $z\in\mathrm{supp}\eta$, there exist sequences $\left\{x_k\right\}_{k\in\mathbb N^+}$ and $\left\{n_k\right\}_{k\in\mathbb N^+}$ with $x_k\in A$ and $n_k\in P(x_k)$ for all $k\in\mathbb N^+$ such that
$n_k\to+\infty$ and $g^{n_k}x_k\to z$, as $k\to+\infty$.
\end{enumerate}
\end{proposition}

\begin{proof}
Apply the existence reductions for
Theorem \ref{thm:B}, respectively Theorem \ref{thm:D},
starting with $R\cap\mathcal H_0$, respectively
$R\cap H_C$.
These reductions select a fixed positive gap and block
length, restrict measurable sets, and make fixed finite
time shifts. Since $R$ is $f$-invariant, all these
operations preserve membership in $R$.
Any shift needed to enter the smaller trapping
neighborhood is made at this stage.
We thus obtain the hypotheses of
Theorem \ref{thm:A}(2), respectively
Theorem \ref{thm:A}(1), on a positive-volume subset of $R$.

Lemma \ref{Lemma 3.1} and the Fubini step in
Section 3.2 give an iterate $g=f^p$, a fixed small
disc $\mathcal D$, and a measurable set
$A_0\subset R\cap\mathrm{Int}\mathcal D$ such that $\mathrm{Leb}_{\mathcal{D}}(A_0)>0$, with
$\inf_{x\in A_0}
d_{\mathcal D}(x,\partial\mathcal D)>0$.
The points of $A_0$ satisfy uniform all-prefix average
domination and admit the original full Pliss-time
sets $P(x)$, whose lower densities exceed some
$\beta>0$.
The suffix estimates hold for every starting index
between zero and the terminal time; deleting finitely
many early labels in the first case of
Lemma \ref{Lemma 3.1} preserves these estimates.
Consequently, Lemma \ref{Lemma 2.6} applies with initial
index zero at every sufficiently large original time,
and the preceding preball lemmas apply on $A_0$.

Apply Lemma \ref{lem:coverage-density} to $R\cap \mathcal D$.
After restricting to a measurable set $A\subset A_0$ with $\mathrm{Leb}_{\mathcal D}(A)>0$, its conclusion is uniform
on $A$. This proves (1).
Keeping $g$, $\mathcal D$ and $P(x)$ fixed,
carry out the construction of the measures of Section 3.2
with the initial set $A$.
All further restrictions preserve (1).

This construction gives measurable sets
$A_{n,M}\subset A$, time sets
$E_{n,M}\subset\{0,\ldots,n-1\}$, and normalized
leaf measures
$m_{n,M}=(\mathrm{Leb}_{\mathcal D})|_{A_{n,M}}$,
with
$\nu_{n,M}
=
\frac1{\#E_{n,M}}
\sum_{j\in E_{n,M}}g^j_{*}(m_{n,M}),$
for all $n\in\mathbf n$ and $M\in\mathbb N^+$ large enough.
Define simultaneously
\begin{equation}\label{eq:coverage-submeasure}
\eta_{n,M}
=
\frac1{\#E_{n,M}}
\int
\sum_{j\in E_{n,M}\cap P(x)}
\delta_{g^jx}\,dm_{n,M}(x).
\end{equation}
In particular, $0\le\eta_{n,M}\le\nu_{n,M}$.
The uniform estimate in Section 3.2 gives
\[
\#(E_{n,M}\cap P(x))
\ge
\left(\beta-\frac1M-\varepsilon_{n,M}\right)n,
\:\forall\:x\in A_{n,M},
\]
where, for each fixed $M$, $\varepsilon_{n,M}\to0$
along the selected $n$-subsequence.
For sufficiently large $M$
and $n$ we obtain
$\eta_{n,M}(\mathbf M^d)
\ge\beta-\frac1M-\varepsilon_{n,M}.$
Without loss of generality, we assume that
$(\nu_{n,M},\eta_{n,M})
\to(\nu_M,\eta_M),
\:\mathbf n\ni n\to+\infty,
\:
(\nu_M,\eta_M)\to(\nu,\eta),
\:M\to+\infty.$
It follows that
$0\le\eta\le\nu,
\:
\eta(\mathbf M^d)\ge\beta>0.$
Lemma \ref{Lemma 3.2} gives the $g$-invariance of $\nu$,
and the trapping property gives $\nu(\Lambda)=1$.
Hence $\eta$ is also supported on $\Lambda$.

The entropy estimates in
Section \ref{sec:entropy-lower} apply to this restricted
initial set and to the same limiting measure $\nu$.
Before any ergodic component is selected, they give
\begin{equation}\label{eq:coverage-whole-entropy}
h_\nu(g)
\ge
\int\log|\det(Dg|_F)|\,d\nu.
\end{equation}
Lemma \ref{Lemma 3.3} gives positive $F$ exponents
$\nu$-almost everywhere and, under the first alternative
of the proposition, negative $E$ exponents as well.

Under the second alternative, put
$\mu=\frac1p\sum_{j=0}^{p-1}f_*^j\nu.$
Then $\mu\in\mathcal M(\Lambda,f)$ and $\nu\le p\mu$.
The total-probability assumption gives $\mu(\Gamma)=1$,
hence $\nu(\Gamma)=1$.
Since $\lambda_E^{\max}(g,x)
=
p\,\lambda_E^{\max}(f,x),$
the $E$ exponents are negative $\nu$-almost everywhere
in this case also.

Write $\nu=\int m\,d\tau(m)$ for the ergodic
decomposition with respect to $g$.
For $\tau$-almost every $m$, the preceding sign
conditions hold $m$-almost everywhere.
Thus the sum of its positive Lyapunov exponents,
counted with multiplicity, equals
$\int\log|\det(Dg|_F)|\,dm$.
Ruelle's inequality gives
$\int\log|\det(Dg|_F)|\,dm-h_m(g)\ge0.$
On the other hand, the entropy decomposition formula
and \eqref{eq:coverage-whole-entropy} imply
\[
\int
\left(
\int\log|\det(Dg|_F)|\,dm-h_m(g)
\right)d\tau(m)
\le0.
\]
Consequently, equality holds for $\tau$-almost every
$m$.
By the $C^{1+\alpha}$ entropy characterization
\cite{Brown22,LY85-1}, these components are hyperbolic
SRB measures. This proves (2).

Finally, fix $z\in\mathrm{supp}\eta$. For $k\in\mathbb N^+$, let
$O_k=\left\{x\in\mathbf M^d:\mathrm{d}(x,z)<\frac{1}{k}\right\},\:a_k=\eta(O_k)>0.$
We choose a sequence $\left\{M_k\right\}_{k\in\mathbb N^+}$ with
$M_k\to+\infty$, as $k\to+\infty$ such that
$\eta_{M_k}(O_k)>\frac{a_k}{2},\:\forall\:k\in\mathbb N^+.$
For sufficiently large $M$,
the counting estimate implies $\#E_{n,M}\to+\infty$.
We may therefore choose a sequence 
$\left\{N_k\right\}_{k\in\mathbb N^+}$ with $N_{k}\to+\infty$, as $k\to+\infty$ such that
\[
\eta_{N_k,M_k}(O_k)>\frac{a_k}{4},
\:
\frac{k+1}{\#E_{N_k,M_k}}<\frac{a_k}{8},\:\forall\:k\in\mathbb N^+.
\]
The total contribution of labels $j\le k$ to
\eqref{eq:coverage-submeasure} is at most the second
quantity.
Removing those labels leaves positive mass in $O_k$.
Hence, for any $k\in\mathbb N^+$, 
there exist $x_k\in A_{N_k,M_k}\subset A$ and
$n_k\in E_{N_k,M_k}\cap P(x_k),
\:n_k>k,$
such that $g^{n_k}x_k\in O_k$.
This proves (3).
\end{proof}

\subsection{Stable saturation and the residual-set argument}
\begin{lemma}\label{lem:coverage-pesin}
Let $\mu$ be an ergodic
SRB measure supported on $\Lambda$ with respect to $g$ such that
\[
\lambda_E^{\max}(g,x)<0<
\lambda_F^{\min}(g,x),\:
\text{$\mu$-a.e. }x\in\Lambda.
\]
Let $\xi$ be a nonzero finite measure
such that $\xi\le c\mu$ for some $c<+\infty$. For any $b_0>0$ and any $B_0>0$, there exist
$z\in\mathrm{supp}\xi$, $\sigma>0$ and $s_0>0$
such that, for any $0<s<s_0$, there are
$\delta_s>0$ and $c_s>0$ with the following property.
Let $\Sigma$ be a $C^1$ embedded $q$-dimensional
submanifold and let $Q\subset\Sigma$ be Borel.
Suppose that
$\mathrm{Leb}_\Sigma(Q)\le B_0s^q$
and $Q$ contains a graph $\left\{
\exp_z\bigl(u+h(u)\bigr):
u\in F(z),\ \|u-a\|<b_0s
\right\},$
where $a\in F(z)$ and $h:\{u\in F(z):\|u-a\|<b_0s\}\to E(z)$
is $C^1$, with
$\sup_{\|u-a\|<b_0s}\|D_uh\|\le\sigma,
\:
\|a\|+\|h(a)\|<\delta_s.$
Then
\[
\mathrm{Leb}_\Sigma(Q\cap B_g(\mu))
\ge c_s\mathrm{Leb}_\Sigma(Q).
\]
\end{lemma}

\begin{proof}
The exponent signs imply that
$$E(x)=E^s_{g}(x):=\oplus_{\lambda_j(g,x)<0}E^j_{g}(x),
\:
F(x)=E_g^u(x):=\oplus_{\lambda_j(g,x)>0}E^j_{g}(x),\:\mu\mathrm{-a.e.\:}x\in\Lambda.$$
By ergodicity, the Lyapunov exponents are uniformly separated from zero on a full $\mu$-measure set of regular points. The regular points with this spectral gap are covered by countably
many common levels of the Pesin estimates for $g$ and $g^{-1}$
(see \cite[Theorems 5.4.6 and 5.6.1]{BarreiraPesin}).
Since $0\ne\xi\ll\mu$, one such level contains a compact set $K$
with $\xi(K)>0$, on which the growth controls are bounded and
the local chart radii are bounded away from zero. Since $\mu$ is ergodic, $\mu(B_g(\mu))=1$.
Thus inner regularity gives a compact set
$\Lambda_\mu\subset K\cap B_g(\mu),
\:
\xi(\Lambda_\mu)>0.$
Disintegrate $\mu$ over a measurable partition
subordinate to local unstable manifolds. The SRB
property and the leafwise Lebesgue differentiation theorem show that
$\xi$-almost every point of $\Lambda_{\mu}$ is a leaf-volume
density point of $\Lambda_{\mu}$ along its local unstable manifold.
Fix such a point
$z\in \Lambda_\mu\cap\mathrm{supp}(\xi|_{\Lambda_\mu})
\subset\mathrm{supp}\xi.$

Work in a sufficiently small exponential-coordinate
rectangle $\mathcal R$ at $z$, using coordinates
$(u,v)\in F(z)\oplus E(z)$.
A neighborhood of $z$ in $W^u_{loc}(g,z)$ is represented
by
$v=\Psi^u_{g,z}(u),\:\Psi^u_{g,z}(0)=0.$
For $w=\exp_z(u_w+v_w)\in K\cap\mathcal R$
sufficiently close to $z$, the chosen local stable
plaque has the representation
$u=u_w+\Psi^s_{g,w}(v-v_w),
\:
\Psi^s_{g,w}(0)=0.$
On the chosen Pesin block, the local stable plaques have uniformly
positive size and vary continuously in the $C^1$ topology \cite[\S8.1, (8.4)--(8.5)]{BarreiraPesin}. After shrinking the rectangle, we may assume that
$\|D\Psi^u_{g,z}\|\le\kappa,
\:
\sup_{w\in K\cap\mathcal R}
\|D\Psi^s_{g,w}\|\le\kappa,
\:
0<\kappa<\frac14.$

Choose $\delta_*>0$ sufficiently small that graphs in the chosen
box with $C^1$ norm less than $2\delta_*$ are transverse to the
local stable plaques wherever they meet. Fix the slope bound $\sigma>0$ in the statement so that $\sigma\le\min\left\{\frac{1}{4},\frac{\delta_{*}}{8}\right\}.$ Since the source transversal is tangent to $F(z)$ at $z$,
its graph function satisfies $\Psi^u_{g,z}(0)=0,
\:
D_0\Psi^u_{g,z}=0.$
By continuity of its derivative, choose $s_0>0$
sufficiently small that $\frac{b_0s_0}2\le1$,
the function $\Psi^u_{g,z}$ is defined on a neighborhood
of $B_{F(z)}(0,\frac{b_0s_0}2)$, and
$\sup_{\|u\|\le\frac{b_0s_0}2}
\|D_u\Psi^u_{g,z}\|\le\sigma.$ Fix $0<s<s_0$ and put $\rho=\frac{b_0s}{2},
\:
B=\{u\in F(z):\left\|u\right\|<\rho\}$.
In addition to the preceding geometric requirements,
choose $\delta_s>0$ so that
$\delta_s\le
\min\left\{
\frac{b_0s}{4},
\frac{\delta_*}{2(1+\sigma)}
\right\}.$
Let $a$ and $h$ satisfy the assumptions of the lemma. For each target set $Q$ satisfying the hypotheses, let
$Q'=\Phi_{Q'}(B)\subset Q,
\:
\Phi_{Q'}(u)
=
\exp_z\bigl(u+h(u)\bigr)$
be its chosen graph core.
For every $u\in\overline B$,
$\|u-a\|
\le\|u\|+\|a\|
<\frac{b_0s}{2}+\frac{b_0s}{4}
<b_0s.$
Thus $\overline B$ is contained in the domain of $h$,
and $h$ is of class $C^1$ on a neighborhood of
$\overline B$, with $\sup_{u\in\overline B}\|D_uh\|\le\sigma.$
Moreover, the segment joining $0$ to $a$ lies in the
domain of $h$, so
\[
\begin{aligned}
\|h(0)\|
&\le \|h(a)\|+\sigma\|a\|\le(1+\sigma)\bigl(\|a\|+\|h(a)\|\bigr)<(1+\sigma)\delta_s
\le\frac{\delta_{*}}{2}.
\end{aligned}
\]
For every $u\in\overline B$, integration along the
segment joining $0$ to $u$ gives
$\|h(u)-\Psi^u_{g,z}(u)\|
\le\|h(0)\|+2\sigma\|u\|
\le\|h(0)\|+2\sigma\rho.$
Also,
$\sup_{u\in\overline B}
\|D_uh-D_u\Psi^u_{g,z}\|\le2\sigma.$
Using the sum of the uniform function and derivative
norms as the $C^1$ norm, we conclude that
\[
\begin{aligned}
\|h-\Psi^u_{g,z}\|_{C^1(\overline B)}
&\le\|h(0)\|+2\sigma(1+\rho)<\frac{\delta_*}{2}+4\sigma\le\delta_*.
\end{aligned}
\]

For this fixed scale $s$, set
$T^u=\Phi_u(B),
\:
\Phi_u(u)=\exp_z\bigl(u+\Psi^u_{g,z}(u)\bigr).$
Choose $r_s>0$ sufficiently small that
$B(z,r_s)\subset\mathcal R$ and every
$w=\exp_z(u_w+v_w)\in K\cap B(z,r_s)$ satisfies $\|u_w\|<\frac{\rho}{8},
\:
\|v_w\|<\frac{\rho}{8}.$
Since $z$ is a leaf-volume density point of
$\Lambda_\mu$ along its local unstable manifold,
there is a compact set
$L=L(s)\subset
\Lambda_\mu\cap T^u\cap B(z,r_s)$
with $\mathrm{Leb}_{T^u}(L)>0$.
This choice is independent of $\Sigma$, $Q$, $a$, and $h$.

Since $\sigma\le1$, the position condition implies
$\|h(0)\|
\le\|h(a)\|+\sigma\|a\|
<\delta_s\le\frac{\rho}{2}.$
For $w\in K\cap B(z,r_s)$ and
$\varphi\in\{\Psi^u_{g,z},h\}$, consider $\mathcal T_{w,\varphi}(u)
=
u_w+\Psi^s_{g,w}\bigl(\varphi(u)-v_w\bigr).$
For $u\in\overline B$, we have
$\|\varphi(u)-v_w\|
\le\frac{\rho}{2}+\sigma\rho+\frac{\rho}{8}
\le\frac{7\rho}{8}.$
The initial choice of $s_0$ is made sufficiently small
that these arguments lie in the uniform domains of
the stable graph functions. Consequently,
\[
\|\mathcal T_{w,\varphi}(u)\|
\le\frac{\rho}{8}
+\kappa\frac{7\rho}{8}
\le\frac{11\rho}{32}<\frac{\rho}{2},
\:
\mathrm{Lip}(\mathcal T_{w,\varphi})
\le\kappa\sigma<1.
\]
The contraction mapping theorem therefore gives a
unique intersection of the chosen stable plaque
through $w$ with each of $T^u$ and $Q'$.
Its base coordinate lies in $B_{F(z)}(0,\frac \rho2)$.
The choices above also give uniform margins from
the boundaries of the stable graph domains.

The chosen stable plaques form a local lamination
after truncation inside $\mathcal R$. Thus they define
a holonomy between the corresponding subsets of
$T^u$ and $Q'$. Let $H_Q:L\to Q'$ be its restriction
to $L$; equivalently, $H_Q(w)$ is the unique
intersection of $Q'$ with the chosen stable plaque
through $w$.
The local lamination property makes $H_Q$ injective.
Continuous dependence of the plaques and the uniform
contraction estimates make $H_Q$ continuous.

We first consider a smooth target graph defined on a neighborhood of
$\overline B$ with $\|h(0)\|<\frac\rho2$ and
$\|Dh\|_{\infty,\overline B}\le\sigma$. In the exponential
coordinates above, write
$P_v(w)=u_w+\Psi^s_{g,w}(v-v_w)$, and put $d_E=\dim E$.
For $w\in L$ and $\|v\|\le\frac{3\rho}4$ the preceding estimates give
$\|v-v_w\|\le\frac{7\rho}8$ and $\|P_v(w)\|<\frac\rho2$.
Thus the parallel discs $D_v=\{\exp_z(u+v):u\in B\}$, in the
$(u,v)$ coordinates, meet these plaques with a common interior
margin. The local volume comparison for Pesin holonomies gives
$c_{\rm flat}(s)>0$ such that for all Borel $S\subset L$,
\[
 |P_v(S)|_q\ge c_{\rm flat}(s)\,\mathrm{Leb}_{T^u}(S),\:
 \|v\|\le\frac{3\rho}4.
\]
Here $|\cdot|_q$ is Euclidean volume in $F(z)$. Apply the
finite-product estimates and the covering argument of
Pugh--Shub \cite[Corollary 3.15 and the proof of Theorem 2
in Section 4]{PughShub89} to the inverse holonomy from $D_v$ to
$T^u$, using the unstable-holonomy argument for $g^{-1}$.
The source disc $T^u$ of the forward holonomy is fixed and
$C^{1+\beta}$ for some $\beta>0$, as allowed in their trial-disc
estimates. After shrinking the rectangle and $s_0$, these discs, extended
slightly beyond the working subdiscs, lie in the trial-disc cone
in the adapted charts based at $w\in L$, since the adapted and
Riemannian norms are uniformly equivalent on $K$ and the tangent
planes approach $F(z)$ as the box shrinks. On the fixed block the parallel comparison discs have
uniform $C^{1+\beta}$ chart bounds, a common lower bound on the
transverse angle, and common interior margins. The distortion
and covering constants, including the conversion to coordinate
volume, can therefore be chosen independent of $v$.

Write $\exp_z^{-1}H_Q(w)=(x(w),t(w))$, so that
$x(w)=P_{t(w)}(w)$ and $t(w)=h(x(w))$. Then $A_h=x(L)$ is compact,
and $\|t(w)\|<\frac{5\rho}8$. For $0<\eta<\frac\rho8$ put
\[
 S_\eta=\{(w,v):w\in L,\ \|v-t(w)\|<\eta\},\:
 \Theta(w,v)=(P_v(w),v).
\]
Use product Lebesgue measure on $F(z)\times E(z)$.
The map $\Theta$ is continuous and injective on the compact set
$L\times\{v\in E(z):\|v\|\le \frac{3\rho}4\}$, so $\Theta(S_\eta)$
is Borel. Applying the parallel-disc estimate on each $v$-slice
and using Fubini's theorem gives
\[
 c_{\rm flat}(s)\,\mathrm{Leb}_{T^u}(L)\,\omega_{d_E}\eta^{d_E}
 \le |\Theta(S_\eta)|_{q+d_E},
\]
where $\omega_{d_E}$ is the volume of the unit $d_E$-ball.
Every $(x,v)=\Theta(w,v)\in\Theta(S_\eta)$ satisfies
$\|x-x(w)\|\le\kappa\eta$ and
$\|v-h(x)\|\le(1+\kappa\sigma)\eta$. Integration in $x$
therefore bounds the same volume from above by
$|(A_h)_{\kappa\eta}|_q\,\omega_{d_E}
 (1+\kappa\sigma)^{d_E}\eta^{d_E},$
with $(A_h)_r$ the closed $r$-neighborhood of $A_h$. Divide by
$\omega_{d_E}\eta^{d_E}$ and let $\eta\downarrow0$.
Since $A_h$ is compact and the base projection is Lipschitz on
the fixed coordinate box, the area formula gives a constant
$c_*(s)>0$, independent of the target, such that
\begin{equation}\label{eq:stable-holonomy-lower}
 \mathrm{Leb}_{Q'}(H_Q(L))
 \ge c_*(s)\,\mathrm{Leb}_{T^u}(L).
\end{equation}

We now treat the $C^1$ target graph $h$ satisfying
the hypotheses of the lemma. Since $$\overline B\subset \mathrm{Int}(B_{F(z)}(a,b_0s)),$$ 
take smooth local
mollifications $h_j$ of $h$, adjusted by constants so that
$h_j(0)=h(0)$. Then $h_j\to h$ in $C^1(\overline B)$ and
$\|Dh_j\|_{\infty,\overline B}\le\sigma$. Set
$\Phi_j(u)=\exp_z(u+h_j(u))$ and $Q_j=\Phi_j(B)$. The same
intersection margins apply, and \eqref{eq:stable-holonomy-lower} gives $\mathrm{Leb}_{Q_j}(H_{Q_j}(L))
 \ge c_*(s)\mathrm{Leb}_{T^u}(L)$
for all sufficiently large $j$.
The continuous dependence of the stable plaques,
uniform transversality, and compactness of $L$ give $H_{Q_j}\to H_Q,
\:\text{uniformly on }L,\:\text{as}\:j\to+\infty.$ 
Therefore the compact sets
$A_j:=\Phi_j^{-1}(H_{Q_j}(L))\to
A:=\Phi_{Q'}^{-1}(H_Q(L)),\:\text{in the Hausdorff metric, as } j\to+\infty.$
Since $A\subset B$ is compact, after discarding finitely
many terms, all $A_j$ lie in a fixed compact subset
of $B$.
Moreover,
$J^q\Phi_j\to J^q\Phi_{Q'},
\:\text{uniformly on }\overline B,\:\text{as }j\to+\infty.$
For every sufficiently small $\varepsilon>0$, let
$A_\varepsilon$ be the $\varepsilon$-neighborhood of $A$
in the base space.
Hausdorff convergence gives $A_j\subset A_\varepsilon$
for all sufficiently large $j$.
The area formula and uniform convergence of the
volume densities imply
\[
\begin{aligned}
\limsup_{j\to+\infty}
\mathrm{Leb}_{Q_j}(H_{Q_j}(L))
&=
\limsup_{j\to+\infty}
\int_{A_j}J^q\Phi_j(u)\,du\le
\int_{A_\varepsilon}J^q\Phi_{Q'}(u)\,du.
\end{aligned}
\]
Letting $\varepsilon\downarrow0$, and using compactness
of $A$ and continuity from above of this finite measure,
we obtain
\[
\limsup_{j\to+\infty}
\mathrm{Leb}_{Q_j}(H_{Q_j}(L))
\le
\int_A J^q\Phi_{Q'}(u)\,du
=
\mathrm{Leb}_{Q'}(H_Q(L)).
\]
Combining this upper bound with \eqref{eq:stable-holonomy-lower} for the smooth targets
$Q_j$ yields
$\mathrm{Leb}_{Q'}(H_Q(L))
 \ge c_*(s)\mathrm{Leb}_{T^u}(L)$
for the given $C^1$ target as well.

Since $L\subset\Lambda_\mu\subset B_g(\mu)$ and
$H_Q(w)$ lies on the chosen local stable plaque
through $w$, we have
$\mathrm d\bigl(g^nH_Q(w),g^nw\bigr)\to0,
\:\text{as }n\to+\infty,
\:w\in L.$
Forward-asymptotic points have the same empirical
measure limit. Hence
$H_Q(L)\subset Q\cap B_g(\mu).$
The set $H_Q(L)$ is compact because $L$ is compact
and $H_Q$ is continuous.
Moreover, $Q'\subset\Sigma$ and both submanifolds
have dimension $q$, so their induced Riemannian
volumes agree on $Q'$. Therefore,
\[
 \mathrm{Leb}_\Sigma(Q\cap B_g(\mu))
 \ge \mathrm{Leb}_\Sigma(H_Q(L))
 = \mathrm{Leb}_{Q'}(H_Q(L))
 \ge c_*(s)\,\mathrm{Leb}_{T^u}(L).
\]
Set $c_s=\min\left\{1,\,
 \frac{c_*(s)\,\mathrm{Leb}_{T^u}(L)}{B_0s^q}\right\}>0.$
Using $\mathrm{Leb}_\Sigma(Q)\le B_0s^q$,
we obtain
\[
\mathrm{Leb}_\Sigma(Q\cap B_g(\mu))
\ge c_s\mathrm{Leb}_\Sigma(Q).
\]
The constant $c_s$ depends on $s$ and the previously
fixed data, but not on $\Sigma$ or $Q$.
This completes the proof.
\end{proof}

\begin{proof}[Proof of Theorem \ref{thm:countable-coverage}]
Every ergodic hyperbolic SRB measure is physical under
the standing regularity assumptions. Moreover, if
$\mu\in\mathscr S_{E,F}(\Lambda,f)$, then $\mu(U)=1$, so every point of
$B_f(\mu)$ eventually enters $U$.
Thus
$B_f(\mu)
=
\bigcup_{j\ge0}
\bigl(B_f(\mu)\cap f^{-j}U\bigr).$
Since $B_f(\mu)$ has positive volume, there exists $j\ge 0$ such that $\mathrm{Leb}(B_f(\mu)\cap f^{-j}U)>0$. Its image under $f^j$
is contained in $B_f(\mu)\cap U$ and still has
positive volume. Hence $\mathrm{Leb}(B_f(\mu)\cap U)>0.$
Distinct basins are disjoint, and the ambient volume
is finite. Therefore $\mathscr S_{E,F}(\Lambda,f)$ is at most countable.
The same argument applies to every iterate $g=f^p$. Each basin is Borel and is unchanged by finite time
shifts. Since $f$ is invertible, each basin is
$f$-invariant. Consequently, $\mathcal B_{E,F}$ and $R=\mathbf M^d\setminus\mathcal B_{E,F}$
are Borel and $f$-invariant.

Suppose, towards a contradiction, that
$\mathrm{Leb}(R\cap\mathcal H_0)>0$, or that
the total-probability assumption in the second
assertion holds and
$\mathrm{Leb}(R\cap H_C)>0$.
Apply Proposition \ref{prop:coverage-localization}.
It gives an iterate $g=f^p$, a disc $\mathcal D$,
a set
$A\subset R\cap\mathrm{Int}\mathcal D$ with $\mathrm{Leb}_{\mathcal D}(A)>0$,
the original full hyperbolic-time sets $P(x),\:x\in A$,
a $g$-invariant probability $\nu$ supported on
$\Lambda$, and a finite nonzero measure
$0\le\eta\le\nu$ with the stated localization
properties.

By the component assertion of that proposition, write $\nu=\sum_{j\in I}c_j\mu_j,
\:
c_j>0,
\:
\sum_{j\in I}c_j=1,$
where $I$ indexes the distinct components of the $g$-ergodic
decomposition of $\nu$ with positive weight and $\left\{\mu_j\right\}_{j\in I}$ consists of hyperbolic ergodic SRB measures for $g$, with negative exponents along $E$ and positive exponents along $F$.
Since
$\mu_j(B_g(\mu_j))=1$ for all $j\in I$
and these basins are disjoint,
$\nu\left(\bigcup_{j\in I}B_g(\mu_j)\right)=1,
\:
\nu_{B_g(\mu_j)}=c_j\mu_j,$ for all $j\in I$, where $\nu_{B_g(\mu_j)}(.)=\nu(.\cap B_g(\mu_j))$.
As $0\ne\eta\le\nu$, there is an index $i\in I$
such that $\eta(B_g(\mu_i))>0$.
Fix this index and put $\xi=\eta_{B_g(\mu_i)}.$
Then
$0\ne\xi\le c_i\mu_i,
\:
\mathrm{supp}\xi\subset\mathrm{supp}\eta.$
Define
$\widehat\mu_i
=
\frac1p\sum_{j=0}^{p-1}f_*^j\mu_i.$ Since every \(f\)-invariant Borel set is \(g\)-invariant and has the same measure under \(\widehat\mu_i\) and \(\mu_i\), the cyclic average \(\widehat\mu_i\) is \(f\)-ergodic. The Lyapunov exponent signs and Pesin’s entropy formula pass to \(\widehat\mu_i\) by the iterate identities used in Section 3.6, so \(\widehat\mu_i\in\mathscr S_{E,F}(\Lambda,f)\). For every $x\in\mathbf M^d$, the empirical measures
satisfy $\mathcal E_{np}^f(x)
=
\frac1p\sum_{j=0}^{p-1}f_*^j\mathcal E_n^g(x).$
Hence convergence of $\mathcal E_n^g(x)$ to $\mu_i$
implies convergence of $\mathcal E_{np}^f(x)$ to
$\widehat\mu_i$. Consequently,
\begin{equation}\label{eq:coverage-basin-inclusion}
B_g(\mu_i)
\subset B_f(\widehat\mu_i)
\subset\mathcal B_{E,F}.
\end{equation}
Let $C_1>0$ and $C_2>0$ be the constants in
Lemma \ref{lem:coverage-preballs}.
Apply Lemma \ref{lem:coverage-pesin} to $\mu_i$
and $\xi$, with $b_0=\frac{C_2}{8},
\:
B_0=C_1.$
It gives $z\in\mathrm{supp}\xi,
\:
\sigma>0,
\:
s_0>0,$
and, for any $0<s<s_0$, constants
$\delta_s,c_s>0$ as in that lemma.
Since $z\in\mathrm{supp}\eta$,
Proposition \ref{prop:coverage-localization}
supplies sequences $\left\{x_k\right\}_{k\in\mathbb N^+}$ and $\left\{n_k\right\}_{k\in\mathbb N^+}$ with $x_k\in A$ and $n_k\in P(x_k)$ for all $k\in\mathbb N^+$ such that
$n_k\to+\infty,
\:
y_k:=g^{n_k}x_k\to z,\;\mathrm{as}\:k\to+\infty.$

We verify the graph hypotheses of
Lemma \ref{lem:coverage-pesin} in the fixed
exponential coordinates at $z$.
For each fixed sufficiently small
$0<s<\min\{r,s_0\}$, put
$Q_k=Q_{n_k}(x_k,s),
\:
\Sigma_k=g^{n_k}(\mathrm{Int}\mathcal D).$
For all sufficiently large $k$, the preball lemma
ensures that $Q_k$ is a Borel subset of $\Sigma_k$,
that
\begin{equation}\label{eq:coverage-terminal-volume}
\mathrm{Leb}_{\Sigma_k}(Q_k)\le C_1s^q,
\end{equation}
and that $Q_k$ contains the graph
$\left\{
\exp_{y_k}\bigl(v+\ell_k(v)\bigr):
v\in F(y_k),\ \|v\|<C_2s
\right\},$
where $\ell_k(0)=0$ and
$\|D\ell_k\|_\infty
\le d_k(s):=C\lambda^{n_k/2}+\omega(s).$
Restricting to sufficiently small $s$ and sufficiently
large $k$, we may assume that $d_k(s)\le1$
and that all the coordinate maps below are defined
on common neighborhoods.

Let $\pi_z^E$ and $\pi_z^F$ be the projections
associated with $E(z)\oplus F(z)$.
Choose linear identifications
$A_k:F(z)\to F(y_k)$
which converge to the identity in a smooth local
trivialization of the tangent bundle, with
$\|A_k\|\le2$ for all sufficiently large $k$.
For all sufficiently large $k$ and $\|u\|<\frac{C_2s}2$, define
\[
\Psi_k(u)
=
\exp_z^{-1}
\left(
\exp_{y_k}\bigl(A_ku+\ell_k(A_ku)\bigr)
\right),
\]
and set $P_k=\pi_z^F\Psi_k,
\:
W_k=\pi_z^E\Psi_k,
\:
a_k=P_k(0).$

Smooth dependence of the exponential charts and
continuity of the splitting give a constant $C'>0$,
independent of sufficiently small $s$, such that,
for each fixed $s$,
$\|DP_k-\mathrm{id}_{F(z)}\|_\infty
\to0,\:\mathrm{as}\:k\to+\infty,
\:
\limsup_{k\to+\infty}\|DW_k\|_\infty\le C'\omega(s).$
The estimates may be taken on
$B_{F(z)}(0,\frac{C_2s}4)$.
For the first estimate, the contribution of
$D\ell_k$, whose range lies in $E(y_k)$,
has $F(z)$-projection tending uniformly to zero.
Only boundedness of $D\ell_k$ is needed here;
no convergence of the graph functions is required.

Put $R_s=\frac{C_2s}4$.
For sufficiently large $k$, we have
$\|DP_k-\mathrm{id}_{F(z)}\|_\infty
\le\frac14,
\:\text{on }B_{F(z)}(0,R_s).$
For any $v\in F(z)$ with
$\|v-a_k\|<\frac{R_s}2=b_0s$, the map
$u\mapsto
v-a_k-\bigl(P_k(u)-a_k-u\bigr)$
is a contraction of the closed ball
$B_{F(z)}(0,R_s)$ into itself.
Indeed, its Lipschitz constant is at most $\frac14$,
and its image lies in the ball of radius $\frac{3R_s}4$.
Thus $P_k$ has a unique inverse on
$\{v:\|v-a_k\|<b_0s\}$.
The derivative estimate and the inverse function
theorem make this inverse $C^1$.
Consequently,
$h_k(v)=W_k(P_k^{-1}(v))$
defines a $C^1$ graph satisfying
$\left\{
\exp_z\bigl(v+h_k(v)\bigr):
\|v-a_k\|<b_0s
\right\}\subset Q_k.$
Moreover, $a_k=\pi_z^F\exp_z^{-1}(y_k),
\:
h_k(a_k)=\pi_z^E\exp_z^{-1}(y_k),$
and there exists a constant $C''>0$ such that
\begin{equation}\label{eq:coverage-fixed-chart-slope}
\limsup_{k\to+\infty}\|Dh_k\|_\infty
\le C''\omega(s).
\end{equation}
Now fix $0<s<\min\{r,s_0\}$ sufficiently small
for the preceding coordinate comparisons and such that
$C''\omega(s)<\frac{\sigma}{3}.$
For this fixed $s$, equation
\eqref{eq:coverage-fixed-chart-slope}
and $y_k\to z$ imply, for all sufficiently large $k$,
\[
\|Dh_k\|_\infty\le\sigma,
\:
\|a_k\|+\|h_k(a_k)\|<\delta_s.
\]
Together with \eqref{eq:coverage-terminal-volume},
these verify every hypothesis of
Lemma \ref{lem:coverage-pesin} with
$\Sigma=\Sigma_k$ and $Q=Q_k$.
Therefore
\[
\frac{\mathrm{Leb}_{\Sigma_k}
          (Q_k\cap B_g(\mu_i))}
     {\mathrm{Leb}_{\Sigma_k}(Q_k)}
\ge c_s>0
\]
for all sufficiently large $k$.

Put $V_k=V_{n_k}(x_k,s)$.
Since $g^{n_k}V_k=Q_k$ and $B_g(\mu_i)$ is
$g$-invariant, the change-of-variables formula and
Lemma \ref{lem:coverage-distortion} give
\begin{equation}\label{eq:coverage-contradiction}
\frac{\mathrm{Leb}_{\mathcal D}
          (V_k\cap B_g(\mu_i))}
     {\mathrm{Leb}_{\mathcal D}(V_k)}
\ge
K_0^{-2}
\frac{\mathrm{Leb}_{\Sigma_k}
          (Q_k\cap B_g(\mu_i))}
     {\mathrm{Leb}_{\Sigma_k}(Q_k)}
\ge K_0^{-2}c_s>0.
\end{equation}
The distortion estimate applies because
$V_k\subset V_{n_k}(x_k,4r)$.

On the other hand, the uniform density conclusion of
Proposition \ref{prop:coverage-localization}
holds on $A$ at radius $r$.
The volume comparison in
Lemma \ref{lem:coverage-density} gives, for the
fixed radius $s\le r$,
\[
\frac{\mathrm{Leb}_{\mathcal D}
          (V_n(x,s)\setminus R)}
     {\mathrm{Leb}_{\mathcal D}(V_n(x,s))}
\le
C(\frac rs)^q
\frac{\mathrm{Leb}_{\mathcal D}
          (V_n(x,r)\setminus R)}
     {\mathrm{Leb}_{\mathcal D}(V_n(x,r))}.
\]
Thus the left-hand side tends to zero uniformly
for $x\in A$ and $n\in P(x)$ as $n\to+\infty$.
In particular,
\[
\frac{\mathrm{Leb}_{\mathcal D}(V_k\setminus R)}
     {\mathrm{Leb}_{\mathcal D}(V_k)}
\to0.
\]
By \eqref{eq:coverage-basin-inclusion},
$B_g(\mu_i)\subset\mathbf M^d\setminus R$.
The last limit therefore contradicts
\eqref{eq:coverage-contradiction}.
This proves
$\mathrm{Leb}(\mathcal H_0\setminus\mathcal B_{E,F})=0,$
and, under the total-probability assumption,
$\mathrm{Leb}(H_C\setminus\mathcal B_{E,F})=0.$

It remains to prove the reverse inclusion in
\eqref{eq:coverage-B}.
Let $x\in B_f(\mu)\cap U$ with $\mu\in\mathscr S_{E,F}(\Lambda,f)$.
Then $p\omega(x)=\{\mu\}$, and the statistical
identity \eqref{stat:identity-equation} gives
$\lim_{m\to+\infty}\liminf_{n\to+\infty}
\frac1{nm}\sum_{j=0}^{n-1}
\Phi_{m,0}(f^{jm}x)
=
\int
\min\{\lambda_F^{\min},-\lambda_E^{\max}\}\,d\mu>0.$
The strict inequality follows from the exponent
signs defining $\mathscr S_{E,F}(\Lambda,f)$.
Hence $x\in\mathcal H_0$.
We have therefore proved the pointwise inclusion
$\mathcal B_{E,F}\cap U\subset\mathcal H_0.$
Combining it with the forward coverage assertion yields
\[
\mathrm{Leb}
\bigl(\mathcal H_0\triangle(\mathcal B_{E,F}\cap U)\bigr)=0,
\]
as required.
\end{proof}
The preceding theorem proves part \textup{(3)}
of Theorem \ref{thm:B}.
To obtain the refined conclusion in part \textup{(3)}
of Theorem \ref{thm:D},
let $\varphi=\log m(Df|_F)$ on $\overline U$,
and extend it continuously to $\mathbf M^d$.
For any $\mu\in\mathscr S_{E,F}(\Lambda,f)$
and any $x\in B_f(\mu)\cap U$, the definition
of the basin gives
$\lim_{n\to+\infty}\frac1n
\sum_{j=0}^{n-1}\log m(Df|_{F(f^jx)})
=
\int\varphi\,d\mu.$
Consequently, $H_C\cap B_f(\mu)$ equals
$B_f(\mu)\cap U$ when
$\int\varphi\,d\mu>0$, and is empty otherwise.
Combining this observation with the preceding
theorem's estimate
$\mathrm {Leb}(H_C\setminus\mathcal B_{E,F})=0$
gives the asserted symmetric-difference identity.
This completes the proof of Theorem \ref{thm:D}.

\subsection{Why countability does not imply finiteness}
\begin{proposition}\label{prop:infinitely-many-basins}
There is a $C^\infty$ diffeomorphism of $\mathbb T^4$ with a proper
trapping neighborhood and a smooth dominated splitting $E\oplus F$,
$F$ uniformly expanding, for which no finite family of invariant
probabilities has basins covering $\mathcal H_0$ modulo volume zero.
The same holds for
$\{x\in U:\tau(G_0)=1\text{ for every }\tau\in p\omega(x)\}$. Furthermore, this example satisfies both alternatives \textup{(ii)}
and \textup{(iii)} of Theorem \ref{thm:E} at $a=0$.
\end{proposition}
\begin{proof}
Let $A=\left(\begin{smallmatrix}2&1\\1&1\end{smallmatrix}\right)$ on
$\mathbb T^2$, with eigenvalues $\lambda^{\pm1}$, $\lambda>1$.
Use a flat adapted metric. On a circle coordinate near zero take a
smooth cutoff $\chi=1$ near zero, and let $h$ be the time-one map of
the smooth vector field
\[
 v(t)=\begin{cases}\chi(t)e^{-\frac1{t^2}}\sin(\frac1t),&t>0,\\0,&t\le0.
 \end{cases}
\]
Extend it by zero outside the coordinate chart. Smoothness at zero
follows because all derivatives of $e^{-\frac1{t^2}}\sin(\frac 1t)$ tend to zero.
For all large $j$, the points $s_j=(2\pi j)^{-1}$ are hyperbolic sinks,
with attraction intervals
$I_j=\left(((2j+1)\pi)^{-1},((2j-1)\pi)^{-1}\right),\:
 h'(s_j)=e^{-\beta_j},\:
 \beta_j=(2\pi j)^2e^{-(2\pi j)^2}>0.$
On another circle, choose a small open interval $D$ about $0$ and a smooth
diffeomorphism $k$ that agrees with $u\mapsto e^{-\gamma}u$ on a neighborhood of
$\overline D$, where $\gamma>0$. Put
\[
 f=A^N\times h\times k,\:
 U=\mathbb T^2\times S^1\times D,\:
 \Lambda=\mathbb T^2\times S^1\times\{0\},\: b=N\log\lambda.
\]
The smooth product splitting
$E=E_A^s\oplus TS^1_h\oplus TS^1_k$, $F=E_A^u$ is dominated for
large $N$, since $m(Df|_F)=e^b$ and
$\|Df|_E\|=\max\{e^{-b},h'(t),e^{-\gamma}\}$ on $\overline U$.

If $m_A$ is Haar measure, the probabilities
$\mu_j=m_A\times\delta_{s_j}\times\delta_0\:(j\:\mathrm{large\:enough})$ are distinct ergodic
hyperbolic SRB measures, with exponents $b,-b,-\beta_j,-\gamma$;
their unstable conditionals are Haar conditionals. If $Z$ is the
full-Haar-measure basin of $m_A$ for $A^N$, then
$Z\times I_j\times D\subset B_f(\mu_j)$ by forward asymptoticity.
For any fixed $p$ and any $t\in I_j$, continuity of $(h^p)'$ gives
\[
 \lim_{n\to+\infty}\frac1{np}\sum_{\ell=0}^{n-1}
   \Phi_{p,0}(f^{\ell p}(z,t,u))
 =\min\{b,\gamma,\beta_j\}>0.
\]
Thus each entire slab belongs to $\mathcal H_0$, and its full-volume
subset $Z\times I_j\times D$ belongs to the stated statistical set,
since $p\omega(x)=\{\mu_j\}$ and $\mu_j(G_0)=1$. Since $B_f(\mu_j)$ has positive volume, $\mu_j$ is strongly
observable. The exponents computed above give
$\mathcal G_0(\mu_j)=\min\{b,\beta_j,\gamma\}>0$, so $\mu_j$
satisfies the requirement in alternative \textup{(ii)} of Theorem D.
Alternative \textup{(iii)} holds because $Z\times I_j\times D$
has positive volume and $p\omega(x)=\{\mu_j\}$ throughout this set.
Basins of different probabilities are disjoint; any finite list misses
one of these positive-volume subsets. This proves both assertions.
Moreover $B_f(\mu_j)\cap U=Z\times I_j\times D$: every orbit of the
one-dimensional flow converges to an equilibrium, and convergence of
its empirical measures to $\delta_{s_j}$ forces its initial coordinate
to lie in $I_j$. Consequently these basin volumes tend to zero.
\end{proof}

\begin{remark}\label{Remark 7.8}
The countable basin coverage in part \textup{(3)} of Theorem \ref{thm:D}
cannot in general be replaced by finite coverage. We use the blocks
of Dolgopyat--Hu--Pesin \cite{DHP01}, recalled in \cite[Section 5]{AV}.
On the $j$th rescaled block, their second-variation calculation gives
a mean center exponent of at least $c_j\tau^2$ for the primary
perturbation of size $\tau$, where $c_j>0$.
Within the smooth center–unstable plane of the product model, which the primary perturbation preserves, the forward and backward graph transforms of the product unstable and center lines, respectively, approximate the corresponding invariant lines after \(N\) steps with angular error \(O_j(\tau q_j^N)\), where \(0<q_j<1\).
On the center-unstable bundle, choose the smooth metric that makes
the approximating lines orthogonal and preserves their lengths
in the background product metric.
Let $g_\tau$ denote the primary perturbation and let $m_j$
be the normalized invariant product volume on the $j$th block.
In the chosen metric, the uniform separation of the center
and unstable expansion factors gives the Gram matrix estimate
\[
  \log m\!\left(D_xg_\tau|_{E^{cu}(x)}\right)
  \ge
  \log\left\|D_xg_\tau|_{E^c(x)}\right\|
  -C_j\tau^2q_j^{2N},
\]
where $C_j$ is uniform in $N$ and sufficiently small $\tau$. All norms in this estimate are taken in the chosen metric.
Integrating with respect to $m_j$ gives the lower bound
$(c_j-C_jq_j^{2N})\tau^2$, since the integral of the center
term is the mean center exponent. Fix $N_j$ so that $C_jq_j^{2N_j}<\frac{c_j}2$. We may then choose
$\tau_j$ small and the subsequent accessibility perturbation still
smaller, preserving positivity and making the maps and metrics
converge to the product data in the $C^j$ norm on the shrinking
blocks, with matching boundary jets. Taking the product with a
circle diffeomorphism having a hyperbolic sink gives a trapping
region with uniformly contracting $E$. Each block yields a distinct
hyperbolic ergodic SRB measure whose basin meets $H_C$ in positive
volume. Since these basins are disjoint, no finite family of basins
covers $H_C$ modulo a set of zero volume.
\end{remark}

\subsection{Proof of Corollaries \ref{cor:srb-expansion-characterization} and \ref{cor:srb-statistical-characterization}}
\label{sec:srb-characterizations}

\begin{proof}[\textbf{Proof of Corollary \ref{cor:srb-expansion-characterization}}]

For any $x\in U$, the trapping property gives $p\omega(x)\subseteq \mathcal M(\Lambda,f)$.  Extend $\varphi$ continuously from $\overline U$ to $\mathbf M^d$. Since $p\omega(x)$ is the compact set of empirical limits of $x$ and
$\tau\mapsto\int\varphi d\tau$ is continuous, we obtain
$\liminf_{n\to+\infty}\frac1n
\sum_{j=0}^{n-1}\varphi(f^jx)
=
\min_{\tau\in p\omega(x)}\int\varphi d\tau.$
Thus $H_C=\{x\in U:p\omega(x)\subseteq X_F^+\}$, which proves
\textup{(2)}$\:\Leftrightarrow\:$\textup{(5)}.
In particular, for $\mu\in\mathscr S_{E,F}(\Lambda,f)$, either
$H_C\cap B_f(\mu)=B_f(\mu)\cap U$ if $\int\varphi d\mu>0$; or $H_C\cap B_f(\mu)=\emptyset$ if $\int\varphi d\mu\le 0$. The preceding observation gives
$\mathcal B_{E,F}^+\cap U\subseteq H_C$.
When $\mathrm {Leb}(H_C)>0$, the basin coverage in
Theorem \ref{thm:D}\,\textup{(3)} gives
$\mathrm {Leb}(H_C\setminus\mathcal B_{E,F}^+)=0$.
This estimate also holds when $\mathrm {Leb}(H_C)=0$.
Consequently,
$\mathrm {Leb}\bigl(
H_C\triangle(\mathcal B_{E,F}^+\cap U)
\bigr)=0.$
Thus conditions \textup{(1)} and \textup{(2)} are equivalent.

 If $\nu\in\mathcal O_f$ and $\nu(U)>0$,
observability provides sequences $\left\{x_k\right\}_{k\in\mathbb N^+}$ with $x_k\in\mathbf M^d$ and $\left\{\tau_k\right\}_{k\in\mathbb N^+}$ with $\tau_k\in p\omega(x_k)$ for all $k\in\mathbb N^+$ such that $\tau_k\to\nu$, as $k\to+\infty$.
Since $U$ is open, the Portmanteau theorem gives
$\tau_k(U)>0$ for all sufficiently large $k$.
Hence the orbit of $x_k$ enters $U$, and the trapping
property implies $\tau_k(\Lambda)=1$.
Then $\nu(\Lambda)=1$. Therefore,
$\nu\in\mathcal O_f\cap X_F^+
 \:\Longrightarrow\:\nu\in \mathcal M(\Lambda,f).$
For $\nu\in \mathcal M(\Lambda,f)$, every probability sufficiently close
to $\nu$ assigns positive mass to $U$.
Thus any orbit having such an empirical limit
eventually enters $U$.
If $A$ is a positive-volume set of these points, then $A=\bigcup_{j\ge0}\bigl(A\cap f^{-j}(U)\bigr).$
For some $j$, the image
$f^j(A\cap f^{-j}(U))$ is a positive-volume subset of
$U$, since diffeomorphisms preserve null sets.
Finite time shifts leave $p\omega(x)$ unchanged.
Consequently, if $\nu\in \mathcal M(\Lambda,f)$ is physical, then
$\mathrm {Leb}(B_f(\nu)\cap U)>0$.
If $\nu\in \mathcal M(\Lambda,f)$ is observable in the ordinary or strong
sense, its defining condition at scale $\varepsilon$
holds on a positive-volume subset of $U$ for every
sufficiently small $\varepsilon>0$.

Let $\nu\in\mathrm P_f\cap X_F^+$.
Its local basin $B_f(\nu)\cap U$ has positive volume and is
contained in $H_C$, since $p\omega(x)=\{\nu\}$ throughout the basin.
By the basin identity, it meets $B_f(\mu)$ for some
$\mu\in\mathscr S_{E,F}(\Lambda,f)\cap X_{F}^+$. Uniqueness of the empirical limit gives
$\nu=\mu$. Conversely, every member of $\mathscr S_{E,F}(\Lambda,f)\cap X_{F}^+$ is physical.
Hence $$\mathrm P_f\cap X_F^+=\mathscr S_{E,F}(\Lambda,f)\cap X_{F}^+.$$

Now let $\nu\in\mathcal O_f^H\cap X_F^+$. By continuity of $\tau\mapsto\int\varphi d\tau$ on $\mathcal M(\Lambda,f)$, for every sufficiently small $\varepsilon>0$ we have 
$B_\rho(\nu,\varepsilon)\cap \mathcal M(\Lambda,f)\subseteq X_F^+$.
By strong observability and localization, there is a
positive-volume set $A_\varepsilon\subset U$ such that
$p\omega(x)\subseteq B_\rho(\nu,\varepsilon)$, for every
$x\in A_\varepsilon$. Thus $A_\varepsilon\subseteq H_C$.
The basin identity gives a point of $A_\varepsilon$ in
$B_f(\mu_\varepsilon)$ for some $\mu_\varepsilon\in\mathscr S_{E,F}(\Lambda,f)\cap X_{F}^+$,
so $\rho(\mu_\varepsilon,\nu)<\varepsilon$.
It follows that $\nu\in\overline{\mathscr S_{E,F}(\Lambda,f)\cap X_{F}^+}^{\,w^*}$.
The reverse inclusion after intersection with $X_F^+$ follows
from $\mathscr S_{E,F}(\Lambda,f)\subseteq\mathcal O_f^H$ and the
closedness of $\mathcal O_f^H$. Hence
\[
\mathcal O_f^H\cap X_F^+
=
\overline{\mathscr S_{E,F}(\Lambda,f)\cap X_{F}^+}^{\,w^*}\cap X_F^+=\overline{\mathscr S_{E,F}(\Lambda,f)}^{\,w^*}\cap X_F^+.
\]
Together with the physical-measure identity, this proves
\textup{(1)}$\:\Leftrightarrow\:$\textup{(3)}
$\:\Leftrightarrow\:$\textup{(4)} and completes the five equivalences.

Finally, suppose that $\mathrm {Leb}(U\setminus H_C)=0$.
For every $\mu\in\mathscr S_{E,F}(\Lambda,f)$, its positive-volume local basin
meets $H_C$, so $\mu\in X_{F}^+$. Thus $\mathscr S_{E,F}(\Lambda,f)\subset  X_{F}^+$,
and the basin identity gives
$\mathrm {Leb}(U\setminus\mathcal B_{E,F}^+)=0$.
For any $\nu\in\mathcal O_f\cap \mathcal M(\Lambda,f)$ and any sufficiently small
$\varepsilon>0$, there is a
positive-volume set $A_\varepsilon\subset U$ such that
$p\omega(x)\cap B_\rho(\nu,\varepsilon)\ne\emptyset$, for every
$x\in A_\varepsilon$. Therefore $B_f(\mu)\cap A_{\varepsilon}\cap U\ne\emptyset$
for some $\mu\in\mathscr S_{E,F}(\Lambda,f)$. For $x\in B_f(\mu)\cap A_{\varepsilon}\cap U$, $p\omega(x)=\{\mu\}$,
so $\rho(\mu,\nu)<\varepsilon$.
This gives $\mathcal O_f\cap \mathcal M(\Lambda,f)\subseteq
\overline{\mathscr S_{E,F}(\Lambda,f)}^{\,w^*}$. The reverse inclusions follow
from $\mathscr S_{E,F}(\Lambda,f)\subseteq\mathcal O_f^H\subseteq\mathcal O_f$
and the closedness of $\mathcal O_f^H$.
\end{proof}

\begin{proof}[\textbf{Proof of Corollary \ref{cor:srb-statistical-characterization}}]
By Lemma \ref{stat:identity},
$\mathcal H_0
 =
 \bigcup_{p\ge1}\{x\in U:b_{p,0}(x)>0\}
 =
 \{x\in U:\mathcal G_0(\tau)>0,\:\forall\:\tau\in p\omega(x)\}.$
Thus conditions \textup{(2)}, \textup{(3)}, and
\textup{(6)} are equivalent.
Theorem \ref{thm:B} gives \textup{(2)}$\:\Rightarrow\:$\textup{(1)}.
Every $\mu\in\mathscr S_{E,F}(\Lambda,f)$ is physical and satisfies
$\mathcal G_0(\mu)>0$, so
\textup{(1)}$\:\Rightarrow\:$\textup{(5)}
$\:\Rightarrow\:$\textup{(4)}, using
$\mathrm P_f\subset\mathcal O_f^H$.
Finally, Theorem \ref{thm:E}\:(ii) gives
\textup{(4)}$\:\Rightarrow\:$\textup{(2)}.

We shall use the basin identity
$\mathrm {Leb}
 \bigl(\mathcal H_0\triangle(\mathcal B_{E,F}\cap U)\bigr)=0$ in all cases.
It follows from Theorem \ref{thm:B}   (3) when
$\mathrm {Leb}(\mathcal H_0)>0$.
Otherwise, the equivalences just proved give
$\mathscr S_{E,F}(\Lambda,f)=\emptyset$, and the identity is immediate. If $\nu\in\mathcal O_f$ and $\nu(U)>0$,
observability provides sequences $\left\{x_k\right\}_{k\in\mathbb N^+}$ with $x_k\in\mathbf M^d$ and $\left\{\tau_k\right\}_{k\in\mathbb N^+}$ with $\tau_k\in p\omega(x_k)$ for all $k\in\mathbb N^+$ such that $\tau_k\to\nu$, as $k\to+\infty$.
Since $U$ is open, the Portmanteau theorem gives
$\tau_k(U)>0$ for all sufficiently large $k$.
Hence the orbit of $x_k$ enters $U$, and the trapping
property implies $\tau_k(\Lambda)=1$.
Then $\nu(\Lambda)=1$. Therefore,
$\nu\in\mathcal O_f,\:\mathcal G_0(\nu)>0
 \:\Longrightarrow\:\nu\in \mathcal M(\Lambda,f).$
For $\nu\in \mathcal M(\Lambda,f)$, every probability sufficiently close
to $\nu$ assigns positive mass to $U$.
Thus any orbit having such an empirical limit
eventually enters $U$.
If $A$ is a positive-volume set of these points, then $A=\bigcup_{j\ge0}\bigl(A\cap f^{-j}(U)\bigr).$
For some $j$, the image
$f^j(A\cap f^{-j}(U))$ is a positive-volume subset of
$U$, since diffeomorphisms preserve null sets.
Finite time shifts leave $p\omega(x)$ unchanged.
Consequently, if $\nu\in \mathcal M(\Lambda,f)$ is physical, then
$\mathrm {Leb}(B_f(\nu)\cap U)>0$.
If $\nu\in \mathcal M(\Lambda,f)$ is observable in the ordinary or strong
sense, its defining condition at scale $\varepsilon$
holds on a positive-volume subset of $U$ for every
sufficiently small $\varepsilon>0$.

Let $\nu\in\mathrm P_f$ satisfy $\mathcal G_0(\nu)>0$.
The preceding observations give $\nu\in \mathcal M(\Lambda,f)$ and
$\mathrm {Leb}(B_f(\nu)\cap U)>0.$
Since $p\omega(x)=\{\nu\}$ on $B_f(\nu)$,
Lemma \ref{stat:identity} implies
$B_f(\nu)\cap U\subset\mathcal H_0$.
The basin identity therefore supplies a point
belonging both to $B_f(\nu)$ and to $B_f(\mu)$
for some $\mu\in\mathscr S_{E,F}(\Lambda,f)$.
Uniqueness of the empirical limit gives $\nu=\mu$.
The reverse inclusion follows from the physicality
and exponent signs of every member of $\mathscr S_{E,F}(\Lambda,f)$. Therefore, $$\mathrm P_f\cap\left\{\nu\in\mathcal M(\mathbf M^d,f):\:\mathcal G_0(\nu)>0\right\}=\mathscr S_{E,F}(\Lambda,f).$$

Let $\nu\in\mathcal O_f^H$ satisfy
$\mathcal G_0(\nu)>0$.
Again $\nu\in\mathcal M(\Lambda,f)$.
Lemma \ref{stat:identity} gives
$\mathcal G_0(\tau)=\sup_{p\ge1}\frac1p\int\Phi_{p,0}\,d\tau,
 \:\tau\in \mathcal M(\Lambda,f),$
so $\mathcal G_0$ is lower semicontinuous on $\mathcal M(\Lambda,f)$.
It is therefore positive on a relative neighborhood
of $\nu$.
For every sufficiently small $\varepsilon>0$,
strong observability and localization provide a
positive-volume set $A_\varepsilon\subset U$ such that
$d_H(p\omega(x),\{\nu\})<\varepsilon,
 \:\text{for every }x\in A_\varepsilon.$
Therefore, for any $x\in A_{\varepsilon}$ and any $\tau\in p\omega(x)$, one has $\tau\in\mathcal M(\Lambda,f)$ and 
$\mathcal G_0(\tau)>0$. Lemma \ref{stat:identity} therefore gives $A_\varepsilon\subset\mathcal H_0$.
The basin identity yields
$\mu_\varepsilon\in\mathscr S_{E,F}(\Lambda,f)$ with
$\rho(\mu_\varepsilon,\nu)<\varepsilon$.
Hence $\nu\in\overline{\mathscr S_{E,F}(\Lambda,f)}^{\,w^*}$.
The reverse inclusion, after intersecting with
$\{\nu\in\mathcal M(\mathbf M^d,f):\:\mathcal G_0(\nu)>0\}$, follows from
$\mathscr S_{E,F}(\Lambda,f)\subset\mathcal O_f^H$ and the 
closedness of $\mathcal O_f^H$.
This closedness follows directly from the triangle
inequality and its defining positive-volume
neighborhood condition. Therefore,
\[
 \mathcal O_f^H\cap\left\{\nu\in\mathcal M(\mathbf M^d,f):\:\mathcal G_0(\nu)>0\right\}
 =
 \overline{\mathscr S_{E,F}(\Lambda,f)}^{\,w^*}\cap\left\{\nu\in\mathcal M(\mathbf M^d,f):\:\mathcal G_0(\nu)>0\right\}.
\]

Finally, assume
$\mathrm {Leb}(U\setminus\mathcal H_0)=0$.
The basin identity gives
$\mathrm {Leb}(U\setminus\mathcal B_{E,F})=0$.
For any $\nu\in\mathcal O_f\cap \mathcal M(\Lambda,f)$ and any sufficiently
small $\varepsilon>0$, ordinary observability and
localization provide a positive-volume set
$A_\varepsilon\subset U$ satisfying
$\rho(\nu,p\omega(x))<\varepsilon,
 \:\text{for every }x\in A_\varepsilon.$
This set meets $B_f(\mu_\varepsilon)$ for some
$\mu_\varepsilon\in\mathscr S_{E,F}(\Lambda,f)$.
At such a point $p\omega(x)=\{\mu_\varepsilon\}$,
so $\rho(\mu_\varepsilon,\nu)<\varepsilon$.
Therefore $\nu\in\overline{\mathscr S_{E,F}(\Lambda,f)}^{\,w^*}$.
Conversely, the inclusions
$\mathscr S_{E,F}(\Lambda,f)\subset\mathcal O_f^H\subset\mathcal O_f$
and the closedness of $\mathcal O_f^H$ give $\overline{\mathscr S_{E,F}(\Lambda,f)}^{\,w^*}
 \subset\mathcal O_f^H\cap \mathcal M(\Lambda,f)
 \subset\mathcal O_f\cap \mathcal M(\Lambda,f).$
This proves the final equality.
\end{proof}

\section{Proof of Theorem \ref{thm:F}}\label{ex:section}
\subsection{The intermittent two-dimensional base}

\begin{lemma}\label{lem:intermittent-base}
Fix $1<\alpha<2$. For every sufficiently large odd integer $K$, there exists
a $C^2$ covering $g:\mathbb T^2\to\mathbb T^2$, of degree $K^2$, with the
following properties:
\begin{enumerate}
\item $g(0)=0$, $D_0g=I$, and $m(D_zg)>1$, for all $z\ne0$.
\item Outside an arbitrarily small ball about $0$, $g(z)=Kz\pmod{\mathbb Z^2}$.
In a smaller ball it is given by
\[
 g(z)=\frac{z}{(1-c|z|^\alpha)^{\frac 1\alpha}},\: c>0.
\]
\item for every nonempty open set $V$, some
$n\ge1$ satisfies $g^n(V)=\mathbb T^2$.
\item $g$ admits exactly one absolutely continuous invariant probability
$\mu_g$. It is ergodic, equivalent to Lebesgue measure, and its density has
a positive essential lower bound. In particular, for every Borel set $D$,
\[
 \lim_{n\to+\infty}\frac1n\sum_{j=0}^{n-1}\mathbf1_D(g^jz)=\mu_g(D),
 \:\text{for Leb-a.e.}\:z\in\mathbb T^2.
\]
\end{enumerate}
\end{lemma}

\begin{proof}
Choose \(0<r_s<\frac R2\) such that \(KR<\frac{1}{4K}\). For \(r_0\in(0,r_s)\) sufficiently small, choose a continuous, nondecreasing function \(q:[0,\infty)\to[1,K]\), smooth on \((0,\infty)\), such that
$q(r)=(1-cr^\alpha)^{-\frac 1\alpha},\:\text{for }0\le r\le r_0;
\:
q(r)=K,\:\text{for }r\ge r_s,$
and \(q(r)>1\) for every \(r>0\). Such a smooth monotone interpolation is possible because \((1-cr^\alpha)^{-1/\alpha}\to1<K\) as \(r\to0\). For each \(v\in\mathbb Z^2\), define $G(v+\xi)=Kv+q(|\xi|)\xi,\: |\xi|<R,$
and set \(G(z)=Kz\) elsewhere. The balls \(B_R(v)\) are pairwise disjoint, and the definitions agree near their boundaries because \(q(r)=K\) for \(r\ge r_s\). By construction, $G(z+v)=G(z)+Kv,\: v\in\mathbb Z^2,$
so \(G\) induces a map \(g:\mathbb T^2\to\mathbb T^2\).

Near each lattice point, write $G(v+\xi)=Kv+\xi+E(\xi).$
The explicit local formula gives
$\|D^j_{\xi}E\|=O\bigl(|\xi|^{1+\alpha-j}\bigr),
\:j=0,1,2,$
for \(0<|\xi|\le r_0\). Since \(\alpha>1\), the extension \(E(0)=0\) is \(C^2\), with \(D_0E=D^2_0E=0\). Thus \(G\) and \(g\) are \(C^2\), and \(D_vG=I\) at every lattice point. At a point \(v+\xi\) with \(0<r=|\xi|<R\), the radial and tangential singular values of \(DG\) are \(q(r)+rq'(r)\) and \(q(r)\), respectively. Since \(q'\ge0\) and \(q(r)>1\) for \(r>0\), we obtain \(m(DG)\ge1\), with equality precisely at the lattice points. Consequently, \(g\) is a local diffeomorphism and hence a covering map. Finally, the equivariance implies that \(g\) is homotopic to \([z]\mapsto[Kz]\), so $\deg g=\det(KI)=K^2.$

\smallskip
\noindent\emph{\textbf{Exactness.}}
Since $G(z)-Kz$ is bounded, $|G(z)|\to+\infty$ as $|z|\to+\infty$,
so $G$ is proper. As a proper local diffeomorphism, $G$ is a covering
map onto $\mathbb R^2$. Since $\mathbb R^2$ is simply connected,
$G$ is a global diffeomorphism. Put $H=G^{-1}$. Then
$\|D_yH\|
 =\frac{1}{m(D_{H(y)}G)}\le1,$
with equality precisely when $y\in K\mathbb Z^2$. In particular,
$H$ is $1$-Lipschitz. Every nondegenerate straight segment meets
$K\mathbb Z^2$ in only finitely many points; hence, for $u\ne v$,
$|H(u)-H(v)|
 \le |u-v|\int_0^1
       \|D_{(1-t)u+tv}H\|\,dt<|u-v|.$
Moreover, the equivariance of $G$ gives $H(y+K\ell)=H(y)+\ell,
 \:\ell\in\mathbb Z^2.$

For $r\ge0$, define $\omega(r)=\sup_{|u-v|\le r}|H(u)-H(v)|.$
By simultaneously translating both arguments by an element of
$K\mathbb Z^2$, we obtain
$\omega(r)=
\max_{\substack{u\in[0,K]^2,\: |h|\le r}}
 |H(u+h)-H(u)|.$
The maximum is attained because the parameter set is compact.
The strict distance inequality above therefore implies $\omega(0)=0,
 \:\omega(r)<r,\:\text{for every }r>0.$

The function $\omega$ is nondecreasing and $1$-Lipschitz.
Indeed, let $0\le s\le t$ and $|u-v|\le t$. If $|u-v|>s$,
choose $w\in[u,v]$ with $|u-w|=s$. Then $|H(u)-H(v)|
 \le \omega(s)+|w-v|
 \le \omega(s)+t-s.$
The same bound is immediate when $|u-v|\le s$. Taking the supremum
yields $0\le\omega(t)-\omega(s)\le t-s.$
Thus $\omega$ is continuous. For any $r>0$, the sequence
$r_n=\omega^{n}(r)$ decreases to a limit $r_{\infty}\ge0$.
Continuity gives $r_{\infty}=\omega(r_{\infty})$, and the strict inequality
$\omega(t)<t$ for $t>0$ forces $r_{\infty}=0$. Consequently,
$\omega^{n}(r)\to0.$

Let $\pi:\mathbb R^2\to\mathbb T^2$ be the quotient map, and fix
$x,y\in\mathbb T^2$. Choose a lift $X$ of $x$. For each $n\ge1$,
there is a lift $Y_n$ of $y$ such that
 $|Y_n-G^n(X)|\le\frac{\sqrt2}{2}.$
Setting $Z_n=H^n(Y_n)$ and using the definition and monotonicity
of $\omega$ repeatedly, we obtain
\[
 \begin{aligned}
 |Z_n-X|
 =|H^n(Y_n)-H^n(G^n(X))|\le\omega^{n}\bigl(|Y_n-G^n(X)|\bigr)\le\omega^{n}(\frac{\sqrt2}{2})\to0.
 \end{aligned}
\]
Also, $g^n(\pi(Z_n))=y$. Since the bound is independent of $x$ and $y$,
for every $\varepsilon>0$ there exists $N(\varepsilon)\ge1$ such that
\[
 g^n(B(x,\varepsilon))=\mathbb T^2,
 \:
 \text{for all }x\in\mathbb T^2
 \text{ and all }n\ge N(\varepsilon).
\]
Every nonempty open subset of $\mathbb T^2$ contains such a ball,
so $g$ is topologically exact.

\smallskip
\noindent\emph{\textbf{Verification of the Hu--Vaienti hypotheses.}}
We apply Theorem A of Hu--Vaienti \cite{HV09}, denoting their
regularity exponent by $\eta\in(0,1)$ to distinguish it from the
neutral exponent $\alpha\in(1,2)$. Let $Q=[-\tfrac12,\tfrac12]^2$. To apply the theorem, take as the
auxiliary phase space
$Q^+=[-\tfrac12-K^{-1},\tfrac12+K^{-1}]^2.$
The cells $U_j=K^{-1}(j+\mathrm {Int}Q)$, indexed by
$\mathcal I
  =\{j\in\mathbb Z^2:\|j\|_\infty\le \frac{K+1}2\},$
partition $Q^+$ up to their boundaries. On $U_j$ set
$T_j(z)=G(z)-j$. Each branch maps $U_j$ diffeomorphically onto
$\mathrm {Int}Q$, and every branch except the central one is
affine. On the grid boundaries choose the representative in $\partial Q$
so that $\pi\circ T=g\circ\pi$. Thus $T(Q^+)\subset Q$. Choose $0<\varepsilon_1<K^{-1}$. Since $G-j$ is a global
diffeomorphism, its restriction to
$\widetilde U_j=(G-j)^{-1}(\mathrm {Int}Q^+)$
extends $T_j$ to an open neighborhood of $\overline{U_j}$ and has
image $\mathrm {Int}Q^+\supset B_{\varepsilon_1}(T_j(U_j))$.
This verifies Assumption 1. The neutral fixed point is $0$, and its
one-step preimages are the centers of the cells, so Assumption 2
also holds. The $C^2$ regularity,
nondegeneracy, and compact branch extensions also give the
relative H\"older estimate in Assumption 4 (a) of \cite{HV09}.

Write $g_0=T_0$ for the central branch and take $\mathcal R=\{z:|z|<R\}.$
Since $q=K$ for $r\ge r_s$, we have
$g_0(\mathcal R)=\mathrm{Int}(B(0,KR)),
 \:
 \overline{\mathcal R}\subset g_0(\mathcal R),
 \:
 \overline{g_0(\mathcal R)}\subset U_0,$
where the last inclusion follows from $KR<\frac{1}{4K}$.
Every full branch covers $\overline{\mathcal R}$. Choose the extension radius $\varepsilon_1>0$ sufficiently small
that $\varepsilon_1<R-r_s$.
For an injective branch $F:D\to\mathbb R^2$, where $D$ is open,
and $x\in D\setminus\{0\}$, define
\[
s(x,F)
=
\sup_{\substack{
y\in D,\:
0<|y-x|\le\min\{\varepsilon_1,0.1|x|\}
}}
\frac{|y-x|}{|F(y)-F(x)|}.
\]
Equivalently, $s(x,F)$ is the least constant for which
$|y-x|\le s(x,F)|F(y)-F(x)|$
holds for every admissible comparison point $y$.
For $x\in U_j\setminus\{0\}$, write
$s(x):=s(x,T_j),$
and set
$s
:=
\sup_{j\in\mathcal I}
\sup_{x\in U_j\setminus\mathcal R}s(x).$
For iterated central branches, the notation
$s(x,g_0^n)$ is understood with
$D=\mathrm{dom}(g_0^n)$. 
If $x\notin\mathcal R$ and $y$ is an admissible comparison point
in the same branch, then both points lie outside the modification
region, and
$T_j(y)-T_j(x)=K(y-x).$
Consequently, $s=K^{-1}.$
The $1$-Lipschitz property of $H$ gives noncontraction on every
branch. In fact, $s(x)<1$ for $x\ne0$. To see this on the central
branch, note that $DG$ is symmetric and
$\langle D_{\zeta}G\xi,\xi\rangle
 \ge q(|\zeta|)|\xi|^2.$
If $|y-x|\le0.1|x|$, the segment $[x,y]$ stays at distance at
least $0.9|x|$ from the origin. Integration along this segment gives
$\langle G(y)-G(x),y-x\rangle
 \ge q(0.9|x|)|y-x|^2,$
and hence
$s(x)\le q(0.9|x|)^{-1}<1.$
On every noncentral branch, $s(x)=K^{-1}$. Thus Assumption 3 (a) of \cite{HV09}
holds. 

We next recall the boundary quantities appearing in
Assumption 3 (b) of \cite{HV09}.
Let $\nu$ denote Lebesgue measure on $\mathbb R^2$, normalized so that
$\nu(Q^+)=1$, and use the global branch extensions
$\widetilde T_j(z)=G(z)-j,
\:j\in\mathcal I.$
For $x\in Q^+$ and $0<\varepsilon\le\varepsilon_0$, define
\[
G_U(x,\varepsilon,\varepsilon_0)
=
\sum_{j\in\mathcal I}
\frac{
\nu\left(
\widetilde T_j^{-1}
\bigl(B(\partial(T_j(U_j)),\varepsilon)\bigr)
\cap B(x,(1-s)\varepsilon_0)
\right)
}{
\nu\bigl(B(x,(1-s)\varepsilon_0)\bigr)
},
\]
and put $G_U(\varepsilon,\varepsilon_0)
=
\sup_{x\in Q^+}G_U(x,\varepsilon,\varepsilon_0).$
Here all neighborhoods and balls are taken in the ambient
Euclidean space. In the present full-branch setting,
$\partial(T_j(U_j))=\partial Q$ for every $j$. For a scale $\varepsilon_2>0$, to be chosen sufficiently small
below, define
$\lambda
=
\max\left\{
2\sup_{0<\varepsilon_0\le\varepsilon_2}
 \sup_{0<\varepsilon\le\varepsilon_0}
 G_U(\varepsilon,\varepsilon_0)
 \left(\frac{\varepsilon_0}{\varepsilon}\right)^\eta,
\frac{3s\gamma_1}{(1-s)\gamma_2}
\right\},$
where $\gamma_1=2$ and $\gamma_2=\pi$ are the volumes of the
unit balls in dimensions one and two, respectively. The additional cells form one outer layer of the same square grid,
so the local boundary multiplicity is still bounded by a constant
$Y$ independent of $K$. By Remark 1.3 of \cite{HV09},
after choosing $\varepsilon_2>0$ sufficiently small,
\[
 G_U(\varepsilon,\varepsilon_0)
 \le \frac{C_Y}{K-1}\frac{\varepsilon}{\varepsilon_0},
\:0<\varepsilon\le\varepsilon_0\le\varepsilon_2.
\]
Here $\varepsilon_2$ may depend on $K$, whereas $C_Y$ does not.
Since $\left(\frac{\varepsilon}{\varepsilon_0}\right)^{1-\eta}\le1,$
$\lambda\le\frac{C_Y'}{K-1}.$
It follows that $s^\eta+\lambda
 \le K^{-\eta}+\frac{C_Y'}{K-1}<1$, for all sufficiently large $K$. This verifies Assumption 3 (b) of \cite{HV09}.

We next establish uniform estimates for the inverse central
branch, including its passage through the interpolation region.
Set $r_{\mathrm{in}}=R,\:r_{\mathrm{out}}=KR,\:
 A_0=g_0(\mathcal R)\setminus\mathcal R.$
Thus $\overline{A_0}=\{z:r_{\mathrm{in}}\le|z|\le r_{\mathrm{out}}\},
 \:
 \partial A_0=\{|z|=r_{\mathrm{in}}\}\cup\{|z|=r_{\mathrm{out}}\}.$
Choose $\varepsilon_*>0$ sufficiently small that
$\varepsilon_*<\min\{\frac {r_{\mathrm{in}}}2,\frac{r_{\mathrm{out}}-r_{\mathrm{in}}}{4}\},
\:A_*:=B(A_0,\varepsilon_*)\subset U_0\setminus\{0\}.$
Let $u(r)=rq(r)$ and let $v=u^{-1}$. Superscripts on $v$ below
denote iteration. Since $0<v(r)<r$ for $r>0$, the sequence
$v^{n}(r_{\mathrm{out}}+\varepsilon_*)$ decreases to zero. We may therefore
choose a single integer $N_*$ such that $v^{N_*}(r_{\mathrm{out}}+\varepsilon_*)<r_0.$
Put $\psi=v^{N_*}$. On the compact interval
$[r_{\mathrm{in}}-\varepsilon_*,r_{\mathrm{out}}+\varepsilon_*]$, both $\psi$ and $\psi'$ are positive
and bounded away from zero, and $\psi(r)<r_0$. In the exact neutral core, the inverse radial map satisfies
$v(t)^{-\alpha}=t^{-\alpha}+c.$
Consequently, for $n\ge N_*$ and
$r\in[r_{\mathrm{in}}-\varepsilon_*,r_{\mathrm{out}}+\varepsilon_*]$, $w_n(r):=v^{n}(r)
 =\frac{\psi(r)}
        {(1+(n-N_*)c\psi(r)^\alpha)^{\frac 1\alpha}},$
and $w_n'(r)
 =\psi'(r)
  (1+(n-N_*)c\psi(r)^\alpha)^{-1-\frac 1\alpha}.$
The Jacobian of the $n$-fold inverse central branch is therefore
\begin{equation}\label{eq:base-inverse-jacobian}
 J_n(r)
 =\frac{w_n(r)w_n'(r)}r
 =\frac{\psi(r)\psi'(r)}r
  (1+(n-N_*)c\psi(r)^\alpha)^{-1-\frac 2\alpha}.
\end{equation}
Writing $m=n-N_*$, we obtain
\[
 \frac{d}{dr}\log J_n(r)
 =
 \frac{d}{dr}\log\frac{\psi(r)\psi'(r)}r
 -(\alpha+2)
  \frac{mc\psi(r)^\alpha}{1+mc\psi(r)^\alpha}
  \frac{\psi'(r)}{\psi(r)}.
\]
Hence $\left|\frac{d}{dr}\log J_n(r)\right|\le C_*$
uniformly on this interval and uniformly in $n\ge N_*$.
The finitely many smaller positive integers $n$ satisfy the same
bound after increasing $C_*$. The radial and tangential
derivatives of the inverse branch also give
$\sup_{x\in A_*}\|D_xH^n\|\le C_0n^{-\frac 1\alpha},
 \:
 \sup_{x\in A_*}|H^n(x)|\le C_0n^{-\frac 1\alpha},
 \: n\ge1.$

We first verify Assumption 4 (b) of \cite{HV09}. Choose $0<\varepsilon_4<\min\{\frac{\varepsilon_*}{4},1\}.$
If $x\in B(A_0,\varepsilon_4)$ and
$|y-x|\le\varepsilon\le\varepsilon_4$, both points lie in $A_*$.
Radiality and the logarithmic derivative bound imply
\[
 \begin{aligned}
 \left|
 \log\frac{|\det D_yg_0^{-n}|}
           {|\det D_xg_0^{-n}|}
 \right|
 &\le C_*\bigl||y|-|x|\bigr|\le C_*|y-x|.
 \end{aligned}
\]
Thus, for a constant $J$ independent of $n$ and $\varepsilon$,
$\frac{|\det D_yg_0^{-n}|}
      {|\det D_xg_0^{-n}|}
 \le e^{C_*\varepsilon}
 \le1+J\varepsilon^\eta.$
Assumption 4 (b) of \cite{HV09} follows by taking its permitted choice
$N(\varepsilon)=+\infty$: the distortion estimate holds for every
$n\ge1$, and the tail sum is empty.

For Assumption 4 (c) of \cite{HV09}, fix $0<\varepsilon_0\le\varepsilon_*$.
Partition the angular variable into finitely many nondegenerate
half-open intervals of length at most
$\frac{\varepsilon_0}{r_{\mathrm{out}}+\varepsilon_*}$. In each sector, split the radial
interval
$[r_{\mathrm{in}}-\varepsilon_0,r_{\mathrm{out}}+\varepsilon_0]$
at $\frac{r_{\mathrm{in}}+r_\mathrm{out}}{2}$, assigning endpoints consistently. This gives a
finite measurable partition of $B(A_0,\varepsilon_0)$.
Each element meets exactly one of the two boundary circles.
Moreover,
$\mathrm{diam}
 \bigl(A\cap B(\partial A_0,\varepsilon_0)\bigr)
 \le2\varepsilon_0+(r_{\mathrm{out}}+\varepsilon_*)
\frac{\varepsilon_0}{r_{\mathrm{out}}+\varepsilon_*}=3\varepsilon_0<10\varepsilon_0.$
Let $\Theta$ be the angular interval associated with $A$.
If $A$ meets the inner boundary circle, then, for
$0<\varepsilon\le\varepsilon_0$,
$\mathrm{Leb}(B(\partial A_0,\varepsilon)\cap A)
 =|\Theta|\int_{r_{\mathrm{in}}-\varepsilon}^{r_{\mathrm{in}}+\varepsilon}r\,dr
 =2r_{\mathrm{in}}|\Theta|\varepsilon.$
The corresponding formula for an element meeting the outer
boundary is $2r_{\mathrm{out}}|\Theta|\varepsilon$. Therefore every denominator
below is positive, and
\[
 \frac{\mathrm{Leb}(B(\partial A_0,\varepsilon)\cap A)}
      {\mathrm{Leb}(B(\partial A_0,\varepsilon_0)\cap A)}
 =\frac{\varepsilon}{\varepsilon_0}
 \le\left(\frac{\varepsilon}{\varepsilon_0}\right)^\eta.
\]
Finally, the logarithmic derivative bound on the entire radial
interval yields
\[
 \frac{|\det D_yg_0^{-n}|}
      {|\det D_xg_0^{-n}|}
 \le e^{C_*(r_{\mathrm{out}}-r_{\mathrm{in}}+2\varepsilon_*)},
 \:\forall\:x,y\in A,\:\forall\: n\ge1.
\]
Thus Assumption 4 (c) of \cite{HV09} holds with $C_p=1$, $\varepsilon_5=\varepsilon_*$, and any $I>\max\left\{
 1,e^{C_*(r_{\mathrm{out}}-r_{\mathrm{in}}+2\varepsilon_*)}
 \right\}.$
The partition may be chosen independently of $n$. 

We now verify Assumption 3 (c) of \cite{HV09} using finite differences.
Set $A_{**}=B(A_0,\frac{\varepsilon_*}{2}),\:\varepsilon_c=\frac{\varepsilon_*}{4}.$
Then $B(x,\varepsilon_c)\subset A_*$ for every $x\in A_{**}$.
Fix such an $x$, put $z=g_0^{-n}(x)=H^n(x)$, and let $z'$ be
an admissible comparison point in the domain of $g_0^n$, with $0<|z'-z|\le\min\{\varepsilon_1,0.1|z|\}.$
Write $w=g_0^n(z')$, so that $z'=H^n(w)$. If
$|w-x|<\varepsilon_c$, integration along the segment $[x,w]$ gives $\frac{|z'-z|}{|w-x|}
 \le C_0n^{-\frac 1\alpha}.$
If $|w-x|\ge\varepsilon_c$, then
$\frac{|z'-z|}{|w-x|}
 \le\frac{0.1|z|}{\varepsilon_c}
\le0.1C_0\varepsilon_c^{-1}n^{-\frac 1\alpha}.$
Consequently, for a constant $C_1$ independent of $x$ and $n$,
$\sup_{x\in A_{**}}
 s(g_0^{-n}(x),g_0^n)
 \le C_1n^{-\frac1\alpha}\to0.$
With $s,\lambda,C_p,I$ already fixed, choose $N_s$ sufficiently
large that
$C_1N_s^{-\frac1\alpha}
 \le \frac{s}{10}
 \left(\frac{\lambda(1-s)^2}{2C_pI^2}\right)^{\frac1\eta}.$
This is precisely Assumption 3 (c) of \cite{HV09}, with
$\varepsilon_3=\frac{\varepsilon_*}{2}$ and dimension $m=2$.

All the hypotheses of Theorem A of \cite{HV09} are now verified.
Furthermore, $\mathrm{Leb}(g_0^{-n}\mathcal R)
 =\pi\bigl(v^{n}(R)\bigr)^2
 \asymp n^{-\frac 2\alpha}.$
Since $\alpha<2$,
$\sum_{n=1}^{\infty}
 \mathrm{Leb}(g_0^{-n}\mathcal R)<+\infty.$
The finiteness criterion in that theorem therefore gives a
nonzero finite absolutely continuous invariant measure.
Since $T(Q^+)\subset Q$, the normalized invariant measure gives
full mass to $Q$. Its projection to $\mathbb T^2$ is an absolutely
continuous invariant probability for $g$.

\smallskip
\noindent\emph{\textbf{Uniqueness and equivalence.}}
Let $Y=Q\setminus\mathcal R$, and let $\widehat g$ be the first
return map to $Y$. It is defined Lebesgue-almost everywhere:
every nonzero point entering $\mathcal R$ leaves it in finite
time, and the points whose orbits hit $0$ form a null set.
Let $\widehat P$ be its Perron--Frobenius operator with respect
to Lebesgue measure restricted to $Y$. Fix a sufficiently small scale $\delta_0>0$.
For $h\in L^1(Y)$, let $\widetilde h$ denote its zero extension
to $\mathbb R^2$. For a measurable set $B$, write
$\mathrm{osc}(\widetilde h,B)
=
\mathrm{ess\,sup}_{B}\widetilde h
-
\mathrm{ess\,inf}_{B}\widetilde h.$
Define the quasi-H\"older seminorm by
\[
|h|_\eta
=
\sup_{0<\varepsilon\le\delta_0}
\varepsilon^{-\eta}
\int_{\mathbb R^2}
\mathrm{osc}
\bigl(\widetilde h,B(x,\varepsilon)\bigr)\,dx,
\]
and set
$
V_\eta(Y)
=
\left\{
h\in L^1(Y): |h|_\eta<+\infty
\right\},
\:
\|h\|_{V_\eta(Y)}
=
\|h\|_{L^1(Y)}+|h|_\eta.$
We abbreviate $V_\eta(Y)$ to $V_\eta$ and
$\|\cdot\|_{L^1(Y)}$ to $\|\cdot\|_1$.

Since $Q$ is forward invariant, zero extension from $Y$ to
$Q^+\setminus\mathcal R$ intertwines the two first-return transfer
operators and preserves the quasi-Hölder seminorm defined by zero
extension to $\mathbb R^2$.
By Proposition 6.2 of \cite{HV09}, there are constants
$q_0\in(0,1)$ and $D<+\infty$ such that
$ |\widehat P h|_\eta
 \le q_0|h|_\eta+D\|h\|_1,
 \:h\in V_\eta.$
We first show that every $L^1$ invariant probability density of
$\widehat P$ belongs to $V_\eta$. Let $f\ge0$ satisfy
$\widehat P f=f$ and $\int_Y f=1$. Choose nonnegative functions
$h_k\in C_c^\infty(\mathrm{Int}Y)$ with $\int_Y h_k=1$ and
$h_k\to f$ in $L^1$. Such an approximation is possible because
$\partial Y$ has Lebesgue measure zero. For
$A_{k,M}=\frac1M\sum_{j=0}^{M-1}\widehat P^j h_k,$
the $L^1$ contraction property gives
$\|A_{k,M}-f\|_1\le\|h_k-f\|_1.$
Iterating the Lasota--Yorke inequality yields
\[
 |A_{k,M}|_\eta
 \le\frac{|h_k|_\eta}{M(1-q_0)}
     +\frac{D}{1-q_0}.
\]
Choose $M=M_k$ so that the first term is at most $1$.
Then $A_{k,M_k}\to f$ in $L^1$, with uniformly bounded
$\eta$-seminorms. Lower semicontinuity of the oscillation
seminorm, obtained from Fatou's lemma, implies $f\in V_\eta$. Now let $\rho$ be any absolutely continuous  $g$-invariant
probability, with density $h_\rho$. Since every point of
$\mathcal R\setminus\{0\}$ eventually leaves $\mathcal R$, $\bigcup_{j\ge0}g^{-j}Y\supset\mathbb T^2\setminus\{0\}.$
Invariance and $\rho(\{0\})=0$ therefore imply $\rho(Y)>0$.
The normalized restriction $\rho|_Y$ is invariant under
the first return map, so its density $f_\rho=\frac{h_\rho|_Y}{\rho(Y)}$
is fixed by $\widehat P$. By the preceding argument,
$f_\rho\in V_\eta$. Proposition 5.2(ii) of \cite{HV09}, applied
to this nonzero nonnegative density, gives a ball
$B\subset \mathrm{Int}Y$ and a constant $c_B>0$ such that
$h_\rho\ge c_B,
 \:\text{Leb-a.e.}\:x\in B.$
By topological exactness, choose $n\ge1$ with
$g^n(B)=\mathbb T^2$, and put
$J_{\max}=\sup_{z\in\mathbb T^2}|\det D_zg|.$
The transfer formula and invariance give
\[
 h_\rho(y)
 =\sum_{g^n(z)=y}
   \frac{h_\rho(z)}{|\det D_zg^n|}
 \ge J_{\max}^{-n}c_B,
 \:\text{for Leb-a.e.}\:y.
\]
Indeed, outside the image of the exceptional null set in $B$,
each $y$ has a preimage in $B$ at which the local lower bound
holds. This image is null because $g^n$ is $C^1$ on the compact
torus and hence Lipschitz. Thus every absolutely continuous
invariant probability is equivalent to Lebesgue measure and
has a density bounded below by a positive constant.

Such a probability must be ergodic. Otherwise, its normalized
restriction to a nontrivial invariant set would be another
absolutely continuous invariant probability vanishing on a set
of positive Lebesgue measure, contradicting the preceding
conclusion. Finally, two equivalent ergodic invariant
probabilities coincide by Birkhoff's ergodic theorem.
Hence $g$ admits a unique absolutely continuous invariant
probability $\mu_g$, which is ergodic and equivalent to
Lebesgue measure.

For each fixed Borel set $D\subset\mathbb T^2$, Birkhoff's
theorem and this equivalence give
\[
 \lim_{n\to+\infty}\frac1n
 \sum_{j=0}^{n-1}\mathbf 1_D(g^jz)
 =\mu_g(D),
 \:\text{for Leb-a.e.}\:z.
\]
\end{proof}

\subsection{Ambient realization in a coordinate ball of the four-torus}
Throughout this lemma, let
$\overline{\mathbb D}^{\,2}
  =\{w\in\mathbb R^2:|w|\le1\}.$
For each positive integer $m$, write
$L_m:\mathbb T^2\longrightarrow\mathbb T^2,
  \:
  L_m([x])=[mx].$
We identify $\mathbb R^2$ with $\mathbb C$ and write
$\mathrm{SO}(2)
  =
  \left\{
    A\in\mathbb R^{2\times2}:
    A^{\mathsf T}A=I_2,\ \det A=1
  \right\}
 .$
Under this identification, multiplication by $e^{i\theta}$
corresponds to the rotation matrix
\[
  \begin{pmatrix}
    \cos\theta&-\sin\theta\\
    \sin\theta& \cos\theta
  \end{pmatrix}.
\]
We use these two representations interchangeably.

\begin{lemma}
\label{lem:solenoid-ambient-realization}
Let $k\ge2$ be an integer. There exist an integer
$d\in\{1,2,3\}$, an integer $K=k^d$, a number $r\in(0,1)$,
and smooth maps
$h_0:\mathbb T^2\to\mathbb R^2,
  \:
  R_0:\mathbb T^2\to\mathrm{SO}(2),$
with $h_0=0$ and $R_0=I$ on a neighborhood of $0$,
having the following property. Suppose that $\Phi:\mathbb T^2\to\mathbb T^2$ is a
$C^2$ diffeomorphism isotopic to the identity through
$C^2$ diffeomorphisms, and that
$\Phi(0)=0,
  \:
  g=L_K\circ\Phi,
  \:
  D_0g=I.$
Set $h=h_0\circ\Phi,
  \:
  R=R_0\circ\Phi,
  \:
  C=r^{-2}-1,
  \:
  \kappa(w)=\frac{w}{\sqrt{1+C|w|^2}}.$
Then
\begin{equation}
\label{eq:intermittent-bundle-embedding}
  F(z,w)
  =
  \bigl(g(z),h(z)+R(z)\kappa(w)\bigr)
\end{equation}
is a \(C^2\) embedding of $N=\mathbb T^2\times\overline{\mathbb D}^{\,2}$ into \(\mathrm {Int}N\). Moreover, the following properties hold.\\ 

$\bullet\:$there exist a smooth embedding
$\phi:N\hookrightarrow\mathbb T^4$, whose image is contained
in a coordinate ball, and a $C^2$ diffeomorphism
$f:\mathbb T^4\to\mathbb T^4$ satisfying
$f\circ \phi=\phi\circ F,$\\

$\bullet\:U=\phi(\mathrm{Int}N)$ gives $\overline U=\phi(N),
  \:
  f(\overline U)\subset U;$\\

$\bullet\:q=\phi(0,0)$ is a fixed point of $f$ with
$D_qf=I;$\\ 

$\bullet\:g^{-1}(0)\setminus\{0\}$ is nonempty;\\

$\bullet\:|h(z)|>2r,
  \:
  \text{for every }z\in g^{-1}(0)\setminus\{0\};$\\

$\bullet\:\kappa(\overline{\mathbb D}^{\,2})
  =r\overline{\mathbb D}^{\,2},
  \:
  D_0\kappa=I,
  \:
  \|D_w\kappa\|\le1,$
with strict inequality in the last assertion whenever
$w\ne0$.
\end{lemma}

\begin{proof}
We begin with the construction of
Ding--Liu--Wang--Yao \cite{DLWY12}.
For $|u|=1$ and $|w|\le1$, consider
$b_k(u,w)
  =
  \left(
    u^k,\frac12u+k^{-2}u^{1-k}w
  \right).$
Distinct roots of the same point under $u\mapsto u^k$
give fiber discs whose center distance is at least
$\sin(\frac \pi k)\ge\frac2k>\frac2{k^2}.$
Since their radii are $k^{-2}$, these closed discs are
disjoint. They lie in the open unit disc because
$\frac12+k^{-2}<1.$
The base and fiber derivatives are invertible, so
$b_k$ is a smooth embedding into the interior of the
solid torus.

Applying this construction first to the second circle
coordinate and then to the first gives the favorite
lifting of  $L_k$ described in \cite{DLWY12}. It is given by $e(z,w)
=\bigl(L_k(z),a(z)+sQ(z)w\bigr),
  \:
  s=k^{-4}$, where $
  a(z)
  =
  \frac12e^{2\pi i z_1}
  +
  \frac1{2k^2}
  e^{2\pi i((1-k)z_1+z_2)},\:
  Q(z)
  =
  e^{2\pi i(1-k)(z_1+z_2)}.
  $ Thus $e:N\hookrightarrow\mathrm{Int}N$ is a smooth
embedding and $Q(0)=I$.

We now spell out the finite-type argument used in the
ambient realization construction of \cite{DLWY12}.
We use their terminology concerning unknotted embeddings,
modular types, and untwisted framing.
For a smooth embedding
$\jmath:N=\mathbb T^2\times\overline{\mathbb D}^{\,2}
\hookrightarrow\mathbb R^4$,
let $\iota(z)=\jmath(z,0)$ denote its parametrized core
embedding. The type of $\jmath$ is the type of $\iota$:
two parametrized core embeddings $\iota_0$ and $\iota_1$
have the same type if there exists a smooth diffeomorphism $A$ of $\mathbb R^4$ such that
$A\circ\iota_0=\iota_1$. Fix $w_*=(1,0)\in\partial\overline{\mathbb D}^{\,2}$.
The corresponding parallel copy of the core is the embedding
$\iota_+:\mathbb T^2
 \to\mathbb R^4\setminus\iota(\mathbb T^2),
 \:
 \iota_+(z)=\jmath(z,w_*).$
We say that $\jmath$ has untwisted framing if
$(\iota_+)_*[\mathbb T^2]=0
 \:\text{in}\:
 H_2\bigl(\mathbb R^4\setminus\iota(\mathbb T^2);\mathbb Z\bigr),$
where $[\mathbb T^2]$ is the fundamental class of the
standard oriented torus. A modular type is a type represented by
$\iota_{\mathrm{std}}\circ\tau_A$, where
$\iota_{\mathrm{std}}$ is the standard unknotted
embedding and
$\tau_A([x])=[Ax],
  \:A\in\mathrm{SL}(2,\mathbb Z).$
 Regard $e$ as a self-embedding of $N$ with
$e(N)\subset\mathrm{Int}N$, and set
$e^0=\mathrm{id}_N$. Choose an unknotted embedding
$\phi_*:N\hookrightarrow\mathbb R^4$
of modular type with untwisted framing.
The framing and unknotting properties of the favorite
lifting proved in \cite{DLWY12} imply that
$\phi_*\circ e^j,
  \:0\le j\le3,$
are all unknotted of modular type with untwisted framing.
The finiteness result in the same paper gives at most
$3$ modular types.
Consequently, there exist integers $0\le a_*<b_*\le3$
such that $\phi_*\circ e^{a_*}$ and $\phi_*\circ e^{b_*}$
have the same type.
The uniqueness, up to compactly supported ambient
diffeomorphism, of unknotted thickenings of the same
type with untwisted framing, established in
\cite{DLWY12}, provides a compactly supported smooth
diffeomorphism $f_*$ of $\mathbb R^4$ satisfying
$f_*\circ(\phi_*\circ e^{a_*})
  =
  \phi_*\circ e^{b_*}.$
Set $d=b_*-a_*,
  \:
  K=k^d,
  \:
  \widehat \phi=\phi_*\circ e^{a_*},
  \:
  e_0=e^d.$
Then $d\in\{1,2,3\}$ and
\begin{equation}
\label{eq:initial-ambient-realization}
  f_*\circ\widehat\phi
  =
  \widehat \phi\circ e_0.
\end{equation}

Since composition preserves affine dependence on $w$,
we can write
$e_0(z,w)
=\bigl(L_K(z),a_0(z)+s_0Q_0(z)w\bigr),\:s_0=s^d,$
where $a_0$ and $Q_0$ are smooth and
$Q_0(z)\in\mathrm{SO}(2)$.
Because $Q(0)=I$ and $L_k(0)=0$, we also have
$Q_0(0)=I.$ Choose $\varepsilon\in(0,1)$ so small that
$\varepsilon
  \bigl(
    \|a_0\|_\infty+s_0+|a_0(0)|
  \bigr)<1.$
First scale the fiber coordinate of $e_0$ by a positive
factor decreasing from $1$ to $\varepsilon$, and then
translate it by $-t\varepsilon a_0(0)$, $0\le t\le1$.
Both stages consist of embeddings into
$\mathrm{Int}N$.
Their endpoint is
$\widetilde e(z,w)
  =
  \bigl(L_K(z),\widetilde h(z)+rQ_0(z)w\bigr),$
where
$\widetilde h(z)
  =
  \varepsilon\bigl(a_0(z)-a_0(0)\bigr),
  \:
  r=\varepsilon s_0.$
Thus $0<r<1$ and $\widetilde h(0)=0$. The collision-pair set $\mathcal C_K
  =
  \{(z,z'):L_K(z)=L_K(z'),\ z\ne z'\}$
is compact, since
$\mathcal C_K
  =
  \bigcup_{\delta\in\ker L_K\setminus\{0\}}
  \{(z,z+\delta):z\in\mathbb T^2\}.$
Injectivity of $\widetilde e$ implies that the closed
discs
$\widetilde h(z)+r\overline{\mathbb D}^{\,2},
  \:
  \widetilde h(z')+r\overline{\mathbb D}^{\,2}$
are disjoint whenever $(z,z')\in\mathcal C_K$.
Hence there exists $\Delta>0$ such that
$|\widetilde h(z)-\widetilde h(z')|
  \ge2r+\Delta,
  \:\forall\:
  (z,z')\in\mathcal C_K.$
There is also $\sigma>0$ such that $|\widetilde h(z)|+r\le1-\sigma,
  \:\forall\:z\in\mathbb T^2.$

Choose a coordinate ball $V$ about $0$ on which $Q_0(z)=e^{i\vartheta(z)}$
for a smooth real-valued function $\vartheta$ satisfying
$\vartheta(0)=0$.
Let $\chi:\mathbb T^2\to[0,1]$ be smooth, supported
compactly in $V$, and equal to one near $0$.
Since $\widetilde h(0)=0$, we can choose its support
small enough that
$\|\chi\widetilde h\|_\infty
  <
  \min\{\frac \Delta4,\frac \sigma2\}.$
Set $h_0=(1-\chi)\widetilde h.$ For every $t\in[0,1]$, the centers $c_t=(1-t\chi)\widetilde h$
satisfy $|c_t(z)-c_t(z')|
  >
  2r+\frac{\Delta}{2},
  \:
  \forall\:(z,z')\in\mathcal C_K$
and $|c_t(z)|+r<1-\frac{\sigma}{2}.$
Therefore
$(z,w)\mapsto\bigl(L_K(z),c_t(z)+rQ_0(z)w\bigr)$
is injective and takes values in $\mathrm{Int}N$.
Its derivative is invertible because it has the
block form
\[
  \begin{pmatrix}
    KI&0\\
    *&rQ_0(z)
  \end{pmatrix}.
\]
Compactness of $N$ then makes it an embedding. Next define
\[
  R_{0,t}(z)
  =
  \begin{cases}
    e^{i(1-t\chi(z))\vartheta(z)},&z\in V,\\
    Q_0(z),&z\notin V.
  \end{cases}
\]
These definitions agree smoothly near $\partial V$.
Changing the fiber rotation does not change the fiber
image disc, so this gives another isotopy through
embeddings into $\mathrm{Int}N$.
With $R_0=R_{0,1}$, its endpoint is
$e_1(z,w)
=\bigl(L_K(z),h_0(z)+rR_0(z)w\bigr).$ By construction, $h_0=0$ and $R_0=I$ near $0$.

We may choose an isotopy $\Phi_t$ from the identity
to $\Phi$ that is smooth in time with values in the
space of $C^2$ maps and stationary near the endpoints.
To justify this choice, start with a $C^2$-continuous
isotopy $\Theta_t$ and subdivide it into sufficiently
small increments
$\Psi_j
=\Theta_{t_{j+1}}\circ\Theta_{t_j}^{-1}.$
Each increment has a periodic $C^2$ displacement
function $u_j$ such that
$\Psi_j([x])=[x+u_j([x])],
  \:
  \|Du_j\|_\infty<1.$
For a smooth function $\beta:[0,1]\to[0,1]$ equal to
zero near $0$ and one near $1$, the maps
$\Psi_{j,s}([x])
  =
  [x+\beta(s)u_j([x])]$
are degree-one local diffeomorphisms and hence global
$C^2$ diffeomorphisms.
Concatenating
$\Psi_{j,s}\circ\Theta_{t_j}$ gives the required isotopy. Precomposing $e_1$ with
$\Phi_t\times\mathrm{id}$ preserves its total
image and gives an isotopy ending at
$e_2(z,w)
=\bigl(g(z),h(z)+rR(z)w\bigr),$
where $h=h_0\circ\Phi$ and $R=R_0\circ\Phi$.
Since $\Phi(0)=0$, we have $h=0$ and $R=I$ near $0$.

Define
\[
  \kappa_t(w)
  =
  (1-t)rw+t\frac{w}{\sqrt{1+C|w|^2}},
  \:
  C=r^{-2}-1>0.
\]
At radius $\rho>0$, its radial and tangential
derivatives are, respectively, $(1-t)r+t(1+C\rho^2)^{-\frac32},
  \:
  (1-t)r+t(1+C\rho^2)^{-\frac12}.$
Both are positive, and
$D_0\kappa_t=((1-t)r+t)I.$
Moreover, $\kappa_t(w)=rw$ when $|w|=1$, since
$(1+C)^{-\frac12}=r$.
Thus every $\kappa_t$ is a smooth diffeomorphism from
$\overline{\mathbb D}^{\,2}$ onto
$r\overline{\mathbb D}^{\,2}$.
It follows that
$(z,w)\mapsto
\bigl(g(z),h(z)+R(z)\kappa_t(w)\bigr)$
is an isotopy through embeddings into
$\mathrm{Int}N$, ending at $F$.
Indeed, the fiber image discs remain unchanged, and
the derivative has invertible diagonal blocks
$D_zg$ and $R(z)D_w\kappa_t$. After smooth reparametrization at the joining times,
the preceding stages give an isotopy
$F_t:N\hookrightarrow\mathrm{Int}N,
  \:
  F_0=e_0,\:F_1=F,$
smooth in time with values in the space of $C^2$ maps
and stationary near the endpoints. We now give the ambient extension argument, including
the finite-regularity details; compare the standard
isotopy extension construction in
\cite[Chapter 8, Section 1]{Hirsch76}.
The explicit bundle formulas extend across $\partial N$.
The fiber separation, boundary margins, and
invertibility of the fiber derivatives are uniform
over the compact parameter interval.
Consequently, for some $\delta>0$, the extended maps
are embeddings
$F_t:U_\delta\to\mathrm{Int}N,
  \:
  U_\delta
  =
  \mathbb T^2\times\{w:|w|<1+\delta\},$
for all $t\in[0,1]$. Put
$\Gamma_t=\widehat \phi\circ F_t,
  \:\text{on }U_\delta.$
This composition is defined because
$F_t(U_\delta)\subset\mathrm{Int}N$.
Extend $\Gamma_t$ constantly in time slightly beyond
the endpoints.
The set
$\mathcal W
  =
  \{(t,y):y\in\Gamma_t(U_\delta)\}$
is open: the map
$(t,x)\mapsto(t,\Gamma_t(x))$
is an injective local diffeomorphism onto $\mathcal W$.
Define on $\mathcal W$ the time-dependent vector field
$V_t(y)
=\partial_t\Gamma_t\bigl(\Gamma_t^{-1}(y)\bigr).$
It is spatially $C^2$, with spatial derivatives through
order two continuous in $(t,y)$.

Choose a smooth compactly supported function on
$\mathcal W$ equal to one on a neighborhood of the
compact set
$\mathcal S
  =
  \{(t,\Gamma_t(x)):
    0\le t\le1,\ x\in N\}.$
Multiplying $V$ by this function and extending by zero
gives a globally defined time-dependent vector field
on $\mathbb R^4$ with uniformly compact support.
Its spatial derivatives through order two are
continuous, so its flow $J_t$, with
$J_0=\mathrm{id}$, consists of $C^2$
diffeomorphisms.

For every $x\in N$, the curves
$t\mapsto\Gamma_t(x)$ and
$t\mapsto J_t(\Gamma_0(x))$ solve the same differential
equation with the same initial value.
Uniqueness therefore gives
$J_t\circ\widehat \phi\circ e_0
=\widehat\phi\circ F_t,
  \:\text{on }N.$
Using \eqref{eq:initial-ambient-realization}, the
compactly supported $C^2$ diffeomorphism
$\widetilde f=J_1\circ f_*$
satisfies
$\widetilde f\circ\widehat \phi
  =
  J_1\circ\widehat \phi\circ e_0
  =
  \widehat \phi\circ F.$

Choose an open Euclidean ball $B$ containing
$\widehat \phi(N)$ and the support of $\widetilde f$
in its interior.
The diffeomorphism $\widetilde f$ is the identity
near $\partial B$ and preserves $B$.
After an affine rescaling, identify $B$ with a
coordinate ball in $\mathbb T^4$.
Transfer $\widetilde f|_B$ to this coordinate ball
and extend it by the identity outside.
The resulting map is a $C^2$ diffeomorphism $f$ of
$\mathbb T^4$.
The corresponding rescaling of $\widehat \phi$ gives a
smooth embedding $\phi$ satisfying
$f\circ \phi=\phi\circ F.$
Since $F(N)\subset\mathrm{Int}N$, setting
$U=\phi(\mathrm{Int}N)$ gives
$\overline U=\phi(N),
  \:
  f(\overline U)\subset U.$

Near $(0,0)$, the equalities $h=0$ and $R=I$ give
$F(z,w)=\bigl(g(z),\kappa(w)\bigr).$
Since $g(0)=0$, $D_0g=I$, $\kappa(0)=0$, and
$D_0\kappa=I$, we obtain
$F(0,0)=(0,0),
  \:
  D_{(0,0)}F=I.$
Differentiating $f\circ \phi=\phi\circ F$ at $(0,0)$ yields
$D_qf
  =
  D_{(0,0)}\phi\,
  D_{(0,0)}F\,
  (D_{(0,0)}\phi)^{-1}
  =
  I.$

For $w\ne0$, the radial and tangential singular values
of $D_w\kappa$ are
$(1+C|w|^2)^{-\frac32},
  \:
  (1+C|w|^2)^{-\frac12}.$
Hence
$\|D_w\kappa\|
  =
  (1+C|w|^2)^{-\frac12}<1.$
At the origin, $D_0\kappa=I$.
The strictly increasing radial function
$\rho\mapsto
  \frac{\rho}{\sqrt{1+C\rho^2}}$
maps $[0,1]$ onto $[0,r]$, proving
$\kappa(\overline{\mathbb D}^{\,2})
  =
  r\overline{\mathbb D}^{\,2}.$

Finally, $g^{-1}(0)=\Phi^{-1}(\ker L_K)$ consists of
exactly $K^2>1$ points, one of which is $0$.
For any other point $z\in g^{-1}(0)$, the images of
the fibers over $z$ and $0$ are
$\{0\}\times
  \bigl(h(z)+r\overline{\mathbb D}^{\,2}\bigr),
  \:
  \{0\}\times r\overline{\mathbb D}^{\,2}.$
They are disjoint by injectivity of $F$.
Since both discs are closed and have radius $r$,
their centers have distance strictly greater than $2r$.
Thus $|h(z)|>2r$.
\end{proof}
\begin{remark}[The isotopy hypothesis for the radial base maps]
\label{rem:radial-source-isotopy}
The radial base maps used in this construction satisfy the
isotopy hypothesis of
Lemma \ref{lem:solenoid-ambient-realization}, as follows. Let $B=\mathrm{Int}(B(0,\rho_*))$, where $0<\rho_*<\frac 12$.
For points in $B$, use the unique representatives
$x\in\mathbb R^2$ with $|x|<\rho_*$.
Let
$q_{\mathrm{rad}}\in C([0,\rho_*])\cap C^2((0,\rho_*])$
satisfy
$q_{\mathrm{rad}}(0)=1,\:
 q_{\mathrm{rad}}(\rho)>0,\:
 q_{\mathrm{rad}}(\rho)+\rho q_{\mathrm{rad}}'(\rho)>0
 \: (0<\rho\le\rho_*).$
Assume that $q_{\mathrm{rad}}(\rho)=K$ for $\rho$
sufficiently close to $\rho_*$, and that the map
$x\longmapsto q_{\mathrm{rad}}(|x|)x$
is of class $C^2$ at the origin.
Define $g([x])=[q_{\mathrm{rad}}(|x|)x]$ on $B$ and $g=L_K$
outside $B$. For $0\le t\le1$, define
\[
 \Phi_t([x])=
 \begin{cases}
  \bigl[(1-t+tK^{-1}q_{\mathrm{rad}}(|x|))x\bigr],
       &[x]\in B,\\
  [x], &[x]\notin B.
 \end{cases}
\]
These maps are $C^2$ and agree with the identity near
$\partial B$. At a point with $0<|x|=\rho<\rho_*$,
their tangential and radial eigenvalues are, respectively, $
 1-t+\frac{t}{K}q_{\mathrm{rad}}(\rho),
 \:
 1-t+\frac{t}{K}\bigl(q_{\mathrm{rad}}(\rho)+\rho q_{\mathrm{rad}}'(\rho)\bigr),$
and both are positive. At the origin, $D_0\Phi_t=\left(1-t+\frac{t}{K}\right)I.$
Thus each $\Phi_t$ is an orientation-preserving local
diffeomorphism of $\mathbb T^2$. Since the torus is compact
and connected, $\Phi_t$ is a covering map. Moreover,
$\Phi_t$ is homotopic to $\Phi_0=\mathrm{id}_{\mathbb T^2}$
and therefore has degree one. Hence it is a global
$C^2$ diffeomorphism. Set $\Phi=\Phi_1$. The family $(\Phi_t)_{0\le t\le1}$
is smooth in $t$ with values in the space of $C^2$ maps
and gives an isotopy from the identity to $\Phi$.
By construction,
$g=L_K\circ\Phi,\:
 \Phi(0)=0,\:
 D_0g=I.$
\end{remark}

\subsection{A continuous invariant splitting on the trapping neighborhood}
\begin{lemma}
\label{lem:parabolic-solenoid-splitting}
Let $X=\mathbb T^2\times\overline{\mathbb D}^{\,2}.$ Suppose that
$f(z,w)=\bigl(g(z),h(z)+R(z)\kappa(w)\bigr)$
is a $C^2$ embedding with $f(X)\subset\mathrm{Int}X$,
where $g:\mathbb T^2\to\mathbb T^2,\:
 h:\mathbb T^2\to\mathbb R^2,\:
 R:\mathbb T^2\to\mathrm{SO}(2)$
are $C^2$. Assume that, for some $z_0\in\mathbb T^2$,
$C,c>0$, and $1<\alpha<2$, the following conditions hold:\\

$\bullet\:m(D_zg)\ge1$ for every $z\in\mathbb T^2$, with strict
inequality when $z\ne z_0$;\\

$\bullet\:$there exists a local chart $\psi$ centered at $z_0$
with $\psi(z_0)=0$, such that for all $\xi$ sufficiently
close to $0$,
\[
 (\psi\circ g\circ\psi^{-1})(\xi)
 =
 \frac{\xi}{(1-c|\xi|^\alpha)^{\frac1\alpha}};
\]

$\bullet\:$on a neighborhood of $z_0$, one has $h=0$ and $R=I$;\\

$\bullet\:$for every $w\in\overline{\mathbb D}^{\,2}$, $\kappa(w)=\frac{w}{\sqrt{1+C|w|^2}}.$\\

Set $
 p_0=(z_0,0),\:
 \Lambda=\bigcap_{n\ge0}f^n(X).$ Then there exists a continuous $Df$-invariant splitting
$T_X(\mathbb T^2\times\mathbb R^2)=E\oplus F,
 \:
 E(z,w)=\{0\}\times\mathbb R^2,$
where $F$ has rank two and
$F(p_0)=\mathbb R^2\times\{0\}.$
Moreover, $D_{p_0}f=I$, and no compact $f$-invariant subset
of $X$ containing $p_0$ admits a nontrivial dominated
splitting of the ambient tangent bundle.
\end{lemma}

\begin{proof}
In the product trivialization, write a complement to $E$
as the graph of a linear map
$S(x):\mathbb R^2\to\mathbb R^2$.
The derivative has the block form
\[
 D_{(z,w)}f
 =
 \begin{pmatrix}
  D_zg & 0\\
  B(z,w) & V(z,w)
 \end{pmatrix},
\]
where $B(z,w)v=D_zh(v)+(D_zR(v))\kappa(w),
 \:
 V(z,w)=R(z)D_w\kappa.$
The local assumptions give
$f(p_0)=p_0,\:
 D_{z_0}g=D_0\kappa=I,\:
 D_{z_0}h=D_{z_0}R=0.$
Thus $D_{p_0}f=I$ and
$p_0\in\Lambda\subset\mathrm{Int}X$.

The graph of $S$ is invariant precisely when
\begin{equation}\label{eq:solenoid-graph-equation}
 S(fx)=b(x)+\mathcal A_xS(x),
 \:
 b(x)=B(x)(D_zg)^{-1},
 \:
 \mathcal A_xS=V(x)S(D_zg)^{-1}.
\end{equation}
Use the Euclidean operator norm on matrices and its induced
norm on operators acting on matrices. Then
$\|\mathcal A_x\|
 \le q(x):=\frac{\|D_w\kappa\|}{m(D_zg)}
 \le1.$
The tangential and radial eigenvalues of $D_w\kappa$ are
$(1+C|w|^2)^{-\frac 12},
 \:
 (1+C|w|^2)^{-\frac 32},$
respectively, with $D_0\kappa=I$. Hence
\begin{equation}\label{eq:solenoid-q-unique-neutral}
 q(x)=1\:\Longleftrightarrow\: x=p_0.
\end{equation}

Define
$P_0(x)=1,\:
 P_n(x)=\prod_{j=0}^{n-1}q(f^jx),$ for all $n\ge 1$. We claim that
\begin{equation}\label{eq:solenoid-pointwise-product}
 P_n(x)\to0,
 \:\text{for every }x\ne p_0.
\end{equation}
Otherwise, the decreasing sequence $P_n(x)$ has a positive
limit, so $q(f^nx)=\frac{P_{n+1}(x)}{P_n(x)}\to1$.
By compactness and \eqref{eq:solenoid-q-unique-neutral},
$f^nx\to p_0$. The orbit eventually remains in the neighborhood where the
local formulas apply. A nonzero base coordinate has strictly
increasing norm there and therefore cannot converge to zero. Write $f^nx=(z_n,w_n)$. Thus the base coordinate equals $z_0$ from some time $N$
onward. The fiber coordinates then satisfy
$w_{n+1}=\kappa(w_n)$.
Since $f$ is injective and fixes $p_0$, one has $w_N\ne0$.
Writing $r_n=|w_n|$, we obtain
$r_n=\bigl(r_N^{-2}+C(n-N)\bigr)^{-\frac12},
 \:
 \prod_{j=N}^{n-1}q(f^jx)
 =
 \prod_{j=N}^{n-1}\frac{r_{j+1}}{r_j}
 =
 \frac{r_n}{r_N}\to0,$
a contradiction. This proves
\eqref{eq:solenoid-pointwise-product}.
Since $P_n$ is continuous and decreases with $n$, Dini's
theorem gives
\begin{equation}\label{eq:solenoid-uniform-products}
 \sup_{x\in L}P_n(x)\to0,
 \:\text{for every compact }
 L\subset X\setminus\{p_0\}.
\end{equation}

We next construct the graph on $\Lambda$.
Put $X_n=f^n(X)$. For $j\ge1$, the notation $f^{-j}$
denotes the inverse of the embedding $f^j:X\to X_j$.
Since $f(\Lambda)=\Lambda$, the restriction $f|_\Lambda$
is a homeomorphism. The function $b$ vanishes on a neighborhood of the fiber
over $z_0$. If $b\equiv0$, then $S\equiv0$ on $X$ solves
\eqref{eq:solenoid-graph-equation}. For the remaining
construction, assume $b\not\equiv0$ and set $\mathcal K=\mathrm{supp}_X b,\:
 \theta=\max_{x\in\mathcal K}q(x)<1,\:
 M_b=\|b\|_\infty.$
For $y\in\Lambda$, write $x_{-j}=f^{-j}y$ and define
\begin{equation}\label{eq:solenoid-backward-series}
 S_\Lambda(y)
 =
 \sum_{j=1}^{+\infty}
 \mathcal A_{x_{-1}}\cdots
 \mathcal A_{x_{-(j-1)}}b(x_{-j}),
\end{equation}
where an empty product is the identity.
If the nonzero terms have indices $j_1<j_2<\cdots$,
the $\ell$th such term has norm at most
$M_b\theta^{\ell-1}$: its operator product contains a
factor associated with each preceding nonzero forcing term.
Thus the series is absolutely convergent and bounded by
$\frac{M_b}{1-\theta}$.

To prove uniform convergence in $y$, set
$\delta_N=\sup_{x\in f(\mathcal K)}P_N(x),
 \:N\ge0.$
Injectivity and $f(p_0)=p_0$ imply
$p_0\notin f(\mathcal K)$, so
\eqref{eq:solenoid-uniform-products} gives $\delta_N\to0$.
If $j_0>N$ is the first nonzero term in the tail, its
operator product has norm at most
$\prod_{i=1}^{j_0-1}q(x_{-i})
 =
 P_{j_0-1}(fx_{-j_0})
 \le\delta_N.$
Every subsequent nonzero term acquires an additional
factor at most $\theta$. Consequently,
\begin{equation}\label{eq:solenoid-uniform-tail}
 \sum_{j>N}
 \left\|
 \mathcal A_{x_{-1}}\cdots
 \mathcal A_{x_{-(j-1)}}b(x_{-j})
 \right\|
 \le\frac{M_b\delta_N}{1-\theta}.
\end{equation}
The same argument applies to finite tails along finite
backward orbits in $X$.

Each term in \eqref{eq:solenoid-backward-series} is
continuous on $\Lambda$, so the uniform convergence makes
$S_\Lambda$ continuous. Separating the first term and shifting the index gives
$S_\Lambda(fx)=b(x)+\mathcal A_xS_\Lambda(x),
 \:x\in\Lambda.$
Since $f^{-j}(p_0)=p_0$ for every $j\ge1$ and
$b(p_0)=0$, every term of the series vanishes at $p_0$.
Hence $S_\Lambda(p_0)=0$.

It remains to extend this graph invariantly to $X$.
Let $Q_n$ denote the $n$th partial sum of
\eqref{eq:solenoid-backward-series}, regarded as a
continuous function on $X_n$ using the inverse maps
defined above. Define
$D_0=X\setminus f(\mathrm{Int}X),
 \:D_n=f^n(D_0).$
Since $f(X)\subset\mathrm{Int}X$,
$X_{n+1}\subset\mathrm{Int}X_n,
\:D_n=X_n\setminus\mathrm{Int}X_{n+1}.$
Hence
$X\setminus\Lambda=\bigcup_{n\ge0}D_n,
 \:
 D_n\cap D_{n+1}=\partial X_{n+1},$
and nonadjacent shells are disjoint.
This family is locally finite in $X\setminus\Lambda$:
if $y\notin\Lambda$, a neighborhood of $y$ disjoint
from some $X_M$ meets only finitely many shells. On the disjoint closed subsets $\partial X$ and
$f(\partial X)$ of $D_0$, prescribe
\begin{equation}\label{eq:solenoid-shell-boundaries}
 S_0(x)=0,\:S_0(fx)=b(x),
 \:x\in\partial X.
\end{equation}
The Tietze extension theorem, applied to the four matrix
entries, gives a continuous extension $S_0$ to $D_0$.
Set $S=S_0$ on $D_0$ and propagate it by
\eqref{eq:solenoid-graph-equation}. Explicitly, for
$x\in D_0$ and $n\ge1$,
$S(f^nx)
 =
 Q_n(f^nx)
 +
 \mathcal A_{f^{n-1}x}\cdots\mathcal A_xS_0(x).$
The boundary conditions
\eqref{eq:solenoid-shell-boundaries} are exactly the
compatibility condition
$S_0(fx)=b(x)+\mathcal A_xS_0(x),
 \:x\in\partial X.$
Thus adjacent definitions agree. By local finiteness,
$S$ is well defined and continuous on $X\setminus\Lambda$,
and satisfies \eqref{eq:solenoid-graph-equation} there.
Moreover,
$
 \sup_{x\in X\setminus\Lambda}\|S(x)\|
 \le M_0+\frac{M_b}{1-\theta},
 \:M_0=\|S_0\|_\infty.$

Define $S=S_\Lambda$ on $\Lambda$.
To verify continuity across $\Lambda$, observe that
$p_0\notin D_0$ and hence
$\varepsilon_N:=\sup_{x\in D_0}P_N(x)\to0.$
Fix $N\ge1$. For $y\in X_N\setminus\Lambda$, let $n$
be the largest integer such that $y\in X_n$. Then $n\ge N$ and
$y=f^nx$ for some $x\in D_0$.
In the propagation formula, the difference between
$S(y)$ and $Q_N(y)$ consists of a finite tail and the
propagated initial value. Their norms are bounded by
$\frac{M_b\delta_N}{1-\theta},
 \:
 M_0P_n(x)\le M_0\varepsilon_N,$
respectively. On $\Lambda$, the first bound follows
from \eqref{eq:solenoid-uniform-tail}, with no initial
value term. Therefore,
\begin{equation}\label{eq:solenoid-near-attractor-continuity}
 \sup_{y\in X_N}\|S(y)-Q_N(y)\|
 \le
 \frac{M_b\delta_N}{1-\theta}+M_0\varepsilon_N
 \to0.
\end{equation}
For each fixed $N$,
$\Lambda\subset X_{N+1}\subset\mathrm{Int}X_N$.
Thus, at any $p\in\Lambda$, one may first choose $N$
so that the approximation error is small and then
use the continuity of $Q_N$ in a neighborhood of $p$.
This proves continuity of $S$ on all of $X$.

Finally, set
$F(x)=\{(v,S(x)v):v\in\mathbb R^2\}.$
This is a continuous rank-two complement to $E$,
and \eqref{eq:solenoid-graph-equation} proves its
$Df$-invariance. The vertical bundle is invariant
by the block form of $Df$, and $S(p_0)=0$ gives
$F(p_0)=\mathbb R^2\times\{0\}$.

Since $D_{p_0}f=I$, any splitting
$T_{p_0}(\mathbb T^2\times\mathbb R^2)=G\oplus H$
with both summands of positive dimension satisfies
\[
 \|D_{p_0}f^n|_G\|\,
 \|(D_{p_0}f^n|_H)^{-1}\|=1,\:\forall\:n\ge 1.
\]
This is incompatible with the exponential bound required
by domination, proving the final assertion.
\end{proof}

\subsection{Verification of the coordinated block condition}
Fix $\alpha\in(1,2)$. Choose an odd integer $k$ sufficiently
large for the base construction, and let $d\in\{1,2,3\}$
and $K=k^d$ be supplied by
Lemma \ref{lem:solenoid-ambient-realization}.
Since $K\ge k$ is odd, construct the base map $g=g_\alpha$
with this value of $K$.
Remark \ref{rem:radial-source-isotopy} verifies the required
isotopy hypothesis, so
Lemma \ref{lem:solenoid-ambient-realization} applies to $g$.
Let
$N=\mathbb T^2\times\overline{\mathbb D}^{\,2},
 \:
 U=\phi(\mathrm{Int}N),$
where $\phi$ is the smooth embedding supplied by that lemma. Denote the bundle embedding by
$\widehat f:N\to\mathrm{Int}N$
and its ambient realization by
$f:\mathbb T^4\to\mathbb T^4$.
These maps satisfy $f\circ\phi=\phi\circ\widehat f$.
In the product coordinates on $N$, we have
\[
 \widehat f(z,w)=\bigl(g(z),h(z)+R(z)\kappa(w)\bigr),
 \:
 \kappa(w)=\frac{w}{\sqrt{1+C|w|^2}},
 \:
 C=r^{-2}-1>0,
\]
where $0<r<1$.

By Lemma \ref{lem:parabolic-solenoid-splitting},
the bundle embedding has a continuous invariant splitting
on $N$, with $E$ vertical and $F$ the graph of a continuous
matrix field $S$. For $U=\phi(\mathrm{Int}N)$, this gives
$\overline U=\phi(N),
 \:
 f(\overline U)=\phi(\widehat f(N))\subset U.$
Since $\phi$ is an embedding between manifolds of the same
dimension, $D_x\phi$ is a linear isomorphism for every
$x\in N$, including boundary points. Define
$\widetilde E(\phi(x))=D_x\phi(E(x)),
 \:
 \widetilde F(\phi(x))=D_x\phi(F(x)).$
These bundles give a continuous splitting
$T_{\overline U}\mathbb T^4
 =\widetilde E\oplus\widetilde F.$
Differentiating the realization identity yields
$D_{\phi(x)}f\circ D_x\phi
 =
 D_{\widehat f(x)}\phi\circ D_x\widehat f.$
The $D\widehat f$-invariance of $E$ and $F$ therefore implies
the $Df$-invariance of $\widetilde E$ and $\widetilde F$.
Moreover,
$\bigcap_{n\ge0}f^n(\overline U)
 =
 \phi\left(\bigcap_{n\ge0}\widehat f^{\,n}(N)\right).$
Thus the ambient system has the required continuous
invariant splitting on the closure of its trapping
neighborhood.

Extend the product metric transported by $\phi$ to a smooth
Riemannian metric on $\mathbb T^4$.
For the local estimates below, we continue to write $\widehat f$
for the coordinate representation and retain $E$ and $F$
for the bundles on $N$. For $x\in N$, let
$$\Phi_p(x)
 =
 \min\left\{
 \log m(D_x\widehat f^p|_{F(x)}),
 -\log\|D_x\widehat f^p|_{E(x)}\|
 \right\},
 \: p\ge1.$$
All block lengths are integers.
For $y\in\overline U$, define the ambient block quantity by
$$\widetilde\Phi_p(y)
 =
 \min\left\{
 \log m(D_yf^p|_{\widetilde F(y)}),
 -\log\|D_yf^p|_{\widetilde E(y)}\|
 \right\},
 \:p\ge1.$$
For every $p\ge1$, differentiating the iterated realization
identity gives
$D_{\phi(x)}f^p\circ D_x\phi
 =
 D_{\widehat f^{\,p}(x)}\phi
 \circ D_x\widehat f^{\,p}.$
Both $D_x\phi$ and $D_{\widehat f^{\,p}(x)}\phi$ are
isometries for the chosen metrics. Hence the corresponding
norms and conorms on the invariant subbundles agree.
Consequently,
$\widetilde\Phi_p(\phi(x))=\Phi_p(x),
 \:x\in N,\:p\ge1.$

The asymptotic block condition is independent of the choice
of smooth metric.
Indeed, if two metrics satisfy, for some $L\ge1$,
$L^{-1}|\xi|\le|\xi|'\le L|\xi|$ on $\phi(N)$, then
$|\Phi'_p(x)-\Phi_p(x)|\le2\log L$
uniformly in $x$ and $p$.
Here $\Phi'_p$ is computed using the pullback of the
alternative ambient metric.
The corresponding normalized block averages differ by at most
$2\frac{\log L}{p}$, which tends to zero as $p\to+\infty$.
We therefore use the transported product metric below.

Let $x=(z,w)\in N$. Define
$J_x\xi=(\xi,S(x)\xi),
 \:
 C_F=\sup_{x\in N}\|J_x\|\ge1.$
Compactness gives $C_F<+\infty$, and
$|\xi|\le|J_x\xi|\le C_F|\xi|.$
Let $\pi(z,w)=z$. The semiconjugacy relation
$\pi\circ \widehat f=g\circ\pi$ and invariance of $F$ imply
$D_x\widehat f^pJ_x=J_{\widehat f^px}D_zg^p.$
Consequently,
$$m(D\widehat f^p|_{F(x)})
 \ge C_F^{-1}m(D_zg^p)
 \ge C_F^{-1}\prod_{j=0}^{p-1}m(D_{g^jz}g).$$
Writing $u(z)=\log m(D_zg),$
we obtain
$\log m(D\widehat f^p|_{F(x)})
 \ge \sum_{j=0}^{p-1}u(g^jz)-\log C_F.$
By construction, $u\ge0$, with $u(z)=0$ precisely at $z=0$.

For the vertical direction, set
$v(w)=-\log\|D_w\kappa\|
      =\frac12\log(1+C|w|^2),
 \:
 (z_j,w_j)=\widehat f^j(z,w).$
Since $R(z_j)$ is orthogonal, submultiplicativity gives
\begin{equation}\label{eq:vertical-bound}
 -\log\|D\widehat f^p|_{E(x)}\|
 \ge\sum_{j=0}^{p-1}v(w_j).
\end{equation}
Define
$\delta(z)=(|h(z)|-r)_+,
 \:
 \beta(z)=\frac12\log(1+C\delta(z)^2),
 \:
 \chi(z)=\min\{u(z),\beta(z)\},$
where $t_+=\max\{t,0\}$.
These functions are continuous and nonnegative.
Because $|\kappa(w)|\le r$ on the closed unit disc,
$|w_{j+1}|
 =
 |h(z_j)+R(z_j)\kappa(w_j)|
 \ge\delta(z_j),$
and hence $v(w_{j+1})\ge\beta(z_j)$.
Thus, for $p\ge2$,
$-\log\|D\widehat f^p|_{E(x)}\|
 \ge\sum_{j=1}^{p-1}v(w_j)
 \ge\sum_{j=0}^{p-2}\beta(g^jz)
 \ge\sum_{j=0}^{p-2}\chi(g^jz).$
The estimate along $F$ and $u\ge0$ also give
$\log m(D\widehat f^p|_{F(x)})
 \ge\sum_{j=0}^{p-2}\chi(g^jz)-\log C_F.$
Taking the minimum and using $C_F\ge1$, we conclude that
\begin{equation}\label{eq:pointwise-block}
 \Phi_p(x)
 \ge\sum_{j=0}^{p-2}\chi(g^jz)-\log C_F,
 \:\forall\:p\ge2,\:\forall\:x=(z,w)\in N.
\end{equation}

Lemma \ref{lem:solenoid-ambient-realization} supplies a
nonzero point $z_*$ such that $g(z_*)=0$ and $|h(z_*)|>2r$.
Therefore $\delta(z_*)>r$, $\beta(z_*)>0$, and $u(z_*)>0$.
It follows that $\chi(z_*)>0$, so $\chi$ is positive on a
nonempty open set.
By Lemma \ref{lem:intermittent-base}, $g$ has a unique
absolutely continuous invariant probability $\nu$, which
is ergodic and equivalent to Lebesgue measure. Hence
$c_0:=\int_{\mathbb T^2}\chi\,d\nu>0.$
Birkhoff's ergodic theorem and this equivalence give a set
$Z\subset\mathbb T^2$ of full Lebesgue measure such that
$\lim_{m\to+\infty}
 \frac1m\sum_{j=0}^{m-1}\chi(g^jz)=c_0,\:\forall\:z\in Z.$

Fix $x=(z,w)\in Z\times\overline{\mathbb D}^{\,2}$ and an integer $p\ge2$.
Summing \eqref{eq:pointwise-block} along
$x,\widehat f^px,\ldots,\widehat f^{(n-1)p}x$ omits exactly the base indices
$jp+p-1$, $0\le j<n$. Therefore,
\[
\begin{aligned}
 \frac1{np}\sum_{j=0}^{n-1}\Phi_p(\widehat f^{jp}x)
 &\ge
 \frac1{np}\sum_{j=0}^{np-1}\chi(g^jz)
 -\frac1{np}\sum_{j=0}^{n-1}\chi(g^{jp+p-1}z)
 -\frac{\log C_F}{p}\\
 &\ge
 \frac1{np}\sum_{j=0}^{np-1}\chi(g^jz)
 -\frac{\|\chi\|_\infty+\log C_F}{p}.
\end{aligned}
\]
Since the ordinary Birkhoff averages converge at every
$z\in Z$, they also converge along each subsequence $m=np$.
Thus the same set $Z$ gives, for every integer $p\ge2$,
$\liminf_{n\to+\infty}
 \frac1{np}\sum_{j=0}^{n-1}\Phi_p(\widehat f^{jp}x)
 \ge
 c_0-\frac{\|\chi\|_\infty+\log C_F}{p}.$
Taking the lower limit as $p\to+\infty$ yields
\begin{equation}\label{eq:aligned-conclusion}
 \liminf_{p\to+\infty}\liminf_{n\to+\infty}
 \frac1{np}\sum_{j=0}^{n-1}\Phi_p(\widehat f^{jp}x)
 \ge c_0>0.
\end{equation}

We now transfer this estimate to the ambient system.
Writing $f$ for the ambient diffeomorphism and $\widehat f$
for the bundle embedding, the realization identity and
the equality of the block quantities give
$\frac1{np}\sum_{j=0}^{n-1}
 \widetilde\Phi_p\bigl(f^{jp}(\phi(x))\bigr)
 =
 \frac1{np}\sum_{j=0}^{n-1}
 \Phi_p(\widehat f^{\,jp}x).$
Consequently, \eqref{eq:aligned-conclusion} gives the same
positive lower bound for the ambient block averages at
every point of $\phi(Z\times\mathbb D^2)$.

By Fubini's theorem, $Z\times\mathbb D^2$ has full
volume in $\mathrm{Int}N$.
Since $\phi$ is a smooth embedding between manifolds of the
same dimension, $\phi(Z\times\mathbb D^2)$ has full volume
in $U$.
Consequently,
\[
\liminf_{p\to+\infty}\liminf_{n\to+\infty}
 \frac1{np}\sum_{j=0}^{n-1}
 \widetilde\Phi_p(f^{jp}y)
 \ge c_0>0,
 \:\text{for Leb-a.e. }y\in U.
\]
In particular, the corresponding condition with an outer
$\limsup$ holds on a set of full volume in $U$, and hence
on a set of positive volume. Therefore, by Theorem \ref{thm:B} with $a=0$, $(\Lambda,f)$ admits a hyperbolic ergodic SRB measure.

\subsection{The support of every ambient SRB measure}
\label{sec:solenoid-support}

We consider only those SRB measures of the ambient diffeomorphism whose support is contained in
the constructed attractor.

Recall that
$N=\mathbb T^2\times\overline{\mathbb D}^{\,2}$.
Let $\widehat f:N\to\mathrm{Int}N$ denote the
bundle embedding and $f:\mathbb T^4\to\mathbb T^4$
its ambient realization, so that
$f\circ\phi=\phi\circ\widehat f$.
In product coordinates,
\[
 \widehat f(z,w)=\bigl(g(z),H_z(w)\bigr),
 \:
 H_z(w)=h(z)+R(z)\kappa(w).
\]
Here $g:\mathbb T^2\to\mathbb T^2$ is a surjective
$C^2$ local diffeomorphism, $R(z)$ is orthogonal, and
$H_z(\overline{\mathbb D}^{\,2})\subset\mathrm{Int}\overline{\mathbb D}^{\,2}$ for every $z$.
The radial fiber map satisfies
\begin{equation}\label{eq:ss-strict}
 |\kappa(u)-\kappa(v)|<|u-v|
 \:
 (u,v\in\overline{\mathbb D}^{\,2},\ u\ne v),
 \:
 \|D_w\kappa\|\le1
 \:
 (w\in\overline{\mathbb D}^{\,2}).
\end{equation}
Indeed, its derivative norm is strictly less than one
except at $0$, so integration along any nondegenerate
line segment in $\overline{\mathbb D}^{\,2}$ gives the strict distance inequality.
The embedding property of $\widehat f$, including
separation of fiber images over distinct preimages of
the same base point, follows from the construction.

Set $\Lambda_N=\bigcap_{n\ge0}\widehat f^{\,n}(N),
 \:
 \Lambda=\phi(\Lambda_N)
        =\bigcap_{n\ge0}f^n(\phi(N)),$
and define
$\pi_N(z,w)=z,
\:\pi=\pi_N\circ\phi^{-1}:\phi(N)\to\mathbb T^2.$
Thus $\pi_N\circ\widehat f=g\circ\pi_N$ and
$\pi\circ f=g\circ\pi$ on $\phi(N)$.

By Lemma \ref{lem:intermittent-base}, $g$ has a unique
absolutely continuous invariant probability $\nu$,
which is ergodic and has full support.
The base construction also gives
\begin{equation}\label{eq:ss-base}
 g(0)=0,
 \:
 m(D_zg)\ge1\:(z\in\mathbb T^2),
 \:
 m(D_zg)>1\:(z\ne0).
\end{equation}
The point $p=\phi(0,0)$ is fixed by $f$ and satisfies
$D_pf=I$.
The continuous invariant splitting constructed above
restricts to
$T_\Lambda\mathbb T^4
 =\widetilde E\oplus\widetilde F,
 \:
 \widetilde E=\ker D\pi,
 \:
 \dim\widetilde F=2.$
The norms in \eqref{eq:ss-strict} and \eqref{eq:ss-base}
are taken with respect to the standard Euclidean
fiber metric and the flat base metric.

\begin{lemma}\label{lem:ss-natural}
The system $(\Lambda,f)$ is topologically conjugate
to the natural extension of $g$.
Every $g$-invariant Borel probability $\tau$ has a unique
$f$-invariant Borel probability $\mu$ on $\Lambda$
satisfying $\pi_*\mu=\tau$.
If $\tau$ is ergodic, then $\mu$ is ergodic.
If $\tau$ has full support on $\mathbb T^2$, then
$\mu$ has full support on $\Lambda$.
\end{lemma}

\begin{proof}
We first establish uniform shrinking of fiber images.
For $0<\delta<\mathrm{diam}\overline{\mathbb D}^{\,2}$, compactness and
\eqref{eq:ss-strict} give
$q_\delta:=
 \max_{\substack{u,v\in \overline{\mathbb D}^{\,2},\:|u-v|\ge\delta}}
 \frac{|\kappa(u)-\kappa(v)|}{|u-v|}
 <1.$
Translations and orthogonal maps preserve distances.
Separating pairs with distance less than $\delta$
from those with distance at least $\delta$, we obtain
$\mathrm{diam}H_z(C)
 \le
 \max\{\delta,q_\delta\mathrm{diam}C\}$
for every nonempty subset $C\subset \overline{\mathbb D}^{\,2}$ and every $z$.
Since each $H_z$ maps $\overline{\mathbb D}^{\,2}$ into itself, iteration gives
$\mathrm{diam}
 \bigl(H_{z_1}\circ\cdots\circ H_{z_n}(\overline{\mathbb D}^{\,2})\bigr)
 \le
 \max\{\delta,q_\delta^n\mathrm{diam}\overline{\mathbb D}^{\,2}\}$
for every finite sequence $z_1,\ldots,z_n$.
Given $\varepsilon>0$, first choose
$0<\delta<\min\{\varepsilon,\mathrm{diam}\overline{\mathbb D}^{\,2}\}$
and then take $n$ sufficiently large.
Thus these diameters tend to zero uniformly over
all sequences of base points.

Consider the natural-extension space
$\widehat X_g
 =
 \left\{
 \mathbf z=(z_0,z_{-1},\ldots):
 g(z_{-j})=z_{-j+1}\ \text{for }j\ge1
 \right\},$
equipped with the subspace topology inherited from
$(\mathbb T^2)^{\mathbb N}$.
It is a compact metric space, and
$\widehat g(z_0,z_{-1},\ldots)
 =
 (g(z_0),z_0,z_{-1},\ldots)$
is a homeomorphism, whose inverse deletes the first
coordinate.

For $\mathbf z\in\widehat X_g$, the sets
$C_n(\mathbf z)
 =
 H_{z_{-1}}\circ\cdots\circ H_{z_{-n}}(\overline{\mathbb D}^{\,2})$
are nonempty, compact, and nested.
Their diameters tend uniformly to zero, so their
intersection is a singleton $\{w(\mathbf z)\}$.
Define
$\mathcal I(\mathbf z)=(z_0,w(\mathbf z)).$
To verify continuity, fix $w_*\in \overline{\mathbb D}^{\,2}$ and set
$w_n(\mathbf z)
 =
 H_{z_{-1}}\circ\cdots\circ H_{z_{-n}}(w_*).$
Each $w_n$ is continuous, and
$|w_n(\mathbf z)-w(\mathbf z)|
 \le\mathrm{diam}C_n(\mathbf z).$
Thus $w_n\to w$ uniformly, proving continuity of
$\mathcal I$.

The defining intersections give
$\widehat f\circ\mathcal I
 =
 \mathcal I\circ\widehat g.$
Moreover, for every $n\ge1$,
$\mathcal I(\mathbf z)
 \in
 \widehat f^{\,n}(\{z_{-n}\}\times \overline{\mathbb D}^{\,2})
 \subset\widehat f^{\,n}(N),$
so $\mathcal I(\widehat X_g)\subset\Lambda_N$.

Since $\widehat f$ is an embedding and
$\widehat f(N)\subset N$, one has
$\widehat f(\Lambda_N)=\Lambda_N$.
Hence $\widehat f|_{\Lambda_N}$ is a homeomorphism.
For any $x\in\Lambda_N$, its unique full backward orbit
has base coordinates defining an element
$\mathbf z\in\widehat X_g$.
The fiber coordinate of $x$ belongs to every
$C_n(\mathbf z)$, and therefore $x=\mathcal I(\mathbf z)$.
Thus $\mathcal I$ is onto $\Lambda_N$.
If $\mathcal I(\mathbf z)=\mathcal I(\mathbf z')$,
applying $(\widehat f|_{\Lambda_N})^{-n}$ and taking
base coordinates gives $z_{-n}=z'_{-n}$ for every $n$.
Thus $\mathcal I$ is injective and, by compactness,
a homeomorphism.

Consequently,
$\mathcal J:=\phi\circ\mathcal I:
 \widehat X_g\to\Lambda$
is a topological conjugacy.
Writing $\rho_0(\mathbf z)=z_0$, we also have
$\pi\circ\mathcal J=\rho_0$.
It therefore remains to prove the measure assertions
on $\widehat X_g$.

Let $\tau$ be a $g$-invariant Borel probability.
Any $\widehat g$-invariant probability
$\widehat\tau$ satisfying $(\rho_0)_*\widehat\tau=\tau$
has marginal distribution $\tau$ at every coordinate.
For Borel sets $A_0,\ldots,A_n\subset\mathbb T^2$,
the compatibility relations therefore imply
\[
 \widehat\tau
 \{\mathbf z:z_{-j}\in A_j,\ 0\le j\le n\}
 =
 \tau\left(
 A_n\cap g^{-1}A_{n-1}\cap\cdots\cap g^{-n}A_0
 \right).
\]
These identities determine $\widehat\tau$, since finite
cylinders generate the Borel sigma-algebra.
Conversely, the displayed formula defines consistent
finite-dimensional distributions because $g_*\tau=\tau$.
Kolmogorov's extension theorem gives a probability
concentrated on the compatible histories
$\widehat X_g$.
The same cylinder formula verifies its
$\widehat g$-invariance and its projection to $\tau$.
Thus the invariant lift exists and is unique.

If $\tau$ is ergodic, the projections of almost all
ergodic components of $\widehat\tau$ equal $\tau$,
by extremality of $\tau$ among $g$-invariant probabilities.
Uniqueness of the invariant lift then forces almost
every component to equal $\widehat\tau$.
Hence $\widehat\tau$ is ergodic.

Finally, let
$\mathcal C
 =
 \{\mathbf z\in\widehat X_g:
       z_{-j}\in U_j,\ 0\le j\le n\}$
be a nonempty basic open cylinder, where each $U_j$
is open.
The set
$V
 =
 U_n\cap g^{-1}U_{n-1}\cap\cdots\cap g^{-n}U_0$
is nonempty and open, and the cylinder formula gives
$\widehat\tau(\mathcal C)=\tau(V)$.
If $\tau$ has full support, this quantity is positive.
Since basic open cylinders form a basis,
$\widehat\tau$ has full support on $\widehat X_g$.

Set $\mu=\mathcal J_*\widehat\tau$.
The conjugacy and the identity
$\pi\circ\mathcal J=\rho_0$
show that $\mu$ is $f$-invariant and satisfies
$\pi_*\mu=\tau$.
Conversely, every invariant lift of $\tau$ to $\Lambda$
pulls back under $\mathcal J$ to its unique invariant
lift on $\widehat X_g$.
Thus $\mu$ is unique, and the conjugacy preserves the
ergodicity and full-support properties established above.
\end{proof}

Let $\mu_\nu$ denote the unique invariant lift of $\nu$
to $\Lambda$. By Lemma \ref{lem:ss-natural},
$\mu_\nu$ is ergodic and
$\mathrm{supp}\mu_\nu=\Lambda$.
We recall that an SRB measure has a
positive Lyapunov exponent almost everywhere and
absolutely continuous conditional measures on its
Pesin unstable manifolds.

\begin{proposition}\label{prop:ss-srb-support}
Every ergodic SRB measure of the ambient diffeomorphism
$f$ supported on $\Lambda$ equals $\mu_\nu$.
In particular, every such measure has support exactly
$\Lambda$.
\end{proposition}

\begin{proof}
Recall that $\pi=\pi_N\circ\phi^{-1}$ is smooth on $U$
and satisfies $\pi\circ f=g\circ\pi$ on $U$.
Let $\mu$ be an ergodic SRB measure supported on
$\Lambda$, and put
$\tau=\pi_*\mu
      =(\pi_N)_*\bigl((\phi^{-1})_*\mu\bigr).$
The semiconjugacy therefore makes $\tau$
an ergodic $g$-invariant measure. Since $f(p)=p$ and $\pi(p)=0$, the measure $\delta_p$
is an invariant lift of $\delta_0$.
Lemma \ref{lem:ss-natural} shows that it is the unique
such lift.
All Lyapunov exponents of $\delta_p$ are zero because
$D_pf=I$, so $\delta_p$ is not an SRB measure under
our convention. Consequently, $\tau\ne\delta_0$.

The continuous function
$q(z)=\log m(D_zg)$
is nonnegative and vanishes only at $0$. Hence
$\beta_\tau:=\int_{\mathbb T^2}q\,d\tau>0.$
Supermultiplicativity of the conorm gives
$\log m(D_zg^n)
 \ge \sum_{j=0}^{n-1}q(g^jz).$
By Birkhoff's theorem,
\begin{equation}\label{eq:ss-base-positive}
 \liminf_{n\to+\infty}
 \frac1n\log m(D_zg^n)
 \ge\beta_\tau>0,
 \:\text{for $\tau$-almost every $z$}.
\end{equation}

We use the transported product metric on $\phi(N)$.
In product coordinates,
$D_{(z,w)}\widehat f(0,v)
 =
 (0,R(z)D_w\kappa\,v).$
Since $R(z)$ is orthogonal and $\|D_w\kappa\|\le1$, $\|D_yf|_{\widetilde E(y)}\|\le1,
 \:y\in\Lambda.$
Thus all Lyapunov exponents on $\widetilde E$
are nonpositive.

For $y\in\Lambda$, define
$A_y=D_y\pi|_{\widetilde F(y)}:
 \widetilde F(y)\longrightarrow T_{\pi(y)}\mathbb T^2.$
Because $\ker D_y\pi=\widetilde E(y)$ and
$\widetilde E(y)\oplus\widetilde F(y)=T_y\mathbb T^4$,
each $A_y$ is an isomorphism.
Continuity and compactness give $C_\pi\ge1$ such that
$\|A_y\|\le C_\pi,
 \:
 \|A_y^{-1}\|\le C_\pi,
 \:y\in\Lambda.$
Differentiating the semiconjugacy relation and using
invariance of $\widetilde F$, we obtain
$A_{f^ny}\circ
 \bigl(D_yf^n|_{\widetilde F(y)}\bigr)
 =
 D_{\pi(y)}g^n\circ A_y.$
Consequently,
$m\bigl(D_yf^n|_{\widetilde F(y)}\bigr)
 \ge C_\pi^{-2}m(D_{\pi(y)}g^n).$
Together with \eqref{eq:ss-base-positive} and
$\pi_*\mu=\tau$, this shows that both Lyapunov
exponents on $\widetilde F$ are positive.
The invariant direct-sum decomposition
$\widetilde E\oplus\widetilde F$ therefore implies
\[
 E^u(y)=\widetilde F(y),
 \:\text{for $\mu$-almost every regular point }y,
\]
where $E^u(y)$ denotes the sum of the positive
Oseledets subspaces.
In particular, the Pesin unstable manifolds have
dimension two.

Let $\mathcal R\subset\Lambda$ be a full-$\mu$-measure
set of regular points satisfying this identity.
For every $y\in\mathcal R$, the derivative of $\pi$
on $T_yW^u(y)=\widetilde F(y)$ is invertible.
The inverse function theorem therefore gives a
neighborhood of $y$ in its unstable manifold,
contained in $U$, on which $\pi$ is a $C^1$
diffeomorphism onto an open subset of $\mathbb T^2$. Choose a measurable partition $\xi$, defined modulo
$\mu$, subordinate to local unstable manifolds,
with its atoms contained in $U$.
Denote its conditional probabilities by $\mu_A^{\xi}$.
For almost every atom $A$,
$\mu_A^{\xi}(A\cap\mathcal R)=1,$
and the SRB property implies that $\mu_A^{\xi}$ is absolutely
continuous with respect to the induced
two-dimensional leaf volume.

Let $B\subset\mathbb T^2$ be a Borel Lebesgue-null set.
By second countability, $A\cap\mathcal R$ is covered
by countably many of the above neighborhoods in
the unstable disc containing $A$.
On each neighborhood, the inverse image of $B$
has zero leaf volume. Hence
$A\cap\mathcal R\cap\pi^{-1}(B)$ has zero leaf volume,
and absolute continuity gives
$\mu_A^{\xi}(\pi^{-1}(B))=0$ for almost every atom $A$.
Disintegration yields $\mu(\pi^{-1}(B))=0.$
Therefore $\tau\ll\mathrm{Leb}_{\mathbb T^2}$.

By Lemma \ref{lem:intermittent-base}, the base map
has a unique absolutely continuous invariant
probability, so $\tau=\nu$.
Uniqueness of the invariant lift in
Lemma \ref{lem:ss-natural} then gives $\mu=\mu_\nu$.
\end{proof}

Thus any ergodic SRB measure on $\Lambda$ equals $\mu_{\nu}$ and has full support on $\Lambda$.

\begin{corollary}\label{cor:ss-domination}
No proper compact $f$-invariant subset of $\Lambda$
supports an ergodic SRB measure of the ambient
diffeomorphism.
Moreover, the ambient tangent bundle
$T_\Lambda\mathbb T^4$ admits no nontrivial
dominated splitting.
Consequently, an SRB existence theorem requiring
a nontrivial dominated splitting on $\Lambda$,
or on the compact support of the resulting
ergodic SRB measure contained in $\Lambda$,
cannot be applied directly to obtain such a measure.
\end{corollary}

\begin{proof}
The first assertion follows from
Proposition \ref{prop:ss-srb-support}, since every
such measure has support exactly $\Lambda$. Suppose that
$T_\Lambda\mathbb T^4=G\oplus H$
were a dominated splitting with both bundles of
positive rank.
Since $p\in\Lambda$ is fixed and $D_pf=I$, for every
$n\ge1$ we would have
$\|D_pf^n|_{G(p)}\|\,
 \|(D_pf^n|_{H(p)})^{-1}\|=1.$
This contradicts the exponential decay required
by domination.
The final assertion follows from these two conclusions.
\end{proof}

The conclusion extends to SRB probabilities supported on
$\Lambda$ without an ergodicity assumption.
Indeed, the entropy characterization in
\cite{LY85-1}, together with
the ergodic decomposition formula for entropy and
Ruelle's inequality, implies that almost every ergodic
component satisfies Pesin's entropy formula.
The positive-exponent condition also passes to almost
every component, so these components are SRB measures.
Since they are supported on $\Lambda$,
Proposition \ref{prop:ss-srb-support} identifies them
with $\mu_\nu$. Hence the original probability also
equals $\mu_\nu$.

\begin{corollary}\label{cor:ss-transitive}
The restriction $f|_\Lambda$ is topologically transitive
and has no nonempty proper relative subattractor.
\end{corollary}

\begin{proof}
Let $\{O_m\}_{m\ge1}$ be a countable basis of nonempty
relatively open subsets of $\Lambda$.
Since $\mu_\nu$ is ergodic and has full support,
Birkhoff's theorem gives a point $x\in\Lambda$ such that
\[
 \lim_{n\to+\infty}
 \frac1n\sum_{j=0}^{n-1}\mathbf 1_{O_m}(f^jx)
 =
 \mu_\nu(O_m)>0,\:
 \text{for every }m\ge1.
\]
Its forward orbit therefore visits every nonempty
relatively open set infinitely often.
Hence $f|_\Lambda$ is topologically transitive.

Suppose that $K\subset\Lambda$ is a nonempty compact
relative subattractor with a relatively open trapping
neighborhood $V\subset\Lambda$. Thus
$f(\overline V^{\,\Lambda})\subset V,
 \:
 K=\bigcap_{n\ge0}f^n(\overline V^{\,\Lambda}),$
where the closure is taken in $\Lambda$.
Since $K\subset V$ and $\mu_\nu$ has full support,
$\mu_\nu(V)>0$.
The inclusion $f(V)\subset V$ and invariance of
$\mu_\nu$ give
$\mu_\nu\bigl(f^{-1}(V)\setminus V\bigr)=0.$
Thus $V$ is invariant modulo $\mu_\nu$, and ergodicity
implies $\mu_\nu(V)=1$.
Consequently,
$\mu_\nu\bigl(f^n(\overline V^{\,\Lambda})\bigr)=1,
 \:n\ge0,$
so $\mu_\nu(K)=1$.
Since $K$ is closed and
$\mathrm{supp}\mu_\nu=\Lambda$, we obtain
$K=\Lambda$.

An ambient trapping neighborhood for a subattractor
contained in $\Lambda$ yields a relative trapping
neighborhood by intersection with $\Lambda$.
Thus no nonempty proper compact subset of \(\Lambda\) is an attractor of the ambient diffeomorphism.
\end{proof}

\subsection{Domination and local conjugacy classes}

When comparing parameters, we write $\phi_\alpha$ for
the corresponding embedding and use subscripts on the
base and fiber maps. In particular,
$p_\alpha=\phi_\alpha(0,0)$.
By Corollary \ref{cor:ss-domination},
$T_{\Lambda_\alpha}\mathbb T^4$ admits no nontrivial
dominated splitting for $f_\alpha$.
The same fixed-point argument applies to every positive
iterate, since
$D_{p_\alpha}(f_\alpha^m)=I$ for all $m\ge1$.

We first observe that $p_\alpha$ is the unique fixed
point in $\Lambda_\alpha$ whose derivative is the identity.
Indeed, at a fixed point $y=\phi_\alpha(z,w)$ with
$z\ne0$, the derivative induces $D_zg_\alpha$ on the
quotient $T_y\mathbb T^4/\widetilde E_\alpha(y).$
Since $m(D_zg_\alpha)>1$, the ambient derivative
cannot be the identity.
At base coordinate $0$, the fixed-point equation is
$\kappa_\alpha(w)=w$, whose only solution is $w=0$. Choose $z\ne0$ sufficiently close to $0$ and set
$x=\phi_\alpha(z,0)$.
The exact inverse germ gives
$f_\alpha^{-n}(x)=\phi_\alpha(z_{-n},0),
 \:
 z_{-n}
 =
 \frac{z}{(1+nc_\alpha|z|^\alpha)^{\frac 1\alpha}},
 \: n\ge0,$
where $c_\alpha>0$ is the coefficient of the base germ.
These points remain in the product-germ neighborhood
and converge to $p_\alpha$.
In particular,
$x\in f_\alpha^n(\phi_\alpha(N))$ for every $n\ge0$,
and hence $x\in\Lambda_\alpha$. The smooth coordinates give two-sided distance bounds
near $p_\alpha$. Since
$|z_{-n}|
 =
 \bigl(|z|^{-\alpha}+nc_\alpha\bigr)^{-\frac 1\alpha},$
we obtain
\begin{equation}\label{eq:ss-parabolic-rate}
 \lim_{n\to+\infty}
 \frac{\log\mathrm{d}
       (f_\alpha^{-n}x,p_\alpha)}
      {\log n}
 =
 -\frac1\alpha.
\end{equation}

Suppose that $f_\alpha|_{\Lambda_\alpha}$ and
$f_{\alpha'}|_{\Lambda_{\alpha'}}$ are locally
$C^1$-conjugate in the sense of
Definition \ref{Definition 1.3}, with conjugating
diffeomorphism $H$.
Then $H(p_\alpha)$ is a fixed point in
$\Lambda_{\alpha'}$, and differentiation of the
conjugacy identity gives
$D_{H(p_\alpha)}f_{\alpha'}
 =
 D_{p_\alpha}H\circ D_{p_\alpha}f_\alpha
 \circ(D_{p_\alpha}H)^{-1}
 =I.$
The preceding uniqueness statement therefore implies
$H(p_\alpha)=p_{\alpha'}$.

All backward iterates of $x$ lie in $\Lambda_\alpha$,
so the conjugacy identity gives
$H(f_\alpha^{-n}x)
 =
 f_{\alpha'}^{-n}(H(x)),
 \:n\ge0.$
Since $H$ and $H^{-1}$ are locally Lipschitz near
the corresponding fixed points,
\eqref{eq:ss-parabolic-rate} implies
\begin{equation}\label{eq:ss-conjugate-rate}
 \lim_{n\to+\infty}
 \frac{\log\mathrm{d}
       (f_{\alpha'}^{-n}(H(x)),p_{\alpha'})}
      {\log n}
 =
 -\frac1\alpha.
\end{equation}

On the other hand, this backward orbit converges to
$p_{\alpha'}$ and is nonconstant, since
$H(x)\ne p_{\alpha'}$.
For all sufficiently large $n$, write
$f_{\alpha'}^{-n}(H(x))
 =
 \phi_{\alpha'}(\zeta_n,\omega_n)$
in the product-germ neighborhood.
The fiber coordinates satisfy
$\omega_n=\kappa_{\alpha'}(\omega_{n+1}).$
If a sufficiently late $\omega_n$ were nonzero, its
norm would increase strictly under every subsequent
backward iteration, contradicting convergence to zero.
Thus $\omega_n=0$ for all sufficiently large $n$. Fix an integer $n_0\ge0$ so that the orbit remains
in this neighborhood and $\omega_n=0$
for every $n\ge n_0$.
Nonconstancy implies $\zeta_{n_0}\ne0$, and the exact
base inverse formula gives
$|\zeta_n|
 =
 \bigl(
 |\zeta_{n_0}|^{-\alpha'}
 +(n-n_0)c_{\alpha'}
 \bigr)^{-\frac1{\alpha'}},
 \:n\ge n_0.$
The same distance comparison therefore shows that
the limit in \eqref{eq:ss-conjugate-rate} equals
$-\frac 1{\alpha'}$.
Hence $\alpha=\alpha'$.

Thus distinct parameters in $(1,2)$ give systems
that are not locally $C^1$-conjugate in the sense of
Definition \ref{Definition 1.3}.
This proves item (6).

\section*{Declarations}

\noindent\textbf{Funding.} W. Lin is supported by the National Natural Science Foundation of China (No. 124B2010). X. Tian is supported by the National Natural Science Foundation of China (No. 12471182) and the Natural Science Foundation of Shanghai (No. 23ZR1405800).

\noindent\textbf{Conflict of interest.} On behalf of all authors, the corresponding author states that there is no conflict of interest.

\noindent\textbf{Data availability.} On behalf of all authors, the corresponding author states that the manuscript has no associated data.

\noindent\textbf{Use of AI tools.} ChatGPT, including GPT-6 Astra,
was used to improve the presentation of the manuscript and played
a key role in developing the proofs of Lemma \ref{lem:coverage-pesin},
Proposition \ref{prop:infinitely-many-basins} and Theorem \ref{thm:F},
as well as the argument in Remark \ref{Remark 7.8}.
The authors take full responsibility for the content of the article.


\begin{thebibliography}{99}

\bibitem{ABV00}
J. F. Alves, C. Bonatti, and M. Viana,
SRB measures for partially hyperbolic systems whose central direction is mostly expanding,
Invent. Math. \textbf{140} (2000), no. 2, 351--398.

\bibitem{AV}
M. Andersson and C. H. V\'asquez,
On mostly expanding diffeomorphisms,
Ergodic Theory Dynam. Systems \textbf{38} (2018), no. 8, 2838--2859.

\bibitem{BarreiraPesin}
L. Barreira and Ya. Pesin,
\emph{Nonuniform Hyperbolicity: Dynamics of Systems with Nonzero Lyapunov Exponents},
Encyclopedia of Mathematics and its Applications, vol. 115,
Cambridge University Press, Cambridge, 2007.

\bibitem{BOB26}
S. Ben Ovadia and D. Burguet,
Generalized $u$-Gibbs measures for $C^\infty$ diffeomorphisms,
preprint, arXiv:2506.18238v2, 2026.

\bibitem{B75}
R. Bowen,
\emph{Equilibrium States and the Ergodic Theory of Anosov Diffeomorphisms},
Lecture Notes in Mathematics, vol. 470,
Springer-Verlag, Berlin--New York, 1975.

\bibitem{Brown22}
A. Brown,
Smoothness of stable holonomies inside center-stable manifolds,
Ergodic Theory Dynam. Systems \textbf{42} (2022), no. 12, 3593--3618.

\bibitem{B24-1}
D. Burguet,
Maximal measure and entropic continuity of Lyapunov exponents for $C^r$ surface diffeomorphisms with large entropy,
Ann. Henri Poincar\'e \textbf{25} (2024), no. 2, 1485--1510.

\bibitem{B24-2}
D. Burguet,
SRB measures for $C^\infty$ surface diffeomorphisms,
Invent. Math. \textbf{235} (2024), no. 3, 1019--1062.

\bibitem{B25}
D. Burguet,
SRB measures for partially hyperbolic systems with one-dimensional center subbundles,
Trans. Amer. Math. Soc. \textbf{378} (2025), no. 10, 7137--7162.

\bibitem{BY25}
D. Burguet and D. Yang,
SRB measures for mostly expanding partially hyperbolic diffeomorphisms via the variational approach,
Int. Math. Res. Not. IMRN (2025), no. 6, Paper No. rnaf053, 23 pp.

\bibitem{CMY22}
Y. Cao, Z. Mi, and D. Yang,
On the abundance of Sinai--Ruelle--Bowen measures,
Comm. Math. Phys. \textbf{391} (2022), no. 3, 1271--1306.

\bibitem{CMZ25}
Y. Cao, Z. Mi, and R. Zou,
$C^1$ Pesin (un)stable manifolds without domination,
Int. Math. Res. Not. IMRN (2025), no. 19, Paper No. rnaf298, 16 pp.

\bibitem{CY16}
Y. Cao and D. Yang,
On Pesin's entropy formula for dominated splittings without mixed behavior,
J. Differential Equations \textbf{261} (2016), no. 7, 3964--3986.

\bibitem{CCE15}
E. Catsigeras, M. Cerminara, and H. Enrich,
The Pesin entropy formula for $C^1$ diffeomorphisms with dominated splitting,
Ergodic Theory Dynam. Systems \textbf{35} (2015), no. 3, 737--761.

\bibitem{CE11}
E. Catsigeras and H. Enrich,
SRB-like measures for $C^0$ dynamics,
Bull. Pol. Acad. Sci. Math. \textbf{59} (2011), no. 2, 151--164.

\bibitem{CZ13}
E. Catsigeras and Y. Zhao,
Observable optimal state points of subadditive potentials,
Discrete Contin. Dyn. Syst. \textbf{33} (2013), no. 4, 1375--1388.

\bibitem{CDP16}
V. Climenhaga, D. Dolgopyat, and Ya. Pesin,
Non-stationary non-uniform hyperbolicity: SRB measures for dissipative maps,
Comm. Math. Phys. \textbf{346} (2016), no. 2, 553--602.


\bibitem{CY05}
W. Cowieson and L.-S. Young,
SRB measures as zero-noise limits,
Ergodic Theory Dynam. Systems \textbf{25} (2005), no. 4, 1115--1138.

\bibitem{CYZ20}
S. Crovisier, D. Yang, and J. Zhang,
Empirical measures of partially hyperbolic attractors,
Comm. Math. Phys. \textbf{375} (2020), no. 1, 725--764.

\bibitem{DLWY12}
F. Ding, Y. Liu, S. Wang, and J. Yao,
Extending $T^p$ automorphisms over $\mathbb{R}^{p+2}$ and realizing DE attractors,
Discrete Contin. Dyn. Syst. \textbf{32} (2012), no. 5, 1639--1655.

\bibitem{DHP01}
D. Dolgopyat, H. Hu, and Ya. Pesin,
An example of a smooth hyperbolic measure with countably many ergodic components,
Appendix B to L. Barreira and Ya. Pesin, ``Lectures on Lyapunov exponents and smooth ergodic theory,''
in \emph{Smooth Ergodic Theory and Its Applications},
Proc. Sympos. Pure Math., vol. 69,
Amer. Math. Soc., Providence, RI, 2001, pp. 95--106.

\bibitem{Hirsch76}
M. W. Hirsch,
\emph{Differential Topology},
Graduate Texts in Mathematics, vol. 33,
Springer-Verlag, New York--Heidelberg, 1976.

\bibitem{HV09}
H. Hu and S. Vaienti,
Absolutely continuous invariant measures for non-uniformly expanding maps,
Ergodic Theory Dynam. Systems \textbf{29} (2009), no. 4, 1185--1215.

\bibitem{LY85-1}
F. Ledrappier and L.-S. Young,
The metric entropy of diffeomorphisms. Part I: Characterization of measures satisfying Pesin's entropy formula,
Ann. of Math. (2) \textbf{122} (1985), no. 3, 509--539.

\bibitem{LY25}
G. Liao and J. Yu,
Pesin entropy formula for extendable continuous Oseledets splittings,
Discrete Contin. Dyn. Syst. \textbf{45} (2025), no. 5, 1577--1590.

\bibitem{Liu98}
P.-D. Liu,
Pesin's entropy formula for endomorphisms,
Nagoya Math. J. \textbf{150} (1998), 197--209.

\bibitem{CMY18}
Z. Mi, Y. Cao, and D. Yang,
SRB measures for attractors with continuous invariant splittings,
Math. Z. \textbf{288} (2018), no. 1--2, 135--165.

\bibitem{MWC24}
Z. Mi, Q. Wang, and Y. Cao,
Physical measures under the hyperbolicity of pseudo-physical measures,
J. Differential Equations \textbf{400} (2024), 230--247.

\bibitem{P76}
Ya. B. Pesin,
Families of invariant manifolds corresponding to nonzero characteristic exponents,
Math. USSR-Izv. \textbf{10} (1976), no. 6, 1261--1305.


\bibitem{PughShub89}
C. Pugh and M. Shub, Ergodic attractors, Trans. Amer. Math. Soc. \textbf{312} (1989), no. 1, 1--54.



\bibitem{R76}
D. Ruelle,
A measure associated with {Axiom-A} attractors,
Amer. J. Math. \textbf{98} (1976), no. 3, 619--654.

\bibitem{RS80}
D. Ruelle and M. Shub,
Stable manifolds for maps,
in \emph{Global Theory of Dynamical Systems},
Lecture Notes in Mathematics, vol. 819,
Springer, Berlin, 1980, pp. 389--392.

\bibitem{Sch98}
S. J. Schreiber,
On growth rates of subadditive functions for semiflows,
J. Differential Equations \textbf{148} (1998), no. 2, 334--350.

\bibitem{S72}
Ya. G. Sinai,
Gibbs measures in ergodic theory,
Uspekhi Mat. Nauk \textbf{27} (1972), no. 4(166), 21--64,
English translation in Russian Math. Surveys \textbf{27} (1972), no. 4, 21--69.

\bibitem{ST12}
W. Sun and X. Tian,
Dominated splitting and Pesin's entropy formula,
Discrete Contin. Dyn. Syst. \textbf{32} (2012), no. 4, 1421--1434.

\bibitem{Viana98}
M. Viana,
Dynamics: a probabilistic and geometric perspective,
Doc. Math., Extra Vol. ICM 1998, Vol. I, 557--578.

\bibitem{Y02}
L.-S. Young,
What are {SRB} measures, and which dynamical systems have them?
J. Statist. Phys. \textbf{108} (2002), no. 5--6, 733--754.

\end{thebibliography}
\end{document}